\documentclass[11pt,reqno]{amsart}
\usepackage{amsmath,amsthm,amssymb,mathrsfs,stmaryrd,color,xcolor,iftex}
\usepackage[all]{xy}
\usepackage{url}
\usepackage{geometry}
\usepackage{amssymb}
\usepackage{amsthm}
\usepackage[utf8]{inputenc}
\usepackage{pst-node}
\usepackage{tikz-cd} 
\usepackage{hyperref}
\newtheorem{theorem}{Theorem}[section]
\newtheorem{lemma}[theorem]{Lemma}
\newtheorem{proposition}[theorem]{Proposition}
\newtheorem{corollary}[theorem]{Corollary}

\usepackage[T1]{fontenc}

\usepackage{csquotes}
\usepackage[english]{babel}
\usepackage[utf8]{inputenc}
\usepackage[style=alphabetic,backend=biber]{biblatex}
\theoremstyle{definition}
\newtheorem{definition}{Definition}[section]
\theoremstyle{remark}
\newtheorem{remark}{Remark}[section]
\theoremstyle{example}

\usepackage{amsmath}
\usepackage{amsfonts}
\usepackage{tikz-cd}

\def\SL{\mathrm{SL}}
\def\HT{\mathrm{HT}}
\def\Act{\mathrm{Act}}
\def\U{\mathrm{U}}
\def\Igs{\mathrm{Igs}}
\def\bb{\mathbb}
\def\Detale{\mathrm{D}_{\text{\'et}}}
\def\mf{\mathfrak}
\def\ra{\rightarrow}
\def\la{\leftarrow}

\def\lim{\mathop{\rm lim}\nolimits}
\def\colim{\mathop{\rm colim}\nolimits}
\def\Spec{\mathop{\rm Spec}}
\def\Spa{\mathop{\rm Spa}}
\def\Spd{\mathop{\rm Spd}}

\def\Sh{\mathop{\textit{Sh}}\nolimits}

\def\Bun{\mathrm{Bun}}

\def\GSpin{\mathrm{GSpin}}

\def\Frob{\mathrm{Frob}}
\def\Hck{\mathrm{Hck}}

\def\Perf{\mathrm{Perf}}
\def\det{\mathrm{det}}

\def\dim{\mathrm{dim}}

\def\End{\mathrm{End}}
\def\Sht{\mathrm{Sht}}
\def\Rep{\mathrm{Rep}}
\def\D{\mathrm{D}}

\def\tr{\mathrm{tr}}

\def\GL{\mathrm{GL}}

\def\GSO{\mathrm{GSO}}

\def\LLC{\mathrm{LLC}}
\def\GU{\mathrm{GU}}
\def\WD{\mathrm{WD}}
\def\Gal{\mathrm{Gal}}
\def\Sh{\mathrm{Sh}}
\newcommand{\Dlis}{\mathrm{D}_{\mathrm{lis}}}
\newcommand{\ul}{\underline}

\newcommand{\Pl}{\mathrm{Pl}}

\newcommand{\mc}{\mathcal}

\newcommand{\Res}{\mathrm{Res}}
\newcommand{\GSp}{\mathrm{GSp}}
\newcommand{\Sp}{\mathrm{Sp}}
\newcommand{\SU}{\mathrm{SU}}
\newcommand{\SO}{\mathrm{SO}}

\newcommand{\ol}{\overline}

\title{Compatibility of the Fargues-Scholze and Gan-Takeda Local Langlands}
\author{Linus Hamann}
\usepackage[english]{babel}
\usepackage[utf8]{inputenc}
\begin{document}
\begin{abstract}
Given a prime $p$, a finite extension $L/\mathbb{Q}_{p}$, a connected $p$-adic reductive group $G/L$, and a smooth irreducible representation $\pi$ of $G(L)$, Fargues-Scholze \cite{FS} recently attached a semisimple $L$-parameter to such $\pi$, giving a general candidate for the local Langlands correspondence. It is natural to ask whether this construction is compatible with known instances of the correspondence after semisimplification. For $G = \GL_{n}$ and its inner forms,  Fargues-Scholze and Hansen-Kaletha-Weinstein \cite{KW} showed that the correspondence is compatible with the correspondence of Harris-Taylor/Henniart \cite{He,HT}. We verify a similar compatibility for $G = \GSp_{4}$ and its unique non-split inner form $G = \GU_{2}(D)$, where $D$ is the quaternion division algebra over $L$, assuming that $L/\mathbb{Q}_{p}$ is unramified and $p > 2$. In this case, the local Langlands correspondence has been constructed by Gan-Takeda and Gan-Tantono \cite{GT1,GT2}. Analogous to the case of $\GL_{n}$ and its inner forms, this compatibility is proven by describing the Weil group action on the cohomology of a local Shimura variety associated to $\GSp_{4}$, using basic uniformization of abelian type Shimura varieties due to Shen \cite{She}, combined with various global results of Kret-Shin \cite{KS} and Sorensen \cite{So} on Galois representations in the cohomology of global Shimura varieties associated to inner forms of $\GSp_{4}$ over a totally real field. After showing the parameters are the same, we apply some ideas from the geometry of the Fargues-Scholze construction explored recently by Hansen \cite{Han}. This allows us to give a more precise description of the cohomology of this local Shimura variety, verifying a strong form of the Kottwitz conjecture in the process. 
\end{abstract} 
\maketitle
\tableofcontents
\section{Introduction}
\subsection{Background and Main Theorems}
Fix distinct primes $\ell \neq p$, let $\mathbb{Q}_{p}$ denote the $p$-adic numbers, and let $G/\mathbb{Q}_{p}$ be a connected reductive group. Set $\mathbb{C}_{p} := \hat{\overline{\mathbb{Q}}}_{p}$ to be the completion of the algebraic closure of $\mathbb{Q}_{p}$. We fix an isomorphism $i: \overline{\mathbb{Q}}_{\ell} \xrightarrow{\simeq} \mathbb{C}$. Let $W_{\mathbb{Q}_{p}}$ be the Weil group of $\mathbb{Q}_{p}$ and set $\hat{G}$ to be the reductive group over $\overline{\mathbb{Q}}_{\ell}$ with root datum dual to $G$. Let $Q$ be the finite quotient through which $W_{\mathbb{Q}_{p}}$ acts on $\hat{G}$. We define the $L$-group $^LG := Q \ltimes \hat{G}$. We let $\Pi(G)$ denote the set of isomorphism classes of smooth irreducible representations of the $p$-adic group $G(\mathbb{Q}_{p})$, and let $\Phi(G)$ denote the set of $L$-parameters, i.e the set of conjugacy classes of homomorphisms
\[ \phi: W_{\mathbb{Q}_{p}} \times \SL_{2}(\mathbb{C}) \rightarrow \phantom{}^{L}G(\mathbb{C}) \]
where $\SL_{2}(\mathbb{C})$ acts via an algebraic representation and $W_{\mathbb{Q}_{p}}$ acts via a continuous semisimple homomorphism in a way that commutes with the natural projection $\phantom{}^{L}G(\mathbb{C}) \rightarrow Q$, where $\phantom{}^{L}G(\mathbb{C})$ is endowed with the discrete topology. The local Langlands correspondence is a conjectural map 
\[ \LLC_{G}: \Pi(G) \rightarrow \Phi(G) \]
\[ \pi \mapsto \phi_{\pi} \]
that builds a bridge between $L$-parameters and the smooth irreducible representations of $G(\mathbb{Q}_{p})$. Conjecturally (under some additional constraints on $\Phi(G)$ if $G$ is not quasi-split), these maps should be surjective with finite fibers called $L$-packets and satisfy various properties such as compatibility with products, maps of $L$-groups, character twists, as well as $L$, $\epsilon$, and $\gamma$-factors. Moreover, one expects that the correspondence is uniquely characterized by some such finite list of properties.
\\\\
In general, the existence and uniqueness of such a correspondence is completely unknown. However, very recently, Fargues and Scholze \cite{FS}, using the action of the excursion algebra on $\ell$-adic sheaves on the moduli space of $G$-bundles on the Fargues-Fontaine curve, were able to construct a completely general candidate, analogous to the work of V. Lafforgue in the function field setting \cite{VL}. Namely, they construct a map
\[ \LLC_{G}^{\mathrm{FS}}: \Pi(G) \rightarrow \Phi^{\mathrm{ss}}(G) \]
\[ \pi \mapsto \phi_{\pi}^{\mathrm{FS}}, \]
where $\Phi^{\mathrm{ss}}(G)$ denotes the set of conjugacy classes of continuous semisimple maps
\[ \phi: W_{\mathbb{Q}_{p}} \rightarrow \phantom{}^{L}G(\overline{\mathbb{Q}}_{\ell}) \]
that commute with the projection $\phantom{}^{L}G(\overline{\mathbb{Q}}_{\ell}) \rightarrow Q$ as above. Fargues and Scholze showed that their map has several good properties such as compatibility with parabolic induction; however, one would also like to check that this correspondence agrees with known instances of the local Langlands correspondence. Precisely, given a candidate for the local Langlands correspondence
\[ \LLC_{G}: \Pi(G) \rightarrow \Phi(G) \]
\[ \pi \mapsto \phi_{\pi} \]
we expect a commutative diagram of the form
\begin{equation*}
\begin{tikzcd}[ampersand replacement=\&]
            \Pi(G)  \ar[rr, "\LLC_{G}"] \arrow[drr,"\LLC^{\mathrm{FS}}_{G}"] \& \&   \Phi(G) \ar[d,"(-)^{\mathrm{ss}}"] \\
            \& \& \Phi^{\mathrm{ss}}(G)
        \end{tikzcd}
\end{equation*} 
where the semisimplification map $(-)^{\mathrm{ss}}$ precomposes an $L$-parameter $\phi \in \Phi(G)$ with the map  \[g \in W_{\mathbb{Q}_{p}} \mapsto (g,\begin{pmatrix}
|g|^{\frac{1}{2}} & 0 \\
0 & |g|^{\frac{-1}{2}}  
\end{pmatrix}) \in W_{\mathbb{Q}_{p}} \times \SL_{2}(\mathbb{C}) \]  
and then applies the fixed isomorphism $i^{-1}: \mathbb{C} \xrightarrow{\simeq} \overline{\mathbb{Q}}_{\ell}$. Here $|\cdot|: W_{\mathbb{Q}_{p}} \rightarrow W_{\mathbb{Q}_{p}}^{ab} \simeq \mathbb{Q}_{p}^{*} \rightarrow \mathbb{C}^{*}$ is the norm character. We make the following definition.
\begin{definition}
For $\pi \in \Pi(G)$, we say that a local Langlands correspondence $\LLC_{G}$ is compatible with the Fargues-Scholze local Langlands correspondence if we have an equality: $\phi^{\mathrm{FS}}_{\pi} = \phi_{\pi}^{\mathrm{ss}}$, as conjugacy classes of semi-simple L-parameters.
\end{definition}
For $\GL_{n}$, the local Langlands correspondence was constructed by Harris-Taylor/Henniart \cite{He, HT} and is uniquely characterized by the preservation of $L$, $\epsilon$, and $\gamma$-factors. In this case, compatibility with the Fargues-Scholze local Langlands correspondence follows from the description of the cohomology of the Lubin-Tate and Drinfeld towers proven in \cite{HT} and was verified by Fargues and Scholze \cite[Theorem~I.9.6 (ix)]{FS}. The main goal of this note is to extend compatibility of the correspondence to $\GSp_{4}$ and its unique non-split inner form. To this end, we now fix a finite extension $L/\mathbb{Q}_{p}$ and set $G$ to be $\mathrm{Res}_{L/\mathbb{Q}_{p}}\GSp_{4}$ and $J$ to be its unique non-split inner form $\mathrm{Res}_{L/\mathbb{Q}_{p}}\GU_{2}(D)$, where $D/L$ is the quaternion division algebra. In this case, the local Langlands correspondence has been constructed by Gan-Takeda and Gan-Tantono, respectively \cite{GT1, GT2}. It is constructed from the local Langlands correspondence for $\GL_{n}$ and theta lifting and admits a similar unique characterization in terms of the preservation of $L$, $\epsilon$, and $\gamma$ factors. We note that we can and do identify $\Phi(G)$ and $\Phi(J)$ with a subset of homomorphisms:
\[ \phi: W_{L} \times \SL_{2}(\mathbb{C}) \rightarrow \hat{G}(\mathbb{C}) = \GSpin_{5}(\mathbb{C}) \simeq \GSp_{4}(\mathbb{C}). \]
This allows us to introduce a bit of terminology. Namely, we say that a parameter $\phi$ in $\Phi(G)$ or $\Phi(J)$ is supercuspidal if the  $\SL_{2}(\mathbb{C})$-factor acts trivially and $\phi$ does not factor through any proper Levi subgroup of $\GSp_{4}$. This terminology is justified by the fact that this is precisely the case when the $L$-packets over $\phi$ with respect to $G$ and $J$ contain only supercuspidal representations. In what follows, we will often abuse notation and drop the superscript $(-)^{\mathrm{ss}}$ when speaking about such parameters, as in this case it merely corresponds to forgetting the trivially acting $\SL_{2}(\mathbb{C})$-factor and applying the isomorphism $i^{-1}$. We now come to our main theorem.
\begin{theorem}
The following is true.
\begin{enumerate}
    \item For any $\pi \in \Pi(G)$ (resp. $\rho \in \Pi(J)$) such that the Gan-Takeda (resp. Gan-Tantono) parameter is not supercuspidal, we have that the Gan-Takeda (resp. Gan-Tantono) correspondence is compatible with the Fargues-Scholze correspondence. 
    \item If $L/\mathbb{Q}_{p}$ is unramified and $p > 2$, we have, for all $\pi \in \Pi(G)$ (resp. $\rho \in \Pi(J)$) such that the Gan-Takeda (resp. Gan-Tantono) parameter is supercuspidal, that the Gan-Takeda (resp. Gan-Tantono) correspondence is compatible with the Fargues-Scholze correspondence. 
\end{enumerate}
\end{theorem}
\begin{remark}
As will be explained more below, the restrictions in the case where the parameter is supercuspidal are necessary to apply basic uniformization of the generic fiber of abelian type Shimura varieties due to Shen \cite{She}. If one were not to impose this assumption, the relevant Shimura varieties would have bad reduction at $p$, which, to the best of our knowledge, prevents the methods of Shen from working. In particular, if one could establish the expected description of the basic locus in the sense of the isomorphism (2) of Definition 4.1, for Shimura varieties associated to the group $\mathrm{Res}_{F/\mathbb{Q}}\mathbf{G}$, where $\mathbf{G}$ is an inner form of $\GSp_{4}$ over a totally real field $F$ with an inert prime $p$ such that $F_{p} \simeq L$ for $L$ any extension, then our result would hold in complete generality.  
\end{remark}
As mentioned above, the proof of compatibility for $\GL_{n}$ uses the results of Harris-Taylor \cite{HT} on the cohomology of the Lubin-Tate/Drinfeld Towers. In particular, if one looks at the rigid generic fiber of the Lubin-Tate tower
\[ \lim_{m \rightarrow \infty} \mathrm{LT}_{n,m,\breve{\mathbb{Q}}_{p}} \]
a tower of ($n - 1$)-dimensional rigid spaces over $\breve{\mathbb{Q}}_{p}$, for fixed $n \geq 1$ and varying $m \geq 1$, where $\breve{\mathbb{Q}}_{p}$ denotes the completion of the maximal unramified extension of $\mathbb{Q}_{p}$. The cohomology of this tower 
\[ R\Gamma_{c}(\mathrm{LT}_{n,\infty},\overline{\mathbb{Q}}_{\ell}) := \colim_{m \rightarrow \infty} R\Gamma_{c}(\mathrm{LT}_{n,m,\mathbb{C}_{p}},\overline{\mathbb{Q}}_{\ell}) \]
base changed to $\mathbb{C}_{p}$ carries commuting actions of $\GL_{n}(\mathbb{Q}_{p})$ and $D^{*}_{\frac{1}{n}}$, the units in the division algebra over $\mathbb{Q}_{p}$ of invariant $\frac{1}{n}$, as well as an action of the Weil group $W_{\mathbb{Q}_{p}}$. In particular, given $\pi \in \Pi(\GL_{n})$ (resp. $\rho \in \Pi(D^{*}_{\frac{1}{n}})$), we can consider the complexes
\[ R\Gamma_{c}(\mathrm{LT}_{n,\infty},\overline{\mathbb{Q}}_{\ell})[\pi] := R\Gamma_{c}(\mathrm{LT}_{n,\infty},\overline{\mathbb{Q}}_{\ell}) \otimes^{\mathbb{L}}_{\mathcal{H}(\GL_{n})} \pi \] 
and
\[ R\Gamma_{c}(\mathrm{LT}_{n,\infty},\overline{\mathbb{Q}}_{\ell})[\rho] := R\Gamma_{c}(\mathrm{LT}_{n,\infty},\overline{\mathbb{Q}}_{\ell}) \otimes^{\mathbb{L}}_{\mathcal{H}(D^*_{\frac{1}{n}})} \rho \] 
where $\mathcal{H}(\GL_{n}) := C^{\infty}_{c}(\GL_{n}(\mathbb{Q}_{p}),\overline{\mathbb{Q}}_{\ell})$ (resp. $\mathcal{H}(D^{*}_{\frac{1}{n}})$) is the usual smooth Hecke algebra of $\GL_{n}(\bb{Q}_{p})$ (resp. $D^{*}_{\frac{1}{n}}$). Then the key result of Harris-Taylor and later refined by Boyer is as follows.
\begin{theorem}{\cite{HT,Boy}} 
Fix $\pi \in \Pi(\GL_{n})$, a supercuspidal representation of $\GL_{n}(\mathbb{Q}_{p})$, let
\[ \mathrm{JL}: \Pi(D^{*}_{\frac{1}{n}}) \rightarrow \Pi(\GL_{n}(\mathbb{Q}_{p})) \]
be the map defined by the Jacquet-Langlands correspondence and $\rho := \mathrm{JL}^{-1}(\pi) \in \Pi(D_{\frac{1}{n}}^{*})$ a Jacquet-Langlands lift of $\pi$. Then the complexes $ R\Gamma_{c}(\mathrm{LT}_{n,\infty},\overline{\mathbb{Q}}_{\ell})[\pi]$ and $R\Gamma_{c}(\mathrm{LT}_{n,\infty},\overline{\mathbb{Q}}_{\ell})[\rho]$ are concentrated in middle degree $n - 1$. The middle degree cohomology of $ R\Gamma_{c}(\mathrm{LT}_{n,\infty},\overline{\mathbb{Q}}_{\ell})[\pi]$ is isomorphic to
\[ \rho \boxtimes \phi_{\pi}^{\vee} \otimes |\cdot|^{(1 - n)/2} \]
as a $D^{*}_{\frac{1}{n}} \times W_{\mathbb{Q}_{p}}$ representation. Similarly, the middle degree cohomology of $ R\Gamma_{c}(\mathrm{LT}_{n,\infty},\overline{\mathbb{Q}}_{\ell})[\rho]$ is isomorphic to 
\[ \pi \boxtimes \phi_{\pi} \otimes |\cdot|^{(1 - n)/2} \] 
where $\phi_{\pi} \in \Phi^{\mathrm{ss}}(\GL_{n})$ is the (semisimplified) L-parameter associated to $\pi$ by Harris-Taylor. 
\end{theorem} 
To see why this result is relevant for compatibility, we invoke the observation, due to Scholze-Weinstein \cite{SW1, SW2}, that the Lubin-Tate tower at infinite level
\[ \mathrm{LT}_{n,\infty} := \lim_{m \rightarrow \infty} \mathrm{LT}_{n,m,\breve{\mathbb{Q}}_{p}} \]
is representable by a space admitting a moduli interpretation as a space of shtukas over the Fargues-Fontaine curve; namely, the space denoted $\Sht(\GL_{n},b,\mu)_{\infty}$ in the notation of \cite{SW2}, where $b \in B(\GL_{n})$ is an element in the Kottwitz set of $\GL_{n}$ corresponding to a rank $n$ isocrystal of slope $\frac{1}{n}$ and $\mu = (1,0,\ldots,0,0)$ is a dominant cocharacter of $\GL_{n}$. If $X$ denotes the Fargues-Fontaine curve, then this parametrizes modifications
\[ \mathcal{O}_{X}(-\frac{1}{n}) \rightarrow \mathcal{O}_{X}^{n}  \]
of type $(1,0,\ldots,0,0)$ (i.e. this map is an embedding with cokernel a length $1$ torsion sheaf on $X$), where $\mathcal{O}_{X}(-\frac{1}{n})$ is the unique semistable rank $n$ vector bundle on $X$ of slope $-\frac{1}{n}$. This interpretation allows one to relate the complex  
$R\Gamma_{c}(\mathrm{LT}_{n,\infty},\overline{\mathbb{Q}}_{\ell})[\pi]$ to the action of a Hecke operator $T_{\mu^{-1}}$ on $\Bun_{G}$, acting on a sheaf $\mathcal{F}_{\pi}$ constructed from the supercuspidal representation $\pi$, where $\mu^{-1} = (0,0,\ldots,0,-1)$ is a dominant inverse of $\mu$. The Fargues-Scholze parameter of $\pi$ is built from the action of the excursion algebra on the sheaf $\mathcal{F}_{\pi}$, which in turn is built from Hecke operators together with a factorization structure coming from geometric Satake. It thus is reasonable to expect that the cohomology group  $R\Gamma_{c}(\mathrm{LT}_{n,\infty},\overline{\mathbb{Q}}_{\ell})[\pi]$ should have $W_{\mathbb{Q}_{p}}$-action given by the Fargues-Scholze parameter $\phi_{\pi}^{\mathrm{FS}}: W_{\mathbb{Q}_{p}} \rightarrow \phantom{}^{L}\GL_{n}(\overline{\mathbb{Q}}_{\ell}) \simeq \GL_{n}(\overline{\mathbb{Q}}_{\ell})$ composed with the highest weight representation of $\phantom{}^{L}\GL_{n}(\overline{\mathbb{Q}}_{\ell})$ corresponding to the dominant cocharacter $\mu^{-1}$. However, this is just the dual of the standard representation of $\GL_{n}(\overline{\mathbb{Q}}_{\ell})$. Thus, using Theorem 1.2, we can see that
\[ \phi_{\pi}^{\vee} = (\phi_{\pi}^{\mathrm{FS}})^{\vee} \]
where the twist by the norm-character $|\cdot|^{(1 - n)/2}$ is cancelled out by a perverse normalization (also related to the middle degree being the relevant one) in the definition of the Hecke operator $T_{\mu^{-1}}$. This implies compatibility for supercuspidal $\pi$, which, by using compatibility of the Fargues-Scholze/semi-simplified Harris-Taylor correspondence with parabolic induction, is enough to conclude the general case. In a similar fashion, using the description of the $\rho$-isotypic part one can prove compatibility for the inner form $D^{*}_{\frac{1}{n}}$. 
\\\\
One may expect, given the above sketch of compatibility for $\GL_{n}$, that, to prove Theorem 1.1, one must similarly provide a description of the cohomology of the $\pi \in \Pi(G)$ (resp. $\rho \in \Pi(J)$)-isotypic part of a local Shimura variety/shtuka space at infinite level associated to $G$. In the case where the associated $L$-parameters are supercuspidal, this is the content of the Kottwitz conjecture \cite[Conjecture~7.4]{RV}.  To this end, we consider the cohomology of a shtuka space associated to the group $G = \mathrm{Res}_{L/\mathbb{Q}_{p}}\GSp_{4}$. Namely, the space denoted $\Sht(G,b,\mu)_{\infty}$, where $\mu$ is the Siegel cocharacter (i.e. it is given by the Siegel cocharacter on one embedding $L \hookrightarrow \ol{\bb{Q}}_{p}$ and trivial on the other factors)  and $b \in B(G)$ is a basic element in the Kottwitz set of $G$, corresponding to a rank $4$ isocrystal with a polarization and automorphism group equal to $J = \mathrm{Res}_{L/\mathbb{Q}_{p}}(\GU_{2}(D))$. This space carries a commuting $G(\mathbb{Q}_{p})$ and $J(\mathbb{Q}_{p})$ action. The quotients
\[ \Sht(G,b,\mu)_{K} := \Sht(G,b,\mu)_{\infty}/\underline{K} \]
for varying compact open $K \subset G(\mathbb{Q}_{p})$ are, as before, representable by rigid analytic varieties over $\Spa(\breve{L})$ of dimension $3$, where $\breve{L} := L\breve{\mathbb{Q}}_{p}$. They define a tower of local Shimura varieties in the sense of Rapoport-Viehmann \cite{RV}, which uniformize the basic locus of certain global Shimura varieties analogous to the Lubin-Tate case described above. Letting $\Sht(G,b,\mu)_{K,\mathbb{C}_{p}}$ be the base-change of these spaces to $\mathbb{C}_{p}$, we can then consider the analog of the complexes described above 
\[ R\Gamma_{c}(G,b,\mu) := \colim_{K \rightarrow 1} R\Gamma_{c}(\Sht(G,b,\mu)_{K,\mathbb{C}_{p}},\overline{\mathbb{Q}}_{\ell}). \]
This complex is concentrated in degrees $0 \leq i \leq 6 = 2\dim(\Sht(G,b,\mu)_{K})$ and admits an action of $G(\mathbb{Q}_{p}) \times J(\mathbb{Q}_{p}) \times W_{L}$. This allows one to consider the $\rho$ and $\pi$-isotypic parts, i.e. we set
\[ R\Gamma_{c}(G,b,\mu)[\rho] := R\Gamma_{c}(G,b,\mu) \otimes^{\mathbb{L}}_{\mathcal{H}(J)}   \rho \]
and 
\[ R\Gamma_{c}(G,b,\mu)[\pi] := R\Gamma_{c}(G,b,\mu) \otimes^{\mathbb{L}}_{\mathcal{H}(G)}   \pi, \]
where $\mathcal{H}(G)$ (resp. $\mathcal{H}(J)$) are the usual smooth Hecke algebra of $G$ (resp. $J$). To deduce compatibility, one needs to realize the (semi-simplified) L-parameter $\phi_{\pi}$ (resp. $\phi_{\rho}$) of Gan-Takeda (resp. Gan-Tantono) in these two cohomology groups. We will sketch how to do this in the next section using uniformization and global methods. Interestingly,  after knowing compatibility, one can use ideas from the geometry of the Fargues-Scholze construction to provide a more precise description of the complexes $R\Gamma_{c}(G,b,\mu)[\rho]$ and $R\Gamma_{c}(G,b,\mu)[\pi]$ for the unique basic element $b \in B(G,\mu)$. Namely, recent work of Hansen \cite{Han} allows us to deduce that if $\phi_{\rho}$ (resp. $\phi_{\pi}$) is supercuspidal the complexes $R\Gamma_{c}(G,b,\mu)[\rho]$ (resp. $R\Gamma_{c}(G,b,\mu)[\pi]$) are concentrated in middle degree $3$. It then follows from work of Hansen-Kaletha-Weinstein \cite{KW} on a weakening of the Kottwitz conjecture and work of Fargues-Scholze \cite[Section~X.2]{FS} describing the Hecke action on objects with supercuspidal Fargues-Scholze parameter that one can actually deduce a strong form of the Kottwitz conjecture for these representations. To state this result, we first note that, if we are given a supercuspidal parameter $\phi: W_{L} \rightarrow \GSp_{4}(\overline{\mathbb{Q}}_{\ell})$, even though the parameter $\phi$ is irreducible, its composition with the standard embedding $\mathrm{std}: \GSp_{4}(\overline{\mathbb{Q}}_{\ell}) \hookrightarrow \GL_{4}(\overline{\mathbb{Q}}_{\ell})$ may not be. In particular, the size of the $L$-packets $\Pi_{\phi}(G) := \LLC_{G}^{-1}(\phi)$ and $\Pi_{\phi}(J) := \LLC_{J}^{-1}(\phi)$ over $\phi$ are governed by this.
\begin{enumerate}
    \item (stable) $\mathrm{std}\circ \phi$ is irreducible. In this case, the $L$-packets each contain one supercuspidal member.
    \item (endoscopic) $\mathrm{std}\circ \phi \simeq \phi_{1} \oplus \phi_{2}$, where $\phi_{i}: W_{L} \rightarrow \GL_{2}(\bar{\mathbb{Q}}_{\ell})$ are distinct irreducible $2$-dimensional representations with $\det(\phi_{1}) = \det(\phi_{2})$. In this case, the $L$-packets over $\phi$ each contain two supercuspidal members.
\end{enumerate}
This allows us to state our main consequence of Theorem 1.1, which (almost) verifies the strong form of the Kottwitz conjecture for $\GSp_{4}/L$ and $\GU_{2}(D)/L$.
\begin{theorem}
Let $L/\mathbb{Q}_{p}$ be an unramified extension with $p > 2$. Let $\pi$ (resp. $\rho$) be members of the $L$-packet over a supercuspidal parameter $\phi: W_{L} \rightarrow \GSp_{4}(\overline{\mathbb{Q}}_{\ell})$ as above. Then the complexes
\[ R\Gamma_{c}(G,b,\mu)[\pi] \]
and 
\[ R\Gamma_{c}(G,b,\mu)[\rho] \]
are concentrated in middle degree $3$. 
\begin{enumerate}
\item If $\phi$ is stable supercuspidal, with singleton $L$-packets $\{\pi\} = \Pi_{\phi}(G)$ and $\{\rho\} = \Pi_{\phi}(J)$, then the cohomology of $R\Gamma_{c}(G,b,\mu)[\pi]$ in middle degree is isomorphic to 
\[ \rho \boxtimes (\mathrm{std}\circ \phi)^{\vee} \otimes |\cdot|^{-3/2}   \]
as a $J(\mathbb{Q}_{p}) \times W_{L}$-module, and the cohomology of $R\Gamma_{c}(G,b,\mu)[\rho]$ in middle degree is isomorphic to
\[ \pi \boxtimes \mathrm{std}\circ \phi \otimes |\cdot|^{-3/2} \]
as a $G(\mathbb{Q}_{p}) \times W_{L}$-module.
\item If $\phi$ is an endoscopic parameter, with $L$-packets $\Pi_{\phi}(G) = \{\pi^{+},\pi^{-}\}$ and $\Pi_{\phi}(J) = \{\rho_{1},\rho_{2}\}$\footnote{For an explanation of the notation, see the discussion at the end of section 2.2.}, the cohomology of $R\Gamma_{c}(G,b,\mu)[\pi]$ in middle degree is isomorphic to
\[ \rho_{1} \boxtimes \phi_{1}^{\vee} \otimes |\cdot|^{-3/2} \oplus \rho_{2} \boxtimes \phi_{2}^{\vee} \otimes |\cdot|^{-3/2} \]
or
\[ \rho_{1} \boxtimes \phi_{2}^{\vee} \otimes |\cdot|^{-3/2} \oplus \rho_{2} \boxtimes \phi_{1}^{\vee} \otimes |\cdot|^{-3/2} \]
as a $J(\mathbb{Q}_{p}) \times W_{L}$-module. Similarly, the cohomology of $R\Gamma_{c}(G,b,\mu)[\rho]$ in middle degree is isomorphic to
\[  \pi^{+} \boxtimes \phi_{1} \otimes |\cdot|^{-3/2} \oplus \pi^{-} \boxtimes \phi_{2} \otimes |\cdot|^{-3/2} \]
or 
\[  \pi^{+} \boxtimes \phi_{2} \otimes |\cdot|^{-3/2} \oplus \pi^{-} \boxtimes \phi_{1} \otimes |\cdot|^{-3/2} \]
as a $G(\mathbb{Q}_{p}) \times W_{L}$-module. Here we write $\mathrm{std}\circ \phi_{\rho} \simeq \phi_{1} \oplus \phi_{2}$, with $\phi_{i}$ distinct irreducible $2$-dimensional representations of $W_{L}$ and $\det(\phi_{1}) = \det(\phi_{2})$. 
\\\\
Moreover, both possibilities for the cohomology of $R\Gamma_{c}(G,b,\mu)[\rho]$ (resp. $R\Gamma_{c}(G,b,\mu)[\pi]$) in the endoscopic case occur for some choice of representation $\rho \in \Pi_{\phi}(J)$ (resp. $\pi \in \Pi_{\phi}(G)$). In particular, knowing the precise form of either $R\Gamma_{c}(G,b,\mu)[\rho]$ or $R\Gamma_{c}(G,b,\mu)[\pi]$ for some $\rho \in \Pi_{\phi}(J)$ or $\pi \in \Pi_{\phi}(G)$ determines the precise form of the cohomology in all other cases. 
\end{enumerate}
\end{theorem}
\begin{remark}
\begin{enumerate}
\item Results of this form when $L = \mathbb{Q}_{p}$ have also been announced by Ito-Mieda \cite{IM}.
\item If one knew Arthur's multiplicity formula for inner forms of $\GSp_{4}$ over totally real fields, one should be able to determine the cohomology in the endoscopic case more precisely, using basic uniformization and the more precise description of the cohomology of the global Shimura variety this multiplicity formula would provide (see for example \cite[Section~3.2]{Ng} for this kind of analysis in the case of unitary groups.). However, to our knowledge the multiplicity formula is unknown in this case. In the case that $L = \mathbb{Q}_{p}$, one can apply what is known about the multiplicity formula for $\GSp_{4}/\mathbb{Q}$ \cite{Ar,GeTa}. This is carried out by Ito-Mieda \cite{IM}. The correct answer, for the $\rho$-isotypic part, should be that, if $\rho = \rho_{1}$, we are in the first case, and if $\rho = \rho_{2}$, we are in the second case. Similarly, for the $\pi$-isotypic part, if $\pi = \pi^{+}$ is the unique generic member of the $L$-packet for a fixed choice of Whittaker datum, we should be in the first case and, if $\pi = \pi^{-}$, we should be in the second case. It might also be possible to show this using a weaker argument. Our analysis reduces us to checking that $R\Gamma_{c}(G,b,\mu)[\rho_{1}]$ admits a sub-quotient isomorphic to $\pi^{+} \boxtimes \phi_{1} \otimes |\cdot|^{-3/2}$, which may be possible to show through basic uniformization and a small global argument.  
\item We hope that the perspective we take on the Kottwitz conjecture in this paper will help provide further advancements in our knowledge of the cohomology of local Shimura varieties. In particular, by invoking the use of these very general geometric tools from the Fargues-Scholze construction, we require global input only to show compatibility, which, as we will see in the next section, requires substantially less than the input needed to determine the precise form of the cohomology such as a multiplicity formula for the automorphic spectrum.   
\end{enumerate}
\end{remark}
We will now conclude the introduction by providing a sketch of the proof of Theorem 1.1.
\subsection{Proof Sketch of the Main Theorems}
As before, we set $G = \mathrm{Res}_{L/\mathbb{Q}_{p}}\GSp_{4}$ and $J = \mathrm{Res}_{L/\mathbb{Q}_{p}}\GU_{2}(D)$.
Similar to the case of $\GL_{n}$, the idea behind proving compatibility for $G$ and $J$ is to use the compatibility of the Fargues-Scholze local Langlands correspondence with parabolic induction to reduce to the case where $\pi$ is a supercuspidal representation. However, this is a little bit more subtle than the case of $\GL_{n}$. Unlike $\GL_{n}$, the local Langlands correspondence for these groups is not a bijection. As seen before, the $L$-packets can be either of size $1$ or $2$. Given an $L$-parameter $\phi: W_{L} \times \SL_{2}(\overline{\mathbb{Q}}_{\ell}) \rightarrow \GSp_{4}(\overline{\mathbb{Q}}_{\ell})$, there are three distinct possibilities. 
\begin{enumerate}
    \item The $L$-packets $\Pi_{\phi}(G)$ and $\Pi_{\phi}(J)$ do not contain any supercuspidal representations. 
    \item The $L$-packets $\Pi_{\phi}(G)$ and $\Pi_{\phi}(J)$ contain a mix of supercuspidal and non-supercuspidal representations.
    \item The $L$-packets $\Pi_{\phi}(G)$ and $\Pi_{\phi}(J)$ contain only supercuspidals.
\end{enumerate}
Case $(1)$ is straightforward. Since compatibility is known for $\GL_{n}$ and its inner forms and any proper Levi subgroup of $G$ (resp. $J$) is a product of such groups, it follows from compatibility of the Fargues-Scholze correspondence with parabolic induction and products that the correspondences are compatible for any representation lying in such an $L$-packet.
\\\\
Case $(2)$ is a bit more subtle, here $\phi^{\mathrm{ss}}$ factors through a Levi subgroup of $\GSp_{4}(\overline{\mathbb{Q}}_{\ell})$, but $\phi$ itself does not. In particular, the restriction to the $\SL_{2}$ factor of $\phi$ is non-trivial. In this case, we can write $\Pi_{\phi}(G) = \{\pi_{disc},\pi_{sc}\}$ (resp. $\Pi_{\phi}(J) := \{\rho_{disc},\rho_{sc}\}$ or $\Pi_{\phi}(J) = \{\rho^{1}_{disc},\rho^{2}_{disc}\}$, depending on whether the parameter is of Saito-Kurokawa or Howe-Piatetski--Shapiro type), where $\pi_{disc}$ (resp. $\rho^{1}_{disc}$, $\rho^{2}_{disc}$, and $\rho_{disc}$) are non-supercuspidal (essentially) discrete series representation of $G$ (resp. $J$), and $\pi_{sc}$ (resp. $\rho_{sc}$) is a supercuspidal representation of $G(\mathbb{Q}_{p})$ (resp. $J(\mathbb{Q}_{p})$) if the parameter is endoscopic. In the stable case (the Soudry type parameters), we can write $\Pi_{\phi}(G) = \{\pi'_{disc}\}$ (resp. $\Pi_{\phi}(J) = \{\rho'_{sc}\}$), where $\pi'_{disc}$ is a discrete series non-supercuspidal and $\rho'_{sc}$ is a supercuspidal representation. We explain the idea in the Saito-Kurokawa case. The key observation is that $\pi_{disc}$ (resp. $\rho_{disc}$) is an irreducible sub-quotient of a parabolic induction, so, in this case, we can apply the same argument as in case (1) to deduce compatibility of the two correspondences. It remains to see that the same is true for $\pi_{sc}$. To do this, we use a description of the $\rho$-isotypic part of the shtuka space $\Sht(G,b,\mu)_{\infty}$ introduced in section $1.1$. Namely, we consider the complex
\[ R\Gamma_{c}(G,b,\mu)[\rho_{disc}] \simeq R\Gamma_{c}(G,b,\mu) \otimes^{\mathbb{L}}_{\mathcal{H}(J)} \rho_{disc}  \]
of $G(\mathbb{Q}_{p}) \times W_{L}$-modules. Recent work of Hansen-Kaletha-Weinstein \cite{KW} then tells us the form of this cohomology group (or rather a small variant thereof) as a $G(\mathbb{Q}_{p})$-representation. In particular, if we let $R\Gamma_{c}(G,b,\mu)[\rho_{disc}]_{sc}$ denote the summand of $R\Gamma_{c}(G,b,\mu)[\rho_{disc}]$ given by the supercuspidal Bernstein components of  $G(\mathbb{Q}_{p})$, then in the Grothendieck group of admissible $G(\mathbb{Q}_{p})$-representations of finite length $R\Gamma_{c}(G,b,\mu)[\rho_{disc}]_{sc}$ is equal to  $-2\pi_{sc}$.
\\\\ Similar to the case of $G = \GL_{n}$, this complex describes the action of the Hecke operator $T_{\mu}$  acting on a sheaf $\mathcal{F}_{\rho_{disc}}$ constructed from $\rho_{disc}$ on $\Bun_{G}$. Moreover, the complex $R\Gamma_{c}(G,b,\mu)[\rho_{disc}]_{sc}$ can be interpreted as a complex of sheaves on the open Harder-Narasimhan(=HN)-strata $\Bun_{G}^{\mathbf{1}} \subset \Bun_{G}$ corresponding to the trivial $G$-bundle on the Fargues-Fontaine curve $X$. It follows from the above description in the Grothendieck group that the excursion algebra will act on a constituent of this complex via eigenvalues valued in the parameter $\phi_{\pi_{sc}}^{\mathrm{FS}}$. However, since the excursion algebra is built from Hecke operators, it will also commute with the action of Hecke operators on $\mathcal{F}_{\rho_{disc}}$. This allows us to conclude that it also must act via eigenvalues valued in $\phi_{\rho_{disc}}^{\mathrm{FS}}$ giving a chain of equalities
\[ \phi_{\pi_{sc}}^{\mathrm{FS}} = \phi_{\rho_{disc}}^{\mathrm{FS}} = \phi_{\rho_{disc}}^{\mathrm{ss}} = \phi_{\pi_{sc}}^{\mathrm{ss}}\] 
of conjugacy classes of parameters, where the first equality follows from the previous argument and the second equality follows from the above analysis of induced representations. Similarly, one deduces compatibility for $\rho_{sc}$ by applying a similar argument to the $\pi$-isotypic part. 
\\\\
Case (3) is by far the most involved and takes up the majority of the paper. This is the case in which the $L$-parameter $\phi$ is supercuspidal. First, one can make a reduction to showing compatibility for just  $\rho \in \Pi(J)$ with supercuspidal Gan-Tantono parameter $\phi$, by using the commutation of Hecke operators and the excursion algebra similar to what was done in case $(2)$. Now, the key point again is that the complex $R\Gamma_{c}(G,b,\mu)[\rho]$ describes the action of the Hecke operator $T_{\mu}$ on a sheaf $\mathcal{F}_{\rho}$ on $\Bun_{G}$. To make further progress towards compatibility, we use that the Hecke operators can be in turn described using the spectral action of the derived category of perfect complexes on the stack of $L$-parameters, as constructed in \cite[Chapter~X]{FS}. In particular, a Hecke operator defines a vector bundle on the stack of $L$-parameters, whose action on the sheaf $\mathcal{F}_{\rho}$ via the spectral action is precisely $T_{\mu}$. Using this, we argue using the support of the spectral action of certain averaging operators, considered by \cite{AL} in the case of $\GL_{n}$, to show that, if $\mathrm{std}\circ \phi \otimes |\cdot|^{-3/2}$ occurs as a $W_{L}$-stable sub-quotient of the complex 
\[ \bigoplus_{\rho' \in \Pi_{\phi}(J)} R\Gamma_{c}(G,b,\mu)[\rho'] \]
we have an equality: 
\[ \mathrm{std}\circ \phi_{\rho}^{\mathrm{FS}} = \mathrm{std}\circ \phi\]
for all $\rho \in \Pi_{\phi}(J)$. Now a $\GSp_{4}$-valued parameter is in turn determined by its composition with $\mathrm{std}$ and its similitude character, which is precisely the central character of $\rho$. Therefore, since the Fargues-Scholze correspondence is compatible with central characters, this is enough to conclude that $\phi = \phi_{\rho}^{\mathrm{FS}}$. This reduces the question of showing compatibility for $\rho \in \Pi(\GU_{2}(D))$ with supercuspidal Gan-Tantono parameter $\phi$ to the following.
\begin{proposition}
Let $\phi$ be a supercuspidal parameter with associated $L$-packet $\Pi_{\phi}(J)$. Then the direct summand of
\[ \bigoplus_{\rho' \in \Pi_{\phi}(J)} R\Gamma_{c}(G,b,\mu)[\rho'] \]
given by the supercuspidal Bernstein components of $G(\mathbb{Q}_{p})$, denoted 
\[ \bigoplus_{\rho' \in \Pi_{\phi}(J)} R\Gamma_{c}(G,b,\mu)[\rho']_{sc}, \]
is concentrated in middle degree $3$ and admits a non-zero $W_{L}$-stable sub-quotient with $W_{L}$-action given by $\mathrm{std}\circ \phi \otimes |\cdot|^{-3/2}$.
\end{proposition}
Just as one does in proving Theorem 1.2, the key idea is to directly relate the complex 
\[ \bigoplus_{\rho' \in \Pi_{\phi}(J)} R\Gamma_{c}(G,b,\mu)[\rho']_{sc} \]
to the cohomology of a global Shimura variety using basic uniformization of the generic fiber as proven by Shen \cite{She} and an analogue of Boyer's trick \cite{Boy1}. This allows us to in turn prove Proposition 1.4 using global results on Galois representations in the cohomology of these Shimura varieties due to Kret-Shin \cite{KS} and Sorensen \cite{So}. More specifically, in this case the relevant Shimura datum is given by $(\mathbf{G},X)$, where $\mathbf{G}$ is a $\mathbb{Q}$-inner form of $\mathbf{G}^{*} := \mathrm{Res}_{F/\mathbb{Q}}(\GSp_{4})$ for $F/\mathbb{Q}$ a totally real extension with $p$ inert and $F_{p} \simeq L$. The relevant uniformization result is then applicable if $L/\mathbb{Q}_{p}$ is an unramified extension and $p > 2$. To state the key consequence of this uniformization result, we introduce some notation. We let $\mathbb{A}$ and $\mathbb{A}_{f}$ denote the adeles and finite adeles of $\mathbb{Q}$, respectively. If $K^{p} \subset \mathbf{G}(\mathbb{A}^{p}_{f})$ denotes the level away from $p$ and $K_{p} \subset G(\mathbb{Q}_{p})$ denotes the level at $p$, we let $\mathcal{S}(\mathbf{G},X)_{K_{p}K^{p}}$ be the rigid analytic Shimura variety over $\mathbb{C}_{p}$ of level $K_{p}K^{p}$. We let $\xi$ be a regular weight of an algebraic representation $\mathcal{V}_{\xi}$ of $\mathbf{G}$ over $\mathbb{Q}$ and let $\mathcal{L}_{\xi}$ denote the associated $\overline{\mathbb{Q}}_{\ell}$ local system on $\mathcal{S}(\mathbf{G},X)_{K_{p}K^{p}}$. We then define 
\[ R\Gamma_{c}(\mathcal{S}(\mathbf{G},X)_{K^{p}},\mathcal{L}_{\xi}) := \colim_{K_{p} \rightarrow \{1\}} R\Gamma_{c}(\mathcal{S}(\mathbf{G},X)_{K^{p}K_{p}},\mathcal{L}_{\xi})   \] 
Now we choose $X$ so that inverse of the Hodge cocharacter transported under the isomorphism $\ol{\mathbb{Q}}_{p} \simeq \mathbb{C}$ identifies with the Siegel cocharacter defined above.
The basic uniformization result of Shen then furnishes a $\mathbb{Q}$-inner form $\mathbf{G}'$ of $\mathbf{G}$ satisfying that $\mathbf{G}'_{\mathbb{Q}_{p}} \simeq J$ together with a  $G(\mathbb{Q}_{p}) \times W_{L}$-equivariant map
\[ \Theta: R\Gamma_{c}(G,b,\mu) \otimes^{\mathbb{L}}_ {\mathcal{H}(J_{b})} \mathcal{A}(\mathbf{G}'(\mathbb{Q})\backslash \mathbf{G}'(\mathbb{A}_{f})/K^{p},\mathcal{L}_{\xi}) \rightarrow R\Gamma_{c}(\mathcal{S}(\mathbf{G},X)_{K^{p}},\mathcal{L}_{\xi})  \]
functorial in the level $K^{p}$. Here $\mathcal{A}(\mathbf{G}'(\mathbb{Q})\backslash \mathbf{G}'(\mathbb{A}_{f})/K^{p},\mathcal{L}_{\xi})$ denotes the space of algebraic automorphic forms  of level $K^{p}$ valued in the algebraic representation $\mathcal{V}_{\xi}$ in the sense of \cite{Gross}. We want to use this uniformization map to apply global results on the cohomology of $R\Gamma_{c}(\mathcal{S}(\mathbf{G},X)_{K^{p}},\mathcal{L}_{\xi})$ to study the action of $W_{L}$ on $R\Gamma_{c}(G,b,\mu)$. To do this, we show an analogue of Boyer's trick, which says that the non-basic Newton strata of the adic flag variety $\mathcal{F}\ell_{G,\mu^{-1}} := (G/P_{\mu^{-1}})^{ad}$ are parabolically induced as spaces with $G(\mathbb{Q}_{p})$-action. Using the Hodge-Tate period map from the Shimura variety $\mathcal{S}(\mathbf{G},X)_{K^{p}}$ to $\mathcal{F}\ell_{G,\mu^{-1}}$, this implies that, if we project the cohomology on both sides to the supercuspidal Bernstein components for $G(\bb{Q}_{p})$, we get a $W_{L} \times G(\mathbb{Q}_{p})$-equivariant isomorphism:
\[ \Theta_{sc}: R\Gamma_{c}(G,b,\mu)_{sc} \otimes^{\mathbb{L}}_ {\mathcal{H}(J_{b})} \mathcal{A}(\mathbf{G}'(\mathbb{Q})\backslash \mathbf{G}'(\mathbb{A}_{f})/K^{p},\mathcal{L}_{\xi}) \xrightarrow{\simeq} R\Gamma_{c}(\mathcal{S}(\mathbf{G},X)_{K^{p}},\mathcal{L}_{\xi})_{sc}  \]
After showing this, we fix a $\rho$ having supercuspidal Gan-Tantono parameter $\phi$ which, using compatibility of Fargues-Scholze and Gan-Tantono with central character twists, we may replace via a suitable unramified character twist. After doing this then,  via an argument using the simple trace formula, choose a globalization of $\rho$ to a cuspidal automorphic representation $\Pi'$ of $\mathbf{G}'$, which occurs as a $J(\mathbb{Q}_{p})$-stable direct summand of $\mathcal{A}(\mathbf{G}'(\mathbb{Q})\backslash \mathbf{G}'(\mathbb{A}_{f})/K^{p},\mathcal{L}_{\xi})$ and is an unramified twist of Steinberg at some non-empty set of places $S_{st}$, for some sufficiently large regular weight $\xi$ and sufficiently small level $K^{p}$. We set $S$ to be a finite set of places outside of which $\Pi'$ is unramified. The Hecke eigenvalues of $\Pi'$ then define a maximal ideal $\mathfrak{m} \subset \mathbb{T}^{S}$ in the abstract commutative Hecke algebra of $\mathbf{G}'$ away from the finite places $S$. Regarding both sides of $\Theta$ as $\mathbb{T}^{S}$-modules, we can localize at $\mathfrak{m}$ to get a map 
\[ \Theta_{\mathfrak{m}}: (R\Gamma_{c}(G,b,\mu) \otimes^{\mathbb{L}}_ {\mathcal{H}(J_{b})} \mathcal{A}(\mathbf{G}'(\mathbb{Q})\backslash \mathbf{G}'(\mathbb{A}_{f})/K^{p},\mathcal{L}_{\xi}))_{\mathfrak{m}} \rightarrow R\Gamma_{c}(\mathcal{S}(\mathbf{G},X)_{K^{p}},\mathcal{L}_{\xi})_{\mathfrak{m}}  \]
We write $K^{p} = K^{S_{st} \cup \{p\}}K_{S_{st}}$
for $K^{S_{st} \cup \{p\}} \subset \mathbf{G}(\mathbb{A}_{f}^{S_{st} \cup \{p\}})$. Taking colimits on both sides as $K_{\{p\} \cup S_{st}} \rightarrow \{1\}$, we see that $\Theta_{\mathfrak{m}}$ induces a map:
\[ (R\Gamma_{c}(G,b,\mu) \otimes^{\mathbb{L}}_ {\mathcal{H}(J_{b})} \mathcal{A}(\mathbf{G}'(\mathbb{Q})\backslash \mathbf{G}'(\mathbb{A}_{f})/K^{S_{st} \cup \{p\}},\mathcal{L}_{\xi}))_{\mathfrak{m}} \rightarrow R\Gamma_{c}(\mathcal{S}(\mathbf{G},X)_{K^{S_{st} \cup \{p\}}},\mathcal{L}_{\xi})_{\mathfrak{m}}  \]
Since we know that $\Theta_{\mathrm{sc}}$ is an isomorphism, we have an isomorphism
\[ (R\Gamma_{c}(G,b,\mu)_{sc} \otimes^{\mathbb{L}}_ {\mathcal{H}(J_{b})} \mathcal{A}(\mathbf{G}'(\mathbb{Q})\backslash \mathbf{G}'(\mathbb{A}_{f})/K^{S_{st} \cup \{p\}},\mathcal{L}_{\xi}))_{\mathfrak{m}} \xrightarrow{\simeq} R\Gamma_{c}(\mathcal{S}(\mathbf{G},X)_{K^{S_{st} \cup \{p\}}},\mathcal{L}_{\xi})_{\mathfrak{m},sc}  \]
Noting that $\mathcal{A}(\mathbf{G}'(\mathbb{Q})\backslash\mathbf{G}'(\mathbb{A}_{f})/K^{p},\mathcal{L}_{\xi})$ is semi-simple, we can project to the summand where $\mathbf{G}(F_{v}) \simeq \mathbf{G}'(F_{v})$ acts via an unramified twist of the  Steinberg representation for all $v \in S_{st}$. This implies that we have an isomorphism
\[ \Theta_{\mathfrak{m},sc}^{st}: (R\Gamma_{c}(G,b,\mu)_{sc} \otimes^{\mathbb{L}}_ {\mathcal{H}(J_{b})} \mathcal{A}(\mathbf{G}'(\mathbb{Q})\backslash \mathbf{G}'(\mathbb{A}_{f})/K^{S_{st} \cup \{p\}},\mathcal{L}_{\xi}))^{st}_{\mathfrak{m}} \xrightarrow{\simeq} R\Gamma_{c}(\mathcal{S}(\mathbf{G},X)_{K^{S_{st} \cup \{p\}}},\mathcal{L}_{\xi})^{st}_{\mathfrak{m},sc}  \]
The key point is now, by analyzing the simple twisted trace formula of Kottwitz-Shelstad \cite{KoShe} and stable trace formulas of Arthur \cite{Ar1}, we can prove a strong multiplicity one type result (Proposition 5.4), for cuspidal automorphic representations that are unramified twists of Steinberg at some sufficiently large non-empty set of places and regular of weight $\xi$ at infinity. This implies that the representations of  $\mathbf{G}'$ occurring on LHS of $\Theta_{\mathfrak{m}}^{st}$ must have local constituent at $p$ with Langlands parameter $\phi$, since we localized at the Hecke eigensystem defined by $\Pi'$ at the unramified places, 
and, since the local constituents of the automorphic representations of $\mathbf{G}'$ at $p$ occurring in the LHS are all in the $L$-packet $\Pi_{\phi}(J)$ by the strong multiplicity one result, we can reduce Proposition 1.4 to showing that $R\Gamma_{c}(\mathcal{S}(\mathbf{G},X)_{K^{p}},\mathcal{L}_{\xi})^{st}_{\mathfrak{m},sc}$ is concentrated in degree $3$ and has $W_{L}$-action given (up to multiplicity) by $\mathrm{std}\circ \phi \otimes |\cdot|^{-3/2}$. This follows from the analysis carried out in  Kret-Shin \cite{KS}. In particular, it follows from their results that this complex will be concentrated in degree $3$ and that the traces of Frobenius in $\Gamma_{F} := \Gal(\overline{F}/F)$ on the \'etale cohomology of the associated global Shimura variety over $\overline{F}$  are given by $\mathrm{std}\circ \phi_{\tau_{v}}$, where $\tau_{v}$ are the local constituents of some weak transfer $\tau$ of $\Pi'$ to an automorphic representation of $\mathrm{Res}_{F/\mathbb{Q}}\GSp_{4} = \mathbf{G}^{*}$ and $\phi_{\tau_{v}}$ is the associated Gan-Takeda parameter. This allows one, up to multiplicities, to describe the Galois action on the global Shimura variety in terms of $\mathrm{std}\circ \rho_{\tau}$, where $\rho_{\tau}$ is a global $\GSp_{4}(\overline{\mathbb{Q}}_{\ell})$-valued representation of the absolute Galois group of $F$ constructed by Sorensen \cite{So} from $\tau$ characterized by the property that $i\WD(\mathrm{std}\circ \rho_{\tau})|_{W_{F_{v}}}^{F-s.s.} \simeq \mathrm{std} \circ \phi_{\tau_{v}} \otimes |\cdot|^{-3/2}$ for all but finitely many places $v$ of $F$. This would give one precisely the desired description of the $W_{L}$-action on $R\Gamma_{c}(G,b,\mu)[\rho]_{sc}$ if one knew that $\phi_{\tau_{p}} = \phi_{\rho}$. Since $\Pi'$ is a globalization of (an unramified twist of)  $\rho$, one needs to choose $\tau$ to be a strong transfer of $\Pi'$ at the prime $p$. This latter goal is accomplished using analysis of the simple trace formula as done in Kret-Shin \cite[Section~6]{KS} combined with the character identities proven by Chan-Gan \cite{CG}. These results on strong transfers also aid us in deducing the strong multiplicity one type result mentioned above. 
\\\\
In section $2$, we give an overview of the Gan-Takeda and Gan-Tantono local Langlands correspondence, putting it in the framework of the refined local Langlands correspondence of Kaletha in preparation for applications to the Kottwitz conjecture. In section $3$, we describe the Fargues-Scholze local Langlands correspondence and related ideas, giving the proof of compatibility in cases (1) and (2) and reducing case (3) to Proposition 1.4, via some properties of the spectral action discussed in section $3.2$. In section $4$, we discuss basic uniformization of the relevant Shimura varieties and prove the aforementioned analogue of Boyer's trick, showing that the uniformization map $\Theta_{sc}$ is an isomorphism. In section $5$, we analyze the simple trace formula with fixed central character in a fashion similar to Kret-Shin \cite{KS} to deduce the existence of the required strong transfers, as well as combine this with analysis of the simple twisted trace formula to deduce the required strong multiplicity one result. In section $6$, we apply the results of section $5$ combined with results of Kret-Shin \cite{KS} and Sorensen \cite{So} to compute the relevant Galois action on the global Shimura variety. Finally, in section $7$, we put the results of the previous sections together to prove Proposition 1.4. We then conclude with the application to the proofs of Theorems 1.1 and 1.3, as well as formally deduce compatibility for the local Langlands correspondence for $\Sp_{4}$ and its non quasi-split inner form $\SU_{2}(D)$, as constructed by Gan-Takeda \cite{GT3} and Choiy \cite{Cho}, respectively. We finish the section with a brief discussion of an application to the cohomology of the related (non-minuscule) local shtuka spaces. 
\section*{Acknowledgments}
It is a pleasure to thank my advisor David Hansen for giving me this project and for numerous suggestions and ideas related to it, as well as Alexander Bertoloni-Meli, Eric Chen, Tasho Kaletha, Arthur-C\'esar Le-Bras, Yoichi Mieda, Jack Sempliner, and Zhiyu Zhang for some nice conversations related to this work. Special thanks also go to Peter Scholze for sharing with me the draft of \cite{FS}, pointing out some errors, and generously initiating me in these ideas during my master's thesis, as well as Thomas Haines, Hao Peng, and Sug Woo Shin for pointing out some errors in section $3$ and section $5$. Lastly, I would like to thank the MPIM Bonn for their hospitality during part of the completion of this project and the anonymous referee for pointing out several mistakes and inaccuracies. Several minor errors were identified by A.I. after publication and have been fixed in the present version.

\section*{Conventions and Notations}
For a diamond or $v$-stack, we freely use the formalism in \cite{Ecod,FS} of $\ell$-adic cohomology of diamonds and $v$-stacks. We will fix isomorphisms $i: \overline{\mathbb{Q}}_{\ell} \xrightarrow{\simeq} \mathbb{C}$ and $j: \overline{\mathbb{Q}}_{p} \xrightarrow{\simeq} \mathbb{C}$ and use the (geometric) normalization of local class field theory that sends the arithmetic Frobenius to the inverse of the uniformizer. For a supercuspidal $L$-parameter, we will often abuse notation and use $\phi$ to denote both the $L$-parameter and the semisimplified parameter $\phi^{\mathrm{ss}}$, as in this case this merely corresponds to forgetting the trivially acting $\SL_{2}(\mathbb{C})$-factor and applying the isomorphism $i$. For a reductive group $H/\mathbb{Q}_{p}$, we will write $R\mathcal{H}om_{H(\mathbb{Q}_{p})}(-,\ol{\mathbb{Q}}_{\ell}): \D(H(\mathbb{Q}_{p}),\ol{\mathbb{Q}}_{\ell})^{\mathrm{op}} \rightarrow  \D(H(\mathbb{Q}_{p}),\ol{\mathbb{Q}}_{\ell})$ for the derived smooth duality functor, where $\D(H(\mathbb{Q}_{p}),\ol{\mathbb{Q}}_{\ell})$ is the unbounded derived category of smooth representations on $\ol{\mathbb{Q}}_{\ell}$ vector spaces. Namely, it is the derived functor induced by the left exact smooth duality functor $V \mapsto (V^{*})^{\mathrm{sm}}$, where $(V^{*})^{\mathrm{sm}}$ is the set of all $\ol{\bb{Q}}_{\ell}$-linear functions $f: V \ra \ol{\mathbb{Q}}_{\ell}$ such that there exists $K \subset H(\mathbb{Q}_{p})$ a compact open such that for all $v \in V$ and $k \in K$ we have that $f(kv) = f(v)$. Normally, in the literature the space $\Sht(G,b,\mu)_{\infty}$ parametrizes modifications $\mathcal{E}_{0} \dashrightarrow \mathcal{E}_{b}$ with meromorphy $\mu$. For us, it will denote the space parametrizing modifications of type $\mu^{-1}$. This convention limits the appearances of duals (cf. Remark 3.8).
\section{Local Langlands for $\GSp_{4}$ and $\GU_{2}(D)$}
\subsection{Local Langlands for $\GSp_{4}$}
In this section, we will describe the local Langlands correspondence of Gan-Takeda for the group $G := \GSp_{4}/L$, where $L/\mathbb{Q}_{p}$ is a finite extension. We fix a choice of Whittaker datum $\mathfrak{m} := (B,\psi)$ throughout section $2$, where $B$ is the Borel and $\psi$ is a non-trivial additive character of $L$.
\\\\
As before, we consider the set $\Phi(G)$ of admissible homomorphisms 
\[ \phi: W_{L} \times \SL_{2}(\mathbb{C}) \rightarrow \hat{G}(\mathbb{C}) = \GSpin_{5}(\mathbb{C}) \simeq \GSp_{4}(\mathbb{C}) \]
taken up to $\hat{G}$-conjugacy, where $W_{L}$ acts via a continuous semisimple homomorphism with respect to the discrete topology and $\SL_{2}(\mathbb{C})$ acts via an algebraic representation. Similarly, let $\Pi(G)$ denote the isomorphism classes of smooth irreducible representations of the group $G(L)$. We can now state the main theorem of Gan-Takeda. 
\begin{theorem}{\cite{GT1}}
There is a surjective finite to one map 
\[ \LLC_{G}: \Pi(\GSp_{4}) \rightarrow \Phi(\GSp_{4}) \]
\[ \pi \mapsto \phi_{\pi} \]
with the following properties:
\begin{enumerate}
    \item $\pi$ is an (essentially) discrete series representation of $\GSp_{4}(L)$ if and only if its $L$-parameter does not factor through any proper Levi subgroup of $\GSp_{4}(\mathbb{C})$.
    \item Given an $L$-parameter $\phi$, we set $S_{\phi} := Z_{\hat{G}}(\mathrm{Im}(\phi))$ to be the centralizer of $\phi$. The fiber $\Pi_{\phi}(G)$ can be naturally parametrized by the set of irreducible characters of the component group  
    \[ A_{\phi} := \pi_{0}(S_{\phi}) \simeq \pi_{0}(S_{\phi}/Z(\GSp_{4}))   \]
    which is either trivial or equal to $\mathbb{Z}/2\mathbb{Z}$. When $A_{\phi} = \mathbb{Z}/2\mathbb{Z}$, exactly one of the two representations in $\Pi_{\phi}(G)$ is generic for the fixed choice of Whittaker datum, and is indexed by the trivial character of $A_{\phi}$.
    \item The similitude character $\mathrm{sim}(\phi_{\pi})$ of $L^{*}$ is equal to the central character $\omega_{\pi}$ via the isomorphism given by local class field theory.
    \item Given a character $\chi$ of $L^{*}$ and letting $\lambda: \GSp_{4} \rightarrow L^{*}$ be the similitude character of $\GSp_{4}(L)$, we have, via local class field theory, that the $L$-parameter of $\pi \otimes (\chi \circ \lambda)$ is equal to $\phi_{\pi} \otimes \chi$.
    \item If $\pi \in \Pi(\GSp_{4})$ is a generic supercuspidal representation or non-supercuspidal representation then, for any $r \geq 1$ and any smooth irreducible representation $\sigma$ of $\GL_{r}(L)$, we have that
    \[ \gamma(s, \pi \times \sigma,\psi) = \gamma(s,\phi_{\pi} \otimes \phi_{\sigma},\psi) \]
    \[ L(s, \pi \times \sigma,\psi) = L(s,\phi_{\pi} \otimes \phi_{\sigma},\psi) \]
    \[ \epsilon(s, \pi \times \sigma,\psi) = \epsilon(s,\phi_{\pi} \otimes \phi_{\sigma},\psi), \]
    where the RHS are the Artin local factors associated to the representations of $W_{L} \times \SL_{2}(\mathbb{C})$, and the LHS are the local factors of Shahidi \cite{Sh} with respect to the morphisms of $L$-groups defined in \cite[Section~4]{GT1}.  In the case that $r \leq 2$ and $\pi$ is a non-generic supercuspidal representation, we have the same relationship where the LHS is given by the local factors defined by Townsend \cite{NT}. 
\end{enumerate}
The map $\LLC_{G}$ is uniquely determined by the properties (1),(3), and (5), where one can take $r \leq 2$ in $(5)$.
\begin{remark}
When the paper of Gan-Takeda was released there was no good theory of $L$,$\epsilon$, and $\gamma$ factors for nongeneric supercuspidal representations satisfying the usual properties. (see the Ten Commandments in \cite{LR}.) Instead, to uniquely characterize the correspondence for these representations, they use an equality between the Plancherel measure on the family of inductions from $\GSpin_{5}(L) \times \GL_{r}(L) \simeq \GSp_{4}(L) \times \GL_{r}(L)$ to $\GSpin_{2r + 5}(L)$ for $r \leq 2$. However, this theory of $L$,$\epsilon$, and $\gamma$ factors was later constructed by Nelson Townsend in his PhD thesis \cite{NT}.
\end{remark}
\end{theorem}
We now make the following definition.
\begin{definition}
Write $\mathrm{std}:\GSp_{4} \hookrightarrow \GL_{4}$ for the standard embedding. We say a discrete $L$-parameter is stable if the $L$-packet $\Pi_{\phi}(G)$ has size $1$ and is endoscopic if it has size $2$. Equivalently, by Theorem 2.1 (2), this is equivalent to saying that the character group $\hat{A}_{\phi}$ of the component group $A_{\phi}$ has cardinality $1$ or $2$, respectively. By \cite[Lemma~6.2]{GT1}, this can be characterized as follows.
\begin{itemize}
    \item (stable) $\mathrm{std}\circ \phi$ is an irreducible representation of $W_{L} \times \SL_{2}(\mathbb{C})$. In this case, $S_{\phi} = Z(\hat{G}) = \mathbb{G}_{m}$, so $A_{\phi}$ is trivial.
    \item (endoscopic) $\mathrm{std}\circ \phi \simeq \phi_{1} \oplus \phi_{2}$, where the $\phi_{i}: W_{L} \times \SL_{2}(\mathbb{C}) \rightarrow \GL_{2}(\mathbb{C})$ for $i = 1,2$ are distinct irreducible $2$-dimensional representations of $W_{L} \times \SL_{2}(\mathbb{C})$ with $\det(\phi_{1}) = \det(\phi_{2})$. In this case, $A_{\phi} \simeq \mathbb{Z}/2\mathbb{Z}$. We recall that $\GSp_{4}$ has a unique proper elliptic endoscopic group $\GSO_{2,2}$, and the endoscopic parameters lie in the image of the map $\Phi(\GSO_{2,2}) \rightarrow \Phi(\GSp_{4})$. More specifically, the dual group of $\GSO_{2,2}$ is $\GSpin_{4}$ and one has an identification $\GSpin_{4} \simeq (\GL_{2}(\mathbb{C}) \times \GL_{2}(\mathbb{C}))^{0} := \{(g_{1},g_{2}) \in \GL_{2}(\mathbb{C}) \times \GL_{2}(\mathbb{C}) |  \det(g_{1}) = \det(g_{2}) \}$ and the map $\Phi(\GSO_{2,2}) \rightarrow \Phi(\GSp_{4}) \xrightarrow{\mathrm{std} \circ} \Phi(\GL_{4})$ comes from the inclusion $(\GL_{2}(\mathbb{C}) \times \GL_{2}(\mathbb{C}))^{0} \subset \GL_{4}(\mathbb{C})$. Using this, we can compute that one has an identification
    \[ S_{\phi} \simeq \{(a,b) \in \mathbb{C}^{*} \times \mathbb{C}^{*}| a^{2} = b^{2} \} \subset (\GL_{2}(\mathbb{C}) \times \GL_{2}(\mathbb{C}))^{0} \]
    where the center $Z(\GSp_{4})(\mathbb{C}) \simeq \mathbb{C}^{*}$ embeds diagonally. 
\end{itemize}
\end{definition}
For an $L$-parameter $\phi$ we see, by Theorem 2.1 (2), that the size of the $L$-packet $\Pi_{\phi}(G)$ is at most $2$. This allows us to subdivide into three cases:
\begin{enumerate}
    \item The $L$-packet $\Pi_{\phi}(G)$ does not contain any supercuspidal representations.
    \item The $L$-packet $\Pi_{\phi}(G)$ contains one supercuspidal and one non-supercuspidal.
    \item The $L$-packet contains only supercuspidals. 
\end{enumerate}
Case $(1)$ contains the case where the parameters are not discrete as well as the almost unipotent and Soudry type parameters. Case (2) can only occur if $\phi$ is discrete but $\phi^{\mathrm{ss}}$
ss is not; more precisely, it consists exactly of the Saito–Kurokawa and Howe–Piatetski-Shapiro type parameters described below. More precisely, Case $(2)$ does not occur when the parameter is a stable discrete parameter, by definition. The relevant case is when the parameter is discrete endoscopic. To understand this, we let $\nu(n)$ denote the unique $n$-dimensional irreducible representation of $\SL_{2}(\mathbb{C})$. There are two cases:
\begin{enumerate}
    \item (Saito-Kurokawa Type) We have $\mathrm{std}\circ \phi = \phi_{0} \oplus \chi \boxtimes \nu(2)$, where $\phi_{0}$ is a $2$-dimensional irreducible representation of $W_{L}$ and $\chi$ is a character, with $\chi^{2} = \det(\phi_{0})$. Therefore, the semisimplification $\phi^{\mathrm{ss}}$ satisfies: $\mathrm{std}\circ \phi^{\mathrm{ss}} = \phi_{0} \oplus \chi \otimes |\cdot|^{\frac{1}{2}} \oplus \chi \otimes |\cdot|^{-\frac{1}{2}}$.
    \item (Howe-Piatetski--Shapiro Type) We have $\mathrm{std}\circ \phi = \chi_{1} \boxtimes \nu(2) \oplus \chi_{2} \boxtimes \nu(2)$, where $\chi_{1}$ and $\chi_{2}$ are distinct characters of $W_{L}$ satisfying $\chi_{1}^{2} = \chi_{2}^{2}$. Therefore, the semisimplification $\phi^{\mathrm{ss}}$ satisfies: $\mathrm{std}\circ \phi^{\mathrm{ss}} = \chi_{1} \otimes |\cdot|^{\frac{1}{2}} \oplus \chi_{1} \otimes |\cdot|^{-\frac{1}{2}} \oplus \chi_{2} \otimes |\cdot|^{\frac{1}{2}} \oplus \chi_{2} \otimes |\cdot|^{-\frac{1}{2}}$.
 \end{enumerate}
 We will say a parameter is mixed supercuspidal if it is of the above form or of Soudry type. More precisely, we say that a parameter is of Soudry type if we have that $\mathrm{std}\circ \phi = \phi_{0} \boxtimes \nu(2)$, where $\phi_{0}$ is a $2$-dimensional irreducible representation of $W_{L}$ with $\phi_{0} \simeq \phi_{0} \otimes eta$, 
 $\eta := \mathrm{sim}(\phi) \mathrm{det}(\phi{0}^{-1})$ is a nontrivial quadratic character. In particular, it is stable. In this case, the semi-simplification satisfies $\mathrm{std}\circ \phi^{\mathrm{ss}} = \phi_{0} \otimes |\cdot|^{\frac{1}{2}} \oplus \phi_{0} \otimes |\cdot|^{-\frac{1}{2}}$. This together with the supercuspidal parameters and the almost unipotent parameters (i.e.. where $\mathrm{std} \circ \phi = \chi \boxtimes \nu(4)$ for $\chi$ a $1$-dimensional representation of $W_{L}$) gives a complete list of the discrete parameters (i.e. those that do not factor through a proper Levi), and as we will see in the next section the mixed supercuspidal case will be precisely when $\Pi_{\phi}(G)$ and $\Pi_{\phi}(J)$ contain a mix of supercuspidal and non-supercuspidal representations (However, in the Soudry type case, the $L$-packet $\Pi_{\phi}(G)$ will be a single non-supercuspidal discrete series representation).
\begin{remark}
\begin{enumerate}
    \item The terminology for the parameters is explained by Arthur's classification \cite{Ar} of the global automorphic representations of $\GSp_{4}$ appearing in the papers \cite{Ku}, \cite{HPS}, and \cite{Soudry}, respectively (by analogy with the shape of the Arthur parameter; here $\nu(n)$ denotes the Weil-Deligne $\SL_{2}$).
    \item We will mention in the next section how to distinguish the three cases discussed above via the number of supercuspidals in the $L$-packet $\Pi_{\phi}(\GU_{2}(D))$ defined by the Gan-Tantono local Langlands correspondence.
\end{enumerate}
\end{remark}
Case (3) in the above division of parameters is the situation where the parameter $\phi$ is supercuspidal as defined in the introduction. In particular, in the supercuspidal case the restriction of the parameter $\phi$ to the $\SL_{2}(\mathbb{C})$ factor is trivial, so the irreducible representations occurring in the decomposition of $\mathrm{std}\circ \phi$ are just representations of $W_{L}$.
\\\\\
For the purposes of applying the weak form of the Kottwitz Conjecture proven in Hansen-Kaletha-Weinstein \cite{KW}, we formulate this correspondence in terms of the refined local Langlands of Kaletha \cite{Ka} with respect to the fixed choice of Whittaker datum $\mathfrak{m}$. Now, given a discrete parameter $\phi$, we have by Theorem 2.1 (2) a correspondence between the $L$-packet $\Pi_{\phi}(G)$ and the set of irreducible characters $A^{\vee}_{\phi}$. This in turn gives rise to an irreducible character of the group $S_{\phi}$ via the composition:
\[ S_{\phi} \rightarrow \pi_{0}(S_{\phi}) = A_{\phi} \]
This allows us to make the following definition.
\begin{definition}
For $\phi$ a discrete parameter and $\pi \in \Pi_{\phi}(G)$, we denote the character of $S_{\phi}$ described above by $\tau_{\pi}$.
\end{definition}
\subsection{Local Langlands for $\GU_{2}(D)$}
In this section, we describe the local Langlands correspondence for the unique non-split inner form $J = \GU_{2}(D)$, the group of similitudes of the unique $2$-dimensional Hermitian vector space over the quaternion division algebra $D/L$. As in the previous section, we let $\Pi(J)$ denote the set of irreducible admissible representations of $J$, and $\Phi(J)$ be the set of $L$-parameters of $J$. This is a subset of the previous set $\Phi(\GSp_{4})$ as we will now explain. The group $J$ has a unique up to conjugacy minimal parabolic whose Levi factor is
\[ D^{*} \times \GL_{1}. \]
This defines a form of the Siegel parabolic of $\GSp_{4}$ and it determines a dual parabolic subgroup $P^{\vee}(\mathbb{C})$ in the dual group $\GSp_{4}(\mathbb{C})$ of $\GU_{2}(D)$. This is the Heisenberg parabolic subgroup of $\GSp_{4}(\mathbb{C})$, its conjugacy class is said to be relevant for $J$ while all other conjugacy classes of proper parabolics are said to be irrelevant. We say $\phi \in \Phi(\GSp_{4})$ is relevant if it does not factor through any irrelevant parabolic subgroups of $\GSp_{4}(\mathbb{C})$. We define $\Phi(J)$ to be the subset of relevant $\phi$ in $\Phi(\GSp_{4})$. We set $B_{\phi} := \pi_{0}(Z_{\Sp_{4}}(\mathrm{Im}(\phi)))$. One has an exact sequence:
\[ \langle \pm 1 \rangle \rightarrow B_{\phi} \rightarrow A_{\phi} \rightarrow 0. \]
This implies that one has an injection on the group of irreducible characters $\hat{A}_{\phi} \hookrightarrow \hat{B}_{\phi}$, which identifies $\hat{A}_{\phi}$ as the subgroup of (of index at most $2$) of characters trivial on the image of the center $Z(\Sp_{4})(\mathbb{C}) $. One can check that $\hat{B}_{\phi} \neq \hat{A}_{\phi}$ if and only if $\phi$ is relevant for $\GU_{2}(D)$. Now we can state the main theorem of Gan and Tantono.
\begin{theorem}{\cite{GT2}}
There is a natural surjective finite-to-one map
\[ \LLC_{J}: \Pi(\GU_{2}(D)) \rightarrow \Phi(\GU_{2}(D)) \]
\[ \rho \mapsto \phi_{\rho} \]
with the following properties:
\begin{enumerate}
    \item $\rho$ is an (essentially) discrete series representation of $\GU_{2}(D)$ if and only if its parameter $\phi_{\rho}$ does not factor through any proper Levi subgroup of $\GSp_{4}(\mathbb{C})$.
    \item For an $L$-parameter $\phi$, the fiber $\Pi_{\phi}(J)$ can be naturally parametrized by the set $\hat{B}_{\phi} \setminus \hat{A}_{\phi}$. This set has size either $1$ or $2$.
    \item The similitude character $\mathrm{sim}(\phi_{\rho})$ of $\phi_{\rho}$ is equal to the central character $\omega_{\rho}$ of $\rho$, via the isomorphism given by local class field theory.
    \item Given a character $\chi$ of $L^{*}$ and letting $\lambda: \GU_{2}(D) \rightarrow L^{*}$ be the similitude character of $\GU_{2}(D)$, we have, via local class field theory, that the $L$-parameter of $\rho \otimes (\chi \circ \lambda)$ is equal to $\phi_{\rho} \otimes \chi$.
    \item If $\rho \in \Pi(\GU_{2}(D))$ is a non-supercuspidal representation then, for any smooth irreducible representation $\sigma$ of $\GL_{r}(L)$, we have that
    \[ \gamma(s, \rho \times \sigma,\psi) = \gamma(s,\phi_{\rho} \otimes \phi_{\sigma},\psi) \]
    \[ L(s, \rho \times \sigma,\psi) = L(s,\phi_{\rho} \otimes \phi_{\sigma},\psi) \]
    \[ \epsilon(s, \rho \times \sigma,\psi) = \epsilon(s,\phi_{\rho} \otimes \phi_{\sigma},\psi) \]
    where the RHS are the Artin local factors associated to the representations of $W_{L} \times \SL_{2}(\mathbb{C})$ and the LHS are the local factors of Shahidi, as defined in \cite[Section~8]{GT2} via the Langlands classification. 
    \item Suppose that $\rho$ is a supercuspidal representation. For any irreducible supercuspidal representation $\sigma$ of $\GL_{r}(L)$ with $L$-parameter $\phi_{\sigma}$, if $\mu(s,\rho \boxtimes \sigma, \psi)$ denotes the Plancherel measure associated to the family of induced representations $I_{P}(\rho \boxtimes \sigma,s)$ on $\GSpin_{r+4,r+1}$, where we have regarded $\rho \boxtimes \sigma$ as a representation of the Levi subgroup $\GSpin_{4,1} \times \GL_{r} \simeq \GU_{2}(D) \times \GL_{r}$, then $\mu(s,\rho \boxtimes \sigma,\psi)$ is equal to
    \[ \gamma(s,\phi_{\rho}^{\vee} \otimes \phi_{\sigma},\psi) \cdot \gamma(-s,\phi_{\rho} \otimes \phi_{\sigma}^{\vee}, \overline{\psi}) \cdot \gamma(2s, \mathrm{Sym}^{2}\phi_{\sigma} \otimes \mathrm{sim}(\phi_{\rho})^{-1},\psi) \cdot \gamma(-2s,\mathrm{Sym}^{2}\phi_{\sigma}^{\vee} \otimes \mathrm{sim}(\phi_{\rho}), \overline{\psi}) \]
 The map $\LLC_{J}$ is uniquely determined by the properties (1), (3), (5), and (6), with $r \leq 4$ in (5) and (6). 
\end{enumerate}
\end{theorem}
We now further elaborate on the structure of the $L$-packets $\Pi_{\phi}(J) := \LLC_{J}^{-1}(\phi)$ in the case where the parameter $\phi$ is mixed supercuspidal. If the parameter $\phi$ is of this form, then, it follows from \cite[Proposition~5.4]{GT2}, the description of $\LLC_{J}$ provided in \cite[Section~7]{GT2}, and the description of the non-generic supercuspidal representations of $\GSp_{4}$ provided in \cite[Theorem~1.1 (i)]{GanTakedaTheta}, that the $L$-packet $\Pi_{\phi}(J)$ has following structure, as alluded to in Remark 2.2 (2). 
\begin{enumerate}
    \item(Saito-Kurokawa Type) The $L$-packet $\Pi_{\phi}(J) = \{\rho_{disc},\rho_{sc}\}$ contains one supercuspidal representation $\rho_{sc}$ and one non-supercuspidal representation $\rho_{disc}$.
    \item(Howe-Piatetski--Shapiro Type) The $L$-packet $\Pi_{\phi}(J) = \{\rho_{disc}^{1},\rho_{disc}^{2}\}$ contains no supercuspidal representations.
    \item(Soudry Type) The $L$-packet $\Pi_{\phi}(J) = \{\rho_{sc}\}$ contains one supercuspidal and the $L$-packet $\Pi_{\phi}(G) = \{\pi_{disc}\}$ contains one non-supercuspidal (essentially) discrete series representation. 
\end{enumerate}
These will be the only cases in which the $L$-packets contain a mix of supercuspidal and non-supercuspidal representations, justifying the terminology of calling them mixed supercuspidal introduced in the previous section. In particular, the only other examples of discrete non-supercuspidal parameters are those of almost unipotent type, and in the case the packets $\Pi_{\phi}(G)$ (resp. $\Pi_{\phi}(J)$) will consist of a single (essentially) discrete series non-supercuspidal representation of $G(L)$ (resp. $J(L)$) (see \cite[Proposition~7.1 (v)]{GT2}).

We now would also like to briefly comment on the structure of the set $\hat{B}_{\phi} \setminus \hat{A}_{\phi}$, confirming the expectation that the size of the $L$-packets $\Pi_{\phi}(G)$ and $\Pi_{\phi}(J)$ is always the same.
\begin{itemize}
    \item (stable) In the case that the $L$-parameter $\phi$ is stable, we have that $B_{\phi} = \mathbb{Z}/2\mathbb{Z}$ and, as noted in section 2.1, $A_{\phi} = 1$. This means the set $\hat{B}_{\phi} \setminus \hat{A}_{\phi}$ consists of one element corresponding to the non-trivial character. 
    \item (endoscopic) In the case that the parameter $\phi$ is endoscopic, we have that the decomposition $\mathrm{std}\circ \phi \simeq \phi_{1} \oplus \phi_{2}$ induces an exact sequence
    \[ \bb{Z}/2\bb{Z} \xrightarrow{\triangle} B_{\phi} = Z(\SL_{2}) \times Z(\SL_{2}) = \mathbb{Z}/2\mathbb{Z} \times \mathbb{Z}/2\mathbb{Z} \ra A_{\phi} = \mathbb{Z}/2\mathbb{Z}  \]
    and so the set $\hat{B}_{\phi} \setminus \hat{A}_{\phi}$ has size $2$ and is indexed by two characters $\eta_{+-}$ and $\eta_{-+}$ each non-trivial on the diagonally embedded $\bb{Z}/2\bb{Z}$. Under the isomorphism 
    \[ S_{\phi} \simeq \{(a,b) \in \mathbb{C}^{*} \times \mathbb{C}^{*}| a^{2} = b^{2} \} \subset (\GL_{2}(\mathbb{C}) \times \GL_{2}(\mathbb{C}))^{0} \]
    from section 2.1 and the natural inclusion $B_{\phi} \subset S_{\phi}$ they extend to characters given by projecting to the two $\bb{C}^{*}$-factors.
\end{itemize}
We now wish to put this local Langlands correspondence for the inner form in the context of the refined local Langlands correspondence of Kaletha \cite{Ka}. We consider the Kottwitz set $B(G)$ \cite{Ko,RR} and let $b \in B(G)$ be the basic element whose associated $\sigma$-centralizer $J_{b} = J$. We take this to be the basic element whose slope homomorphism is the dominant rational cocharacter of $G$ given by $(1/2,1/2,1/2,1/2)$. Let $Z(\GSp_{4}) \simeq \mathbb{G}_{m}$ be the center. We recall that we have an isomorphism $\pi_{1}(G) \simeq X^*(Z(\hat{G})) \simeq \mathbb{Z}$ and that the $\kappa$-invariant of $b$ is sent to the element $1 \in \mathbb{Z}$ under this isomorphism. This indexes the identity representation of $\mathbb{G}_{m}$, denoted  $id_{\mathbb{G}_{m}}$. Thus, given a discrete parameter $\phi: W_{L} \times \SL_{2}(\mathbb{C}) \rightarrow \GSp_{4}(\mathbb{C})$, the refined local Langlands correspondence asserts bijections
\[ \Pi_{\phi}(G) \longleftrightarrow \{ \text{irreducible algebraic representations $\tau$ of $S_{\phi}$ s.t }   \tau|_{Z(\hat{G})} = \mathbf{1} \} \]
\[ \pi \mapsto \tau_{\pi} \]
\[ \Pi_{\phi}(J) \longleftrightarrow \{\text{irreducible algebraic representations $\tau$ of $S_{\phi}$ s.t }   \tau|_{Z(\hat{G})} = id_{\mathbb{G}_{m}} \}\]
\[ \rho \mapsto \tau_{\rho} \]
where $\mathbf{1}$ is the trivial representation. In section $2.1$, we saw how for $\pi \in \Pi_{\phi}(G)$ to construct the desired $\tau_{\pi}$. Here it is uniquely pinned down by the property that the trivial representation corresponds to the unique $\mf{m}$-generic representation. In the case of the inner form, the situation is a bit more tricky. Consider $\rho \in \Pi(J)$ with associated $L$-parameter $\phi_{\rho}$. If $\phi_{\rho}$ is stable then $S_{\phi} = \mathbb{G}_{m}$ and $\tau_{\rho}$ is simply $id_{\mathbb{C}^{*}}$. If $\phi$ is endoscopic, then, as noted in section 2.1, we have an inclusion:
\[ Z(\GSp_{4})(\mathbb{C}) = \mathbb{C}^{*} \xrightarrow{\Delta} S_{\phi} \simeq \{(a,b) \in \mathbb{C}^{*} \times \mathbb{C}^{*}| a^{2} = b^{2} \}. \]
We consider the characters $\tau_{i}: S_{\phi} \rightarrow \mathbb{C}^{*}$ for $i = 1,2$ given by projecting to the first and second coordinate. These satisfy the property that $\tau_{i}|_{\mathbb{C}^{*}} = id_{\mathbb{C}^{*}}$ on the diagonally embedded center as desired. Similarly, under the parametrization of Gan-Tantono $\Pi_{\phi}(J) = \{\rho_{+-},\rho_{-+}\}$, where $\rho_{+-}$ and $\rho_{-+}$ correspond to the characters $\eta_{+-}$ and $\eta_{-+}$ described above. Specifically, if $\pi_{1}$ and $\pi_{2}$ are the unique discrete series representations of $\GL_{2}(L)$ in the $L$-packet over $\phi_{1}$ and $\phi_{2}$ then
\[ \rho_{+-} := \theta(\mathrm{JL}^{-1}(\pi_{2}) \boxtimes \pi_{1}) \text{ and } \rho_{-+} := \theta(\mathrm{JL}^{-1}(\pi_{1}) \boxtimes \pi_{2}) \]
where $\theta$ denotes the non-zero local theta lift from $D^{*} \times \GL_{2}(L)$ to $\GU_{2}(D)$, as in \cite[Proposition~5.4]{GT2}, and $\mathrm{JL}: \Pi(D^{*}) \ra \Pi(\GL_{2})$ is the Jacquet-Langlands correspondence. Now we would like to match these two representations with $\tau_{1}$ and $\tau_{2}$ under the refined local Langlands of Kaletha. Suppose we fix such a matching. We consider a refined endoscopic datum $\mf{c}$ for the quasi-split reductive group $G$, which we recall is a tuple $(H,s,\mc{H},\eta)$ which consists of
\begin{itemize}
    \item a quasi-split group $H$ over $L$,
    \item an extension $\mc{H}$ of $W_L$ by $\widehat{H}$ such that the map $W_L \to \mathrm{Out}(\widehat{H})$ coincides with the map $\rho_H: W_L \to \mathrm{Out}(\widehat{H})$ induced by the action of $W_L$ on $\widehat{H} \subset {}^LH$,
    \item an element $s \in Z(\widehat{H})^{\Gamma}$,
    \item an $L$-homomorphism $\eta: \mc{H} \to {}^LG$,
\end{itemize}
satisfying the condition:
\begin{itemize}
    \item we have $\eta(\widehat{H}) = Z_{\widehat{G}}(\eta(s))^{\circ}$.
\end{itemize}
Considering $J$ as the $\sigma$-centralizer of the unique basic element $b_{1} \in B(G)$ of Kottwitz invariant $1 \in \bb{Z}$, this defines for us an extended pure inner twisting $(\xi,b_{1}): G \ra J_{b_{1}}$ in the sense of \cite[Section~5.2]{Kott1}, and we can attach a canonical transfer factor $\Delta[\mf{m},\mf{c},b_{1}]$ to this datum, as defined in \cite[Section~4.1]{Ka}. Given a test function $f \in C^{\infty}_{c}(J(L),\ol{\mathbb{Q}}_{\ell})$, we can use these transfer factors to say what it means for $f^{\mf{c}} \in C^{\infty}_{c}(H(L),\ol{\mathbb{Q}}_{\ell})$ to be matching.  In particular, for $f \in C_c^\infty(J(L),\overline{\mathbb Q}_\ell)$, we say that
$f^{\mathfrak c} \in C_c^\infty(H(L),\overline{\mathbb Q}_\ell)$ matches
$f$ if, for every strongly $G$-regular semisimple element
$\gamma \in H(L)$,
\[
SO_\gamma(f^{\mathfrak c})
=
\sum_{[\delta]}
\Delta[\mathfrak m,\mathfrak c,b_1](\gamma,\delta)\,O_\delta(f),
\]
where the sum runs over the $J(L)$-conjugacy classes of strongly
regular semisimple elements $\delta$ matching $\gamma$. Here
$O_\delta$ and $SO_\gamma$ denote ordinary and stable orbital
integrals, respectively, with compatible choices of Haar measures.
We use the convention in which the Weyl discriminant factor
$\Delta_{\mathrm{IV}}$ is included in the transfer factor.

Now, suppose we have a discrete parameter $\phi \in \Phi(J)$ and a refined endoscopic datum $\mf{c}$, such that $\phi = \eta \circ \phi^{\mf{c}}$ as conjugacy classes of parameters for an $L$-parameter $\phi^{\mf{c}}: W_{L} \times \SL_{2}(\bb{C}) \ra \mathcal{H}$. Then the matching is uniquely described using the endoscopic character identities. This asserts an equality
\[ \Theta^{1}_{\phi^{\mf{c}}}(f^{\mf{c}}) = \sum_{\pi \in \Pi_{\phi^{\mf{c}}}(H)} \tr(1|\tau_{\pi})\theta_{\pi}(f^{\mf{c}}) = e(J)\sum_{\rho \in \Pi_{\phi}(J)} \mathrm{tr}(s|\tau_{\rho}) \theta_{\rho}(f) = \Theta_{\phi}^{s}(f), \]
where $e(J)$ is the Kottwitz sign of $J$, as defined in \cite{Kott} and $\theta_{\pi}$ denotes the Harish-Chandra character of $\pi$. Using the linear independence of the distributions $\theta_{\pi}$ and the fact that the packets $\Pi_{\phi}(J)$ are disjoint, we can see that the matching between $\rho \mapsto \tau_{\rho}$ is uniquely characterized by these relations. To show that there exists a matching between the representations $\tau_{1}$ and $\tau_{2}$ and $\rho_{+-}$ and $\rho_{-+}$, we need to show that these identities are satisfied under the parametrization of Gan-Takeda-Tantono. This will follow from the endoscopic character identities verified by Chan-Gan \cite{CG}. Namely, in the case that the parameter $\phi$ is stable, one only needs to consider the trivial endoscopic datum $\mf{c}_{\mathrm{triv}} = (G,1,\phantom{}^{L}G,\mathrm{id})$, and these identities follow from \cite[Proposition~11.1 (1)]{CG}. In the case that the parameter $\phi$ is endoscopic, the case $\mf{c}_{\mathrm{triv}} = (G,1,\phantom{}^{L}G,\mathrm{id})$ follows from \cite[Proposition~11.1 (1)]{CG}, but one also needs to consider the refined endoscopic datum given by $\mf{c} = (\GSO_{2,2},(1,-1),\GSpin_{4},i: \GSpin_{4} \ra \GSpin_{5} \simeq \GSp_{4})$, where $i$ is the map from before and $(1,-1) \in S_{\phi} \subset (\GL_{2}(\mathbb{C}) \times \GL_{2}(\mathbb{C}))^{0}$. Here we have committed an abuse of notation in the endoscopic datum, writing $\GSpin_{4}$ instead of its split extension $\GSpin_{4} \times W_{L}$ by the Weil group and similarly for $\GSp_{4}$.

In this case, the identities (up to a sign) follow from combining \cite[Proposition~11.1 (1)]{CG} as before and \cite[Proposition~11.1 (2)]{CG}, where we note that since $\tau_{1}$ and $\tau_{2}$ are by definition given by the two projection maps their traces against $(1,-1)$ have the opposite sign, which is consistent with \cite[Proposition~11.1 (2)]{CG}. 

In order for the above argument for deducing the refined local Langlands from the results of Chan-Gan to work, one needs to compare the transfer factors used in \cite[Section~4]{CG} with those of Kaletha, so that the notion of matching functions between $f$ and $f^{\mf{c}}$ coincide. We comment briefly on this now. When discussing the transfer factors $\Delta[\mf{m},\mf{c}]$ used for transferring between two quasi-split groups both Kaletha and Chan-Gan work with the Whittaker normalized transfer factors considered in \cite{KoShe2}. Kaletha on the one hand works with the renormalized factors $\Delta'_{\lambda}$ in the notation of \cite[Section~5.5]{KoShe2} (see \cite[Page~221]{Ka}) while Chan-Gan work with the original factors denoted $\Delta_{\lambda}$ considered in \cite[Section~5.3]{KoShe} (see \cite[Section~4.1]{CG}); however, as noted at the end of \cite[Section~5.5]{KoShe2}, when the element $s$ in the refined endoscopic datum satisfies $s = s^{-1}$, the two transfer factors agree, and this will always be true for the endoscopic datum mentioned above. We also note that Chan-Gan unlike Kaletha do not include the Weyl discriminant term $\Delta_{IV}$ in their definition of the Whittaker normalized transfer factors, as noted in \cite[Section~4.1]{CG}. Instead they include it in their definition of the orbital integrals used when defining the notion of matching functions. Moreover, both authors use the arithmetic normalization of local class field theory. Since we use the geometric normalization (see Conventions and Notations), we note that this discrepancy does not affect the argument: switching the normalization of local class field theory replaces $\Delta'_\lambda$ by $\Delta_D$ in the notation of \cite[Section~5.5]{KS12}, and these agree for the refined endoscopic datum $c$ above, since $s = s^{-1}$. Moreover, the labeling on the $G$-side is fixed by genericity, which does not depend on this choice, while the matching $\rho \mapsto \tau_\rho$ on the $J$-side is in any case only characterized by the character identities above. These observations allow one to see that the transfer factors $\Delta[\mf{m},\mf{c}]$ considered in \cite[Page~221]{Ka} and the transfer factors considered \cite{CG} give rise to the same definition of matching functions when transferring between two quasi-split groups.

Therefore, the main difficulty when comparing the two conventions of transfer factors lies when considering transfer between a quasi-split group and a non-quasi-split inner form. In particular, one needs to compare the transfer factors considered in \cite[Section~4.3]{CG} with the transfer factor $\Delta[\mf{m},\mf{c},b_{1}]$ mentioned above. This is determined up to a norm one element in $\mathbb{C}$, as mentioned in \cite[Section~4]{CG}. For Chan-Gan, fixing this normalization is accomplished by matching the Levi subgroups $D^{*} \times \GL_{1} \subset \GU_{2}(D)$ and $\GL_{2} \times \GL_{1} \subset \GSp_{4}$ under the inner twisting, and then fixing the normalization by setting the transfer factor to be equal to the Kottwitz sign $e(\D^{*} \times \GL_{1}) = -1$ on the pairs $(\gamma,\delta)$ of strongly regular elements, where $\delta \in (D^{*} \times \GL_{1})(L)$ is the norm of an element $\gamma \in (\GL_{2} \times \GL_{1})(L)$, where we reiterate that Chan-Gan incorporate the Weyl discriminant into their orbital integrals. Then one chooses normalization of the transfer factors between $\GSp_{4}$ and $\GU_{2}(D)$ to be compatible with the one on the Levi, as defined in \cite[Section~4.3]{CG}. Similarly, when contemplating the transfer factors between $\GU_{2}(D)$ and $\GSO_{2,2}$ for the endoscopic case, Chan-Gan perform the same procedure where one matches up the Levi $D^{*} \times \GL_{1} \subset \GU_{2}(D)$ with either $\GL_{2} \times \GL_{1} \subset \GSO_{2,2}$  or $\GL_{1} \times \GL_{2} \subset \GSO_{2,2}$, since both of these match up with $D^{*} \times \GL_{1}$ after passing to dual groups and applying the map $i: \GSpin_{4} \ra \GSp_{4}$. After fixing the choice, we can normalize the transfer factors such that their value is equal to $-1$ when evaluated on the strongly regular elements in norm correspondence between $D^{*} \times \GL_{1}$ and the other choice of Levi. However, the two choices determine transfer factors that differ up to a sign, since they are interchanged under the unique outer automorphism of $\GSO_{2,2}$ which switches the two factors, as remarked in \cite[Section~4.4]{CG}. Then for the endoscopic character identities in \cite[Proposition~11.1 (2)]{CG} an ad hoc choice is made for which Levi to match it with at the beginning of \cite[Section~11]{CG}. On the other hand, as mentioned above Kaletha normalizes his choice by realizing $\GU_{2}(D)$ as an extended pure inner twist of $\GSp_{4}$, where one obtains an identity 
\[ \Delta[\mf{m},\mf{c},b_{1}](\gamma,\delta) = \Delta[\mf{m},\mf{c}](\gamma,\delta^{*})\langle \mathrm{inv}[b_{1}](\delta^{*},\delta), s \rangle. \]
Here $\gamma$ and $\delta$ are strongly regular elements of $H(L)$ and $\GU_{2}(D)(L)$, respectively, and the terms $\delta^{*} \in \GSp_{4}(L)$ and $\langle \mathrm{inv}[b_{1}](\delta^{*},\delta), s \rangle$ are as constructed in the discussion around \cite[Equation~6]{Ka} (Here Kaletha is discussing pure inner twists as opposed to extended pure inner twists, but the constructions are almost identical). The element $\langle \mathrm{inv}[b_{1}](\delta^{*},\delta), s \rangle$ is a certain cocycle evaluated on the element $s$,  and, by virtue of $s$ being an element of order $2$ in the cases of interest, this is given by $\pm 1$. This sign can in general depend on $(\delta^{*},\delta)$, but if  $H = G$ in the endoscopic datum $\mf{c}$ then its value is determined by evaluating $\kappa(b_{1}) \in X^{*}(Z(\hat{G})^{\Gamma})$ on the element $s \in Z(\hat{G})^{\Gamma}$ (see \cite[Fact~3.2.3]{KW}). By studying the restriction of the construction of $\langle \mathrm{inv}[b_{1}](\delta^{*},\delta), s \rangle$ to Levi subgroups (where one can check that if $(\gamma,\delta)$ lie in Levis which transfer to each other then $\delta^{*}$ can also be chosen to lie in the associated Levi of $\GSp_{4}$) and using the previous observations, one can conclude that the transfer factors used in \cite{Ka} and \cite{CG} differ by a sign. More specifically, we can for example look at the inclusion of the Levi $M = \GL_{2} \times \GL_{1} \subset \GSp_{4}$ corresponding to the Siegel parabolic. This induces a map of Kottwitz sets $B(\GL_{2} \times \GL_{1}) \rightarrow B(\GSp_{4})$. The element $b_{1}$ admits a (necessarily unique) basic lift to an element $b_{1,M} \in B(M)_{\mathrm{basic}}$ with the property that $J_{b_{1,M}} \simeq D^{*} \times \GL_{1}$ (see \cite[Lemma~4.26,Lemma~4.22]{GH}). This defines an inner twisting $(\xi_{M},b_{1,M}): \GL_{2} \times \GL_{1} \ra D^{*} \times \GL_{1}$, which will be compatible with the inner twisting by $(\xi,b_{1})$ in the obvious sense. 

Choose the embedding of the dual Levi $\widehat M$ compatibly with
the chosen Levi subgroup of $H$. Under the identification
\[
\widehat H
=
\bigl\{(g_1,g_2)\in\mathrm{GL}_2(\mathbb C)^2:
\det(g_1)=\det(g_2)\bigr\}
\subset\widehat G,
\]
we may, after exchanging the two factors if necessary, realize
this dual Levi as
\[
\widehat M
=
\bigl\{(g,\operatorname{diag}(u,v)):
g\in\mathrm{GL}_2(\mathbb C),\
u,v\in\mathbb C^\times,\
\det(g)=uv\bigr\}
\subset\widehat H.
\]
In particular, the endoscopic element $s=(I_2,-I_2)$ belongs to
$Z(\widehat M)^\Gamma$. For strongly regular elements
$\delta^*\in M(L)$ and $\delta\in J_{b_{1,M}}(L)$ related by the
Levi inner twist $(\xi_M,b_{1,M})$, compatibility with the inner
twist of $G$ and functoriality of the Kottwitz invariant give
\[
\bigl\langle\operatorname{inv}[b_1](\delta^*,\delta),s\bigr\rangle
=
\bigl\langle\operatorname{inv}[b_{1,M}](\delta^*,\delta),s\bigr\rangle
=
\kappa_M(b_{1,M})(s).
\]
Indeed, applying \cite[Fact~3.2.3]{KW}  to $M$, the image of the Levi
invariant under the Kottwitz map is $\kappa_M(b_{1,M})$; evaluation
at $s\in Z(\widehat M)^\Gamma$ therefore gives the displayed
equality. Since $s^2=1$, this value lies in $\{\pm1\}$ and is
independent of the chosen related pair in the Levi.

We then can match the sign of the transfer factors up with the ad-hoc choice fixed by Chan-Gan in the endoscopic case after restricting to Levis by evaluating both transfer factors on strongly regular elements $(\gamma,\delta)$ which are in norm correspondence between $\GL_{2} \times \GL_{1} \subset \GSO_{2,2}$ or $\GL_{1} \times \GL_{2} \subset \GSO_{2,2}$, depending on the rather inexplicit choice Chan-Gan make for the statement of \cite[Proposition~11.1 (2)]{CG} and then dividing out the Weyl discriminant term in Kaletha's normalization. Determining the precise sign however, would require carefully tracing through all the constructions.

The above argument shows the refined local Langlands correspondence of Kaletha holds for the group $G$; however, in order to describe the exact matching between $\rho_{+-}$ and $\rho_{-+}$ with $\tau_{1}$ and $\tau_{2}$ one needs to exactly compare the transfer factors. The above argument only shows that there is some matching; nonetheless, for our purposes the choice ends up being irrelevant, so we denote the representations in the $L$-packet $\Pi_{\phi}(J)$ corresponding to the projections $\tau_{1}$ and $\tau_{2}$ by $\rho_{1}$ and $\rho_{2}$, respectively. Similarly, for the representations obtained by pre-composing a character with the composition 
\[ S_{\phi} \rightarrow A_{\phi} \]
we denote the elements of the $L$-packet $\Pi_{\phi}(G)$ corresponding to the trivial (non-trivial) character of $A_{\phi}$ by $\pi^{+}$ (resp. $\pi^{-}$). We note that, by Theorem 2.1 (2), $\pi^{+}$ can be characterized as the unique $\mathfrak{m}$-generic representation of this $L$-packet. 
\begin{definition}
Given a discrete $L$-parameter $\phi$ as above and $\rho \in \Pi_{\phi}(J)$, we let $\tau_{\rho}$ be the irreducible representation of $S_{\phi}$ associated to it via the matching described above. Given $\pi \in \Pi_{\phi}(G)$ and $\rho \in \Pi_{\phi}(J)$, we set
\[ \delta_{\pi,\rho} := \tau_{\pi}^{\vee} \otimes \tau_{\rho}\]
where $\tau_{\pi}^{\vee}$ denotes the contragredient. 
\end{definition}
\begin{remark}
Changing the choice of Whittaker datum scales the representations by a $1$-dimensional character of $S_{\phi}$ that is trivial when restricted to the center, so in particular this pairing is independent of the choice of Whittaker datum (see \cite[Lemma~2.3.3]{KW}). 
\end{remark}
\section{The Fargues-Scholze Local Langlands Correspondence}
We will now discuss the Fargues-Scholze local Langlands correspondence and deduce compatibility in the cases where the Gan-Takeda/Gan-Tantono parameter is not supercuspidal. We will then conclude by reducing the question of compatibility in the supercuspidal case to Proposition 1.4.
\subsection{Overview of the Fargues-Scholze Local Langlands Correspondence}
For now, let $G$ be any connected reductive group over $\mathbb{Q}_{p}$. Since we are going to be using geometric Satake, we fix a choice of the square root of $p$ in $\overline{\mathbb{Q}}_{\ell}$, so that half Tate-twists are well-defined. For us, we will always take $i^{-1}(\sqrt{p})$, where $i: \ol{\mathbb{Q}}_{\ell} \xrightarrow{\simeq} \mathbb{C}$ is the fixed isomorphism. Fargues-Scholze \cite{FS} consider the moduli space of $G$-bundles on the Fargues-Fontaine curve $X$, denoted $\Bun_{G}$. This moduli space is an Artin $v$-stack (in the sense of \cite[Section~IV.1]{FS}) and has the structure that the underlying points of its topological space $|\Bun_{G}|$ are in natural bijection with elements of the Kottwitz set $B(G)$, where the slopes of the $G$-isocrystal associated to $b \in B(G)$ are the negatives of the slopes of the associated vector bundle $\mathcal{E}_{b}$. Moreover, the specializations between points of $|\Bun_{G}|$ are dictated by the partial ordering on $B(G)$ induced by the kappa invariant and the slope homomorphism \cite{Vi}. In particular, the connected components of $\Bun_{G}$ are in bijection with $B(G)_{basic} \xrightarrow{\simeq} \pi_{1}(G)_{\Gamma}$. Specifically, for any $b \in B(G)_{basic}$, there is a unique open Harder-Narasimhan stratum $\Bun_{G}^{b} \subset \Bun_{G}$ dense inside the associated connected component. We recall that the elements of $B(G)_{basic}$ parametrize extended pure inner forms of $G$, via sending an element $b \in B(G)_{basic}$ to its $\sigma$-centralizer $J_{b}/\mathbb{Q}_{p}$. For such a basic $b$, we have an identification $\Bun_{G}^{b} \simeq  [\ast/\underline{J_{b}(\mathbb{Q}_{p})}] =: B\underline{J_{b}(\mathbb{Q}_{p})}$ of the HN-strata defined by $b$ and the classifying stack of $J_{b}(\mathbb{Q}_{p})$. For any Artin $v$-stack $Z$, Fargues-Scholze define a triangulated category $\D_{\blacksquare}(Z,\overline{\mathbb{Q}}_{\ell})$ of solid $\overline{\mathbb{Q}}_{\ell}$-sheaves \cite[Section~VII.1]{FS} and isolate a nice full subcategory $\Dlis(Z,\overline{\mathbb{Q}}_{\ell}) \subset \D_{\blacksquare}(Z,\overline{\mathbb{Q}}_{\ell})$ of lisse-\'etale $\overline{\mathbb{Q}}_{\ell}$-sheaves \cite[Section~VII.6]{FS}, which may be roughly (up to the difference between the $\ell$-adic and discrete topology) thought of as the unbounded derived category of \'etale $\overline{\mathbb{Q}}_{\ell}$ sheaves on $Z$, where one has made an enlargement to capture information about the topology of $p$-adic groups. In any case, the key point for us is that we have the following basic result.
\begin{lemma}{\cite[Proposition~VII.7.1]{FS}}
There is an equivalence of categories
\[ \Dlis(B\underline{J_{b}(\mathbb{Q}_{p})},\overline{\mathbb{Q}}_{\ell})\simeq \D(J_{b}(\mathbb{Q}_{p}),\overline{\mathbb{Q}}_{\ell}) \]
where the RHS denotes unbounded derived category of smooth $J_{b}(\mathbb{Q}_{p})$-representations with coefficients in $\overline{\mathbb{Q}}_{\ell}$. Under this equivalence, Verdier duality corresponds to smooth duality. 
\end{lemma}
\begin{remark}
The main reason for constructing this category $\Dlis$ is that, if one were to take the usual definition for the category of \'etale $\overline{\mathbb{Q}}_{\ell}$-sheaves on $B\underline{J_{b}(\mathbb{Q}_{p})}$, this equivalence would no longer be true. In particular, one would obtain the bounded derived category of representations of $J_{b}(\mathbb{Q}_{p})$ admitting a $J_{b}(\mathbb{Q}_{p})$-stable $\overline{\mathbb{Z}}_{\ell}$-lattice, where the representation is continuous with respect to the $\ell$-adic topology on $\ol{\bb{Z}}_{\ell}$. This would limit the scope of the Fargues-Scholze LLC as, in general, one wants to consider the full category of smooth $\overline{\mathbb{Q}}_{\ell}$-representations of $J_{b}(\mathbb{Q}_{p})$, and hence the need for the enlargement of the derived category to $\Dlis$.   
\end{remark}
Lemma 3.1 tells us that, given an irreducible smooth representation $\pi$ of $G(\mathbb{Q}_{p})$, we can consider the associated sheaf, denoted $\mathcal{F}_{\pi}$, on $\Bun_{G}^{\mathbf{1}}$ the open HN-strata corresponding to the trivial element $\mathbf{1} \in B(G)$, and take the extension by zero along the open inclusion $j_{!}(\mathcal{F}_{\pi})$\footnote{The shriek push-forward is not in general well-defined in the context of solid $\ol{\mathbb{Q}}_{\ell}$-sheaves. However, for the inclusion of HN-strata into $\Bun_{G}$, its existence follows from \cite[Proposition~VII.7.3]{FS}, using \cite[Proposition~VII.6.7]{FS}.}. This realizes the representation $\pi$ in terms of a sheaf on the moduli space $\Bun_{G}$ in an analogous way to how the function-sheaf dictionary realizes cuspidal automorphic forms as functions associated to sheaves in the context of curves over finite fields. Following V. Lafforgue \cite{VL}, Fargues and Scholze construct a semisimple L-parameter associated to this sheaf by looking at the action of the excursion algebra on this category $\Dlis(\Bun_{G},\overline{\mathbb{Q}}_{\ell})$. This relies on a form of the geometric Satake correspondence for the $B_{dR}^{+}$-affine Grassmannians. For any finite set $I$, let $X^{I}$ be the product of $I$-copies of the diamond $X = \Spd(\Breve{\mathbb{Q}}_{p})/\Frob^{\mathbb{Z}}$. We then have the global Hecke stack
\[ \begin{tikzcd}
& & \arrow[dl,"h^{\leftarrow}"] \Hck \arrow[dr,"h^{\rightarrow} \times supp"] & & \\
& \Bun_{G} & & \Bun_{G} \times X^{I}  & 
\end{tikzcd} \]
defined as the functor that parametrizes, for $S$ a perfectoid space in characteristic $p$ together with a map $S \rightarrow X^{I}$ defining a tuple of Cartier divisors in the relative Fargues-Fontaine $X_{S}$ over $S$, corresponding to characteristic $0$ untilts $S_{i}^{\sharp}$ for $i \in I$ of $S$, a pair of $G$-torsors $\mathcal{E}_{1}$, $\mathcal{E}_{2}$ together with an isomorphism
\[ \beta:\mathcal{E}_{1}|_{X_{S} \setminus \bigcup_{i \in I} S_{i}^{\sharp}} \xrightarrow{\simeq} \mathcal{E}_{2}|_{X_{S} \setminus \bigcup_{i \in I} S_{i}^{\sharp}}\]
meromorphic along the $S_{i}^{\sharp}$, where $h^{\leftarrow}((\mathcal{E}_{1},\mathcal{E}_{2},\beta,(S_{i}^{\sharp})_{i \in I})) = \mathcal{E}_{1}$ and $h^{\rightarrow} \times supp((\mathcal{E}_{1},\mathcal{E}_{2},\beta,(S_{i}^{\sharp})_{i \in I})) = (\mathcal{E}_{2},(S_{i}^{\sharp})_{i \in I})$, as in \cite[Page~320]{FS}. We set $\phantom{}^{L}G^{I}$ to be $I$-copies of the Langlands dual group of $G$, i.e. $^{L}G = Q \ltimes \hat{G}(\overline{\mathbb{Q}}_{\ell})$, where $\hat{G}$ is the reductive group having dual root datum to $G$ and is viewed as a reductive group over $\overline{\mathbb{Q}}_{\ell}$. The Weil group acts on $\hat{G}$ via the induced action on root datum through some finite quotient $Q$, which we now fix. Let $\Rep_{\overline{\mathbb{Q}}_{\ell}}(^{L}G^{I})$ denote the category of algebraic $\overline{\mathbb{Q}}_{\ell}$-representations of $I$-copies of $^{L}G$. For each element $W \in \Rep_{\overline{\mathbb{Q}}_{\ell}}(^{L}G^{I})$, the geometric Satake correspondence of Fargues-Scholze \cite[Chapter~VI]{FS} furnishes a solid $\overline{\mathbb{Q}}_{\ell}$-sheaf $\mathcal{S}_{W}$ on the global Hecke $\Hck$ (\cite[Page~319]{FS}). This is given by pulling back along the natural map $\Hck \ra \mathcal{H}\mathrm{ck}$ from the global Hecke stack to the local one, defined by restricting a modification to a formal neighborhood of the Cartier divisor (as in \cite[Page~16]{FS}), where the Satake category is by definition an abelian subcategory of all sheaves on $\mathcal{H}\mathrm{ck}_{G}$. This sheaf on the global Hecke stack allows us to define Hecke operators.
\begin{definition}
For each $W \in \Rep_{\overline{\mathbb{Q}}_{\ell}}(^{L}G^{I})$, we define the Hecke operator
\[ T_{W}: \Dlis(\Bun_{G},\overline{\mathbb{Q}}_{\ell}) \rightarrow \D_{\blacksquare}(\Bun_{G} \times X^{I}) \]
\[ A \mapsto R(h^{\rightarrow} \times supp)_{\natural}(h^{\leftarrow *}(A) \otimes^{\mathbb{L}} \mathcal{S}_{W})\]
where $\mathcal{S}_{W}$ is a solid $\overline{\mathbb{Q}}_{\ell}$-sheaf and the functor $R(h^{\rightarrow} \times supp)_{\natural}$ is the natural push-forward. I.e. the left adjoint to the restriction functor in the category of solid $\overline{\mathbb{Q}}_{\ell}$-sheaves \cite[Proposition~VII.3.1]{FS}. 
\end{definition}
\begin{remark}
These satisfy various compatibilities with respect to composition and restriction to the diagonal. In particular, given two representations $V,W \in \Rep_{\overline{\mathbb{Q}}_{\ell}}(\phantom{}^{L}G)$, we have that
\[ (T_{V} \times id)(T_{W})(\cdot)|_{\Delta} \simeq T_{V \otimes W}(\cdot) \]
where $\Delta: X \rightarrow X^{2}$ is the diagonal map. 
\end{remark}
We then consider $\Dlis(\Bun_{G},\overline{\mathbb{Q}}_{\ell})^{BW_{\mathbb{Q}_{p}}^{I}}$, the category of objects in $\Dlis(\Bun_{G},\overline{\mathbb{Q}}_{\ell})$ with continuous action by $W_{\mathbb{Q}_{p}}^{I}$. Examples of objects in this category are objects of $\Dlis(\Bun_{G},\overline{\mathbb{Q}}_{\ell})$ tensored by a continuous representation of $W_{\mathbb{Q}_{p}}^{I}$, for a more precise description see \cite[Section~IX.1]{FS}.  With this in hand, we then have the following theorem of Fargues-Scholze.
\begin{theorem}{\cite[Theorem~I.7.2,Proposition~IX.2.1,Corollary~IX.2.3]{FS}}
The Hecke operator $T_{W}$ for $W \in \Rep_{\overline{\mathbb{Q}}_{\ell}}(\phantom{}^{L}G^{I})$   
\[ T_{W}: \Dlis(\Bun_{G},\overline{\mathbb{Q}}_{\ell}) \rightarrow \D_{\blacksquare}(\Bun_{G} \times X^{I}) \] 
induces a functor 
\[ \Dlis(\Bun_{G},\overline{\mathbb{Q}}_{\ell}) \rightarrow \Dlis(\Bun_{G},\overline{\mathbb{Q}}_{\ell})^{BW_{\mathbb{Q}_{p}}^{I}} \]
and the induced endofunctors of $\Dlis(\Bun_{G},\overline{\mathbb{Q}}_{\ell})$ given by forgetting the Weil group action preserve compact and ULA objects.
\end{theorem}
\begin{remark}
This should be thought of as a manifestation of Drinfeld's Lemma, where (roughly) the \'etale fundamental group of $\Spd(\breve{\mathbb{Q}}_{p})/\Frob^{\mathbb{Z}} = X$ should be the same as $W_{\mathbb{Q}_{p}}$. 
\end{remark}
From now on, when talking about Hecke operators we shall always refer to this induced functor, which we will also abusively denote by $T_{W}$. Theorem 3.2 has direct implications for the cohomology of local Shimura varieties. To study this, consider a minuscule  cocharacter $\mu$ with field of definition $E$, and let $b \in B(G,\mu)$ be the unique basic element in the $\mu$-admissible locus (see \cite[Definition~2.5]{RV}). We say that the triple $(G,b,\mu)$ defines a local Shimura datum in the sense of Rapoport-Viehmann \cite{RV}. Attached to such a datum, Scholze-Weinstein \cite{SW2} construct a tower of diamonds
\[ p_{K}: (\Sht(G,b,\mu)_{K})_{K \subset G(\mathbb{Q}_{p})} \rightarrow \Spd(\Breve{E})\]
for varying open compact $K \subset G(\mathbb{Q}_{p})$. This is obtained by considering the space $\Sht(G,b,\mu)_{\infty}$ which parametrizes modifications $ \mathcal{E}_{b} \rightarrow \mathcal{E}_{0}$ with meromorphy bounded by $\mu$, where $\mathcal{E}_{b}$ (resp. $\mathcal{E}_{0}$) is the bundle corresponding to $b \in B(G)$ (resp. the trivial bundle) on the Fargues-Fontaine curve. It has commuting actions by $G(\mathbb{Q}_{p})$ and $J_{b}(\mathbb{Q}_{p})$ given by acting via automorphisms on $\mathcal{E}_{0}$ and $\mathcal{E}_{b}$, respectively. The tower is then given by considering the quotients of this space under $K$, for varying open compact $K \subset G(\mathbb{Q}_{p})$ under the action of $G(\mathbb{Q}_{p})$. 
\begin{definition}
Let $\Sht(G,b,\mu)_{K,\mathbb{C}_{p}}$ be the base-change of the above tower to $\mathbb{C}_{p}$. We define the complex 
\[ R\Gamma_{c}(G,b,\mu) := \colim_{K \rightarrow \{1\}} R\Gamma_{c}(\Sht(G,b,\mu)_{K,\mathbb{C}_{p}},\mathbb{Z}_{\ell}) \]
a colimit of $W_{E} \times J_{b}(\mathbb{Q}_{p})$-modules with a $G(\mathbb{Q}_{p})$-action, where $W_{E}$ is the Weil group of $E$. The complex $R\Gamma_{c}(\Sht(G,b,\mu)_{K,\mathbb{C}_{p}},\mathbb{Z}_{\ell})$ is naturally a smooth $J_{b}(\mathbb{Q}_{p})$-representation with a continuous action of $W_{E}$, by \cite[Theorem~IX.3.1]{FS} (cf. \cite[Proposition~2.3]{Han}). A priori it only has an action by the inertia group, but this space admits a non-effective Frobenius descent datum. We then define, for $\rho$ a smooth admissible $J_{b}(\mathbb{Q}_{p})$-representation, the complex 
\[ R\Gamma_{c}(G,b,\mu)[\rho] := \colim_{K \ra \{1\}} R\Gamma_{c}(G,b,\mu)_{K} \otimes^{\mathbb{L}}_{\mathcal{H}(J_{b})} \rho, \]
where $\mathcal{H}(J_{b})$ is the usual smooth Hecke algebra. We also define 
\[ R\Gamma_{c}^{\flat}(G,b,\mu)[\rho] := \colim_{K \ra \{1\}} R\mathcal{H}om_{J_{b}(\mathbb{Q}_{p})}(R\Gamma_{c}(G,b,\mu)_{K},\rho)[-2d](-d). \]
Similarly, for $\pi$ a smooth admissible $G(\mathbb{Q}_{p})$-representation, we define $R\Gamma_{c}(G,b,\mu)[\pi]$ and $R\Gamma_{c}^{\flat}(G,b,\mu)[\pi]$. 
\end{definition}
\begin{remark}
 We note that, by Hom-Tensor duality, $R\mathcal{H}om_{\D(G(\bb{Q}_{p})}(R\Gamma_{c}(G,b,\mu)[\rho],\overline{\mathbb{Q}}_{\ell})[-2d](-d)$ is isomorphic to $R\Gamma_{c}^{\flat}(G,b,\mu)[\rho^{*}]$, where $\rho^{*}$ is the contragredient. We will end up using both of these cohomology groups throughout this manuscript. The former is more natural from the point of view of basic uniformization, while the latter is amenable to analysis using the results of Hansen-Kaletha-Weinstein \cite{KW} on the Kottwitz conjecture.
\end{remark}
To study these complexes, we specialize the above discussion of Hecke operators to the case where $W = V_{\mu^{-1}}$ is specified by the highest weight representation of highest weight $\mu^{-1}$ a dominant inverse of $\mu$ and $I = \{\ast\}$ is a singleton. The sheaf $\mathcal{S}_{W}$ will then be supported on the closed subspace $\Hck_{\leq \mu^{-1}} = \Hck_{\mu^{-1}}$ of $\Hck$, parametrizing modifications with meromorphy bounded by or equal to $\mu^{-1}$, where the equality follows by the minuscule assumption. The space $\Hck_{\mu^{-1}}$ is cohomologically smooth of dimension $d := \langle 2\rho_{G},\mu \rangle$, and the sheaf $\mathcal{S}_{W}$, as in the geometric Satake correspondence of \cite{MV}, behaves like the intersection cohomology of this space, so we have $\mathcal{S}_{W} \simeq \overline{\mathbb{Q}}_{\ell}[d](\frac{d}{2})$. This implies that, to study the action of the Hecke operator $T_{W}$ on $\Bun_{G}$, we can look at the restriction of the diagram defining the Hecke correspondence to this subspace
\[ \begin{tikzcd}
& & \arrow[dl,"h_{\mu^{-1}}^{\leftarrow}"] \Hck_{\mu^{-1}} \arrow[dr,"h_{\mu^{-1}}^{\rightarrow} \times supp"] & & \\
& \Bun_{G} & & \Bun_{G} \times \Spd(\Breve{E})/\Frob^{\mathbb{Z}}  & 
\end{tikzcd} \]
In particular, we have an isomorphism:
\[ T_{\mu^{-1}}(A) := T_{W}(A) \simeq R(h_{\mu^{-1}}^{\rightarrow} \times supp)_{\natural}(h_{\mu^{-1}}^{\leftarrow *}(A))[d](\frac{d}{2}). \]
Now consider a smooth admissible representation $\pi$ of $G(\mathbb{Q}_{p})$ and apply the Hecke operator to the sheaf:
\[ j_{!}(\mathcal{F}_{\pi}). \]
Then the fiber of $\Hck_{\mu^{-1}}$ of $h_{\mu^{-1}}^{\leftarrow}$ over $\Bun_{G}^{\mathbf{1}}$ is identified with
\[ [Gr_{G,\mu^{-1}}/\underline{G(\mathbb{Q}_{p})}] \]
the Schubert cell/variety associated to $\mu^{-1}$ in the $B_{dR}^{+}$-affine Grassmannian, quotiented out by $G(\mathbb{Q}_{p})$ acting on the trivial bundle via automorphisms. The sheaf
\[ T_{\mu^{-1}}j_{!}(\mathcal{F}_{\pi}) \]
is then supported on the HN-strata given by the Kottwitz elements in $B(G,\mu)$ since, by \cite[Proposition~A.9]{R}, any $G$-bundle occurring as a modification of type $\mu^{-1}$ of the trivial bundle has associated Kottwitz element lying in this set. We then consider the restriction
\[ j_{b}^{*}T_{\mu^{-1}}j_{!}(\mathcal{F}_{\pi}) \in \D(J_{b}(\mathbb{Q}_{p}),\overline{\mathbb{Q}}_{\ell})^{BW_{E}} \]
where $j_{b}: \Bun_{G}^{b} \hookrightarrow \Bun_{G}$ is the inclusion of the open HN-strata defined by $b$. The complex $j_{b}^{*}T_{\mu^{-1}}j_{!}(\mathcal{F}_{\pi})$ will be computed in terms of the cohomology of sheaves supported on the Newton strata
\[ [Gr_{G,\mu^{-1}}^{b}/\underline{G(\mathbb{Q}_{p})}] \]
parametrizing modifications of type $\mu^{-1}$ of the trivial bundle such that the resulting bundle has associated Kottwitz element of type $b$ after pulling back to each geometric point, modulo automorphisms of the trivial bundle. The space $\Sht(G,b,\mu)_{\infty}$ defined above is a pro-\'etale $\underline{J_{b}(\mathbb{Q}_{p})}$-torsor with respect to the $J_{b}(\mathbb{Q}_{p})$-action by automorphisms of $\mathcal{E}_{b}$
\[ \Sht(G,b,\mu)_{\infty} \rightarrow Gr_{G,\mu^{-1}}^{b} \]
over this Newton strata. Using this description of the infinite level Shimura variety, it then follows from base change (see \cite[Section~IX.3]{FS} for details) that we have an isomorphism 
\[ R\Gamma_{c}(G,b,\mu)[\pi][d](\frac{d}{2}) \simeq  j_{b} ^{*}T_{\mu^{-1}}j_{!}(\mathcal{F}_{\pi}) \in \D(J_{b}(\mathbb{Q}_{p}),\overline{\mathbb{Q}}_{\ell})^{BW_{E}}  \]
of $J_{b}(\mathbb{Q}_{p}) \times W_{E}$-modules. For our purposes, it will be also useful to have a description of the $\rho$-isotypic part of this cohomology in terms of Hecke operators, for $\rho$ a smooth irreducible representation of $J_{b}(\mathbb{Q}_{p})$. In particular, analysis similar to the above gives us an isomorphism
\[ R\Gamma_{c}(G,b,\mu)[\rho][d](\frac{d}{2}) \simeq j_{\mathbf{1}}^{*}T_{\mu}j_{b!}(\mathcal{F}_{\rho}) \]
as $G(\mathbb{Q}_{p}) \times W_{E}$-modules. We record these two isomorphisms as a corollary of the above discussion. 
\begin{corollary}{\label{cor: isotypcpartcalculation}}
Consider a local Shimura datum $(G,b,\mu)$ as above, and let $\pi$ (resp. $\rho$) a smooth irreducible representation of $G(\mathbb{Q}_{p})$ (resp. $J_{b}(\mathbb{Q}_{p})$). There exists an isomorphism
\[ R\Gamma_{c}(G,b,\mu)[\rho][d](\frac{d}{2}) \simeq  j_{\mathbf{1}}^{*}T_{\mu}j_{b!}(\mathcal{F}_{\rho}) \]
of complexes of $G(\mathbb{Q}_{p}) \times W_{E}$-modules and an isomorphism
\[ R\Gamma_{c}(G,b,\mu)[\pi][d](\frac{d}{2}) \simeq j_{b}^{*}T_{\mu^{-1}}j_{\mathbf{1}!}(\mathcal{F}_{\pi}) \]
of complexes of $J_{b}(\mathbb{Q}_{p}) \times W_{E}$-modules. 
\end{corollary}
We have the following basic structural result which, in more generality, follows from the analysis in Fargues-Scholze, but, in the case of a local Shimura datum, also partially follows from standard finiteness results for rigid spaces (see \cite[Section~6]{RV}). In particular, one can show the following.  
\begin{theorem}{\cite[Corollary~I.7.3,Page~317]{FS}}
For a local Shimura datum $(G,b,\mu)$ as above and $\rho$ a smooth irreducible representation, the cohomology groups of $R\Gamma_{c}^{\flat}(G,b,\mu)[\rho]$ and $R\Gamma_{c}(G,b,\mu)[\rho]$ are valued in smooth admissible $G(\mathbb{Q}_{p})$-representations of finite length with an action of $W_{E}$. 
\end{theorem}
\begin{remark}{\label{remark: ULAsheaves}}
A sheaf $\mathcal{F} \in \Dlis(\Bun_{G},\overline{\mathbb{Q}}_{\ell})$ being ULA is equivalent to its stalks at different HN-strata being valued in complexes of smooth admissible representations \cite[Theorem~V.7.1,Proposition~VII.7.9]{FS}, so indeed the admissibility of the above complex is a consequence of Theorem 3.2 and Corollary 3.3.
\end{remark}
Fargues-Scholze use the endofunctors defined by the Hecke algebra on $\Dlis(\Bun_{G},\overline{\mathbb{Q}}_{\ell})$  to define the excursion algebra.
\begin{definition}
For a finite set $I$, a representation $W \in \Rep_{\overline{\mathbb{Q}}_{\ell}}(^LG^{I})$, maps $\alpha: \overline{\mathbb{Q}}_{\ell} \rightarrow \Delta^{*}W$ and $\beta: \Delta^{*}W \rightarrow \overline{\mathbb{Q}}_{\ell}$, and elements $\gamma_{i} \in W_{\mathbb{Q}_{p}}$ for $i \in  I$, one defines the excursion operator on $\Dlis(\Bun_{G},\overline{\mathbb{Q}}_{\ell})$ to be the composition:
\[ id = T_{\overline{\mathbb{Q}}_{\ell}} \xrightarrow{\alpha} T_{\Delta^{*}W} = T_{W} \xrightarrow{(\gamma_{i})_{i \in I}} T_{W} = T_{\Delta^{*}W}  \xrightarrow{\beta} T_{\overline{\mathbb{Q}}_{\ell}} = id \]
where $\Delta^{*}W$ is the precomposition of $W$ with the diagonal embedding $\phantom{}^{L}G \ra \phantom{}^{L}G^{I}$.
\end{definition}
This defines a natural endomorphism of the identity functor on $\Dlis(\Bun_{G},\overline{\mathbb{Q}}_{\ell})$. If one looks at the induced endofunctor given by the inclusion $\D(G(\mathbb{Q}_{p}),\overline{\mathbb{Q}}_{\ell}) \subset \Dlis(\Bun_{G},\overline{\mathbb{Q}}_{\ell})$ induced by the open immersion $j: \Bun_{G}^{\mathbf{1}} \hookrightarrow \Bun_{G}$ 
then one obtains a natural endomorphism of the identity functor on $\D(G(\mathbb{Q}_{p}),\overline{\mathbb{Q}}_{\ell})$. In other words, we get a family of compatible endomorphisms for all complexes of smooth representations of $G(\mathbb{Q}_{p})$; namely, an element of the Bernstein center. One can verify that this excursion algebra satisfies similar properties to that considered by V. Lafforgue, so, using Lafforgue's reconstruction theorem \cite[Proposition~11.7]{VL}, one can show the following.
\begin{theorem}{\cite[Proposition~I.9.1]{FS}}
To an irreducible smooth $\overline{\mathbb{Q}}_{\ell}$-representation $\pi$ of $G(\mathbb{Q}_{p})$ (or more generally $A \in \Dlis(\Bun_{G},\overline{\mathbb{Q}}_{\ell})$ any Schur-irreducible object (i.e. $\End(A) = \overline{\mathbb{Q}}_{\ell}$)), there is a unique (up to conjugacy) continuous semisimple map
\[ \phi^{\mathrm{FS}}_{\pi}: W_{\mathbb{Q}_{p}} \rightarrow \phantom{}^{L}G(\overline{\mathbb{Q}}_{\ell}) \]
characterized by the property that for all $I,W,\alpha,\beta$, and $\gamma_{i} \in W_{\mathbb{Q}_{p}}$ for $i \in I$, the corresponding endomorphism of $\pi$ defined above is given by multiplication by the scalar that results from the composite
\[ \overline{\mathbb{Q}}_{\ell} \xrightarrow{\alpha} \Delta^{*}W = W \xrightarrow{(\phi_{\pi}(\gamma_{i}))_{i \in I}} W = \Delta^{*}W \xrightarrow{\beta} \overline{\mathbb{Q}}_{\ell} \]
\end{theorem} 
By further studying the geometry of $\Bun_{G}$ and the Hecke stacks, one can deduce various good properties of this correspondence.
\begin{theorem}{\cite[Theorem~I.9.6]{FS},\cite[Theorem~I.0.3]{KW}}
The mapping defined above 
\[ \pi \mapsto \phi^{\mathrm{FS}}_{\pi} \]
enjoys the following properties:
\begin{enumerate}
    \item (Compatibility with Local Class Field Theory) If $G = T$ is a torus, then $\pi \mapsto \phi_{\pi}$ is the usual local Langlands correspondence 
    \item The correspondence is compatible with character twists, passage to contragredients, and central characters.
    \item (Compatibility with products) Given two irreducible representations $\pi_{1}$ and $\pi_{2}$ of two connected reductive groups $G_{1}$ and $G_{2}$ over $\mathbb{Q}_{p}$, respectively, we have
    \[ \pi_{1} \boxtimes \pi_{2} \mapsto \phi^{\mathrm{FS}}_{\pi_{1}} \times \phi^{\mathrm{FS}}_{\pi_{2}}\]
    under the Fargues-Scholze local Langlands correspondence for $G_{1} \times G_{2}$. 
    \item (Compatibility with parabolic induction) Given a parabolic subgroup $P \subset G$ with Levi factor $M$ and an irreducible representation $\pi_{M}$ of $M$, then the semisimple L-parameter corresponding to any sub-quotient of $ind_{P}^{G}(\pi_{M})$ the (normalized) parabolic induction is the composition
    \[ W_{\mathbb{Q}_{p}}\xrightarrow{\phi^{\mathrm{FS}}_{\pi_{M}}} \\  ^{L}M(\overline{\mathbb{Q}}_{\ell}) \rightarrow ^{L}G(\overline{\mathbb{Q}}_{\ell}) \]
    where the map $\phantom{}^{L}M(\overline{\mathbb{Q}}_{\ell}) \rightarrow \phantom{}^{L}G(\overline{\mathbb{Q}}_{\ell})$ is the natural embedding. 
    \item (Compatibility with Harris-Taylor/Henniart LLC)
    For $G = \GL_{n}$ or an inner form of $G$ the semisimple L-parameter associated to $\pi$ is the (semi-simplified) parameter $\phi_{\pi}^{\mathrm{ss}}$ associated to $\pi$ by Harris-Taylor/Henniart.
    \item (Compatibility with Restriction of Scalars) The above story works the same for $G'$ a connected reductive group over any finite extension $E'/\mathbb{Q}_{p}$, where one then gets a semisimple L-parameter valued on $W_{E'}$. If $G = \mathrm{Res}_{E'/\mathbb{Q}_{p}}G'$ is the Weil restriction of some $G'/E'$ then $L$-parameters for $G/\mathbb{Q}_{p}$ agree with $L$-parameters for $G'/E'$ in the usual sense.
    \item (Compatibility with Isogenies) If $G' \rightarrow G$ is a map of reductive groups inducing an isomorphism of adjoint groups, $\pi$ is an irreducible smooth representation of $G(E)$ and $\pi'$ is an irreducible constituent of $\pi|_{G'(E)}$ then $\phi_{\pi'}$ is the image of $\phi_{\pi}$ under the induced map $\hat{G} \rightarrow \hat{G'}$. 
\end{enumerate}
\end{theorem}
\begin{remark}
In (5), the compatibility of the Fargues-Scholze local Langlands correspondence with the Harris-Taylor/Henniart local Langlands correspondence for an arbitrary inner form of $\GL_{n}$ is not included in the paper of Fargues-Scholze \cite{FS}. However, it follows from the work of Hansen-Kaletha-Weinstein on the Kottwitz conjecture \cite[Theorem~1.0.3]{KW}.
\end{remark}
\subsection{The Spectral Action}
With these basic structural properties out of the way, we turn our attention to the "spectral action" on $\Dlis(\Bun_{G},\overline{\mathbb{Q}}_{\ell})$, which will be very important to proving compatibility of the two correspondences in the case where the parameter is supercuspidal, as well as deducing applications to the Kottwitz conjecture. We recall that an $L$-parameter over $\overline{\mathbb{Q}}_{\ell}$ can be thought of as a continuous (not necessarily semisimple, but usually Frobenius semi-simple) homomorphism 
\[ \phi: W_{\mathbb{Q}_{p}} \rightarrow \phantom{}^{L}G(\overline{\mathbb{Q}}_{\ell}) \]
commuting with the natural projection to $Q$. Assuming the Frobenius semi-simplicity of $\phi$, one can use the classical construction of Deligne-Grothendieck to see that this coincides with the definition given in section $2$ for $\GSp_{4}$ after applying the isomorphism $i$, where the monodromy operator is recovered by taking the logarithm of the action of a sufficiently small open subgroup of inertia, normalized by the $\ell$-adic tame character. Such a continuous map can be thought of as a continuous $1$-cocycle $W_{\mathbb{Q}_{p}} \rightarrow \hat{G}(\overline{\mathbb{Q}}_{\ell})$, with respect to the action of $W_{\mathbb{Q}_{p}}$ on $\hat{G}(\overline{\mathbb{Q}}_{\ell})$. If we let $A/\mathbb{Z}_{\ell}$ be any $\mathbb{Z}_{\ell}$-algebra endowed with a topology given by writing $A = \colim_{A' \subset A} A'$, where $A'$ is a finitely generated $\mathbb{Z}_{\ell}$-module with its $\ell$-adic topology, then we can define a moduli space, denoted $\mathcal{Z}^{1}(W_{\mathbb{Q}_{p}},\hat{G})$, over $\mathbb{Z}_{\ell}$, whose $A$-points are the continuous $1$-cocycles $W_{\mathbb{Q}_{p}} \rightarrow \hat{G}(A)$ with respect to the natural action of $W_{\mathbb{Q}_{p}}$ on $\hat{G}(A)$. This defines a scheme considered in \cite{DH} and \cite{Zhu1} which, by \cite[Theorem~I.8.1]{FS}, can be written as a union of open and closed affine subschemes $\mathcal{Z}^{1}(W_{\mathbb{Q}_{p}}/P,\hat{G})$ as $P$ runs through sufficiently small open normal subgroups of wild inertia of $W_{\bb{Q}_{p}}$, where each $\mathcal{Z}^{1}(W_{\mathbb{Q}_{p}}/P,\hat{G})$ is a flat local complete intersection over $\mathbb{Z}_{\ell}$ of dimension $\dim(G)$. This allows us to consider the Artin stack quotient $[\mathcal{Z}^{1}(W_{\mathbb{Q}_{p}},\hat{G})/\hat{G}]$, where $\hat{G}$ acts via conjugation. We then consider the base change to $\overline{\mathbb{Q}}_{\ell}$, denoted $[\mathcal{Z}^{1}(W_{\mathbb{Q}_{p}},\hat{G})_{\overline{\mathbb{Q}}_{\ell}}/\hat{G}]$ and referred to as the stack of Langlands parameters, as well as the category $\Perf([\mathcal{Z}^{1}(W_{\mathbb{Q}_{p}},\hat{G})_{\overline{\mathbb{Q}}_{\ell}}/\hat{G}])$ of perfect complexes on this space. We let $\Dlis(\Bun_{G},\overline{\mathbb{Q}}_{\ell})^{\omega}$ denote the triangulated sub-category of compact objects in $\Dlis(\Bun_{G},\overline{\mathbb{Q}}_{\ell})$ (which are precisely the objects with quasi-compact support on $\Bun_{G}$ and which restrict to compact objects in $\D(J_{b}(\mathbb{Q}_{p}),\overline{\mathbb{Q}}_{\ell})$ for all $b \in B(G)$ by \cite[Theorem~V.4.1,Proposition~VII.7.4]{FS}). We then have the key theorem of Fargues-Scholze.
\begin{theorem}{\cite[Corollary X.1.3]{FS}}
There exists a natural compactly supported $\overline{\mathbb{Q}}_{\ell}$-linear action of $\Perf(\mathcal{Z}^{1}(W_{\mathbb{Q}_{p}},\hat{G})_{\overline{\mathbb{Q}}_{\ell}}/\hat{G})$ on $\Dlis(\Bun_{G},\overline{\mathbb{Q}}_{\ell})^{\omega}$ satisfying the property that the restriction along the map
\[ \Rep_{\overline{\mathbb{Q}}_{\ell}}(\phantom{}^{L}G^{I}) \rightarrow \Perf([\mathcal{Z}^{1}(W_{\mathbb{Q}_{p}},\hat{G})_{\overline{\mathbb{Q}}_{\ell}}/\hat{G}])^{BW^{I}_{\mathbb{Q}_{p}}} \]
induces the action of Hecke operators
\[ \Rep_{\overline{\mathbb{Q}}_{\ell}}(\phantom{}^{L}G^{I}) \rightarrow \End(\Dlis(\Bun_{G},\overline{\mathbb{Q}}_{\ell})^{\omega})^{BW_{\mathbb{Q}_{p}}^{I}} \]
for a varying finite index set $I$.
\end{theorem}
\begin{remark}
\begin{enumerate}
    \item Here the map
    \[ \Rep_{\overline{\mathbb{Q}}_{\ell}}(\phantom{}^{L}G) \rightarrow \Perf([\mathcal{Z}^{1}(W_{\mathbb{Q}_{p}},\hat{G})_{\overline{\mathbb{Q}}_{\ell}}/\hat{G}])^{BW_{\mathbb{Q}_{p}}} \]
    associates to a representation $V$, with associated map $r_{V}: \phantom{}^{L}G(\overline{\mathbb{Q}}_{\ell}) \rightarrow GL(V)(\overline{\mathbb{Q}}_{\ell})$ a vector bundle on $[\mathcal{Z}^{1}(W_{\mathbb{Q}_{p}},\hat{G})_{\overline{\mathbb{Q}}_{\ell}}/\hat{G})]$ of rank equal to $\dim(V)$ with $W_{\mathbb{Q}_{p}}$-action. This bundle, denoted $C_{V}$, has the property that its evaluation at a $\overline{\mathbb{Q}}_{\ell}$-point corresponding to a parameter $\phi: W_{\mathbb{Q}_{p}} \rightarrow \phantom{}^{L}G(\overline{\mathbb{Q}}_{\ell})$ is precisely $r_{V} \circ \phi$, and it can be described as pulling back along the natural map to $[\mathrm{Spec}(\overline{\bb{Q}}_{\ell})/\hat{G}]$.
    \item The compactly supported condition means that, for all $\mathcal{F} \in \Dlis(\Bun_{G},\overline{\mathbb{Q}}_{\ell})^{\omega}$, the functor $\Perf([\mathcal{Z}^{1}(W_{\mathbb{Q}_{p}},\hat{G})_{\overline{\mathbb{Q}}_{\ell}}/\hat{G}]) \rightarrow \Dlis(\Bun_{G},\overline{\mathbb{Q}}_{\ell})^{\omega}$ induced by acting on $\mathcal{F}$ factors through an action of  $\Perf([\mathcal{Z}^{1}(W_{\mathbb{Q}_{p}}/P,\hat{G})_{\overline{\mathbb{Q}}_{\ell}}/\hat{G})])$, where $P$ is a subgroup of wild inertia. Fargues and Scholze state this action in terms of a $(\infty,1)$-category acting on an $(\infty,1)$-category, we suppress this technicality for simplicity. However, we note that this enhanced action has the concrete implication that if one has a morphism, cone, or homotopy limit/colimit in the derived category $\Perf([\mathcal{Z}^{1}(W_{\mathbb{Q}_{p}},\hat{G})_{\overline{\mathbb{Q}}_{\ell}}/\hat{G}])$ that acting on an object $A \in \Dlis(\Bun_{G},\overline{\mathbb{Q}}_{\ell})^{\omega}$ will produce a corresponding morphism, cone, or homotopy limit/colimit. \item In fact, Fargues-Scholze show that giving such a compactly supported $\overline{\mathbb{Q}}_{\ell}$-linear action is equivalent (when properly formulated) to giving $\Rep_{\overline{\mathbb{Q}}_{\ell}}(Q^{I})$-linear monoidal functors
    \[ \Rep_{\overline{\mathbb{Q}}_{\ell}}(\phantom{}^{L}G^{I}) \rightarrow \End_{\overline{\mathbb{Q}}_{\ell}}(\Dlis(\Bun_{G},\overline{\mathbb{Q}}_{\ell})^{\omega})^{BW^{I}_{\mathbb{Q}_{p}}}\]
    functorial in $I$. In the case that $I = \{*\}$, the fact that the Hecke action satisfies this monoidal property is precisely Remark 3.2.
    \end{enumerate}
\end{remark}
To study this spectral action, we consider, as in \cite[Section~VIII.3]{FS}, the coarse quotient in the category of schemes
\[ \mathcal{Z}^{1}(W_{\mathbb{Q}_{p}},\hat{G})_{\overline{\mathbb{Q}}_{\ell}} / / \hat{G}\]
of 
$\mathcal{Z}^{1}(W_{\mathbb{Q}_{p}},\hat{G})_{\overline{\mathbb{Q}}_{\ell}}$ by the action of $\hat{G}$ via conjugation. Given an $L$-parameter $\phi:W_{\mathbb{Q}_{p}} \rightarrow \phantom{}^{L}G(\overline{\mathbb{Q}}_{\ell})$, it follows by \cite[Theorem~4.13]{DH} or \cite[Proposition~VIII.3.2]{FS}, that the $\hat{G}$-orbit of $\phi$ in $\mathcal{Z}^{1}(W_{\mathbb{Q}_{p}},\hat{G})_{\overline{\mathbb{Q}}_{\ell}}$ is closed if and only if $\phi$ is semisimple. Consequently, the $\ol{\bb{Q}}_{\ell}$-points of $\mathcal{Z}^{1}(W_{\mathbb{Q}_{p}},\hat{G})_{\overline{\mathbb{Q}}_{\ell}} / / \hat{G}$ are naturally indexed by semi-simple parameters. Moreover, the natural map 
\[ \pi: [\mathcal{Z}^{1}(W_{\mathbb{Q}_{p}},\hat{G})_{\overline{\mathbb{Q}}_{\ell}}/\hat{G}] \rightarrow \mathcal{Z}^{1}(W_{\mathbb{Q}_{p}},\hat{G})_{\overline{\mathbb{Q}}_{\ell}} / / \hat{G}\]
evaluated on a $\overline{\mathbb{Q}}_{\ell}$-point in the stack quotient defined by an $L$-parameter $\phi$ defines a closed $\overline{\mathbb{Q}}_{\ell}$-point in the coarse moduli space given by its semisimplification $\phi^{\mathrm{ss}}$. We can fit excursion operators into this picture as follows. We let $\mathcal{Z}^{spec}(G,\overline{\mathbb{Q}}_{\ell}) := \mathcal{O}(\mathcal{Z}^{1}(W_{\mathbb{Q}_{p}},\hat{G})_{\overline{\mathbb{Q}}_{\ell}})^{\hat{G}}$ be the ring of functions on the stack of $L$-parameters/the coarse moduli space, which we refer to as the spectral Bernstein center. As noted above, the excursion operators define a family of commuting endomorphisms of the identity functor on $\Dlis(\Bun_{G},\overline{\mathbb{Q}}_{\ell})$. We let $\mathcal{Z}^{geom}(G,\overline{\mathbb{Q}}_{\ell})$ be the ring of such endomorphisms as in \cite[Definition~IX.0.2]{FS}, which we refer to as the geometric Bernstein center. In \cite[Corollary~IX.0.3]{FS}, Fargues and Scholze construct a canonical map of rings
\[ \mathcal{Z}^{spec}(G,\overline{\mathbb{Q}}_{\ell}) \rightarrow \mathcal{Z}^{geom}(G,\overline{\mathbb{Q}}_{\ell}) \]
which is given by excursion operators in the following sense. By \cite[Theorem~VIII.3.6]{FS}, there is an identification between $\mathcal{Z}^{spec}(G,\overline{\mathbb{Q}}_{\ell})$ and the algebra of excursion operators. In particular, an excursion operator, as in Definition 3.3, associated to the datum $I$, $W$, $\alpha$, $\beta$, and $\gamma_{i} \in W_{\mathbb{Q}_{p}}$ for $i \in I$ defines a function $f_{I,W,\alpha,\beta,(\gamma_{i})_{i \in I}} \in \mathcal{O}_{[\mc{Z}^{1}(W_{\mathbb{Q}_{p}},\hat{G})_{\overline{\mathbb{Q}}_{\ell}}/\hat{G}]} = \mathcal{Z}^{spec}(G,\overline{\mathbb{Q}}_{\ell})$ on  $\mathcal{Z}^{1}(W_{\mathbb{Q}_{p}},\hat{G})_{\overline{\mathbb{Q}}_{\ell}} / / \hat{G}$, whose evaluation on the closed point of the coarse moduli space associated to a semisimple parameter $\phi: W_{\mathbb{Q}_{p}} \rightarrow \phantom{}^{L}G(\overline{\mathbb{Q}}_{\ell})$ is precisely the scalar that results from the endomorphism:
 \[ \overline{\mathbb{Q}}_{\ell} \xrightarrow{\alpha} \Delta^{*}W = W \xrightarrow{(\phi(\gamma_{i}))_{i \in I}} W = \Delta^{*}W \xrightarrow{\beta} \overline{\mathbb{Q}}_{\ell} \]
We note that multiplication by $f_{I,W,\alpha,\beta,(\gamma_{i})_{i \in I}}$ defines an endomorphism
\[ \mathcal{O}_{[\mc{Z}^{1}(W_{\mathbb{Q}_{p}},\hat{G})_{\overline{\mathbb{Q}}_{\ell}}/\hat{G}]} \rightarrow \mathcal{O}_{[\mc{Z}^{1}(W_{\mathbb{Q}_{p}},\hat{G})_{\overline{\mathbb{Q}}_{\ell}}/\hat{G}]} \]
of the structure sheaf on the Artin stack $[\mc{Z}^{1}(W_{\mathbb{Q}_{p}},\hat{G})_{\overline{\mathbb{Q}}_{\ell}}/\hat{G}]$. If we act on a Schur-irreducible object $A \in \Dlis(\Bun_{G},\overline{\mathbb{Q}}_{\ell})^{\omega}$ then we obtain an endomorphism 
\[ \{ \mathcal{O}_{[\mc{Z}^{1}(W_{\mathbb{Q}_{p}},\hat{G})_{\overline{\mathbb{Q}}_{\ell}}/\hat{G}]} \star A = A \rightarrow \mathcal{O}_{[\mc{Z}^{1}(W_{\mathbb{Q}_{p}},\hat{G})_{\overline{\mathbb{Q}}_{\ell}}/\hat{G}]} \star A = A \} \in \End(A) = \overline{\mathbb{Q}}_{\ell} \]
which will be precisely the scalar given by evaluating $\phi_{A}^{\mathrm{FS}}$ on the excursion datum (see \cite[Theorem~5.2.1]{Zou}). In this way, we see that the action of excursion operators can be obtained through the spectral action. We will leverage this interpretation of excursion operators to prove the following key lemma.
\begin{lemma}{\label{lemma: SchurIrreducibleSupport}}
Let $A \in \Dlis(\Bun_{G},\overline{\mathbb{Q}}_{\ell})^{\omega}$ be any Schur-irreducible object with Fargues-Scholze parameter $\phi_{A}^{\mathrm{FS}}$. Set $x$ to be the closed point defined by the parameter $\phi_{A}^{\mathrm{FS}}$ in the coarse moduli space $\mathcal{Z}^{1}(W_{\mathbb{Q}_{p}},\hat{G})_{\overline{\mathbb{Q}}_{\ell}} / / \hat{G}$. Let $\pi^{-1}(x)$ denote the closed subset defined by the preimage in $[\mathcal{Z}^{1}(W_{\mathbb{Q}_{p}},\hat{G})_{\overline{\mathbb{Q}}_{\ell}}/\hat{G}]$.  Suppose we have $C \in \Perf([\mathcal{Z}^{1}(W_{\bb{Q}_{p}},\hat{G})_{\overline{\mathbb{Q}}_{\ell}}/\hat{G}])$ with support disjoint from $\pi^{-1}(x)$ then $C$ acts by zero on $A$ via the spectral action.
\end{lemma}
\begin{proof}
If we look at the action on $A \in \Dlis(\Bun_{G},\overline{\mathbb{Q}}_{\ell})$ via the map
\[ \mathcal{Z}^{spec}(G,\overline{\mathbb{Q}}_{\ell}) \rightarrow \mathcal{Z}^{geom}(G,\overline{\mathbb{Q}}_{\ell}) \]
given by excursion operators this factors through the quotient by the maximal ideal $\mathfrak{m}_{A} \subset \mathcal{Z}^{spec}(G,\overline{\mathbb{Q}}_{\ell}) = \mathcal{O}(\mc{Z}^{1}(W_{\mathbb{Q}_{p}},\hat{G})_{\overline{\mathbb{Q}}_{\ell}})^{\hat{G}}$ defined by the closed point $\phi_{A}^{\mathrm{FS}}$ in the coarse moduli space. By the conditions on the support of $C$, this implies that we can write the identity element as $1 = 1_C + 1_A \in Z^{\mathrm{spec}}(G,\overline{\mathbb{Q}}_\ell)$, where $1_C$ is a function that annihilates $C$ and $1_A \in \mathfrak{m}_A$. Indeed, since $\hat{G}$ is reductive and we are in characteristic $0$, the map $\pi\colon [Z^1(W_{\mathbb{Q}_p},\hat{G})_{\overline{\mathbb{Q}}_\ell}/\hat{G}] \ra Z^1(W_{\mathbb{Q}_p},\hat{G})_{\overline{\mathbb{Q}}_\ell}/\!\!/\hat{G}$ is a good quotient. Hence it sends the closed $\hat{G}$-stable subset $\mathrm{supp}(C)$ to a closed subset $\pi(\mathrm{supp}(C))$, and $x \notin \pi(\mathrm{supp}(C))$ since $\mathrm{supp}(C)\cap\pi^{-1}(x) = \emptyset$. As the coarse quotient is affine, we can therefore choose $g \in \mathcal{Z}^{\mathrm{spec}}(G,\overline{\mathbb{Q}}_\ell)$ vanishing on $\pi(\mathrm{supp}(C))$ with $g(x) = 1$. Since $g$ vanishes on the support of the perfect complex $C$, it acts nilpotently on each of the finitely many cohomology sheaves of $C$, so $g^N$ acts by zero on $C$ for $N$ sufficiently large. We then set $1_C := g^N$ and $1_A := 1 - g^N$, which lies in $\mathfrak{m}_A$ since $g(x) = 1$

 We consider the spectral action of $C$ on $A$
\[ C \star A \in \Dlis(\Bun_{G},\overline{\mathbb{Q}}_{\ell}) \]
and look at the endomorphism induced by multiplication by $1$ on $C$
\[ C \star A \rightarrow C \star A \]
which is just the identity. However, since $1_{C}$ annihilates $C$, this is the same as the action of $1_{A}$ on $C \star A$, but, it follows by the above discussion that acting via multiplication by $1_{A}$ is the same as acting via the map
\[ \mathcal{Z}^{spec}(G,\overline{\mathbb{Q}}_{\ell}) \rightarrow \mathcal{Z}^{geom}(G,\overline{\mathbb{Q}}_{\ell}) \]
given by excursion operators, and the action of $1_{A}$ after applying this map is zero. This would lead to a contradiction unless $C \star A$ is also zero.
\end{proof}
To take advantage of this lemma, we now introduce the following endofunctors of $\Dlis(\Bun_{G},\overline{\mathbb{Q}}_{\ell})^{\omega}$, which are analogues of the averaging operators considered by \cite{AL} in the Fargues-Scholze geometric Langlands correspondence for $\GL_{n}$ and by \cite{FGV, Gait} in the classical geometric Langlands correspondence over function fields.
\begin{definition}
Let $\phi$ be a representation of $W_{\mathbb{Q}_{p}}$ and $V$ a representation of $\phantom{}^{L}G$ with $T_{V}$ the associated Hecke operator. We consider the endofunctor of $\Dlis(\Bun_{G},\overline{\mathbb{Q}}_{\ell})$
\[ A \mapsto R\Gamma(W_{\mathbb{Q}_{p}},T_{V}(A) \otimes \phi^{\vee})\]
where $R\Gamma(W_{\mathbb{Q}_{p}},-): \Dlis(\Bun_{G},\overline{\mathbb{Q}}_{\ell})^{BW_{\mathbb{Q}_{p}}} \rightarrow \Dlis(\Bun_{G},\overline{\mathbb{Q}}_{\ell})$ is the derived functor given by continuous group cohomology with respect to $W_{\mathbb{Q}_{p}}$. We denote this endofunctor by $Av_{V,\phi}: \Dlis(\Bun_{G},\overline{\mathbb{Q}}_{\ell}) \rightarrow \Dlis(\Bun_{G},\overline{\mathbb{Q}}_{\ell})$. 
\end{definition}
We now would like to realize the functor $Av_{V,\phi}$ as the spectral action of an object in $\Perf([\mathcal{Z}^{1}(W_{\mathbb{Q}_{p}},\hat{G})_{\overline{\mathbb{Q}}_{\ell}}/\hat{G})])$ similar to \cite[Section~5.5]{AL}. An obvious guess would be that one should take the vector bundle $C_{V}$ corresponding to the Hecke operator $T_{V}$, as in Remark 3.7 (1), and then twist this by the constant sheaf defined by $\phi^{\vee}$, which we denote by
\[ C_{V} \otimes \phi^{\vee} \in \Perf([\mathcal{Z}^{1}(W_{\mathbb{Q}_{p}},\hat{G})_{\overline{\mathbb{Q}}_{\ell}}/\hat{G}])^{BW_{\mathbb{Q}_{p}}}\]
More precisely, this is the vector bundle with $W_{\mathbb{Q}_{p}}$-action whose evaluation at a $\overline{\mathbb{Q}}_{\ell}$-point corresponding to a $L$-parameter $\tilde{\phi}: W_{\mathbb{Q}_{p}} \rightarrow \phantom{}^{L}G(\overline{\mathbb{Q}}_{\ell})$ is the vector space with $W_{\mathbb{Q}_{p}}$-action given by tensoring the representation
\[ \tilde{\phi}: W_{\mathbb{Q}_{p}} \rightarrow \phantom{}^{L}G(\overline{\mathbb{Q}}_{\ell}) \xrightarrow{r_{V}} GL(V)\]
with $\phi^{\vee}$. To obtain the desired perfect complex, it is natural to apply $R\Gamma(W_{\mathbb{Q}_{p}},-)$ to the vector bundle
$C_{V} \otimes \phi^{\vee}$, which we denote by $\mathcal{A}v_{V,\phi}$. We note that $\mathcal{A}v_{V,\phi}$ is a perfect complex. Indeed, as $p$ is invertible in $\overline{\mathbb{Q}}_{\ell}$, the wild inertia $P \subset W_{\mathbb{Q}_{p}}$ will always act through a finite quotient on $C_{V} \otimes \phi^{\vee}$ and has no higher cohomology which implies that the invariants
\[ (C_{V} \otimes \phi^{\vee})^{P} \]
are a direct summand of the vector bundle $C_{V} \otimes \phi^{\vee}$. If we choose a generator $\tau \in I/P$ in the tame quotient of the inertia subgroup $I \subset W_{\mathbb{Q}_{p}}$ together with a Frobenius lift $\sigma \in W_{\mathbb{Q}_{p}}/P$ then the complex $\mathcal{A}v_{V,\phi}$ is computed as the homotopy limit of the diagram: 
\begin{equation}
\begin{tikzcd}
& (C_{V} \otimes \phi^{\vee})^{P} \arrow[r,"\tau - 1"] \arrow[d,"\sigma - 1"] & (C_{V} \otimes \phi^{\vee})^{P} \arrow[d,"\sigma(1 + \tau + \ldots + \tau^{p - 1}) - 1"] \\
& (C_{V} \otimes \phi^{\vee})^{P} \arrow[r,"\tau - 1"] & (C_{V} \otimes \phi^{\vee})^{P} 
\end{tikzcd} 
\end{equation}
This gives a presentation of $\mathcal{A}v_{V,\phi}$ as a perfect complex, as in \cite[Page~21]{AL}. To see this, note that $\sigma$ and $\tau$ generate a dense discrete subgroup of $W_{\mathbb{Q}_{p}}/P$ subject to the relationship that $\sigma^{-1}\tau\sigma = \tau^{p}$ (where $\sigma$ is a Lift of geometric Frobenius), and it follows that the limit of the diagram is just $((C_{V} \otimes \phi^{\vee})^{P})^{W_{\mathbb{Q}_{p}}/P}$ and that the homotopy limit is in turn the derived functor $R\Gamma(W_{\mathbb{Q}_{p}},C_{V} \otimes \phi^{\vee}) \simeq R\Gamma(W_{\mathbb{Q}_{p}}/P,(C_{V} \otimes \phi^{\vee})^{P})$, where we have used the vanishing of the higher Galois cohomology with respect to $P$ for the last isomorphism.

Then we have the following Lemma which is a verbatim generalization of \cite[Lemma~5.7]{AL}.
\begin{lemma}
There exists a canonical identification 
\[ Av_{V,\phi}(-) \simeq \mathcal{A}v_{V,\phi} \star (-)\]
of endofunctors of $\Dlis(\Bun_{G},\overline{\mathbb{Q}}_{\ell})^{\omega}$
\end{lemma}
\begin{proof}
Equation (1) gives a diagram of perfect complexes on $[\mathcal{Z}^{1}(W_{\mathbb{Q}_{p}},\hat{G})_{\overline{\mathbb{Q}}_{\ell}}/\hat{G}]$. Acting via the spectral action on an object $\mathcal{F} \in \Dlis(\Bun_{G},\overline{\mathbb{Q}}_{\ell})$ then gives a diagram
\[
\begin{tikzcd}
& (T_{V}(\mathcal{F}) \otimes \phi^{\vee})^{P} \arrow[r,"\tau - 1"] \arrow[d,"\sigma - 1"] & (T_{V}(\mathcal{F}) \otimes \phi^{\vee})^{P} \arrow[d,"\sigma(1 + \tau + \ldots + \tau^{p - 1}) - 1"] \\
& (T_{V}(\mathcal{F}) \otimes \phi^{\vee})^{P} \arrow[r,"\tau - 1"] & (T_{V}(\mathcal{F}) \otimes \phi^{\vee})^{P} 
\end{tikzcd}\]
However, if we take the homotopy limit of the diagram in $(1)$, the claim follows from the fact that the spectral action commutes with homotopy limits, as noted in Remark 3.7 (2). 
\end{proof}
With this identification in hand, we can apply Lemma 3.8 to prove the following key consequence.
\begin{lemma}
Let $A \in \Dlis(\Bun_{G},\overline{\mathbb{Q}}_{\ell})^{\omega}$ be an ULA Schur-irreducible object with Fargues-Scholze parameter $\phi_{A}^{\mathrm{FS}}$, $V$ a representation of $\phantom{}^{L}G$, and $\phi$ an irreducible semisimple representation of $W_{\mathbb{Q}_{p}}$. If the cohomology sheaves of $T_{V}(A) \in \Dlis(\Bun_{G},\overline{\mathbb{Q}}_{\ell})^{BW_{\mathbb{Q}_{p}}}$ with respect to the standard $t$-structure on $\Dlis(\Bun_{G},\overline{\mathbb{Q}}_{\ell})^{BW_{\mathbb{Q}_{p}}}$  have a non-zero sub-quotient as $W_{\mathbb{Q}_{p}}$-modules with $W_{\mathbb{Q}_{p}}$-action given by $\phi$  then $r_{V} \circ \phi_{A}^{\mathrm{FS}}$ also has such a sub-quotient.
\end{lemma}
\begin{remark}
During the creation of this manuscript, a similar result was obtained by Koshikawa through a similar but simpler proof \cite[Theorem~1.3]{Kosh}. As we will similarly conclude in Section 3.3, he shows that if the cohomology $R\Gamma_{c}(G,b,\mu)[\rho](\frac{d}{2})$ admits a sub-quotient with $W_{E_{\mu}}$-action given by an irreducible representation $\phi$ then the Fargues-Scholze parameter $\phi_{\rho}^{\mathrm{FS}}$ admits a sub-quotient given by $\phi^{\vee}$. However, in his paper, the shtuka space $\Sht(G,b,\mu)_{\infty}$ parametrizes modifications of the form $\mathcal{E}_{0} \dashrightarrow \mathcal{E}_{b}$ of type $\mu$, whereas for us it parametrizes modifications of type $\mu^{-1}$. This explains why, under our conventions, there is no dual appearing. 
\end{remark}
\begin{proof}
We consider $\nu(n)$, the $n$-dimensional $W_{\bb{Q}_{p}}$-representation given as the iterated non-split extension of $W_{\bb{Q}_{p}}$-representations with graded isomorphic to $(\ol{\bb{Q}}_{\ell}(n - 1),\ldots,\ol{\bb{Q}}_{\ell}(1),\ol{\bb{Q}}_{\ell})$; where each stage of the extension is specified by the non-zero class in $H^{1}(W_{\bb{Q}_{p}},\ol{\bb{Q}}_{\ell}(i)^{\vee} \otimes \ol{\bb{Q}}_{\ell}(i + 1)) 
 \simeq H^{1}(W_{\bb{Q}_{p}},\ol{\bb{Q}}_{\ell}(1)) \simeq \ol{\bb{Q}}_{\ell}$. Here the last isomorphism follows from local Tate-duality and the fact that Euler-Poincar\'e characteristic is $0$. For an $L$-parameter $\tilde{\phi}: W_{\bb{Q}_{p}} \ra \phantom{}^{L}G(\ol{\bb{Q}}_{\ell})$, it is easy to see, by looking at the long exact cohomology sequences attached to this iterated extension defining $\nu(n)$, that, for any finite dimensional $W_{\bb{Q}_{p}}$-representation $M$, $R\Gamma(W_{\bb{Q}_{p}},M \otimes \phi^{\vee} \otimes \nu(n)^{\vee})$ will be non-trivial if and only if the semi-simplification $M^{\mathrm{ss}}$ admits a subquotient with $W_{\bb{Q}_{p}}$-action given by $\phi$ or $\phi(n)$ for sufficiently large $n$.

We first show that the assumption on the cohomology sheaves of $T_{V}(A)$ implies that the averaging operator 
\[ Av_{V,\phi \otimes \nu(n)}(A) = R\Gamma(W_{\mathbb{Q}_{p}},T_{V}(A) \otimes \phi^{\vee} \otimes \nu(n)^{\vee})\]
is non-trivial for some $n$ sufficiently large. Since $A$ is compact we know by Theorem 3.2 that $T_{V}(A)$ is also compact and therefore supported on a finite number of HN-strata. By applying excision with respect to the HN-strata of $\Bun_{G}$ \cite[Proposition~VII.7.3]{FS}, we can assume, without loss of generality, that $T_{V}(A)$ is supported on a single stratum, where it is given by a complex of smooth representations of $J_{b}(\mathbb{Q}_{p})$ for $b \in B(G)$. Moreover, since $A$ is ULA, by Theorem 3.2 the sheaf $T_{V}(A)$ is also ULA and therefore it follows that, for all open compact $K \subset J_{b}(\mathbb{Q}_{p})$, $T_{V}(A)^{K}$ is a perfect complex of $\ol{\mathbb{Q}}_{\ell}$-vector spaces. Writing $T_{V}(A) \simeq \colim_{K \ra \{1\}} T_{V}(A)^{K}$ and using that cohomology commutes with colimits, this allows us to apply results from the Galois cohomology of $W_{\mathbb{Q}_{p}}$ on finite-dimensional $\ol{\mathbb{Q}}_{\ell}$ vector spaces to $T_{V}(A)$. To do this, we consider the spectral sequence
\[ E_{2}^{p,q} = H^{p}(W_{\mathbb{Q}_{p}},H^{q}(T_{V}(A) \otimes (\phi \otimes \nu(n))^{\vee})) \Longrightarrow H^{p + q}( R\Gamma(W_{\mathbb{Q}_{p}},T_{V}(A) \otimes (\phi \otimes \nu(n))^{\vee})), \]
where cohomology is being taken with respect to the standard $t$-structure on $\Dlis(\Bun_{G},\overline{\mathbb{Q}}_{\ell})$.

Choose $q$ such that $\phi$ occurs as a sub-quotient of $H^q(T_V(A))$. By the above discussion, for all sufficiently large $n$, we have
\[ 
R\Gamma(W_{\bb{Q}_p},
H^q(T_V(A))\otimes(\phi\otimes\nu(n))^\vee)\neq 0. \]

Now recall that the cohomological dimension of $W_{\mathbb{Q}_{p}}$ acting on finite dimensional $\overline{\mathbb{Q}}_{\ell}$-vector spaces is $2$.  Since the Euler--Poincaré characteristic of $W_{\mathbb{Q}_{p}}$ acting on finite-dimensional $\overline{\mathbb{Q}}_{\ell}$-vector spaces is zero, this implies
\[ 
H^{1}(W_{\mathbb{Q}_p},
H^q(T_V(A))\otimes (\phi\otimes\nu(n))^\vee) \neq0. \]
But this is the term $E_2^{1,q}$ in the above spectral sequence, and, since the cohomological dimension of $W_{\mathbb{Q}_p}$ is $2$, the $E_{2}^{1,q}$ term admits no non-zero incoming or outgoing differentials, therefore it  gives rise to a non-zero contribution on the $E_{\infty}$-page. This implies that we have
\[
\mathrm{Av}_{V,\phi\otimes\nu(n)}(A)\neq0 \]
for every sufficiently large $n$.

Lemma 3.9 therefore tells us that the spectral action of the perfect complex $\mathcal{A}v_{V,\phi \otimes \nu(n)}$ on $A$ is non-trivial for sufficiently large $n$. Now choose $n$ so that $\phi(n)$ is not a subquotient of the fixed semisimple representation $r_{V} \circ \phi_{A}^{\mathrm{FS}}$. If $x$ denotes the closed $\overline{\mathbb{Q}}_{\ell}$-point in the coarse moduli space of Langlands parameters defined by $\phi_{A}^{\mathrm{FS}}$ with preimage $\pi^{-1}(x)$ in the stack of Langlands parameters, Lemma 3.8 tells us that $\mathcal{A}v_{V,\phi \otimes \nu(n)}$ must have non-zero support on $\pi^{-1}(x)$. The $\overline{\mathbb{Q}}_{\ell}$-points of $\pi^{-1}(x)$ correspond to the set of Langlands parameters whose semisimplification is precisely $\phi_{A}^{\mathrm{FS}}$. The previous analysis tells us that the evaluation of the perfect complex $R\Gamma(W_{\mathbb{Q}_{p}},C_{V} \otimes (\phi \otimes \nu(n))^{\vee})$ at some such point, corresponding to an $L$-parameter $\tilde{\phi}: W_{\mathbb{Q}_{p}} \rightarrow \phantom{}^{L}G(\overline{\mathbb{Q}}_{\ell})$, must be non-zero. This evaluation is precisely the complex 
\[ R\Gamma(W_{\mathbb{Q}_{p}}, r_{V} \circ \tilde{\phi} \otimes (\phi \otimes \nu(n))^{\vee}). \]
 However, this can only be non-zero for the chosen $n$ if $r_{V} \circ \tilde{\phi}^{\mathrm{ss}} = r_{V} \circ \phi_{A}^{\mathrm{FS}}$ has the  property that $r_{V} \circ \phi_{A}^{\mathrm{FS}}$ has a subquotient given by $\phi$ or $\phi(n)$, where the latter is not a subquotient of $r_{V} \circ \phi_{A}^{\mathrm{FS}}$ by our choice of $n$. We have concluded that if $T_{V}(A)$ has  a cohomology sheaf with a non-zero subquotient with $W_{\bb{Q}_{p}}$-action given by $\phi$ then $r_{V} \circ \phi_{A}^{\mathrm{FS}}$ also has a subquotient given by $\phi$, as desired.
\end{proof}
We conclude this section by reviewing how the spectral action behaves on objects with supercuspidal Fargues-Scholze parameter along the lines of \cite[Section~X.2]{FS}.  These results will be used to deduce a strong form of the Kottwitz conjecture from compatibility. We recall that a supercuspidal parameter, viewed as a continuous $1$-cocycle with respect to the $W_{\mathbb{Q}_{p}}$-action on $\hat{G}(\overline{\mathbb{Q}}_{\ell})$, denoted $\phi: W_{\mathbb{Q}_{p}} \rightarrow \hat{G}(\overline{\mathbb{Q}}_{\ell})$, satisfies the property that it doesn't factor through any $\phantom{}^{L}P$ for $P$ a proper parabolic subgroup of $\phantom{}^{L}G$. Equivalently, $\phi$ is semi-simple and $S_{\phi}/Z(\hat{G}^{\Gamma})$ is finite. From this, it follows by deformation theory (see the proof of \cite[Theorem~1.7]{DH} for more details) that the unramified twists of the parameter $\phi$ define a connected component
\[ C_{\phi} \hookrightarrow [\mathcal{Z}^{1}(W_{\mathbb{Q}_{p}},\hat{G})_{\overline{\mathbb{Q}}_{\ell}}/\hat{G}]\]
giving rise to a direct summand
\[ \Perf(C_{\phi}) \hookrightarrow \Perf([\mathcal{Z}^{1}(W_{\mathbb{Q}_{p}},\hat{G})_{\overline{\mathbb{Q}}_{\ell}}/\hat{G}]) \]
Therefore, the spectral action gives rise to a corresponding direct summand
\[ \Dlis^{C_{\phi},'}(\Bun_{G},\overline{\mathbb{Q}}_{\ell})^{\omega} \subset  \Dlis(\Bun_{G},\overline{\mathbb{Q}}_{\ell})^{\omega} \]
such that the Schur-irreducible objects in this subcategory all have Fargues-Scholze parameter given by an unramified twist of $\phi$. Now, it follows by Proposition 3.14 in the next section and Theorem 3.6 (4), that, since $\phi$ is supercuspidal, the restriction of any object in $\D^{C_{\phi}}_{\mathrm{lis}}(\Bun_{G},\overline{\mathbb{Q}}_{\ell})^{\omega}$ to any non-basic HN-strata of $\Bun_{G}$ must be zero. Therefore, one obtains a decomposition
\[ \Dlis^{C_{\phi},'}(\Bun_{G},\overline{\mathbb{Q}}_{\ell})^{\omega} \simeq \bigoplus_{b \in B(G)_{basic}} \D^{C_{\phi},'}(J_{b}(\mathbb{Q}_{p}),\overline{\mathbb{Q}}_{\ell})^{\omega}\]
where $J_{b}$ is the $\sigma$-centralizer of $b$ and  $\D^{C_{\phi}}(J_{b}(\mathbb{Q}_{p}),\overline{\mathbb{Q}}_{\ell})^{\omega} \subset \D(J_{b}(\mathbb{Q}_{p}),\overline{\mathbb{Q}}_{\ell})^{\omega}$ is a full subcategory of the derived category of compact objects in smooth representations of $J_{b}(\mathbb{Q}_{p})$. It also follows again by Theorem 3.6 (4) that the Schur-irreducible objects of any $\D^{C_{\phi}}(J_{b}(\mathbb{Q}_{p}),\overline{\mathbb{Q}}_{\ell})^{\omega}$ must lie only in the supercuspidal components of the Bernstein center. Now to further analyze this we fix a character $\chi$ of $Z(G)(\mathbb{Q}_{p})$ and consider the subcategory  
\[ \D_{\mathrm{lis},\chi}^{C_{\phi},'}(\Bun_{G},\overline{\mathbb{Q}}_{\ell})^{\omega} := \bigoplus_{b \in B(G)_{basic}} \D_{\chi}^{C_{\phi},'}(J_{b}(\mathbb{Q}_{p}),\overline{\mathbb{Q}}_{\ell})^{\omega} \]
where $\D_{\chi}^{C_{\phi},'}(J_{b}(\mathbb{Q}_{p}),\overline{\mathbb{Q}}_{\ell})^{\omega}$ is the derived (non-full) subcategory of $\D^{C_{\phi},'}(J_{b}(\mathbb{Q}_{p}),\overline{\mathbb{Q}}_{\ell})^{\omega}$ given by compact objects with fixed central character $\chi$, via the natural isomorphism $Z(J_{b})(\mathbb{Q}_{p}) \simeq Z(G)(\mathbb{Q}_{p})$ (where we recall that $J_{b}$ is an extended pure inner form of $G$). One can see that the spectral action of $\Perf(C_{\phi})$ preserves this subcategory. Indeed, it follows by \cite[Theorem~I.8.2]{FS} that $\Perf([\mathcal{Z}^{1}(W_{\mathbb{Q}_{p}},\hat{G})_{\overline{\mathbb{Q}}_{\ell}}/\hat{G}])$ and in turn $\Perf(C_{\phi})$ is generated under cones and retracts by the image of $\Rep_{\overline{\mathbb{Q}}_{\ell}}(\phantom{}^{L}G)$ in $\Perf([\mathcal{Z}^{1}(W_{\mathbb{Q}_{p}},\hat{G})_{\overline{\mathbb{Q}}_{\ell}}/\hat{G}])$. This reduces us to checking that Hecke operators preserve this subcategory, which, in turn reduces to the observation that, if one looks at the simultaneous action of $J_{b}(\mathbb{Q}_{p}) \times J_{b'}(\mathbb{Q}_{p})$ on the space parametrizing modifications $\mathcal{E}_{b} \rightarrow \mathcal{E}_{b'}$, for $b$ and $b'$ in $B(G)_{basic}$ that, under the canonical identification $Z(J_{b'})(\mathbb{Q}_{p}) \simeq Z(J_{b})(\mathbb{Q}_{p})$, the diagonally embedded center acts trivially. This follows since an element in the center of $J_{b}(\mathbb{Q}_{p})$ acts on the modification by the inverse of an element in the corresponding center of $J_{b'}(\mathbb{Q}_{p})$, where the inverse appears from the fact that $J_{b}(\mathbb{Q}_{p})$ is acting on the left and $J_{b'}(\mathbb{Q}_{p})$ is acting on the right. 
\\\\
Now we take $\chi$ to be the central character determined by $\phi$ and local class field theory (as in \cite[Section~10.1]{Bor}). We note that all Schur-irreducible objects in $\D_{\chi}^{C_{\phi},'}(J_{b}(\mathbb{Q}_{p}),\overline{\mathbb{Q}}_{\ell})$ lie in supercuspidal Bernstein components by Theorem 3.6 (4), and have fixed central character $\chi$. By \cite[Section~VI.3.6]{RenReprp}, supercuspidal representations are injective/projective in the category of smooth representations with fixed central character, we can therefore write $\D_{\chi}^{C_{\phi},'}(J_{b}(\mathbb{Q}_{p}),\overline{\mathbb{Q}}_{\ell})^{\omega} = \bigoplus_{\pi} \Perf(\overline{\mathbb{Q}}_{\ell}) \otimes \pi$, where $\pi$ runs over all supercuspidal representations of $J_{b}(\mathbb{Q}_{p})$ with central character $\chi$ and Fargues-Scholze parameter given by an unramified twist of $\phi$. We may then further project to the summand of supercuspidals whose Fargues-Scholze parameter is exactly equal to $\phi$ and use this to define  $\D_{\chi}^{C_{\phi}}(J_{b}(\mathbb{Q}_{p}),\overline{\mathbb{Q}}_{\ell})^{\omega}$ (where this summand is also preserved by the Hecke-action and in turn the spectral action, since the Hecke action and excursion algebra commute).

All in all, we conclude that we have a decomposition  
\[ \D_{\mathrm{lis},\chi}^{C_{\phi}}(\Bun_{G},\overline{\mathbb{Q}}_{\ell})^{\omega} = \bigoplus_{b \in B(G)_{basic}} \bigoplus_{\pi_{b}} \Perf(\overline{\mathbb{Q}}_{\ell}) \otimes \pi_{b} \] 
where the $\pi_{b}$ runs over all supercuspidal representations of $J_{b}(\mathbb{Q}_{p})$ with Fargues-Scholze parameter $\phi_{\pi_{b}}^{\mathrm{FS}} = \phi$. Moreover, the RHS carries an action of $\Perf(\mc{C}_{\phi})$.
\\\\
Now, the parameter $\phi$ defines a closed substack $i: [\Spec{(\ol{\mathbb{Q}}_{\ell})}/S_{\phi}] \hookrightarrow C_{\phi}$ of the connected component $C_{\phi}$. Moreover, vector bundles on $[\Spec{(\ol{\mathbb{Q}}_{\ell})}/S_{\phi}]$ can be described in terms of algebraic representations of $S_{\phi}$.  Also, since the Fargues-Scholze parameter of the representations in $\D_{\mathrm{lis},\chi}^{C_{\phi}}(\Bun_{G},\overline{\mathbb{Q}}_{\ell})^{\omega}$ is equal to $\phi$, we would now like to claim that the action of $\Perf(C_{\phi})$ on $\D_{\mathrm{lis},\chi}^{C_{\phi}}(\Bun_{G},\overline{\mathbb{Q}}_{\ell})^{\omega}$ factors through the natural $*$-pullback operation $\Perf(C_{\phi}) \ra \Perf([\Spec{(\ol{\mathbb{Q}}_{\ell})}/S_{\phi}])$ using Lemma \ref{lemma: SchurIrreducibleSupport}.

However, we note that the action of $\mathrm{Perf}(C_\phi)$ on $\mathrm{D}^{C_\phi}_{\mathrm{lis},\chi}(\mathrm{Bun}_G,\overline{\mathbb{Q}}_\ell)^\omega$ does not factor through the $*$-pullback $i^*\colon \mathrm{Perf}(C_\phi)\to\mathrm{Perf}([\mathrm{Spec}(\overline{\mathbb{Q}}_\ell)/S_\phi])$ on morphisms, since $C_\phi$ has a positive-dimensional direction given by unramified twists by $Z(\hat{G})^\Gamma$. Instead, we argue on objects as follows. By the proof of Lemma \ref{lemma: SchurIrreducibleSupport}, the maximal ideal $\mathfrak{m}_x\subset \mathcal{Z}^{\mathrm{spec}}(G,\overline{\mathbb{Q}}_\ell)$ defined by $\phi$ acts by zero on every Schur-irreducible object $\pi_b$ with Fargues--Scholze parameter $\phi$, so the action of $\mathrm{Perf}(C_\phi)$ on such objects factors through the full subcategory defined by the formal completion $\widehat{C}_\phi$ of $C_\phi$ at $x$. Since $Z(\hat{G})^\Gamma$ is central, $S_\phi$ acts trivially on the unramified-twist direction, and $\widehat{C}_\phi\simeq \widehat{Y}\times[\mathrm{Spec}(\overline{\mathbb{Q}}_\ell)/S_\phi]$, where $\widehat{Y}$ is the formal completion at the identity of the variety of unramified twists. In particular, every $W\in\mathrm{Rep}_{\overline{\mathbb{Q}}_\ell}(S_\phi)$ defines a vector bundle $\mathcal{C}_W$ on $\widehat{C}_\phi$ by pullback, with $i^*\mathcal{C}_W\simeq W$ and $\mathcal{C}_W\otimes\mathcal{C}_{W'}\simeq\mathcal{C}_{W\otimes W'}$. We set $\mathrm{Act}_W:=\mathcal{C}_W\star(-)$. By monoidality of the spectral action, we have
\[
\mathrm{Act}_W\circ\mathrm{Act}_{W'}\simeq\mathrm{Act}_{W\otimes W'}.
\]
Moreover, each $\mathcal{C}_W$ with $W$ irreducible is a direct summand of $\mathcal{C}_V|_{\widehat{C}_\phi}$ for some $V\in\mathrm{Rep}_{\overline{\mathbb{Q}}_\ell}(^LG)$, so $\mathrm{Act}_W(\pi_b)$ is a direct summand of the Hecke operator $T_V(\pi_b)$ and therefore lies in $\mathcal{D}^{C_\phi}_{\mathrm{lis},\chi}(\mathrm{Bun}_G,\overline{\mathbb{Q}}_\ell)^\omega$ by the discussion above.

Therefore, given $W \in \Rep_{\overline{\mathbb{Q}}_{\ell}}(S_{\phi})$, we can spectrally act to get an object: 
\[ \Act_{W}(\pi_{b}) \in \bigoplus_{b' \in B(G)_{basic}} \bigoplus_{\pi_{b'}} \Perf(\overline{\mathbb{Q}}_{\ell}) \otimes \pi_{b'}. \]
Assume that $W|_{Z(\hat{G})^{\Gamma}}$ is isotypic, given by some character $\eta: Z(\hat{G})^{\Gamma} \rightarrow \overline{\mathbb{Q}}_{\ell}^{*}$. As $Z(\hat{G})^{\Gamma}$ is diagonalizable with characters given by $B(G)_{basic} \xrightarrow{\simeq} \pi_{1}(G)_{\Gamma}$ via the $\kappa$ map, we obtain an element $b_{\eta} \in B(G)_{basic}$. Then $\Act_{W}(\pi_{b})$ is concentrated on the basic HN-strata given by $b' = b - b_{\eta}$. Therefore, we get an isomorphism
\[ \Act_{W}(\pi_{b}) \simeq \bigoplus_{\pi_{b'}} V_{\pi_{b'}} \otimes \pi_{b'}, \]
where $V_{\pi_{b'}} \in \Perf(\overline{\mathbb{Q}}_{\ell})$ and $\pi_{b'}$ runs over all supercuspidals of $J_{b'}$ with Fargues-Scholze parameter  $\phi^{\mathrm{FS}}_{\pi_{b'}} = \phi$. With this in hand, we can elucidate the $W_{\mathbb{Q}_{p}}$-action on the Hecke operator applied to a smooth irreducible object with supercuspidal Fargues-Scholze parameter, similar to what was done in Lemma 3.10 in the general case. Namely, given $V \in \Rep_{\overline{\mathbb{Q}}_{\ell}}(\phantom{}^{L}G)$, we obtain a vector bundle on $[\Spec{(\overline{\mathbb{Q}}_{\ell})}/S_{\phi}]$ with $W_{\mathbb{Q}_{p}}$-action given by $\phi$ via pulling back to the closed substack defined by the parameter $\phi$. In other words, we have a functor:
\[ \Rep_{\overline{\mathbb{Q}}_{\ell}}(\phantom{}^{L}G) \rightarrow \Rep_{\overline{\mathbb{Q}}_{\ell}}(S_{\phi})^{BW_{\mathbb{Q}_{p}}}. \]
Theorem 3.7 and the above discussion imply that the action of the image of $V$ in $\Rep_{\overline{\mathbb{Q}}_{\ell}}(S_{\phi})^{BW_{\mathbb{Q}_{p}}}$ acting via the spectral action on $\D_{\mathrm{lis},\chi}^{C_{\phi}}(\Bun_{G},\overline{\mathbb{Q}}_{\ell})$ described above is precisely the Hecke operator $T_{V}$. In particular, we note that, since every object $\D_{\mathrm{lis},\chi}^{C_{\phi}}(\Bun_{G},\overline{\mathbb{Q}}_{\ell})$ has Fargues-Scholze parameter equal to $\phi$, the spectral action of the Hecke operator on objects in $\D_{\mathrm{lis},\chi}^{C_{\phi}}(\Bun_{G},\overline{\mathbb{Q}}_{\ell})$.

This tells us that, if we decompose $r_{V} \circ \phi$ viewed as a representation of $S_{\phi} \times W_{\bb{Q}_{p}}$ as a direct sum $\bigoplus_{i \in I} W_{i} \otimes \sigma_{i}$ where $W_{i} \in \Rep_{\ol{\mathbb{Q}}_{\ell}}(S_{\phi})$ is irreducible and $\sigma_{i}$ is a continuous finite-dimensional representation of $W_{\mathbb{Q}_{p}}$, then we have an isomorphism 
\[ T_{V}(\pi_{b}) \simeq \bigoplus_{i \in I} \Act_{W_{i}}(\pi_{b}) \otimes \sigma_{i} \]
as $J_{b'}(\mathbb{Q}_{p}) \times W_{\mathbb{Q}_{p}}$-modules.
We now summarize the above discussion as a corollary for future use.
\begin{corollary}
Let $\phi$ be a supercuspidal parameter of $G$, $b \in B(G)_{basic}$ a basic element, $V \in \Rep_{\overline{\mathbb{Q}}_{\ell}}(\phantom{}^{L}G)$ an irreducible representation of some highest weight $\mu$ with dominant inverse $\mu^{-1}$, and $\pi_{b}$ a representation of $J_{b}(\mathbb{Q}_{p})$ with Fargues-Scholze parameter equal to $\phi$. We set $b_{\mu} \in B(G,\mu^{-1})$ to be the unique basic element and $b'$ the unique basic element such that $\kappa(b') = \kappa(b) + \kappa(b_{\mu})$.  If we decompose $r_{V} \circ \phi$ viewed as representation of $S_{\phi}$ as a direct sum $\bigoplus_{i \in I} W_{i} \otimes \sigma_{i}$, where $W_{i} \in \Rep_{\ol{\bb{Q}}_{\ell}}(S_{\phi})$ is irreducible and $\sigma_{i}$ is a continuous finite-dimensional representation of $W_{\mathbb{Q}_{p}}$, then there exists an isomorphism of $W_{\mathbb{Q}_{p}} \times J_{b'}(\mathbb{Q}_{p})$-modules
\[ T_{\mu}(\pi_{b}) \simeq \bigoplus_{i \in I} \Act_{W_{i}}(\pi_{b}) \otimes \sigma_{i} \]
where $\Act_{W_{i}}(\pi_{b}) \simeq \bigoplus_{\pi_{b'}} V_{\pi_{b'}} \otimes \pi_{b'}$ with $V_{\pi_{b'}} \in \Perf(\overline{\mathbb{Q}}_{\ell})$ and $\pi_{b'}$ ranging over supercuspidal representations of $J_{b'}(\mathbb{Q}_{p})$ with Fargues-Scholze parameter equal to $\phi$.
\end{corollary}
\begin{remark}
As we will start to see in the next section, the work of Hansen \cite{Han}, Hansen-Kaletha-Weinstein \cite{KW}, and compatibility of the Fargues-Scholze and Gan-Takeda/Gan-Tantono local Langlands correspondence will allow us to use this Corollary to prove Theorem 1.3. This is suggested already by Corollary 3.3, which shows us that $T_{\mu}(\pi_{b})$ can be computed explicitly using the cohomology of local Shimura varieties.
\end{remark}
\subsection{Compatibility with the Local Langlands for $\GSp_{4}$ and $\GU_{2}(D)$}
With the results of the previous section in place, we can now start making progress towards our goal of proving compatibility. So we again let $G = \mathrm{Res}_{L/\mathbb{Q}_{p}}\GSp_{4}$ and $J = \mathrm{Res}_{L/\mathbb{Q}_{p}}\GU_{2}(D)$, for $L/\mathbb{Q}_{p}$ a finite extension. As mentioned in section $1$, the case where a representation $\pi \in \Pi(G)$ (resp. $\rho \in \Pi(J)$), is a sub-quotient of a parabolic induction easily follows from Theorem 3.6 (3), (4), (5), (6) and compatibility of the (semi-simplified) Gan-Takeda (resp. Gan-Tantono) parameter with parabolic induction. We record this as a corollary now.
\begin{corollary}
Let $\pi \in \Pi(G)$ (resp. $\rho \in \Pi(J)$) be representations occurring as a sub-quotient of a parabolic induction. For such $\pi$ (resp. $\rho$), the Fargues-Scholze and Gan-Takeda (resp. Gan-Tantono) local Langlands correspondences are compatible. 
\end{corollary}
To tackle the remaining cases where the $L$-parameter $\phi$ is mixed supercuspidal or supercuspidal (so that the $L$-packets of $G$ or $J$ contain supercuspidal representaitons), we note that these are the cases where the $L$-parameter is discrete (i.e. the $L$-parameter does not factor through a proper Levi subgroup). This case is amenable to analysis using the results of Hansen-Kaletha-Weinstein \cite{KW}. We will now let $\mu$ be the Siegel cocharacter of $G$ so that the $\sigma$-centralizer of the  unique basic element $b \in B(G,\mu)$ is $J$. We now state the main result of \cite{KW} specialized to the case of the Shimura datum $(G,b,\mu)$.
\begin{theorem}{\cite[Theorem~1.0.2]{KW}}
Let $\phi$ be a discrete parameter and $S_{\phi} := Z_{\hat{G}}(\mathrm{Im}(\phi))$ as before. Let $\Pi_{\phi}(G)$ and $\Pi_{\phi}(J)$ denote the $L$-packets over $\phi$. Set $\pi \in \Pi_{\phi}(G)$ (resp. $\rho \in \Pi_{\phi}(J)$) to be smooth irreducible representations of $G(\mathbb{Q}_{p})$ (resp. $J(\mathbb{Q}_{p})$). If $\phi$ is supercuspidal or mixed supercuspidal, we have the following equality in the elliptic part $K_{0}(G(\mathbb{Q}_{p}))^{ell}$ of the Grothendieck group of admissible representations of $G(\mathbb{Q}_{p})$ of finite length
\[ [R\Gamma_{c}^{\flat}(G,b,\mu)[\rho]] = -\sum_{\pi \in \Pi_{\phi}(G)} Hom_{S_{\phi}}(\delta_{\pi,\rho},\mathrm{std}\circ\phi)\pi \]
and the following equality 
\[ [R\Gamma_{c}^{\flat}(G,b,\mu)[\pi]] = -\sum_{\rho \in \Pi_{\phi}(J)} Hom_{S_{\phi}}(\delta_{\pi,\rho}^{\vee},(\mathrm{std}\circ\phi)^{\vee})\rho \]
in the Grothendieck group of elliptic admissible $J(\mathbb{Q}_{p})$-representations of finite length, where $\delta_{\pi,\rho}$ is the algebraic representation of $S_{\phi}$ in Definition 2.3. Moreover, if the Fargues-Scholze parameter of $\pi$ (resp. $\rho$) is supercuspidal this is true in the Grothendieck group $K_{0}(G(\mathbb{Q}_{p}))$ (resp. $K_{0}(J(\mathbb{Q}_{p}))$) of all admissible representations of finite length. 
\end{theorem}
\begin{remark}
\begin{enumerate}
\item To deduce the result for the $\pi$-isotypic part, we have implicitly used the two towers isomorphism $\Sht(G,b,\mu)_{\infty} \simeq \Sht(J,\hat{b},\mu^{-1})_{\infty}$, where $\hat{b} \in B(J,\mu^{-1})$ is the unique basic element \cite[Corollary 23.3.2.]{SW2} and $\mu^{-1}$ is the dominant inverse of $\mu$. This explains the dual appearing in the $\pi$-isotypic part. 
\item We see that, via Corollary 3.3, this, in the case that $\phi$ is supercuspidal, should provide us insight into the multiplicity spaces $V_{\pi_{b'}}$ appearing in Corollary 3.11, assuming compatibility of the Fargues-Scholze and Gan-Tantono/Gan-Takeda local Langlands correspondences. Namely, we will see later (Theorem 3.17 and 3.18) that $R\Gamma_{c}(G,b,\mu)[\rho] \simeq R\Gamma_{c}^{\flat}(G,b,\mu)[\rho]$ and is concentrated in middle degree $3$ if $\phi_{\rho}^{\mathrm{FS}}$ is supercuspidal (assuming $L/\bb{Q}_{p}$ is unramified). Assuming compatibility, Corollary 3.11 will therefore tell us that $R\Gamma_{c}(G,b,\mu)[\rho]$ will be a direct sum over representations $\pi \in \Pi_{\phi}(G)$ with $W_{L}$-action given by $\mathrm{std}\circ \phi_{\rho}^{\mathrm{FS}} = \mathrm{std}\circ \phi_{\rho}$ decomposed as a representation of $S_{\phi}$ up to Tate twists. The summands in the decomposed $S_{\phi}$-representation correspond to the weight spaces appearing in the above description in the Grothendieck group. 
\end{enumerate}
\end{remark}
We now wish to write out the precise formula for the $\rho$ and $\pi$-isotypic parts, using the refined local Langlands discussed in section $2$.
\begin{enumerate}
    \item ($\phi$ stable)
    In this case, the $L$-packet $\Pi_{\phi}(G) = \{\pi\}$ is a singleton so the RHS of the above formula for the $\rho$-isotypic part has one term
    \[ -\pi Hom_{S_{\phi}}(\delta_{\pi,\rho},\mathrm{std}\circ\phi_{\rho}) \]
    In this case, $S_{\phi} = \mathbb{G}_{m}$ and $\delta_{\pi,\rho}$ is simply the identity representation, so the formula reduces to
    \[ -4\pi \]
    \item ($\phi$ endoscopic) 
    In this case, the $L$-packet has size $2$ and, as seen in section $2.1$, $\Pi_{\phi}(G) = \{ \pi^{+}, \pi^{-} \}$, where $\pi^{+}$ (resp. $\pi^{-}$) corresponds to the trivial (resp. non-trivial) character of the component group. $\rho$ can be either of the two representations corresponding to the irreducible representation of $S_{\phi}$ given by $\tau_{i}$ for $i = 1,2$ the projection to the two coordinates of $S_{\phi}$. However, the RHS remains the same regardless of which one it corresponds to. So, without loss of generality, we assume that $\rho = \rho_{1}$. Then the RHS of the above formula for the $\rho$-isotypic part has two terms
    \[ -\pi^{+}Hom_{S_{\phi}}(\tau_{1},\mathrm{std}\circ\phi_{\rho})\]
    and
    \[ -\pi^{-}Hom_{S_{\phi}}(\tau_{1} \otimes \tau_{\pi^{-}} \simeq \tau_{2}, \mathrm{std}\circ \phi_{\rho}) \]
    However, writing $\mathrm{std}\circ \phi_{\rho} \simeq \phi_{1} \oplus \phi_{2}$, these get identified with
    \[ -\pi^{+}Hom_{\overline{\mathbb{Q}}_{\ell}^{*}}(\overline{\mathbb{Q}}_{\ell}^{*},\phi_{1})\]
    and
    \[ -\pi^{-}Hom_{\overline{\mathbb{Q}}_{\ell}^{*}}(\overline{\mathbb{Q}}_{\ell}^{*}, \phi_{2}). \]
    Thus, the RHS is equal to
    \[ - 2\pi^{+} - 2\pi^{-}\]
\end{enumerate}
Similarly, for the $\pi$-isotypic part, we get that the RHS of the above formula is given by
\[ -4\rho \]
in the stable case and
\[ -2\rho_{1} - 2\rho_{2} \]
in the endoscopic case. 
\\\\
As mentioned in section 1.2, we will now use the previous result to perform a bootstrap to the supercuspidal representations occurring in the $L$-packets $\Pi_{\phi}(G)$ (resp. $\Pi_{\phi}(J)$), for $\phi$ a mixed supercuspidal parameter. For this, we will mention one last result from the Fargues-Scholze local Langlands correspondence. 
\begin{proposition}{\cite[Section~IX.7.1]{FS}}
Fix $G$ any connected reductive group over $\mathbb{Q}_{p}$, the action of the excursion algebra on $\Dlis(\Bun_{G},\overline{\mathbb{Q}}_{\ell})$ commutes with Hecke operators. Moreover, it is compatible with restriction to the HN-strata $\Bun_{G}^{b}$ for $b \in B(G)$ in the following sense. Given a Schur irreducible object $A \in \Dlis(\Bun_{G},\ol{\mathbb{Q}}_{\ell})$ and $b \in B(G)$, if we let $B$ be a smooth irreducible constituent of $j_{b}^{*}(A)$ then we can view it as a sheaf on the neutral stratum of $\Bun_{J_{b}}^{1}$ and let $\phi_{B}^{\mathrm{FS}}$ denote its Fargues-Scholze parameter with respect to the excursion algebra on $\Bun_{J_{b}}$. Then the Fargues-Scholze parameter of $A$ is the composition 
\[ W_{\mathbb{Q}_{p}} \xrightarrow{\phi_{B}^{\mathrm{FS}}} \phantom{}^{L}J_{b}(\ol{\mathbb{Q}}_{\ell}) \ra \phantom{}^{L}G(\ol{\mathbb{Q}}_{\ell}) \]
where $\phantom{}^{L}J_{b}(\ol{\mathbb{Q}}_{\ell}) \ra \phantom{}^{L}G(\ol{\mathbb{Q}}_{\ell})$ is the \emph{twisted} embedding, as defined in \cite[Page~327]{FS}.
\end{proposition}
\begin{remark}
The commutation of Hecke operators and the excursion algebra follows from the interpretation of the excursion algebra in terms of endomorphisms coming from multiplication by the ring of global functions of the stack of $L$-parameters, as discussed in section 3.2. 
\end{remark} 
From this, we can deduce the following useful corollary.
\begin{corollary}
Fix $G$ any connected reductive group with $(G,b,\mu)$ a local Shimura datum and $\pi \in \Pi(G)$ and $\rho \in \Pi(J_{b})$ smooth irreducible representations. All smooth irreducible representations occurring in the cohomology of $R\Gamma_{c}(G,b,\mu)[\pi]$ have Fargues-Scholze parameter equal to $\phi_{\pi}^{\mathrm{FS}}$. Similarly, all smooth irreducible representations occurring in the cohomology of $R\Gamma_{c}(G,b,\mu)[\rho]$ have Fargues-Scholze parameter equal to $\phi_{\rho}^{\mathrm{FS}}$. The same is also true for $R\Gamma_{c}^{\flat}(G,b,\mu)[\rho]$ and $R\Gamma_{c}^{\flat}(G,b,\mu)[\pi]$.
\end{corollary}
\begin{proof}
The first part follows immediately from Proposition 3.14 and Corollary 3.3. It remains to see the same is true for the complexes $R\Gamma_{c}^{\flat}(G,b,\mu)[\rho]$ and $R\Gamma_{c}^{\flat}(G,b,\mu)[\pi]$. This can be done by writing them as $j_{1}^{*}T_{\mu}Rj_{b*}(\rho)$ and $j_{b}^{*}T_{\mu^{-1}}Rj_{1*}(\pi)$, as in \cite[Section~IX.7.1]{FS}, where it again follows from Proposition 3.14 and Corollary 3.3. 
\end{proof}
We now exploit this corollary to deduce compatibility in the mixed supercuspidal case. 
\begin{corollary}
Let $\phi$ be an $L$-parameter of Howe-Piatetski--Shapiro, Saito-Kurokawa, or Soudry type. Then, for any $\pi \in \Pi_{\phi}(G)$ (resp. $\rho \in \Pi_{\phi}(J)$), the Fargues-Scholze and Gan-Takeda (resp. Gan-Tantono) local Langlands correspondences are compatible.
\end{corollary}
\begin{proof}
We first give the proof for the Gan-Takeda local Langlands correspondence. If $\phi$ is of Saito-Kurokawa type then, as seen in section $2.2$, we can write $\Pi_{\phi}(G) = \{\pi_{sc},\pi_{disc}\}$ and $\Pi_{\phi}(J) = \{\rho_{sc},\rho_{disc}\}$, where $\pi_{sc}$ (resp. $\rho_{sc}$) is a supercuspidal representation of $G(\mathbb{Q}_{p})$ (resp. $J(\mathbb{Q}_{p})$) and $\pi_{disc}$ (resp. $\rho_{disc}$) is a non-supercuspidal representation. We note that the Gan-Takeda (resp. Gan-Tantono) correspondences are compatible with the Fargues-Scholze correspondence for $\pi_{disc}$ (resp. $\rho_{disc}$), by Corollary 3.12.
\\\\
If we let $\mu$ be the Siegel cocharacter and $b \in B(G,\mu)$ be the unique basic element then the $\sigma$-centralizer $J_{b}$ is isomorphic to  $J$. We can consider the complex
\[ R\Gamma_{c}^{\flat}(G,b,\mu)[\rho_{disc}] \]
of $G(\mathbb{Q}_{p}) \times W_{L}$-representations. We then let $R\Gamma_{c}^{\flat}(G,b,\mu)[\rho_{disc}]_{sc}$ denote the direct summand of $R\Gamma_{c}^{\flat}(G,b,\mu)[\rho_{disc}]$, given by the supercuspidal Bernstein components of $G(\mathbb{Q}_{p})$. Theorem 3.13 tells us that we can describe this complex in the Grothendieck group of admissible $G(\mathbb{Q}_{p})$-representations of finite length as
\[ [R\Gamma_{c}^{\flat}(G,b,\mu)[\rho_{disc}]_{sc}] = -2\pi_{sc}, \]
which tells us that $\pi_{sc}$, occurs as a non-zero sub-quotient of the complex $R\Gamma_{c}^{\flat}(G,b,\mu)[\rho_{disc}]$. By Corollary 3.15, we know that we have an equality
\[ \phi_{\rho_{disc}}^{\mathrm{FS}} = \phi_{\pi_{sc}}^{\mathrm{FS}} \]
of conjugacy classes of parameters. However, Corollary 3.12 tells us that $\phi_{\rho_{disc}}^{\mathrm{FS}} = \phi_{\rho_{disc}}^{\mathrm{ss}}$, which is equal to $\phi_{\pi_{sc}}^{\mathrm{ss}}$, so we get the desired equality. The analysis in the Howe-Piatetski--Shapiro case is the same, where one can look at the $\rho$-isotypic part for any of the two non-supercuspidals in $\Pi_{\phi}(J)$, and in the Soudry type case the claim for the Gan-Takeda local Langlands just follows from Corollary 3.12. Similarly, for the Gan-Tantono local Langlands correspondence one looks at the $\pi$-isotypic part for $\pi_{disc} \in \Pi_{\phi}(G)$ a discrete series non-supercuspidal representation. In particular, in the Soudry type case, one needs to show compatibility for the unique supercuspidal representation in $\Pi_{\phi}(J)$, this can be done by looking at the $\pi_{disc}$ isotypic part for $\pi_{disc}$ the unique non-supercuspidal discrete series representation in the L-packet $\Pi_{\phi}(G)$.
\end{proof}
In the remaining part of this section, we address proving compatibility in the case where the parameter $\phi$ is supercuspidal. Before tackling the question of compatibility, we address some geometric properties of the sheaves $\mathcal{F}_{\rho}$, for $\rho$ with supercuspidal Fargues-Scholze parameter. This will be leveraged in proving the strong form of the Kottwitz Conjecture for the $\rho$ and $\pi$-isotypic parts in section $8$, as mentioned in Remark 3.10 (2). 
\\\\
Now, consider again $G$ a general connected reductive group, a basic element $b \in B(G)_{basic}$, and a smooth irreducible representation $\rho$ of the $\sigma$-centralizer $J_{b}(\mathbb{Q}_{p})$. We will now address some further consequences of the Fargues-Scholze parameter $\phi^{\mathrm{FS}}_{\rho}$ being supercuspidal. It turns out that the sheaves defined by representations with these parameters have interesting geometric properties, which were leveraged in \cite{Han} to prove various general results on the cohomology groups. 
In particular, Hansen shows the following:
\begin{theorem}{\cite[Theorem~1.1.]{Han}}
Let $(G,b,\mu)$ be a basic local Shimura datum with $E$ the reflex field of $\mu$ as before, and let $\rho$ be a smooth irreducible representation of $J_{b}(\mathbb{Q}_{p})$. Suppose the following conditions hold: 
\begin{enumerate}
    \item The spaces $(\Sht(G,b,\mu)_{K})_{K \subset G(\mathbb{Q}_{p})}$ occur in the basic uniformization at $p$ of a global Shimura variety in the sense of Definition 4.1.
    \item The Fargues-Scholze parameter $\phi_{\rho}^{\mathrm{FS}}: W_{\mathbb{Q}_{p}} \rightarrow \phantom{}^{L}G(\overline{\mathbb{Q}}_{\ell})$ is supercuspidal. 
\end{enumerate}
Then the complex $R\Gamma_{c}(G,b,\mu)[\rho]$ defined above is concentrated in the middle degree $d = \dim(\Sht(G,b,\mu)_{\infty}) = \langle 2\rho_{G},\mu \rangle$.
\end{theorem}
One of the key ideas in the argument is to exploit the behavior of the sheaf $j_{b!}(\mathcal{F}_{\rho})$ under Verdier duality, where $j_{b}: \Bun_{G}^{b} \hookrightarrow \Bun_{G}$ is the inclusion of the open HN-strata corresponding to $b \in B(G)_{basic}$. In particular, by Proposition 3.14 and Theorem 3.6 (4), one can see that the natural map $j_{b!}(\mathcal{F}_{\rho}) \rightarrow Rj_{b*}(\mathcal{F}_{\rho})$ is an isomorphism. Namely, Proposition 3.14 implies that a non-zero restriction of $Rj_{b*}(\mathcal{F}_{\rho})$ to any non-basic HN-strata must be valued in representations having Fargues-Scholze parameter $\phi_{\rho}^{\mathrm{FS}}$ under the relevant twisted embedding, which is impossible since the $\sigma$-centralizers of non-basic elements are extended pure inner forms of proper Levi subgroups of $G$ and, by assumption, the parameter $\phi_{\rho}^{\mathrm{FS}}$ is supercuspidal. This implies that, if we apply Verdier duality to both sides of 
the isomorphism 
\[ j_{\mathbf{1}}^{*}T_{\mu}j_{b!}(\mathcal{F}_{\rho}) \simeq R\Gamma_{c}(G,b,\mu)[\rho][d](\frac{d}{2}) \]
supplied by Corollary 3.3, we see that the LHS is isomorphic to
\[ j_{\mathbf{1}}^{*}T_{\mu}j_{b!}(\mathcal{F}_{\rho^{*}}) \simeq R\Gamma_{c}(G,b,\mu)[\rho^{*}][d](\frac{d}{2})\]
On the other hand, on the RHS we act through Verdier duality on the tower $(\Sht(G,b,\mu)_{K})_{K \subset G(\mathbb{Q}_{p})}$, which are smooth rigid spaces of dimension $d$. So, in particular, the dualizing object is isomorphic to $\overline{\mathbb{Q}}_{\ell}[2d](d)$. This allows one to deduce the following consequence for the cohomology groups $R\Gamma_{c}(G,b,\mu)[\rho]$. 
\begin{theorem}{\cite[Theorem~1.3,Theorem~2.23]{Han}}
Fix a basic local Shimura datum $(G,b,\mu)$ and let $\rho$ be representation of $J_{b}(\mathbb{Q}_p)$ with supercuspidal Fargues-Scholze parameter. Then there is a natural isomorphism
\[ R\mathcal{H}om(R\Gamma_{c}(G,b,\mu)[\rho],\overline{\mathbb{Q}}_{\ell}) \simeq R\Gamma_{c}(G,b,\mu)[\rho^{*}][2d](d) \]
as $W_{E}$-equivariant objects of $\D(G(\mathbb{Q}_{p}),\overline{\mathbb{Q}}_{\ell})$, where $\rho^{*}$ is the contragredient of $\rho$. In particular, we have a natural $W_{E}$-equivariant isomorphism of admissible $G(\mathbb{Q}_{p})$-representations for all $0 \leq i \leq 2d$
\[ H^{i}(R\Gamma_{c}(G,b,\mu)[\rho])^{*} \simeq H^{2d - i}(R\Gamma_{c}(G,b,\mu)[\rho^{*}])(d) \]
\end{theorem}
\begin{remark}
As noted in Remark 3.4, the LHS of the above formula is isomorphic to $R\Gamma_{c}^{\flat}(G,b,\mu)[\rho^{*}][2d](d)$, so it follows by cancelling the shifts and Tate twists and relaxing contragredients that one has an isomorphism
\[ R\Gamma_{c}^{\flat}(G,b,\mu)[\rho] \simeq R\Gamma_{c}(G,b,\mu)[\rho] \]
as $G(\mathbb{Q}_{p}) \times W_{E}$-representations for all such $\rho$. 
\end{remark}
Now we turn our attention to the question of showing compatibility for supercuspidal parameters assuming Proposition 1.4. So again let $L/\mathbb{Q}_{p}$ be a finite extension and let $G := \mathrm{Res}_{L/\mathbb{Q}_{p}}(\GSp_{4})$ be the restriction of scalars of $\GSp_{4}$ and $J := \mathrm{Res}_{L/\mathbb{Q}_{p}}\GU_{2}(D)$ the unique non-split inner form as before. As we will see in section $8$, it essentially follows from Theorem 3.13 and Corollary 3.15 that showing compatibility for $\rho \in \Pi(J)$ with supercuspidal Gan-Tantono parameter implies the corresponding statement for $\pi \in \Pi(G)$ with supercuspidal Gan-Takeda parameter. So we fix such a $\rho$ and assume that the Gan-Tantono parameter $\phi_{\rho}$ is endoscopic supercuspidal with the stable case being strictly easier. We will write $\mathrm{std}\circ \phi_{\rho} \simeq \phi_{1} \oplus \phi_{2}$ for $\phi_{i}$ distinct irreducible $2$-dimensional representations of $W_{L}$ and let $\mu$ be the Siegel cocharacter. The shtuka space $\Sht(G,b,\mu)_{\infty}$ in this case will have dimension $d := \langle 2\rho_{G},\mu \rangle = 3$. We will assume for the rest of this section that Proposition 1.4 is true.
\begin{proposition}
For $L/\bb{Q}_{p}$ unramified and $p > 2$, we let $\phi$ be a supercuspidal parameter with associated $L$-packet $\Pi_{\phi}(J)$. Then the direct summand of
\[ \bigoplus_{\rho' \in \Pi_{\phi}(J)} R\Gamma_{c}(G,b,\mu)[\rho'] \]
given by the supercuspidal Bernstein components of $G(\mathbb{Q}_{p})$, denoted
\[ \bigoplus_{\rho' \in \Pi_{\phi}(J)} R\Gamma_{c}(G,b,\mu)[\rho']_{sc}, \]
is concentrated in middle degree $3$ and admits a non-zero $W_{L}$-stable sub-quotient with $W_{L}$-action given by $\mathrm{std}\circ \phi \otimes |\cdot|^{-3/2}$.
\end{proposition}
First, we combine this with the following lemma.
\begin{lemma}
Let $\phi$ be a supercuspidal parameter then all representations in the $L$-packet $\Pi_{\phi}(J)$ have the same Fargues-Scholze parameter.
\end{lemma}
\begin{proof}
We choose a $\pi \in \Pi_{\phi}(G)$ and then apply Corollary 3.15 to deduce that all representations occurring in the cohomology of $R\Gamma_{c}^{\flat}(G,b,\mu)[\pi]$ have Fargues-Scholze parameter equal to $\phi_{\pi}^{\mathrm{FS}}$. However, by Theorem 3.13, we have that all representations in $\rho \in \Pi_{\phi}(J)$ occur in the cohomology of $R\Gamma_{c}^{\flat}(G,b,\mu)[\pi]$, so their Fargues-Scholze parameters are the same as desired. 
\end{proof}
With this in hand, we are ready to prove the key consequence of Proposition 1.4 using the results on the spectral action obtained in section 3.2. 
\begin{corollary}
Assume that $L/\mathbb{Q}_{p}$ is an unramified extension and that $p > 2$ and that Proposition 1.4 is true. Then, for $\rho$ a smooth irreducible representation of $J(\mathbb{Q}_{p})$ with supercuspidal Gan-Tantono parameter $\phi$, the Gan-Tantono and Fargues-Scholze correspondences coincide.
\end{corollary}
\begin{proof}
As mentioned in the introduction, the key will be the isomorphism
\[ \bigoplus_{\rho' \in \Pi_{\phi}(J)} j_{\mathbf{1}}^{*}T_{\mu}j_{b!}(\mathcal{F}_{\rho'}) \simeq \bigoplus_{\rho' \in \Pi_{\phi}(J)} R\Gamma_{c}(G,b,\mu)[\rho'][3](\frac{3}{2}) \]
supplied by Corollary 3.3. Now Proposition 1.4 tells us that one of the summands on the RHS admits a sub-quotient with $W_{L}$-action given by $\phi_{1}$ and one of them admits a sub-quotient with $W_{L}$-action given by $\phi_{2}$. Applying Lemmas 3.20 and 3.10 therefore tells us that $\mathrm{std}\circ \phi_{\rho}^{\mathrm{FS}}$ admits a sub-quotient isomorphic to $\phi_{1}$ and a sub-quotient isomorphic to $\phi_{2}$.  Therefore, we conclude that $\mathrm{std}\circ \phi = \mathrm{std}\circ \phi_{\rho}^{\mathrm{FS}}$. Moreover, we also know that the similitude character of $ \phi_{\rho}^{\mathrm{FS}}$ must coincide with the central character of $\rho$ by Theorem 3.6 (2), which agrees with the similitude character of $\mathrm{std}\circ \phi_{\rho} = \phi_{1} \oplus \phi_{2}$. By combining these observations with \cite[Lemma~6.1]{GT1}, this is enough to conclude that $\phi_{\rho} = \phi_{\rho}^{\mathrm{FS}}$, as conjugacy classes of $\GSp_{4}$-valued parameters.
\end{proof}
\section{Basic Uniformization}
In this section, we will briefly review what basic uniformization of the generic fiber of a global Shimura variety means, following \cite[Section~3.1]{Han}. Then we will apply it to our particular case and derive an analogue of Boyer's trick, providing useful consequences for the proof of Proposition 1.4. 
\subsection{A Review of Basic Uniformization}
We now recall briefly what basic uniformization means. Let $\mathbf{G}/\mathbb{Q}$ be a connected reductive group over $\mathbb{Q}$ and let $(\mathbf{G},X)$ be a Shimura datum, with associated inverse conjugacy class of Hodge cocharacters $\mu: \mathbb{G}_{m,\mathbb{C}} \rightarrow \mathbf{G}_{\mathbb{C}}$. Fix a prime $p$, and set $G := \mathbf{G}_{\mathbb{Q}_{p}}$. Using our fixed isomorphism $\mathbb{C} \simeq \overline{\mathbb{Q}}_{p}$, we can and do regard $\mu$ as a conjugacy class of cocharacters $\mu: \mathbb{G}_{m,\overline{\mathbb{Q}}_{p}} \rightarrow G_{\overline{\mathbb{Q}}_{p}}$. This allows us to consider the $\mu$-admissible locus $B(G,\mu)$ in the Kottwitz set of $G$. Let $\mathbb{A}$ (resp. $\mathbb{A}_{f}$) denote the adeles (resp. finite adeles) of $\mathbb{Q}$ and $\mathbb{A}_{f}^{p}$ denote the finite adeles away from $p$. For any compact open subgroup $K \subset \mathbf{G}(\mathbb{A}_{f})$, let $\mathcal{S}(\mathbf{G},X)_{K}$ be the associated rigid analytic Shimura variety over $\mathbb{C}_{p}$ of level $K$, regarded as a diamond over $\Spd(\bb{C}_{p})$. We let $K = K^{p}K_{p}$, where $K^{p} \subset \mathbf{G}(\mathbb{A}_{f}^{p})$ and $K_{p} \subset G(\mathbb{Q}_{p})$ are open compact subgroups. We set
\[ \mathcal{S}(\mathbf{G},X)_{K^{p}} := \lim_{K_{p} \rightarrow \{1\}} \mathcal{S}(\mathbf{G},X)_{K^{p}K_{p}}. \]
If $(\mathbf{G},X)$ is of pre-abelian type, this is representable by a perfectoid space and in general it can be described as a diamond. By the results of \cite{Han1}, there exists a canonical $G(\mathbb{Q}_{p})$-equivariant Hodge-Tate period map
\[ \pi_{HT}: \mathcal{S}(\mathbf{G},X)_{K^{p}} \rightarrow \mathcal{F}\ell_{G,\mu^{-1}}, \]
where $\mathcal{F}\ell_{G,\mu^{-1}} := (G_{\mathbb{C}_{p}}/P_{\mu^{-1}})^{ad}$ is the adic space associated to the flag variety defined by the parabolic $P_{\mu^{-1}} \subset G_{\mathbb{C}_{p}}$ given by the dominant inverse of $\mu$ via the dynamical method again regarded as a diamond. We identify $\mathcal{F}\ell_{G,\mu^{-1}}$ with the space of modifications of type $\mathcal{E} \dashrightarrow \mathcal{E}_{0}$ of type $\mu$ via the Bialynicka-Birula isomorphism. We note that the map $\pi_{\HT}$ is an fdcs map of diamonds, in the sense of \cite[Definition~5.4]{Mann2022NuclearSheaves}.\footnote{The maps $\mathcal{S}(\mathbf{G},X)_{K^{p}} \ra \Spa(\bb{C}_{p})$ and $\mathcal{F}\ell_{G,\mu^{-1}} \ra \Spa(\bb{C}_{p})$ are easily checked to be fdcs, and therefore this follows by \cite[Lemma~5.5 (iv)]{Mann2022NuclearSheaves}.} in particular all six operations constructed in \cite{Ecod} exist for this morphism.

By the $G(\mathbb{Q}_{p})$-equivariance, $\pi_{HT}$ descends to a map:
\[ \pi_{HT,K_{p}}: \mathcal{S}(\mathbf{G},X)_{K^{p}K_{p}} \rightarrow [\mathcal{F}\ell_{G,\mu^{-1}}/\underline{K_{p}}]. \]
We let $b \in B(G,\mu)$ be the unique basic element, and let $\mathcal{F}\ell_{G,\mu^{-1}}^{b}$ be the basic Newton stratum. This parametrizes, for $S$ a perfectoid space in characteristic $p$, modifications $\mathcal{E} \dashrightarrow \mathcal{E}_{0}$ of type $\mu$ between the trivial $G$-bundle $\mathcal{E}_{0}$ on the relative Fargues-Fontaine curve $X_{S}$ and $\mathcal{E}$ a bundle isomorphic to the $G$-bundle $\mathcal{E}_{b}$ corresponding to $b \in B(G)$ after pulling back to a geometric point of $S$. The triple $(G,b,\mu)$ defines a local Shimura datum, as in section $3.1$, so we may consider the infinite level Shimura variety/shtuka space $\Sht(G,b,\mu)_{\infty}$ and its base change $\Sht(G,b,\mu)_{\infty,\mathbb{C}_{p}}$. By pulling back along $\pi_{HT}$, we get an open subspace $\mathcal{S}(\mathbf{G},X)^{b}_{K^{p}} \subset \mathcal{S}(\mathbf{G},X)_{K^{p}}$, which descends to an open subspace $\mathcal{S}(\mathbf{G},X)^{b}_{K}$, for $K \subset \mathbf{G}(\mathbb{A}_{f})$ an open compact. We now have the key definition.
\begin{definition}
We say a global Shimura datum $(\mathbf{G},X)$ satisfies basic uniformization at $p$ if there exists (in fact this is always true see \cite[Proposition~3.1]{Han}) a  $\mathbb{Q}$-inner form $\mathbf{G}'$ of $\mathbf{G}$ satisfying
\begin{itemize}
    \item $\mathbf{G}'_{\mathbb{A}^{p}_{f}} \simeq \mathbf{G}_{\mathbb{A}^{p}_{f}}$ as algebraic groups over $\mathbb{A}^{p}_{f}$,
    \item $\mathbf{G}'_{\mathbb{Q}_{p}} \simeq J_b$, where $J_{b}$ is the inner form of $G$ given by the $\sigma$-centralizer of the basic element $b \in B(G,\mu)$,
    \item $\mathbf{G}'(\mathbb{R})$ is compact modulo center,
\end{itemize}
and a $\mathbf{G}(\mathbb{A}_{f})$-equivariant isomorphism of diamonds over $\mathbb{C}_{p}$
\begin{equation} 
\lim_{K^{p} \rightarrow \{1\}} \mathcal{S}(\mathbf{G},X)^{b}_{K^{p}} \simeq (\underline{\mathbf{G}'(\mathbb{Q})}\backslash\underline{\mathbf{G}'(\mathbb{A}_{f})} \times_{\Spd(\mathbb{C}_{p})} \Sht(G,b,\mu)_{\infty,\mathbb{C}_{p}})/\underline{J_{b}(\mathbb{Q}_{p})}, 
\end{equation}
where $J_{b}(\mathbb{Q}_{p})$ acts diagonally, such that, under the identification $\mathcal{F}\ell_{G,\mu^{-1}}^{b} \simeq \Sht(G,b,\mu)_{\infty,\mathbb{C}_{p}}/\underline{J_{b}(\mathbb{Q}_{p})}$, the morphism
\[ \pi_{HT}: \lim_{K^{p} \rightarrow \{1\}} \mathcal{S}(\mathbf{G},X)^{b}_{K^{p}} \rightarrow \mathcal{F}\ell_{G,\mu^{-1}}^{b} \]
identifies with the projection
\[ (\underline{\mathbf{G}'(\mathbb{Q})}\backslash\underline{\mathbf{G}'(\mathbb{A}_{f})} \times_{\Spd(\mathbb{C}_{p})} \Sht(G,b,\mu)_{\infty,\mathbb{C}_{p}})/\underline{J_{b}(\mathbb{Q}_{p})} \rightarrow \Sht(G,b,\mu)_{\infty,\mathbb{C}_{p}}/\underline{J_{b}(\mathbb{Q}_{p})}, \]
where $\mathbf{G}(\mathbb{A}_{f}) \simeq \mathbf{G}'(\mathbb{A}_{f}^{p}) \times G(\mathbb{Q}_{p})$ acts on the RHS via the natural action of $\mathbf{G}'(\mathbb{A}_{f}^{p})$ on $\mathbf{G}'(\mathbb{Q})\backslash\mathbf{G}'(\mathbb{A}_{f})$ and $G(\mathbb{Q}_{p})$ acts on $\Sht(G,b,\mu)_{\infty,\mathbb{C}_{p}}$ compatible with Weil group descent datum. Moreover, if the reflex field of the cocharacter $\mu: \mathbb{G}_{m,\overline{\mathbb{Q}}_{p}} \rightarrow G_{\overline{\mathbb{Q}}_{p}}$ is $E/\mathbb{Q}_{p}$ then this isomorphism descends to an isomorphism of diamonds over $\breve{E} := E\breve{\mathbb{Q}}_{p}$.  
\end{definition}
We now mention some consequences of uniformization which will be key to us in what follows. Let $\mathcal{H}(J_{b}) := C^{\infty}_{c}(J_b(\mathbb{Q}_{p}),\overline{\mathbb{Q}}_{\ell})$ be the usual smooth Hecke algebra. We fix an algebraic representation of  $\mathbf{G}/\mathbb{Q}$, denoted $\mathcal{V}_{\xi}$, of some regular highest weight $\xi$. The isomorphism $i: \overline{\mathbb{Q}}_{\ell} \xrightarrow{\simeq} \mathbb{C}$ then determines a $\overline{\mathbb{Q}}_{\ell}$-local system   $\mathcal{L}_{\xi}$ on the Shimura variety $\mathcal{S}(G,X)_{K^{p}}$. We now consider the space of algebraic automorphic forms valued in $\mathcal{V}_{\xi}$: \[\mathcal{A}(\mathbf{G}'(\mathbb{Q})\backslash \mathbf{G}'(\mathbb{A}_{f})/K^{p},\mathcal{L}_{\xi}) := \colim_{K_{p} \rightarrow \{1\}} \mathcal{A}(\mathbf{G}'(\mathbb{Q})\backslash \mathbf{G}'(\mathbb{A}_{f})/K^{p}K_{p},\mathcal{L}_{\xi})\] 
in the sense of Gross \cite{Gross}. Namely, it is the space of all continuous functions $\phi: \mathbf{G}'(\mathbb{A}_{f}) \ra \mathcal{V}_{\xi}(\ol{\mathbb{Q}}_{\ell})$ with respect to the pro-finite topology on the source and the discrete topology on the target such that, for all $k \in K^{p}$, $\gamma \in \mathbf{G}'(\mathbb{Q})$, and $g \in \mathbf{G}'(\mathbb{A}_{f})$, we have that $\phi(gk) = \phi(g)$ and $\phi(\gamma g) = \gamma\phi(g)$, where $\gamma$ acts via $\mathcal{V}_{\xi}$ (up to issues of center). The isomorphism (2) will then allow us to deduce an isomorphism
\[ R\Gamma_{c}(G,b,\mu) \otimes^{\mathbb{L}}_ {\mathcal{H}(J_{b})} \mathcal{A}(\mathbf{G}'(\mathbb{Q})\backslash \mathbf{G}'(\mathbb{A}_{f})/K^{p},\mathcal{L}_{\xi}) \xrightarrow{\simeq} R\Gamma_{c}(\mathcal{S}(\mathbf{G},X)^{b}_{K^{p}},\mathcal{L}_{\xi})  \]
of $G(\mathbb{Q}_{p}) \times W_{E}$-modules, which when composed with the morphism 
\[ R\Gamma_{c}(\mathcal{S}(\mathbf{G},X)^{b}_{K^{p}},\mathcal{L}_{\xi}) \rightarrow R\Gamma_{c}(\mathcal{S}(\mathbf{G},X)_{K^{p}},\mathcal{L}_{\xi}) \]
coming from excision with $\mathbb{Z}_{\ell}$-coefficients with respect to the open strata $\mathcal{S}(\mathbf{G},X)^{b}_{K^{p}} \hookrightarrow \mathcal{S}(\mathbf{G},X)_{K^{p}}$, gives rise to the uniformization map mentioned in the introduction. We show this now.
\begin{proposition}
Assume that $(\mathbf{G},X)$ satisfies basic uniformization at $p$, then there exists a $G(\mathbb{Q}_{p}) \times W_{E}$-equivariant map
\[ \Theta: R\Gamma_{c}(G,b,\mu) \otimes^{\mathbb{L}}_ {\mathcal{H}(J)} \mathcal{A}(\mathbf{G}'(\mathbb{Q})\backslash \mathbf{G}'(\mathbb{A}_{f})/K^{p},\mathcal{L}_{\xi}) \rightarrow R\Gamma_{c}(\mathcal{S}(\mathbf{G},X)_{K^{p}},\mathcal{L}_{\xi})  \]
functorial in the level $K^{p}$.  
\end{proposition}
\begin{proof}
We will want to use excision in this argument with respect to the Newton stratification. However, with $\ol{\mathbb{Q}}_{\ell}$-coefficients there are some subtleties if one just does this naively, as noted at the end of the introduction of \cite{Han} (see also \cite[Proposition~2.6, Example~2.7]{Huber} for examples in more classical language). To remedy this, as in \cite{Han} we first construct a map $\hat{\Theta}$ with $\mathbb{Z}_{\ell}$-coefficients, where we use the formalism of adic sheaves defined in \cite[Section~26]{Ecod}, which defines a full six functor formalism with the usual excision triangles. We then deduce the statement with $\ol{\mathbb{Q}}_{\ell}$-coefficients from this. We write $\hat{\mathcal{L}}_{\xi}$ for the $\mathbb{Z}_{\ell}$-adic local system on $\mathcal{S}(\mathbf{G},X)_{K^{p}}$ defined by fixing a $K^{p}$-stable $\mathbb{Z}_{\ell}$-lattice in the $\mathbb{Q}_{\ell}$-realization of the algebraic representation $\mathcal{V}_{\xi}$ considered above, as in \cite[Proposition~2.9]{Han}

We fix a level $K_{p} \subset \mathbf{G}(\mathbb{Q}_{p})$ and write $\pi_{HT,K_{p}}: \mathcal{S}(\mathbf{G},X)_{K^{p}K_{p}} \ra [\mathcal{F}\ell_{G,\mu^{-1}}/\underline{K_{p}}]$ for the induced Hodge-Tate period map. Then we have an isomorphism: 
\[ R\Gamma_{c}(\mathcal{S}(\mathbf{G},X)_{K^{p}},\hat{\mathcal{L}}_{\xi}) \simeq \colim_{K_{p} \ra \{1\}}  R\Gamma_{c}([\mathcal{F}\ell_{G,\mu^{-1}}/\underline{K_{p}}], R\pi_{HT,K_{p}!}(\hat{\mathcal{L}}_{\xi})), \]
where the LHS is defined as the colimit of the cohomology of the Shimura varieties at finite level (or alternatively, since we are working with compactly supported cohomology which is, in particular the colimit over the cohomology of the qcqs opens by definition, the cohomology of the Shimura variety at infinite level).
Write $j_{K_{p}}: [\mathcal{F}\ell_{G,\mu^{-1}}^{b}/\underline{K_{p}}] \hookrightarrow [\mathcal{F}\ell_{G,\mu^{-1}}/\underline{K_{p}}]$ for the open inclusion of the basic locus quotiented out by $K_{p}$.  Now, we claim that, under this identification, the desired map is given by the map
\[ R\Gamma_{c}([\mathcal{F}\ell_{G,\mu^{-1}}^{b}/\underline{K_{p}}], j_{K_{p}}^{*}R\pi_{HT,K_{p}!}(\hat{\mathcal{L}}_{\xi})) \ra   R\Gamma_{c}([\mathcal{F}\ell_{G,\mu^{-1}}/\underline{K_{p}}], R\pi_{HT,K_{p}!}(\hat{\mathcal{L}}_{\xi}))  \]
coming from applying excision with respect to $j_{K_{p}}$, for varying $K_{p}$. There is a natural map 
\[ q_{K_{p}}^{b}: [\mathcal{F}\ell_{G,\mu^{-1}}^{b}/\underline{K_{p}}] \ra [\Spd(\mathbb{C}_{p})/\underline{J_{b}(\mathbb{Q}_{p})}], \]
which is cohomologically smooth as in \cite[Proposition~2.16]{Han}. Since we are working with the category of $\mathbb{Z}_{\ell}$-adic sheaves, the category of sheaves on the target is not identified with the unbounded derived category of smooth representations of $J_{b}(\mathbb{Q}_{p})$, but rather the left-completed derived category of smooth $\ell$-complete representations by  \cite[Proposition~2.6 (1)]{Han} (cf. Remark 3.1). However, since $\hat{\mathcal{L}}_{\xi}$ was defined by fixing a $K^{p}$-stable lattice inside the $\mathbb{Q}_{\ell}$-realization of $\mathcal{V}_{\xi}$, the representation $\Pi := \mathcal{A}(\mathbf{G}'(\mathbb{Q})\backslash \mathbf{G}'(\mathbb{A}_{f})/K^{p},\hat{\mathcal{L}}_{\xi})$ defined by the $\ell$-adic completion of the space of $\mathbb{Z}_{\ell}$-algebraic automorphic forms defined above has the structure of such a $J_{b}(\mathbb{Q}_{p})$-representation, and we write $\mathcal{F}_{\Pi}$ for this sheaf. By \cite[Corollary~3.7]{Han}, the fact that we know basic uniformization at $p$ implies we have a natural in $K^{p}$ isomorphism:
\[ (q_{K_{p}}^{b*})(\mathcal{F}_{\Pi}) \simeq j_{K_{p}}^{*}R\pi_{HT,K_{p}!}(\hat{\mathcal{L}}_{\xi}). \]
Applying $R\Gamma_{c}([\mathcal{F}\ell_{G,\mu^{-1}}^{b}/K_{p}],-)$, we obtain 
\[  R\Gamma_{c}([\mathcal{F}\ell_{G,\mu^{-1}}^{b}/K_{p}],(q_{K_{p}}^{b*})(\mathcal{F}_{\Pi})) \simeq  R\Gamma_{c}([\mathcal{F}\ell_{G,\mu^{-1}}^{b}/K_{p}],j_{K_{p}}^{*}R\pi_{HT,K_{p}!}(\hat{\mathcal{L}}_{\xi})), \]
but now, by \cite[Proposition~2.17]{Han}, we have an isomorphism
\[ R\Gamma_{c}(\Sht(G,b,\mu)_{K_{p}},\mathbb{Z}_{\ell}) \otimes^{\mathbb{L}}_{\mathcal{H}(J_{b})} \mathcal{A}(\mathbf{G}'(\mathbb{Q})\backslash \mathbf{G}'(\mathbb{A}_{f})/K^{p},\hat{\mathcal{L}}_{\xi}) \simeq R\Gamma_{c}([\mathcal{F}\ell_{G,\mu^{-1}}^{b}/K_{p}],(q_{K_{p}}^{b*})(\mathcal{F}_{\Pi})) \]
natural in $K_{p}$. However, by proper base-change, we have an identification
\[ R\Gamma_{c}(\mathcal{S}(\mathbf{G},X)^{b}_{K^{p}K_{p}},\hat{\mathcal{L}}_{\xi}) \simeq R\Gamma_{c}([\mathcal{F}\ell_{G,\mu^{-1}}^{b}/K_{p}],j_{K_{p}}^{*}R\pi_{HT,K_{p}!}(\hat{\mathcal{L}}_{\xi})), \]
and by composing the previous isomorphism with the natural map
\[ R\Gamma_{c}(\mathcal{S}(\mathbf{G},X)^{b}_{K^{p}K_{p}},\hat{\mathcal{L}}_{\xi}) \ra R\Gamma_{c}(\mathcal{S}(\mathbf{G},X)_{K^{p}K_{p}},\hat{\mathcal{L}}_{\xi}) \]
coming from excision applied to the adic sheaf $\hat{\mathcal{L}}_{\xi}$, we obtain a map of the desired form. Since all these constructions are natural in $K_{p}$, if we take the colimit over $K_{p}$ we obtain a uniformization map:
\[ \widehat{\Theta}: R\Gamma_{c}(G,b,\mu) \otimes^{\mathbb{L}}_{\mathcal{H}(J_{b})} \mathcal{A}(\mathbf{G}'(\mathbb{Q})\backslash \mathbf{G}'(\mathbb{A}_{f})/K^{p},\hat{\mathcal{L}}_{\xi}) \rightarrow R\Gamma_{c}(\mathcal{S}(\mathbf{G},X)_{K^{p}},\hat{\mathcal{L}}_{\xi}).  \]
However, using that both sides of $\widehat{\Theta}$ are admissible $G(\mathbb{Q}_{p})$-representations (i.e. its invariants under varying $K_{p}$ give rise to a perfect complex), we can obtain the desired map with $\ol{\mathbb{Q}}_{\ell}$-coefficients by taking smooth vectors with respect to $G(\mathbb{Q}_{p})$ and tensoring by $\ol{\mathbb{Q}}_{\ell}$ (cf. the equivalence \cite[Proposition~2.3]{Han} and Definition 3.2). Here the admissibility of the target of $\hat{\Theta}$ follows from standard finiteness results for Shimura varieties, and the admissibility of the source follows from the fact that $\mathbf{G}'(\mathbb{Q})\backslash\mathbf{G}'(\mathbb{A}_{f})/K^{p}K_{p}$ is a finite set for any compact open $K_{p} \subset \mathbf{G}'(\mathbb{Q}_{p})$ and the fact that the isotypic component of $R\Gamma_{c}(G,b,\mu)$ with respect to any complex of admissible representations is again admissible. This follows for example by the fact that Hecke operators preserve ULA sheaves (\cite[Theorem~I.7.2]{FS}), the relationship between Hecke operators and isotypic parts as in Corollary \ref{cor: isotypcpartcalculation}, and the relationship between ULAness and admissibility (cf. Remark \ref{remark: ULAsheaves} and see \cite[Theorem~I.5.1 (v)]{FS})
\end{proof}
\subsection{Boyer's Trick}
We will now be interested in applying uniformization to the situation we are interested in, proving an analogue of Boyer's trick \cite{Boy1} and deducing some relevant consequences. The first relevant result is due to Shen. \begin{theorem}{\cite[Corollary~6.14]{She}}
If $(\mathbf{G},X)$ is a Shimura datum of abelian type and $p > 2$ is a prime where $G$ is unramified then $(\mathbf{G},X)$ satisfies basic uniformization at $p$.  
\end{theorem}
\begin{remark}
To see a full proof of this statement for the full integral model at hyperspecial level, one can also look at the proof of \cite[Theorem~D]{LiHue}.
\end{remark}
Let $F/\bb{Q}$ be a totally real field and $q$ some odd inert prime. For the rest of the section, we take $\mathbf{G}$ to be a $\mathbb{Q}$-inner form of $\mathbf{G}^{*} := \mathrm{Res}_{F/\mathbb{Q}}\GSp_{4}$, with $F/\mathbb{Q}$ a totally real field satisfying the condition that $\mathbf{G}(\mathbb{R}) \simeq \GSp_{4}(\mathbb{R}) \times \GU_{2}(\mathbb{H})^{[F:\mathbb{Q}] - 1}$, and that it is split at all finite places if $[F:\mathbb{Q}]$ is odd and non-split only at $q$ if $[F:\mathbb{Q}]$ is even. Here $\mathbb{H}$ denotes the Hamilton quaternions. We assume that $p \neq q$ is totally inert and $F_{p} \simeq L$, a fixed unramified extension of $\mathbb{Q}_{p}$ so that $\mathbf{G}_{\mathbb{Q}_{p}} \simeq \mathrm{Res}_{L/\mathbb{Q}_{p}}\GSp_{4} = G$. We fix a level $K = K_{p}K^{p} \subset \mathbf{G}(\mathbb{A}_{f})$ as before, and assume from now on that $(\mathbf{G},X)$ is such that the corresponding cocharacter $\mu$ is the Siegel cocharacter. Therefore, the unique basic $b \in B(G,\mu)$ will have $\sigma$-centralizer given by $\mathrm{Res}_{L/\mathbb{Q}_{p}}\GU_{2}(D)$, with $D/L$ the quaternionic division algebra. Since $L/\mathbb{Q}_{p}$ is unramified and $p > 2$, we can apply Theorem $4.2$ to deduce basic uniformization at $p$. Let $\mathbf{G}'$ be the $\mathbb{Q}$-inner form of $\mathbf{G}$ defined above. Now we prove the following result, which plays a similar role to Boyer's trick \cite{Boy1} in the study of the cohomology of the Lubin-Tate/Drinfeld towers.
\begin{lemma}
For $b \in B(G,\mu)$ non-basic there exists a proper parabolic $P = LU \subset G$, a dominant cocharacter $\mu_{L}$ of the Levi $L$, and an element $b_{L} \in B(L,\mu_{L})$, such that the adic Newton stratum $\mathcal{F}\ell_{G,\mu^{-1}}^{b}$ is parabolically induced as a space with $G(\mathbb{Q}_{p})$-action from the flag variety $\mathcal{F}\ell_{L,\mu_{L}^{-1}}^{b_{L}}$, where $P(\mathbb{Q}_{p})$ acts via the surjection $P(\mathbb{Q}_{p}) \ra L(\mathbb{Q}_{p})$.
\end{lemma}
\begin{proof}
We recall \cite[Definition~4.28]{RV} that we say $b \in B(G,\mu)$ is Hodge-Newton reducible if there exists a proper Levi subgroup $L$ together with a basic element $b_{L} \in B(L,\mu_{L})$ mapping to $b \in B(G,\mu)$ under the natural map $B(L) \ra B(G)$, where $\mu_{L}$ is a choice of representative for $\mu$ as a geometric dominant cocharacter of $L$. Given such a $b$, we let $P$ be the standard parabolic of $G$ with respect to a choice of Borel such that its Levi factor is $L$, and fix $\mu_{L}$ to be the conjugacy class of dominant cocharacters of $L$ that is dominant with respect to $B$. We will also assume that $b_{L}$ can always be chosen so that $J_{b_{L}} \xrightarrow{\simeq} J_{b}$ via the map induced by $B(L) \ra B(G)$. 

We recall that, if $\mathcal{J}_{b}$ denotes the group diamond parametrizing automorphisms of the bundle corresponding to $b \in B(G)$, this has a semi-direct product decomposition 
\[ \mathcal{J}_{b} \simeq \ul{J_{b}(\mathbb{Q}_{p})} \ltimes \mathcal{J}_{b}^{U} \]
given by the splitting of the HN-filtration of $\mathcal{E}_{b}$ \cite[Proposition~III.5.1]{FS}. The flag variety $\mathcal{F}\ell_{G,\mu^{-1}}^{b}$ identifies with the moduli space $\mathcal{E} \dashrightarrow \mathcal{E}_{0}$ parametrizing modifications on $X_{S}$ of meromorphy $\mu$, where $\mathcal{E}_{0}$ denotes the trivial $G$-bundle and $\mathcal{E}$ is a bundle which is isomorphic to $\mathcal{E}_{b}$ after pulling back to a geometric point of $S$ under our conventions. We then have the local Hodge-Tate period map 
\[ \pi_{\mathrm{HT}}^{b}: \Sht(G,b,\mu)_{\infty,\mathbb{C}_{p}} \ra \mathcal{F}\ell_{G,\mu^{-1}}^{b}, \]
which is the $\mathcal{J}_{b}$-torsor defined by rigidifying $\mathcal{E} \simeq \mathcal{E}_{b}$. The space $\Sht(G,b,\mu)_{\infty,\mathbb{C}_{p}}$ has an action of $G(\mathbb{Q}_{p})$ and $\mathcal{J}_{b}$ coming from the automorphisms of the $G$-bundles $\mathcal{E}_{0}$ and $\mathcal{E}_{b}$, respectively. It now follows by \cite[Proposition~4.13]{IG} that $P(\mathbb{Q}_{p}) \subset G(\mathbb{Q}_{p})$ stabilizes a subspace $\mathcal{C}^{\mu_{L}}_{b_{L}} \subset \Sht(G,b,\mu)_{\infty}$, and that we have a $G(\mathbb{Q}_{p}) \times \mathcal{J}_{b}$-equivariant isomorphism of diamonds:
\[ \mathcal{C}_{b_{L}}^{\mu_{L}} \times^{P(\mathbb{Q}_{p})} G(\mathbb{Q}_{p}) \simeq \Sht(G,b,\mu)_{\infty}. \] 
By \cite[Proposition~4.24]{IG} and our assumption that $J_{b_{L}} \xrightarrow{\simeq} J_{b}$, the space $\mathcal{C}_{b_{L}}^{\mu_{L}}$ is isomorphic to 
\[ \Sht(L,b_{L},\mu_{L})_{\infty} \times \mathcal{J}_{b}^{U}, \]
where one can see by looking at the definition of the isomorphism in the proof of \cite[Proposition~4.24]{IG} that the resulting action of $P(\mathbb{Q}_{p})$ on the $\Sht(L,b_{L},\mu_{L})_{\infty}$ factor is via the surjection $P(\mathbb{Q}_{p}) \ra L(\mathbb{Q}_{p})$ and the natural action of $L(\mathbb{Q}_{p})$ on the trivial $L$-bundle $\mathcal{E}_{L}^{0}$ by automorphisms. Now, by quotienting out by the $\mathcal{J}_{b} \simeq J_{b}(\mathbb{Q}_{p}) \ltimes \mathcal{J}_{b}^{U}$-action on $\Sht(G,b,\mu)_{\infty}$, we deduce that 
\[ [\Sht(G,b,\mu)_{\infty}/\mathcal{J}_{b}] \simeq \mathcal{F}\ell_{G,\mu^{-1}}^{b} \] is parabolically induced from 
\[ [\Sht(L,b_{L},\mu_{L})_{\infty} \times \mathcal{J}_{b}^{U}/J_{b}(\mathbb{Q}_{p}) \ltimes \mathcal{J}_{b}^{U}] \simeq [\Sht(L,b_{L},\mu_{L})_{\infty}/J_{b_{L}}(\mathbb{Q}_{p})] \simeq \mathcal{F}\ell_{L,\mu^{-1}_{L}}^{b_{L}} \]
where $P(\mathbb{Q}_{p})$ acts on $\mathcal{F}\ell_{L,\mu^{-1}_{L}}^{b_{L}}$ via the natural surjection $P(\mathbb{Q}_{p}) \ra L(\mathbb{Q}_{p})$ and the last isomorphism is given by the local Hodge-Tate period map for the local Shimura datum $(L,b_{L},\mu_{L})$.

Therefore, the claim will follow from checking that the non-basic elements in $B(G,\mu)$ are Hodge-Newton reducible in such a way that the basic element $b_{L}$ can be chosen in such a way that $J_{b_{L}} \simeq J_{b}$. This follows from the classification of this condition in \cite[Theorem~3.5]{GHN}. For clarity, we recall how this works in our case. There are two non-basic elements $b \in B(G,\mu)$. The $\mu$-ordinary element, defined by the maximal element $b^{max} \in B(G,\mu)$ with respect to the partial ordering on $B(G,\mu)$, and the intermediate stratum corresponding to the element lying between the basic element and $b^{max}$ with respect to the partial ordering on $B(G)$. The element $b^{\max}$ admits a reduction to the unique element $b_{T} \in B(T,\mu_{T})$, where $T$ is the maximal torus inside $\GSp_{4}/L$. If we take the image of $b_{T}$ inside the Kottwitz set defined by the Levi of the Siegel parabolic then we can arrange that $J_{b_{L}} \simeq J_{b}$, as well. Similarly, the element $b \in B(G,\mu)$ corresponding to the intermediate strata admits a reduction to the unique basic element $b_{L} \in B(L,\mu_{L})$ such that $J_{b_{L}} \xrightarrow{\simeq} J_{b}$, where $L$ is the Levi factor of the Klingen parabolic of $\GSp_{4}$.
\end{proof}
We now consider some irreducible algebraic representation of $\mathbf{G}/\bb{Q}$ with highest weight $\xi$ as in the previous section, and look at the uniformization map
\[ \Theta: R\Gamma_{c}(G,b,\mu) \otimes^{\mathbb{L}}_ {\mathcal{H}(J_{b})} \mathcal{A}(\mathbf{G}'(\mathbb{Q})\backslash \mathbf{G}'(\mathbb{A}_{f})/K^{p},\mathcal{L}_{\xi}) \rightarrow R\Gamma_{c}(\mathcal{S}(\mathbf{G},X)_{K^{p}},\mathcal{L}_{\xi})  \]
furnished by Proposition 4.1 and Theorem 4.2. We let $R\Gamma_{c}(G,b,\mu)_{sc}$ and $R\Gamma_{c}(\mathcal{S}(\mathbf{G},X)_{K^{p}},\mathcal{L}_{\xi})_{sc}$ be the direct summands given by the supercuspidal Bernstein components of $G(\bb{Q}_{p})$. Then we have the following key consequence of the previous lemma, which justifies why we are referring to this as Boyer's trick. 
\begin{proposition}
The uniformization map $\Theta$ induces an isomorphism
\[ \Theta_{sc}: R\Gamma_{c}(G,b,\mu)_{sc} \otimes^{\mathbb{L}}_ {\mathcal{H}(J_{b})} \mathcal{A}(\mathbf{G}'(\mathbb{Q})\backslash \mathbf{G}'(\mathbb{A}_{f})/K^{p},\mathcal{L}_{\xi}) \xrightarrow{\simeq} R\Gamma_{c}(\mathcal{S}(\mathbf{G},X)_{K^{p}},\mathcal{L}_{\xi})_{sc}  \]
on the summand given by the supercuspidal Bernstein components of $G(\mathbb{Q}_{p})$. 
\end{proposition}
\begin{proof}
As we will want to apply excision, we will show this by first proving a claim in the setting of adic $\mathbb{Z}_{\ell}$-complete sheaves, as in the proof of Proposition 4.1. We will use the notation introduced there. We let $\mathcal{F}\ell_{G,\mu^{-1}}^{nbas}$ denote the closed complement of $\mathcal{F}\ell_{G,\mu^{-1}}^{b}$ for $b \in B(G,\mu)$ the unique basic element. The space $\mathcal{F}\ell_{G,\mu^{-1}}^{nbas}$ is stratified by $\mathcal{F}\ell_{G,\mu^{-1}}^{b}$ for $b \in B(G,\mu)$ non-basic. For varying $K_{p} \subset G(\mathbb{Q}_{p})$, we set $[\mathcal{F}\ell_{G,\mu^{-1}}^{nbas}/\ul{K_{p}}]$ to be the $v$-stack quotient. We let $\mathcal{S}(\mathbf{G},X)_{K^{p}K_{p}}^{nbas}$ denote the preimage of $[\mathcal{F}\ell_{G,\mu^{-1}}^{nbas}/\ul{K_{p}}]$ under the map $\pi_{HT,K_{p}}: \mathcal{S}(\mathbf{G},X)_{K^{p}K_{p}} \ra [\mathcal{F}\ell_{G,\mu^{-1}}/\ul{K_{p}}]$. It follows from the proof of Proposition 4.1 that the cone of the $\mathbb{Z}_{\ell}$-complete uniformization map $\widehat{\Theta}$ is identified with the $G(\mathbb{Q}_{p})$-representation
\[ R\Gamma_{c}(\mathcal{S}(\mathbf{G},X)_{K^{p}}^{nbas},\hat{\mathcal{L}}_{\xi}) := \colim_{K_{p} \ra \{1\}} R\Gamma_{c}(\mathcal{S}(\mathbf{G},X)_{K^{p}K_{p}}^{nbas},\hat{\mathcal{L}}_{\xi}). \]
We now consider the Hodge-Tate period morphism quotiented out by $G(\mathbb{Q}_{p})$ restricted to this locus
\[ [\mathcal{S}(\mathbf{G},X)^{nbas}_{K^{p}}/\ul{G(\mathbb{Q}_{p})}] \ra [\mathcal{F}\ell_{G,\mu^{-1}}^{nbas}/\ul{G(\mathbb{Q}_{p})}], \]
which we will denote as $\pi_{HT}^{nbas}$.

We write $h_{G}^{\ra}: [\mathcal{F}\ell_{G,\mu^{-1}}^{nbas}/\ul{G(\mathbb{Q}_{p})}] \ra [\ast/\ul{G(\mathbb{Q}_{p})}]$ for the natural map remembering the isomorphism class of the trivial $G$-bundle. We have an identification: 
\[ h_{G!}^{\ra}R\pi^{nbas}_{HT!}(\hat{\mathcal{L}}_{\xi}) \simeq R\Gamma_{c}(\mathcal{S}(\mathbf{G},X)_{K_{p}}^{nbas},\hat{\mathcal{L}}_{\xi}) \]
of objects in $\hat{D}(G(\mathbb{Q}_{p}),\mathbb{Z}_{\ell})$, the derived category of $\mathbb{Z}_{\ell}$-complete smooth representations, as in \cite[Definition~2.5]{Han} (see also \cite[Proposition~2.6 (1)]{Han} for the identification of this category with $\bb{Z}_{\ell}$-adic sheaves on the classifying stack). Here we again recall that compactly supported cohomology respects taking limits of spaces since it is defined in terms of the colimit over the qcqs opens (however, this fails for regular cohomology unless the spaces in the limit are qcqs). Now we can apply excision to the sheaf $R\pi^{nbas}_{HT!}(\hat{\mathcal{L}}_{\xi})$, and the stratification of $[\mathcal{F}\ell^{nbas}_{G,\mu^{-1}}/\ul{G(\mathbb{Q}_{p})}]$ by $[\mathcal{F}\ell^{b}_{G,\mu^{-1}}/\ul{G(\mathbb{Q}_{p})}]$ for $b \in B(G,\mu)$ non-basic, which is well-defined since the adic Newton stratification is $G(\mathbb{Q}_{p})$-equivariant. We set $A_{b} := j_{b}^{*}R\pi^{nbas}_{HT!}(\hat{\mathcal{L}}_{\xi})$. It follows by Lemma 4.3 and the definition of being parabolically induced that we have an isomorphism of $v$-stacks \[ [\mathcal{F}\ell_{G,\mu^{-1}}^{b}/\ul{G(\mathbb{Q}_{p})}] \simeq [\mathcal{F}\ell_{L,\mu^{-1}_{L}}^{b_{L}}/\ul{P(\mathbb{Q}_{p})}],  \] and, therefore, $A_{b}$ gives rise to a sheaf $A_{b}^{P}$ on $\mathcal{F}\ell_{L,\mu^{-1}_{L}}^{b_{L}}$ with $P(\mathbb{Q}_{p})$-equivariant structure. We consider the commutative diagram 
\[ 
\begin{tikzcd}
\left[\mathcal{F}\ell^{b}_{G,\mu^{-1}}/\underline{G(\mathbb{Q}_{p})}\right] \arrow[r,"h^{\ra,b}_{G}"]& \left[\Spd(\mathbb{C}_{p})/\underline{G(\mathbb{Q}_{p})}\right] & \\
\left[\mathcal{F}\ell^{b_{L}}_{L,\mu_{L}^{-1}}/\underline{P(\mathbb{Q}_{p})}\right] \arrow[u,"\simeq"]  \arrow[r,"h_{P}^{\ra,b}"]  & \left[\Spd(\mathbb{C}_{p})/\ul{P(\mathbb{Q}_{p})}\right]  \arrow[u,"p"]   & 
\end{tikzcd}
\]
Here $h^{\ra,b}_{G} := h_{G}^{\ra} \circ j_{b}$, where $j_{b}: [\mathcal{F}\ell_{G,\mu^{-1}}^{b}/\ul{G(\mathbb{Q}_{p})}] \hookrightarrow [\mathcal{F}\ell_{G,\mu^{-1}}/\ul{G(\mathbb{Q}_{p})}]$ is the natural inclusion of the locally closed adic Newton stratum. By proper base-change, we have an  identification
\[ h^{\ra}_{G!}j_{b!}j_{b}^{*}R\pi^{nbas}_{HT!}(\hat{\mathcal{L}}_{\xi}) = h^{\ra}_{G!}j_{b!}(A_{b}) \simeq h^{\ra,b}_{G!}(A_{b}), \]
as $\mathbb{Z}_{\ell}$-complete $G(\mathbb{Q}_{p})$-representations. Using the above commutative diagram, we deduce that this is isomorphic to
\[ p_{!}h_{P!}^{\ra,b}(A^{P}_{b}). \]
We now invoke the following.
\begin{theorem}
Consider the composition $h^{\la,b}_{G}: [\mathcal{F}\ell_{G,\mu^{-1}}^{b}/\underline{G(\bb{Q}_{p})}] \ra [\Spa(\bb{C}_{p})/\mathcal{J}_{b}] \ra [\Spa(\bb{C}_{p})/\underline{J_{b}(\bb{Q}_{p})}]$. Here the first morphism corresponds to the $\mathcal{J}_{b}$-torsor given by the local Hodge-Tate period map $\pi_{\mathrm{HT}}^{b}$, and the second morphism corresponds to the natural projection $[\Spa(\bb{C}_{p})/\mathcal{J}_{b}] \ra [\Spa(\bb{C}_{p})/\underline{J_{b}(\bb{Q}_{p})}]$. The sheaf $j_{b}^{*}R\pi_{HT!}(\hat{\mathcal{L}}_{\xi})$ is the pullback along $h^{\la,b}_{G}$ of a complex of $\bb{Z}_{\ell}$-complete smooth $J_{b}(\bb{Q}_{p})$-representations.
\end{theorem}
\begin{proof}
First, we note that if $F = \bb{Q}$ then $(\mathbf{G},X)$ is the Shimura datum for the Siegel threefold. In particular, it is of PEL type $C$ and the claim then follows from \cite[Corollary~3.17,Remark~3.14]{HamLeeTorsVan} and proper base-change together with the equivalence  $\Detale(\Bun_{G}^{b},\bb{Z}_{\ell}) \simeq \hat{\D}(J_{b}(\bb{Q}_{p}),\bb{Z}_{\ell})$ (\cite[Proposition~V.2.2]{FS}).

We can therefore assume that $F \neq \bb{Q}$. In this case, the group $\mathbf{G}^{\mathrm{der}}$ is anisotropic and therefore the Shimura variety is compact by \cite[Theorem~1]{BailyBorel}. It follows by the compactness (so that the good reduction locus gives the entire Shimura variety) and the main result of \cite{MingjiaPolII} (see also \cite{MingjiaPolI} for the Hodge-type case, which the abelian type case can be reduced to), that there exists a $v$-stack $\Igs_{K^{p}}$ mapping to $\Bun_{G}$ via a map $\pi_{\Igs}$, and a Cartesian diagram
\[
\begin{tikzcd}
\left[\mathcal{S}(\mathbf{G},X)_{K^{p}}/\underline{G(\bb{Q}_{p})}\right] \arrow[r,"\pi_{\HT}"] \arrow[d] & \left[\mathcal{F}\ell_{G,\mu^{-1}}/\underline{G(\bb{Q}_{p})}\right] \arrow[d,"h^{\la}_{G}"] \\
\mathrm{Igs}_{K^{p}} \arrow[r,"\pi_{\Igs}"] & \Bun_{G,\bb{C}_{p}},
\end{tikzcd}
\]
where $h^{\la}_{G}$ is the map sending a modification $\mathcal{E} \dashrightarrow \mathcal{E}_{0}$ to $\mathcal{E}$. In particular, note that $h_{G}^{\la}$ pulls back along the locally closed immersion $\Bun_{G,\bb{C}_{p}}^{b} \hookrightarrow \Bun_{G,\bb{C}_{p}}$ to the morphism $[\mathcal{F}\ell^{b}_{G,\mu^{-1}}/\underline{G(\bb{Q}_{p})}] \ra [\Spa(\bb{C}_{p})/\mathcal{J}_{b}] \simeq \Bun_{G,\bb{C}_{p}}^{b}$ corresponding to the torsor given by $\pi_{\HT}^{b}$. The claim now follows just as in the case where $F = \bb{Q}$.
\end{proof}
We write $h^{\la,b}_{P}: [\mathcal{F}\ell_{L,\mu_{L}^{-1}}^{b_{L}}/\underline{P(\bb{Q}_{p})}] \simeq [\mathcal{F}\ell_{G,\mu^{-1}}^{b}/\underline{G(\bb{Q}_{p})}] \xrightarrow{h^{\la,b}_{G}} [\Spa(\bb{C}_{p})/\underline{J_{b}(\bb{Q}_{p})}]$. It follows by the proof of Lemma 4.3  that the map $h^{\la,b}_{P}$ identifies with the composition  $[\mathcal{F}\ell_{L,\mu_{L}^{-1}}^{b_{L}}/\underline{P(\bb{Q}_{p})}] \ra [\mathcal{F}\ell_{L,\mu_{L}^{-1}}^{b_{L}}/\underline{L(\bb{Q}_{p})}] \ra \Bun_{L}^{b_{L}} = [\Spa(\bb{C}_{p})/\underline{J_{b_{L}}(\bb{Q}_{p})}] \simeq [\Spa(\bb{C}_{p})/\underline{J_{b}(\bb{Q}_{p})}]$, where the last morphism is the map corresponding to the local Hodge-Tate period morphism $\pi_{\HT}^{b_{L}}$ attached to $\Sht(L,b_{L},\mu_{L})_{\infty} \ra \mathcal{F}\ell_{L,\mu_{L}^{-1}}^{b_{L}}$, and the first morphism is the natural map induced by $P(\bb{Q}_{p}) \ra L(\bb{Q}_{p})$. In particular, we deduce that the complex $A_{b}^{P}$ is pulled back from a sheaf $A_{b}^{L}$ along the map $[\mathcal{F}\ell_{L,\mu_{L}^{-1}}^{b_{L}}/\underline{P(\bb{Q}_{p})}] \ra [\mathcal{F}\ell_{L,\mu_{L}^{-1}}^{b_{L}}/\underline{L(\bb{Q}_{p})}]$. We have a Cartesian diagram
\[ \begin{tikzcd}
\left[\mathcal{F}\ell^{b_{L}}_{L,\mu_{L}^{-1}}/\underline{P(\mathbb{Q}_{p})}\right]  \arrow[r,"h_{P}^{\ra,b}"] \arrow[d] & \left[\Spa(\mathbb{C}_{p})/\ul{P(\mathbb{Q}_{p})}\right]  \arrow[d,"q"]  & \\
\left[\mathcal{F}\ell_{L,\mu_{L}^{-1}}^{b_{L}}/\ul{L(\mathbb{Q}_{p})}\right] \arrow[r,"h_{L}^{\ra,b}"]  & \left[\Spa(\bb{C}_{p})/\ul{L(\bb{Q}_{p})}\right], 
\end{tikzcd} \] 
by virtue of the fact that $U(\bb{Q}_{p})$ acts trivially on $\mathcal{F}\ell_{L,\mu_{L}^{-1}}^{b_{L}}$ as explained in the proof of Lemma 4.3, where $U$ denotes the unipotent radical of $P$. By applying proper base-change to this Cartesian square, we obtain an isomorphism 
\[ h_{P!}^{\ra,b}A_{b}^{P} \simeq q^{*}h_{L!}^{\ra,b}A_{b}^{L} \]
and combining with the above this implies that 
\[ h^{\ra}_{G!}j_{b!}j_{b}^{*}R\pi^{nbas}_{HT!}(\hat{\mathcal{L}}_{\xi}) \simeq p_{!}q^{*}h_{L!}^{\ra,b}A_{b}^{L}.  \]
However, the functor $p_{!}q^{*}: \hat{D}(L(\bb{Q}_{p}),\bb{Z}_{\ell}) \ra \hat{D}(G(\bb{Q}_{p}),\bb{Z}_{\ell})$ is easily identified with the unnormalized parabolic induction functor. Therefore, we have constructed a filtration on the cone of $\widehat{\Theta}$ as a $\mathbb{Z}_{\ell}$-complete $G(\mathbb{Q}_{p})$-representation whose graded pieces are parabolically induced. If we now take smooth vectors and tensor by $\ol{\mathbb{Q}}_{\ell}$ as in the proof of Proposition 4.1 then this implies the desired claim.
\end{proof}
Proposition 4.4 will be a key tool in allowing us to describe the $W_{L}$-action on the local Shimura variety by global methods. We start this global analysis by constructing strong transfers between $\GSp_{4}$ and its inner forms over a number field $F$ and proving a strong multiplicity one result. 
\section{Existence of Strong Transfers and a Strong Multiplicity One Result}
In this section, we will show the existence of strong transfers of certain automorphic representations of an inner form of $\GSp_{4}$ over a number field $F$, using the analysis of the trace formula similar to that of \cite[Section~6]{KS}. We will then combine this with analysis of the simple twisted trace formula of Kottwitz-Shelstad \cite{KoShe}, to deduce a kind of strong multiplicity one result for inner forms of $\GSp_{4}$. 
\subsection{The Simple Trace Formula and Existence of Strong Transfers}
In order to describe the Galois action on the global Shimura variety, we will need to construct strong transfers for inner forms of $\GSp_{4}/F$ over a totally real field $F$. This will allow us to compute the traces of Frobenius on the global Shimura variety in terms of the Langlands parameters of the strong transfer. The construction of strong transfers will be accomplished by applying the elliptic part of the stable trace formula with respect to the Lefschetz functions constructed by Kret-Shin \cite{KS} at the Steinberg/Infinite places and pseudo-coefficients at some finite number of places where the representation has supercuspidal $L$-parameter, applying the character identities of Chan-Gan \cite{CG} to conclude equality of the orbital integrals at these latter places. First, we recall the key results of Kret-Shin on the trace formula with fixed central character. For now, we will work generally. Let $\mathbf{G}$ denote a connected reductive group over a number field $F$ with center $Z$, write $A_{Z}$ for the maximal $\mathbb{Q}$-split torus of $\mathrm{Res}_{F/\mathbb{Q}}Z$, and set $A_{Z,\infty} = A_{Z}(\mathbb{R})^{0}$ to be the connected component of the identity. Let $\mathbb{A}_{F}$ be the adeles of $F$ and write $\mathbf{G}(\mathbb{A}_{F})^{1}$ for a choice of subgroup so that $\mathbf{G}(\mathbb{A}_{F}) = \mathbf{G}(\mathbb{A}_{F})^{1} \times A_{Z,\infty}$, as in \cite[Page~11]{Ar2}. We consider a closed subgroup $\mathfrak{X} \subset Z(\mathbb{A}_{F})$ which contains $A_{Z,\infty}$ such that $Z(F)\mathfrak{X}$ is closed inside $Z(\mathbb{A}_{F})$ and a continuous character $\chi: (\mathfrak{X} \cap Z(F))\backslash\mathfrak{X} \rightarrow \mathbb{C}^{*}$. 
\\\\
We write $\Pl_{F}$ for the set of places of $F$, and, for $v \in \Pl_{F}$, let $\mathfrak{X}_{v} \subset Z(F_{v})$ denote a closed subgroup. We let $\chi_{v}: \mathfrak{X}_{v} \rightarrow \mathbb{C}^{*}$ be a smooth character. Write $\mathcal{H}(\mathbf{G}(F_{v}),\chi_{v}^{-1})$ for the space of smooth compactly supported functions modulo $\mathfrak{X}_{v}$ on $\mathbf{G}(F_{v})$ which transform under $\mathfrak{X}_{v}$ via $\chi_{v}^{-1}$. We also require the functions to be $K_{v}$-finite for some maximal compact subgroup $K_{v}$ of $\mathbf{G}(F_{v})$ if $v$ is archimedean. We now assume that $\mf{X}$ as above can be written as restricted product $\prod'_{v \in \mathrm{Pl}_{F}} \mathfrak{X}_{v}$ in $Z(\mathbb{A}_{F})$, for $\mathfrak{X}_{v} \subset Z(F_{v})$ a closed subgroup. Given a semisimple element $\gamma_{v} \in \mathbf{G}(F_{v})$ and an admissible representation $\pi_{v}$ of $\mathbf{G}(F_{v})$ with central character $\chi_{v}$ on $\mathfrak{X}_{v}$, we define the orbital integral and trace character for $f_{v} \in \mathcal{H}(\mathbf{G}(F_{v}),\chi_{v}^{-1})$ as follows. Let $I_{\gamma_{v}}$ denote the connected centralizer of $\gamma_{v}$ in $\mathbf{G}$. We have
\[ O_{\gamma_{v}}(f_{v}) := \int_{I_{\gamma_{v}}(F_{v})\backslash \mathbf{G}(F_{v})} f_{v}(x^{-1}\gamma_{v}x)dx \]
and
\[ \tr(f_{v}|\pi_{v}) := \tr(\int_{\mathbf{G}(F_{v})/\mathfrak{X}_{v}} f_{v}(g)\pi_{v}(g)dg) \]
where we have fixed compatible choices of Haar measure on $\mathbf{G}$ and $\mathfrak{X}_{v}$ throughout. We note that this operator is well defined since the above operator is of finite rank if $v$ is finite and is of trace class if $v$ is infinite, by the $K_{v}$-finiteness assumption.
\\\\
We define the adelic Hecke algebra $\mathcal{H}(\mathbf{G}(\mathbb{A}_{F}),\chi^{-1})$, as well as the global orbital integrals by taking a restricted tensor product over the local Hecke algebras defined above and products of the local integrals. Write $\Gamma_{ell}(\mathbf{G})$ to be the set of $\mathfrak{X}$-orbits of $F$-elliptic conjugacy classes in $\mathbf{G}(F)$. Let $\mathbb{A}_{f,F}$ denote the finite adeles of $F$. For our purposes, it will suffice to consider a central character datum $(\mathfrak{X},\chi)$, where $\mathfrak{X} = Z(\mathbb{A}_{F})$, and we write $\chi = \bigotimes_{v \in \Pl_{F}} \chi_{v}$ for smooth characters $\chi_{v}: Z(F_{v}) \ra \mathbb{C}^{*}$ and all $v \in \Pl_{F}$. We let $L^{2}_{disc,\chi}(\mathbf{G}(F)\backslash \mathbf{G}(\mathbb{A}_{F}))$ denote the discrete summand of the space of functions on $\mathbf{G}(F) \backslash \mathbf{G}(\mathbb{A}_{F})$ transforming under $\mathfrak{X}$ by $\chi$ and square-integrable on $\mathbf{G}(F)\backslash \mathbf{G}(\mathbb{A}_{F})^{1}/\mathfrak{X}(\mathbb{A}_{F}) \cap G(\mathbb{A}_{F})^{1}$. Write $\mathcal{A}_{cusp,\chi}(\mathbf{G})$ for the set of isomorphism classes of cuspidal automorphic representations of $\mathbf{G}(\mathbb{A}_{F})$ whose central characters restricted to $\mathfrak{X}$ are $\chi$. For $f \in \mathcal{H}(\mathbf{G}(\mathbb{A}_{F}),\chi^{-1})$, define the invariant distributions $T_{ell,\chi}^{\mathbf{G}}$ and $T_{disc,\chi}^{\mathbf{G}}$ by
\[ T_{ell,\chi}^{\mathbf{G}}(f) := \sum_{\gamma \in \Gamma_{ell}(\mathbf{G})} i(\gamma)^{-1}vol(I_{\gamma}(F)\backslash I_{\gamma}(\mathbb{A}_{F})/\mathfrak{X}(\mathbb{A}_{F}))O_{\gamma}(f) \]
\[ T^{\mathbf{G}}_{disc,\chi}(f) := \tr(f|L^{2}_{disc,\chi}(\mathbf{G}(F)\backslash \mathbf{G}(\mathbb{A}_{F}))), \]
where $i(\gamma)$ is the number of connected components in the centralizer of $\gamma$.
Analogously, we define $T^{\mathbf{G}}_{cusp,\chi}(f)$ by taking the trace on the space of square-integrable cusp forms whose central character restricted to $\mathfrak{X}$ is $\chi$. Let $\mathbf{G}^{*}$ denote the quasi-split inner form of $\mathbf{G}$ over $F$, with a fixed inner twist $\mathbf{G}^{*} \simeq \mathbf{G}$ over $\overline{F}$.  Since $Z$ is canonically identified with the center of $\mathbf{G}^{*}$, we may view the character $\chi$ as a central character datum for $\mathbf{G}^{*}$. We then let $f^{*}$ denote a Langlands-Shelstad transfer of $f$ to $\mathbf{G}^{*}$. One can construct such a transfer by lifting $f$ along the surjection $\mathcal{H}(\mathbf{G}(\mathbb{A}_{F})) \rightarrow  \mathcal{H}(\mathbf{G}(\mathbb{A}_{F}),\chi^{-1})$ applying the transfer with trivial central character due to Waldspurger and then taking the image along the analogous surjection for $\mathbf{G}^{*}(\mathbb{A}_{F})$. We let $\Sigma_{ell,\chi}(\mathbf{G}^{*})$ denote the set of $Z(\mathbb{A}_{F})$ orbits of stable $F$-elliptic conjugacy classes in $\mathbf{G}^{*}(F)$. We define the stable elliptic distribution
\[ ST^{\mathbf{G}^{*}}_{ell,\chi}(f^{*}) := \tau(\mathbf{G}^{*}) \sum_{\gamma \in \Sigma_{ell,\chi}(\mathbf{G}^{*})} \tilde{i}(\gamma)^{-1}SO_{\gamma,\chi}^{\mathbf{G}^{*}}(f^{*}) \]
where $SO^{\mathbf{G}^{*}}_{\gamma,\chi}(f^{*})$ denotes the stable orbital integral of $f^{*}$ at $\gamma$, $\tau(\mathbf{G}^{*})$ is the Tamagawa number of $\mathbf{G}^{*}$, and $\tilde{i}(\gamma)$ is the number of Galois fixed connected components of the centralizer of $\gamma$ in $\mathbf{G}^{*}$. Let $\xi$ be an irreducible representation of $\mathbf{G}_{F_{\infty}}$. Denote by $\chi_{\xi}: Z(F_{\infty}) \rightarrow \mathbb{C}^{*}$ the restriction of $\xi^{\vee}$ to $Z(F_{\infty})$. Write $f_{\xi}^{\mathbf{G}} \in \mathcal{H}(\mathbf{G}(F_{\infty}),\chi_{\xi}^{-1})$ for a Lefschetz function associated with $\xi$. In other words, a function such that $\tr(f^{\mathbf{G}}_{\xi}|\pi_{\infty})$ computes the Euler-Poincar\'e characteristic for the relative Lie algebra cohomology of $\pi_{\infty} \otimes \xi$ for every irreducible admissible representation $\pi_{\infty}$ of $\mathbf{G}(F_{\infty})$ with central character $\chi_{\xi}$. It follows by the Vogan-Zuckerman classification \cite{VZ} that, if $\xi$ is regular and $\pi_{\infty}$ is essentially unitary, this will be non-zero if and only if $\pi_{\infty}$ is an (essentially) discrete series representation cohomological of regular weight $\xi$. For a finite place $v_{st}$ of $F$, we let $f_{v_{st}} := f_{Lef,v_{st}}^{\mathbf{G}} \in \mathcal{H}(\mathbf{G}(F_{v_{st}}))$ denote a Lefschetz function at $v_{st}$. Morally, this should be characterized by the property that $\tr(f_{v_{st}}|\pi_{v_{st}})$ computes the Euler-Poincar\'e characteristic of the continuous group cohomology of $\mathbf{G}(F_{v_{st}})$ valued in $\pi_{v_{st}}$. In the case that the center is anisotropic, it follows from the computations in \cite[Theorem~XI.3.9]{BW} that this means that the trace of $f_{Lef,v_{st}}^{\mathbf{G}}$ will only be non-zero if $\pi_{v}$ is $1$-dimensional or an unramified twist of the Steinberg representation, for all irreducible admissible unitary $\pi_{v}$. For finite places $v = v_{st}$, they were originally constructed by Kottwitz; however, these results do not apply for the desired application, since the center of $\GSp_{4}$ is not compact. For the construction of these functions in this case and the proof of the property that their traces detect when a representation is $1$-dimensional or an unramified twist of Steinberg, see \cite[Appendix~A]{KS}. 

We now assume for simplicity that the center $Z$ of $\mathbf{G}$ is split, which will be the case in all our applications. We consider a character $\eta: \mathbf{G}(F_{v_{st}}) \ra \mathbb{C}^{*}$ such that $\eta|_{Z(F_{v_{st}})} = \chi_{v_{st}}$. We can define a function $f_{Lef,\eta}^{\mathbf{G}} \in \mathcal{H}(\mathbf{G}(F_{v_{st}}),\chi_{v_{st}}^{-1})$ as $f_{Lef,\eta}^{\mathbf{G}}(g) := \eta^{-1}(g)f_{Lef}^{G/Z}(\ol{g})$, where $\ol{g} \in G(F_{v_{st}})/Z(F_{v_{st}})$ denotes the image of $g$ under the quotient map. This (up to scaling by a non-zero constant) forms a pseudo-coefficient for the Steinberg representation twisted by $\eta$, as constructed in \cite{Kaz,SSt}. In particular, by \cite[Corollary~A.8 (2)]{KS}, we have, for $\pi_{v_{st}}$ an irreducible (essentially) unitary representation of $\mathbf{G}(F_{v_{st}})$ with central character $\chi_{v_{st}}$, that $\tr(f_{Lef,\eta}^{\mathbf{G}}|\pi_{v_{st}}) \neq 0$ if and only if $\pi_{v_{st}}$ is isomorphic to $\eta$ or Steinberg twisted by $\eta$.

Now we have the key lemma, which tells us that with respect to these choices of test functions we get the following "simple trace formula".
\begin{lemma}{\cite[Lemma~6.1,6.2]{KS}}
Fix a central character datum $(Z(\mathbb{A}_{F}),\chi)$ with $\chi_{\xi} = \chi|_{Z(F_{\infty})}$ as above for $\chi_{\xi}$ attached to some regular weight $\xi$. For an element $f \in \mathcal{H}(\mathbf{G}(\mathbb{A}_{F}),\chi^{-1})$ in the global Hecke algebra as above, assume that $f_{\infty} = f_{\xi}^{\mathbf{G}} \in \mathcal{H}(\mathbf{G}(F_{\infty}),\chi_{\xi}^{-1})$ is a Lefschetz function at $\infty$ and assume that $f_{v_{st}} = f^{\mathbf{G}}_{Lef,\eta}$ is the Lefschetz function described above at $v_{st}$ for a character $\eta$ of $\mathbf{G}(F_{v_{st}})$ such that $\eta|_{Z(F_{v_{st}})} = \chi_{v_{st}}$. Then we have an equality:
\[ ST_{ell,\chi}^{\mathbf{G}^{*}}(f^{*}) = T^{\mathbf{G}}_{ell,\chi}(f) = T^{\mathbf{G}}_{disc,\chi}(f) = T^{\mathbf{G}}_{cusp,\chi}(f). \]
\end{lemma}
\begin{proof}
Strictly speaking, the cited Lemmas only prove this in the case where the central character datum is $(Z(F_{\infty}),\chi_{\xi})$; however, the result in this case easily follows using \cite[Corollary~A.8 (3)]{KS} (cf. the proof of \cite[Corollary~8.5]{KS}). 
\end{proof}
Let $S_{st}$ and $S_{sc}$ be disjoint finite sets of finite places of $F$, and let $S_{0}$ be a finite set of places contained in  $S_{st} \cup S_{sc}$. Let $S_{\infty}$ denote the infinite places of $F$. Set $S$ to be a finite set of places containing $S_{st} \cup S_{sc} \cup S_{\infty}$. We assume that the inner twist $\mathbf{G}^{*}$ of $\mathbf{G}$ is trivialized away from $S_{0}$ and $S_{\infty}$ (i.e. for all $v \notin S_{0} \cup S_{\infty}$, the inner twisting gives rise to an isomorphism $\mathbf{G}_{F_{v}} \simeq \mathbf{G}_{F_{v}}^{*}$). In particular, $\mathbf{G}$ is unramified outside $S$ and we can fix a reductive model for $\mathbf{G}$ and $\mathbf{G}^{*}$ over $\mathcal{O}_{F}[1/S]$. By abuse of notation, we write $\mathbf{G}$ and $\mathbf{G}^{*}$ for the integral models of these groups. The inner twist gives an isomorphism $\mathbf{G}^{*}_{F_{v}} \simeq \mathbf{G}_{F_{v}}$ and isomorphisms $\mathbf{G}^{*}_{\mathcal{O}_{F_{v}}} 
\simeq \mathbf{G}_{\mathcal{O}_{F_{v}}}$ of the hyperspecial subgroups determined by this model for the finite places $v \notin S$. The notion of unramified local representation on either side will be defined with respect to this fixed choice of hyperspecial level.
\\\\
For the rest of the section, we will assume that $\mathbf{G}^{*} = \GSp_{4}$. We assume throughout that $\pi$ is a global cuspidal automorphic representation of the group $\mathbf{G}(\mathbb{A}_{F})$ satisfying the following properties:
\begin{enumerate}
    \item $\pi$ is cohomological of some regular weight $\xi$ at infinity.
    \item $\pi_{v}$ is unramified at all finite places $v \notin S$.
    \item $\pi_{v}$ has supercuspidal $L$-parameter for $v \in S_{sc}$.  
    \item $\pi_{v}$ is an unramified twist of Steinberg at all finite places $v \in S_{st}$.
\end{enumerate}
 Note, as in Definition 2.1, we can further partition $S_{sc}$ into $S_{ssc}$, where $\mathrm{std}\circ \phi_{\pi_{v}}$ is irreducible (ssc = stable supercuspidal) and $S_{esc}$, where $\mathrm{std}\circ \phi_{\pi_{v}}$ is reducible (esc = endoscopic supercuspidal).
\\\\
The main result of this section shows the existence of strong transfers from $\mathbf{G}$ to $\mathbf{G}^{*}$ at the places in $S_{sc} \cup S_{st} \cup S_{\infty}$ for a certain class of automorphic representations of $\mathbf{G}$. It is essentially a more refined version of \cite[Proposition~6.3]{KS} in the particular case that $\mathbf{G}^{*} = \GSp_{4}$. 
\begin{theorem}
Suppose that $S_{st}$ is non-empty. Given a $\pi$ as above, there exists a cuspidal automorphic representation $\tau$ of $\mathbf{G}^{*}(\mathbb{A}_{F})$ satisfying the following:
\begin{itemize}
\item $\tau^{S} \simeq \pi^{S}$.
\item At all $v \in S_{sc} \cup S_{st} \cup S_{\infty}$, $\tau_{v}$ has the same Langlands parameter as $\pi_{v}$.
\end{itemize}
Moreover, we can choose $\tau$ to be globally generic. For the first part of the claim, the same is true with the roles of $\tau$ and $\pi$ reversed. 
\end{theorem}
\begin{proof}
First off note that, since $\pi$ is an unramified twist of Steinberg at some finite place, $\tau$, if it exists, is automatically (essentially) tempered at all places (cf. Remark 5.1). It follows, by \cite[Remark~7.4.7]{GeTa}, that the global $L$-packet of $\tau$ therefore contains a globally generic representation. So, if we can show the existence of some $\tau$ with the desired properties, that means we can find $\tau$ globally generic with the same properties. For the former, we now apply the trace formula.
\\\\
Let $\mathfrak{X} = Z(\mathbb{A}_{F})$ and $\chi$ be the central character of $\pi$. We set $f = \bigotimes_{v \in \mathrm{Pl}_{F}} f_{v}$ to be a test function on $\mathbf{G}(\mathbb{A}_{F})$ satisfying the following:
\begin{enumerate}
    \item $f_{\infty} = f_{\xi}^{\mathbf{G}}$ is a Lefschetz/Euler-Poincar\'e function of weight $\xi$ of $\mathbf{G}(F_{\infty})$.
    \item At $v \in S_{st}$, $f_{v} = f_{Lef,\eta_{v}}^{\mathbf{G}}$ is a Lefschetz function for $\mathbf{G}(F_{v})$, where $\eta_{v}$ is the unique character such that $\pi_{v}$ is the Steinberg twisted by $\eta_{v}$.
    \item At $v \in S_{ssc}$, $f_{v} = f_{\pi_{v}}$ is the pseudo-coefficient of $\pi_{v}$, as constructed in \cite{Kaz,SSt}. 
    \item At $v \in S_{esc}$, $f_{v} = f_{\pi_{v}^{+}} + f_{\pi_{v}^{-}}$, where $\{\pi_{v}^{+},\pi_{v}^{-}\}$ is the $L$-packet over $\phi_{\pi_{v}}$ and $f_{\pi_{v}^{\pm}}$ is the pseudo-coefficient of $\pi_{v}^{\pm}$. 
    \item At the finite places $v \notin S$, $f_{v}$ is an arbitrary element of the unramified Hecke algebra.
    \item For $v \in S \setminus (S_{st} \cup S_{\infty} \cup S_{sc})$,  choose $K_{v}$ to be sufficiently small such that $\pi_{v}^{K_{v}} \neq 0$ and let $f_{v}$ be the function defined by the idempotent $K_{v}$ with the appropriate central character (so that $\dim(\pi_{v}^{K_{v}}) =  \mathrm{tr}(f_{v}|\pi_{v}) > 0$). 
\end{enumerate}
Given such a $f$, we choose a test function $f^{*} = \bigotimes_{v \in \mathrm{Pl}_{F}} f_{v}^{*}$ on $\mathbf{G}^{*}(\mathbb{A}_{F})$ satisfying the following:
\begin{enumerate}
    \item $f_{\infty}^{*} = f_{\xi}^{\mathbf{G}^{*}}$ is a Lefschetz/Euler-Poincar\'e for the representation $\xi$ of $\mathbf{G}^{*}(F_{\infty})$.
    \item At $v \in S_{st}$, $f_{v}^{*} = f^{\mathbf{G}^{*}}_{Lef,\eta_{v}}$ is the Lefschetz function for $\mathbf{G}^{*}(F_{v})$.
    \item At $v \in S_{ssc}$, $f_{v}^{*} = f_{\tau_{v}}$ is a pseudo-coefficient of $\tau_{v}$, where $\tau_{v}$ is the unique supercuspidal representation of $\mathbf{G}^{*}(F_{v})$ with Langlands parameter $\phi_{\pi_{v}}$. 
    \item At $v \in S_{esc}$, $f_{v}^{*} = f_{\tau_{v}^{+}} + f_{\tau_{v}^{-}}$, where $\{\tau_{v}^{+},\tau_{v}^{-}\}$ is the $L$-packet over $\phi_{\pi_{v}}$ of $\mathbf{G}^{*}(F_{v})$ and $f_{\tau_{v}^{\pm}}$ is the pseudo-coefficient of $\tau_{v}^{\pm}$.
    \item At the finite places $v \notin S$, $f_{v}^{*} = f_{v}$ is the same element of the unramified Hecke algebra.
    \item For $v \in S \setminus S_{st} \cup S_{\infty} \cup S_{sc}$, choose $f_{v}^{*} = f_{v}$.
\end{enumerate}
Now we wish to check that $f$ and $f^{*}$ are matching up to a non-zero constant $c$, in the sense of \cite[Section~5.5]{KoShe2}. We can check this place by place. For the finite places $v \notin S_{st} \cup S_{sc} \cup S_{\infty}$, this is tautological. For all $v \in S_{\infty} \cup S_{st}$, this follows from \cite[Lemma~A.4,A.11]{KS}. For $v \in S_{sc}$, this follows from the character identities of Chan-Gan \cite[Proposition~11.1]{CG}. Namely, recall for a regular semisimple elliptic element $\gamma$ the orbital integrals of the pseudo-coefficients of a discrete series representation $\pi$ are given by the Harish-Chandra character of $\pi^{\vee}$ evaluated at $\gamma$, and it vanishes if $\gamma$ is non-elliptic \cite[Theorem~K]{Kaz} \cite[Proposition~3.2]{KST}. It follows that the orbital integrals for a sum of pseudo-coefficients over an $L$-packet of $\GSp_{4}$ or its inner form is already a stable distribution, by \cite[Main Theorem]{CG} combined with \cite[Proposition~11.1 (1)]{CG} for the inner form. Therefore, the desired equality of stable orbital integrals reduces to showing an equality of the sum of Harish-Chandra characters over the $L$-packets of supercuspidal parameters, and this is precisely \cite[Proposition~11.1 (1)]{CG}. 

Since the test functions are matching and $S_{st} \neq \emptyset$, we can apply Lemma 5.1 to conclude:
\[ T^{\mathbf{G}^{*}}_{cusp,\chi}(f^{*}) = ST^{\mathbf{G}^{*}}_{ell,\chi}(f^{*}) =  cT^{\mathbf{G}}_{cusp,\chi}(f). \]
Then, by linear independence of characters, we have a relationship
\[ \sum_{\substack{\Pi' \in \mathcal{A}_{cusp,\chi}(\mathbf{G}^{*}) \\\\ \Pi'^{S} \simeq \pi^{S}}} m(\Pi')\tr(f_{S}^{*}|\Pi'_{S})  = c \cdot \sum_{\substack{\Pi \in \mathcal{A}_{cusp,\chi}(\mathbf{G}) \\\\ \Pi^{S} \simeq \pi^{S}}} m(\Pi)\tr(f_{S}|\Pi_{S}), \]
where $m(\Pi)$ (resp. $m(\Pi')$) denotes the multiplicity of $\Pi$ (resp. $\Pi'$) in the cuspidal spectrum of $\mathbf{G}$ (resp. $\mathbf{G}^*$). Now at the infinite places, as soon as $\tr(f_{v}|\Pi_{v}) \neq 0$ at $v \mid \infty$ the regularity condition on $\xi$ implies that $\Pi_{v}$ is an (essentially) discrete series representation cohomological of regular weight $\xi$ and that $\tr(f_{v}|\Pi_{v}) = (-1)^{q(\mathbf{G}_{v})}$ by the Vogan-Zuckerman classification of unitary cohomological representations, where, for each $v \mid \infty$, we set
\[
q(\mathbf{G}_v)
:= \frac{1}{2}\dim_{\mathbb{R}}
\bigl(\mathbf{G}^{\mathrm{ad}}(F_v)/K_v^{\mathrm{ad}}\bigr),
\]
with $K_v^{\mathrm{ad}}$ a maximal compact subgroup of
$\mathbf{G}^{\mathrm{ad}}(F_v)$. At $v_{st} \in S_{st}$ it follows by \cite[Corollary~A.8]{KS} that $\Pi_{v_{st}}$ is either the $\eta_{v_{st}}$-twist of the Steinberg or trivial representation. If $\Pi_{v_{st}}$ were one-dimensional, then the global representation $\Pi$ would also be one-dimensional by a strong-approximation argument \cite[Lemma~6.2]{KST}. In particular, $\Pi_v$ would be one-dimensional for $v \mid \infty$. Since $\xi$ is regular, a one-dimensional representation is not $\xi$-cohomological, so $\mathrm{tr}(f_v | \Pi_v) = 0$ for $v | \infty$ and such $\Pi$ does not contribute to the trace identity. Therefore, every $\Pi$ contributing non-trivially satisfies that $\Pi_{v_{st}}$ is the $\eta_{v_{st}}$-twist of the Steinberg representation.

At the remaining $v \in S_{ssc}$ (resp. $v \in S_{esc}$), it follows from the definition of pseudo-coefficients that, if $\tr(f_{v}|\Pi_{v}) > 0$, we have $\Pi_{v} \simeq \pi_{v}$ (resp. $\Pi_{v} \in \{ \pi_{v}^{+},\pi_{v}^{-} \}$). Similar considerations apply for $\Pi' \in \mathcal{A}_{cusp,\chi}(\mathbf{G}^{*})$ occurring non-trivially in the LHS. In summary, by the above analysis, we can deduce that the RHS of the previous equation is non-zero for the term corresponding to $\pi$ and that all the non-trivial terms on the RHS have the same sign. Therefore, the LHS is also non-zero, and we see, by choosing any non-zero term, that we obtain the desired $\tau$. The converse direction works similarly, where the role of $\mathbf{G}$ and $\mathbf{G}^{*}$, are swapped. 
\end{proof}
\begin{remark}
We note that we crucially used at the places $v \in S_{sc}$ that $\phi_{\pi_{v}}$ was supercuspidal. Otherwise, $\tr(f_{v}|\Pi_{v}) \neq 0 $ wouldn't necessarily imply that $\Pi_{v}$ lies in the $L$-packet over $\phi_{\pi_{v}}$ without assuming that $\Pi_{v}$ is tempered. We note that \cite[Lemma~2.7]{KS} gives temperedness of cuspidal $\xi$-cohomological automorphic representations of $\mathrm{GSp}_4$ which are a twist of the Steinberg representation at some place.
\end{remark}
\subsection{The Stable and $\sigma$-twisted Simple Trace Formula}
For the proof of strong multiplicity one, we will need some more refined analysis of trace formulae. Namely, we will be interested in the discrete part of the stable trace formula in the particular case of $\mathbf{G}^{*} = \GSp_{4}/F$, as discussed in \cite[Section~7.1]{CG}. To this end, fix a central character datum $(\mathfrak{X},\chi) = (Z(\mathbb{A}_{F}),\chi)$ as before. We recall that the unique elliptic proper endoscopic group of $\GSp_{4}/F$ is $C = \GSO_{2,2} \simeq (\GL_{2} \times \GL_{2})/\{(t,t^{-1})| t \in \GL_{1}\}$. Then, for a test function $f^{*}$ on $\mathbf{G}^{*}(\mathbb{A}_{F})$ as above, the discrete part of the stable trace formula is an equality:
\[ I^{\mathbf{G}^{*}}_{disc,\chi}(f^{\mathbf{*}}) = ST^{\mathbf{G}^{*}}_{disc,\chi}(f^{*}) + \frac{1}{4}ST^{C}_{disc,\chi}(f^{C}) \]
for $f^{C}$ a matching test function on $C$. Here
\[ I_{disc,\chi}^{\mathbf{G}^{*}}(f^{*}) = \sum_{M} |W(G,M)|^{-1}\cdot \sum_{s \in W(M,G)_{reg}}|\det(s - 1)_{\mathfrak{a}_{M}/\mathfrak{a}_{G}}|^{-1} \cdot \tr(M_{P}(s,0)\cdot I^{P}_{disc,\chi}(0,f^{*})) \]
is a sum indexed over classes of standard Levi subgroups of $\mathbf{G}^{*}$. The precise definition of the terms will not be important for our purposes, but the interested reader can look at \cite[Section~3]{Ar1}. We simply note that the term corresponding to $M = \mathbf{G}^{*}$ is precisely equal to
$T_{disc,\chi}^{\mathbf{G}^{*}}(f^{*})$, as defined in section 5.1. The distribution $ST^{\mathbf{G}^{*}}_{disc,\chi}$ is a stable distribution on $\mathbf{G}^{*}$, similar to $ST_{ell,\chi}^{\mathbf{G}^{*}}$, and $ST^{C}_{disc,\chi}$ is the analogous stable distribution on $C(\mathbb{A}_{F})$. However, since $C$ has no proper elliptic endoscopic group, we have
\[ ST^{C}_{disc,\chi}(f^{C}) = I_{disc,\chi}^{C}(f^{C}) = T_{disc,\chi}^{C}(f^{C}) + (\text{other terms}) \]
with the other terms indexed by proper standard Levi subgroups of $C$ as above.

We will be interested in combining this with the elliptic part of the twisted trace formula as described by Kottwitz-Shelstad \cite{KoShe} for the particular group $\tilde{\mathbf{G}} := \GL_{4} \times \GL_{1}/F$ with respect to involution
 \[ \sigma: (g,e) \mapsto (J^{t}g^{-1}J^{-1},e\det(g))\]
where 
\[
   J :=
  \left[ {\begin{array}{cccc}
    &  &  & 1 \\
    &  & 1 &  \\
    & -1 &  &   \\ 
  -1 &  &  &   \\
  \end{array} } \right]
\]
We can enumerate the elliptic $\sigma$-twisted endoscopic groups as follows. 
\begin{enumerate}
    \item $\mathbf{G}^{*} = \GSp_{4}$
    \item $C_{E} = \mathrm{Res}_{E/F}\GL_{2}' := \{(g_{1},g_{2}) \in \mathrm{Res}_{E/F}\GL_{2} | \det(g_{1}) = \det(g_{2})\}$
    \item $C_{+}^{E} = (\GL_{2} \times \mathrm{Res}_{E/F}\GL_{1})/\GL_{1}$
\end{enumerate}
where $E$ is an \'etale quadratic $F$-algebra and $E$ is not split in case (3).  By the stabilization of the twisted trace formula \cite{MW16} (see also \cite[Section~7]{CG} for the case at hand), for $\tilde f$ a test function on $\tilde{\mathbf{G}}$, which is moreover  cuspidal at at least $3$ finite places, i.e.\ its regular twisted orbital integrals vanish off the elliptic set there (so that this identity reduces to a simple twisted trace formula), then we have an identity
\[ I^{\tilde{\mathbf{G}},\sigma}_{disc,\chi}(\tilde{f}) = \frac{1}{2}ST_{disc,\chi}^{\mathbf{G}^{*}}(f^{*}) + \frac{1}{4}\sum_{E}ST^{C_{E}}_{disc,\chi}(f^{C_{E}}) + \frac{1}{8}\sum_{E \neq F^{\oplus 2}} ST^{C_{+}^{E}}_{disc,\chi}(f^{C_{+}^{E}}) \]
where
\begin{itemize}
    \item $\sum_{E}$ is a sum over \'etale quadratic $F$-algebras $E$,
    \item $\tilde{f}$, $f^{*}$, $f^{C_{E}}$, and $f^{C_{+}^{E}}$ are matching test functions,
    \item $ST^{\mathbf{G}^{*}}_{disc,\chi}$, $ST^{C_{E}}_{disc,\chi}$, and $ST^{C_{+}^{E}}_{disc,\chi}$ are the stable distributions appearing in the discrete part of the stable trace formula, as described above,
    \item $I^{\tilde{\mathbf{G}},\sigma}_{disc,\chi}$ is the invariant distribution which is the twisted analogue of $I^{\mathbf{G}^{*}}_{disc,\chi}$. It is given by \cite[Theorem~14.3.1 and Proposition~14.3.2]{LW} and has the form
    \[ I_{disc,\chi}^{\tilde{\mathbf{G}},\sigma}(\tilde{f}) = \sum_{M} |W(G,M)|^{-1}\cdot \sum_{s \in W(M,G)_{reg}}|\det(s - 1)_{\mathfrak{a}_{M}/\mathfrak{a}_{G}}^{s\cdot\sigma}|^{-1} \cdot \tr(M_{P}(s,0)\cdot I^{P}_{disc,\chi}(0,\tilde{f})I_{P,disc}(\sigma)), \]
    where the sum runs over standard Levi subgroups $M$ of $G$. 
\end{itemize}
Now we want to apply these trace formulae with respect to appropriately chosen test functions. We will assume that $S_{st}$ is a finite set of places such that $|S_{st}| \geq 3$. Then, for all $v \in S_{st}$, we let $f^{*}_{v}$ be a pseudo-coefficient for the unramified twist of Steinberg by $\eta_{v}$, where $\eta_{v}$ is a character of $\mathbf{G}(F_{v})$ such that the restriction to $Z(F_{v})$ is $\chi_{v}$. In particular, we will take $f^{*}_{v}$ to be the Lefschetz function $f_{Lef,\eta_{v}}^{\mathbf{G}^{*}} \in \mathcal{H}(\mathbf{G}^{*}(F_{v}),\chi_{v}^{-1})$ considered in the previous section which is a pseudo-coefficient for Steinberg up to scaling. It follows, by \cite[Corollary~10.8]{CG}, that we can choose the local constituent of the matching function $\tilde{f}$ at $v$ to be the $\sigma$-twisted pseudo-coefficient of the Steinberg twisted by $\eta_{v}$, as defined in \cite{MW}. The regular twisted orbital integrals of these functions vanish off the elliptic set (see the construction in \cite[Section~7.2]{MW}), and therefore we can apply the simple twisted trace formula. Moreover, $\tilde f_v$ has stable twisted orbital integrals, so its transfers to the elliptic
$\sigma$-twisted endoscopic groups $C_{E,F_v}$ and $C^E_{+,F_v}$ of $(\tilde{\mathbf G}_{F_v},\sigma)$ vanish.
Indeed, write $\eta_v=\mu_v\circ\mathrm{sim}$ and let
$\tilde\pi_v:=(\mathrm{St}_{\mathrm{GL}_4}\otimes\mu_v\circ\det)\boxtimes\mu_v^2$ be the $\sigma$-stable
Steinberg representation of $\tilde{\mathbf G}(F_v)$ of which $\tilde f_v$ is a twisted pseudo-coefficient.
\begin{enumerate}
  \item By the twisted analogue of \cite[Theorem~K]{Kaz} (cf. \cite[Section~7.2]{MW}), the twisted orbital integral
        $O_\delta(\tilde f_v)$ vanishes unless $\delta$ is $\sigma$-elliptic. For $\sigma$-elliptic regular $\delta$
        it equals $c_v\,\overline{\Theta_{\tilde\pi_v,\sigma}(\delta)}$ for a constant $c \neq 0$.
  \item The parameter of $\tilde\pi_v$ is $(\nu(4)\otimes\mu_v,\mu_v^2)$. Its $4$-dimensional part is irreducible
        and symplectic, so it factors through ${}^L\mathbf G^*=\mathrm{GSp}_4(\mathbb C)$ and through no $L$-group of
        $C_E$ or $C^E_+$. By the twisted character identity [CG15, Proposition 9.1], $\Theta_{\tilde\pi_v,\sigma}$ is
        the transfer of the character of the Steinberg packet $\{\mathrm{St}\otimes\eta_v\}$ of $\mathbf G^*(F_v)$.
        This packet is a singleton, so its character is stable. Hence $\Theta_{\tilde\pi_v,\sigma}(\delta)$ depends
        only on the stable $\sigma$-conjugacy class of $\delta$.
  \item Since $H^1(F_v,\tilde{\mathbf G})=0$, the $\sigma$-conjugacy classes in the stable class of a
        $\sigma$-elliptic regular $\delta$ are parametrized by the finite abelian group $H^1(F_v,I_\delta)$. For an
        elliptic datum other than $\mathbf G^*$ the associated character $\kappa$ of this group is non-trivial.
        Hence $O^\kappa_\delta(\tilde f_v)=c_v\,\overline{\Theta_{\tilde\pi_v,\sigma}(\delta)}\sum_{\delta'}\kappa(\delta')=0$.
\end{enumerate}
Therefore, it follows that the transfers of $\tilde{f}_{v}$ to all elliptic twisted endoscopic groups $(\tilde{\mathbf{G}}_{F_{v}},\sigma)$ vanish, except possibly for $\mathbf{G}^{*}_{F_{v}}$. Thus, the simple twisted trace formula simplifies giving an equality: 
\[ I^{\tilde{\mathbf{G}},\sigma}_{disc,\chi}(\tilde{f}) = \frac{1}{2}ST_{disc,\chi}^{\mathbf{G}^{*}}(f^{*}) \]
Now we apply the discrete part of the stable trace formula for $\mathbf{G}^{*}$ to the RHS. This gives us an equality: 
\[ I^{\mathbf{G}^{*}}_{disc,\chi}(f^{\mathbf{*}}) - \frac{1}{4}ST^{C}_{disc,\chi}(f^{C}) = ST^{\mathbf{G}^{*}}_{disc,\chi}(f^{*}).  \]
However, for any $v \in S_{st}$, the orbital integral of $f_{v}^{*}$ is again stable, so we see that its transfer to the endoscopic group $C$ is zero. Hence, the second term on the LHS vanishes. All in all, we obtain the following lemma.
\begin{lemma}
For $S_{st}$ a finite set of finite places with $|S_{st}| \geq 3$, $f^{*}$ and $\tilde{f}$ matching test functions on $\mathbf{G}^{*}$ and $\tilde{\mathbf{G}}$, respectively, such that $f^{*}_{v}$, for every $v \in S_{\mathrm{st}}$, is a pseudo-coefficient for the (essentially) discrete series Steinberg representation twisted by an unramified character $\eta_{v}$ such that $\eta_{v}|_{Z(F_{v})} = \chi_{v}$ and $\tilde{f}_{v}$ is the $\sigma$-twisted pseudo-coefficient for the Steinberg representation of $\tilde{\mathbf{G}}_{F_{v}}$ twisted by $\eta_{v}$, we have an equality:
\[ \frac{1}{2}I^{\mathbf{G}^{*}}_{disc,\chi}(f^{\mathbf{*}}) =  I_{disc,\chi}^{\tilde{\mathbf{G}},\sigma}(\tilde{f}), \]
relating spectral information on $\mathbf{G}^{*}$ to $\tilde{\mathbf{G}}$.
\end{lemma}
\subsection{Strong Multiplicity One}
We now would like to combine the analysis of sections 5.1 and 5.2 to deduce a strong multiplicity one result for $\mathbf{G}^{*} = \GSp_{4}/F$ and certain inner forms. Our analysis is very similar to \cite[Sections~10.5 and 10.6]{CG} and benefited from reading the proofs of \cite[Proposition~10.1 and Theorem~11.4]{RW} in a paper of Rosner and Weissauer, where they prove a similar multiplicity one result using Weselmann's topological twisted trace formula \cite{We} instead of the simple twisted trace formula of Kottwitz-Shelstad. Let $S_{st}$ and $S_{sc}$ be disjoint finite sets of finite places. Let $S_{\infty}$ denote the set of infinite places. Set $S_{0} \subset S_{st} \cup S_{sc}$ and $S_{sc} \cup S_{st} \cup S_{\infty} \subset S$ to be finite sets of places as before. We let $\mathbf{G}$ be an inner form over $F$, as in Theorem 5.2, trivialized outside of $S_{0} \cup S_{\infty}$. We have the following.
\begin{proposition}
Assume that $|S_{st}| \geq 3$. Let $\pi$ be a cuspidal automorphic representation of $\mathbf{G}^{*} = \GSp_{4}/F$ or the above inner form $\mathbf{G}$ satisfying the following:
\begin{enumerate}
    \item $\pi$ is cohomological of regular weight $\xi$ at infinity,
    \item $\pi$ is unramified outside of $S$,
    \item $\pi$ is an unramified twist of Steinberg at all places in $S_{st}$,
    \item $\pi$ has supercuspidal $L$-parameter at all places in $S_{sc}$.
\end{enumerate}
If $\pi'$ is a cuspidal automorphic representation of $\mathbf{G}^{*}$ satisfying (1), (3), and $\pi'^{S} \simeq \pi^{S}$ then its Langlands parameter at all places in $S$ agrees with $\pi$. If $\pi'$ is a cuspidal automorphic representation of $\mathbf{G}$ satisfying conditions (1), (3), and $\pi'^{S} \simeq \pi^{S}$ then its Langlands parameter at all places in $S_{st} \cup S_{sc} \cup S_{\infty}$ agrees with $\pi$.
\end{proposition}
\begin{proof}
Set $\chi$ to be the central character of $\pi$. We note that the central character of $\pi'$ is also $\chi$, since
$(\pi')^S \simeq \pi^S$.

We apply the above trace formulae with central character datum $(Z(\mathbb{A}_{F}),\chi)$. If $\pi$ is a representation of $\mathbf{G}^{*} = \GSp_{4}/F$, we take $\tau$ to be a globally generic member of the global $L$-packet of $\pi$, as in the proof of Theorem 5.2. If $\pi$ is a representation of the inner form $\mathbf{G}$ then, using Theorem 5.2, we take $\tau$ to be a globally generic strong transfer $\tau$ of $\pi$ to a cuspidal automorphic representation of $\mathbf{G}^{*}$, with Langlands parameter equal to $\phi_{\pi_{v}}$ at all places $v \in S_{sc} \cup S_{st} \cup S_{\infty}$. Now we apply \cite[Theorem~12.1]{GT1} to $\tau$ to deduce the existence of a strong transfer to a globally generic automorphic representation of $\GL_{4}(\mathbb{A}_{F})$, denoted $\tilde{\tau}$. It satisfies the following:
\begin{enumerate}
    \item $\tilde{\tau}$ is a global theta lift of $\tau$.
    \item $\tilde{\tau}^{\vee} \otimes \chi \simeq \tilde{\tau}$.
    \item For all places $v$, we have that $\phi_{\tilde{\tau}_{v}} = \mathrm{std}\circ \phi_{\tau_{v}} $ as conjugacy classes of parameters.
    \item Its form falls into the two cases:
    \begin{enumerate}
        \item $\tilde{\tau}$ is cuspidal 
        \item $\tilde{\tau} = \sigma \boxplus \sigma'$ for $\sigma \neq \sigma'$ a cuspidal automorphic representation of $\GL_{2}$, where $\omega_{\sigma} = \omega_{\sigma'}$.
    \end{enumerate}
    In the latter case, $\tau$ is the theta lift of a cusp form $\sigma \otimes \sigma'$ on $C= \GSO_{2,2}$. 
\end{enumerate}
To distinguish these two cases, we say that $\tilde{\tau}$ is a stable or endoscopic lift. We note that actually, for $v \in S_{st}$, we have that $\phi_{\tilde{\tau}_{v}} = \mathrm{std} \circ \phi_{\tau_{v}}$, and $\phi_{\tau_{v}}$ is an unramified twist of the Steinberg parameter. Therefore, $\tilde{\tau}$ is necessarily a stable lift.

For each $v \in S_{\mathrm{st}}$, choose an arbitrary unramified
character $\eta_v$ whose restriction to $Z(F_v)$ is $\chi_v$.
We choose matching test functions $\widetilde f$ and $f^*$ on
$\tilde{\mathbf{G}}$ and $\mathbf{G}^*$ such that
\[
f_v^* = f_{\mathrm{Lef},\eta_v}^{G^*}
\]
for $v \in S_{\mathrm{st}}$, with the corresponding twisted
pseudo-coefficient on $\tilde{\mathbf{G}}$, as in Lemma~5.3.
We take $f_\infty^*$ to be the Lefschetz function associated
with $\xi$. In particular, the trace identities that follow hold for
every such choice of the characters $(\eta_v)_{v\in S_{\mathrm{st}}}$. At the remaining finite places of $S$, we allow the test
functions $f_v^*$ to vary, and outside $S$ we vary the
spherical Hecke functions, choosing matching functions
on $\tilde{\mathbf{G}}$. Such matching functions exist: at each place $v$, every stable orbital integral on $G^*(F_v)$ is the twisted transfer of a twisted orbital integral on $\tilde{G}(F_v)$ by the surjectivity of twisted transfer \cite{MW16}, and outside $S$ the twisted transfer of spherical Hecke functions is given by the twisted fundamental lemma.

 Lemma 5.3 then gives us an equality:
\[ \frac{1}{2}I^{\mathbf{G}^{*}}_{disc,\chi}(f^{\mathbf{*}}) =  I_{disc,\chi}^{\tilde{\mathbf{G}},\sigma}(\tilde{f}). \]
Consider the part of the RHS corresponding to the cuspidal representation $\tilde{\tau}$ constructed above. By using linear independence of the unramified characters and the strong multiplicity one property for $\tilde{\mathbf{G}}$, the above identity implies an equality
\begin{equation}
 c_{1}\cdot\sum_{(\Pi')^{S} \simeq \pi^{S}} m(\Pi')\tr(f^{*}_{S}|\Pi'_{S}) =  \tr_{\sigma}(\tilde{\tau}_{S}|\tilde{f}_{S}) 
\end{equation}
where $c_{1}$ is a non-zero constant. Here $\tr_{\sigma}(\tilde{\tau}_{S}|\tilde{f}_{S})$ is the $\sigma$-twisted trace, as defined in \cite[Section~5.16]{CG}. The LHS runs over automorphic representations satisfying the following:
\begin{enumerate}
    \item  $\Pi'$ has non-zero contribution to the discrete part of the trace formula $I_{disc,\chi}^{\mathbf{G}^{*}}$. 
    \item The coefficient $m(\Pi')$ is the coefficient associated with the trace of $\Pi'$ in $I_{disc,\chi}^{\mathbf{G}^{*}}$.
\end{enumerate}
We can further simplify the LHS of (3) by noting that non-discrete spectrum representations which intervene in $I_{disc,\chi}^{\mathbf{G}^{*}}$ are parabolically induced from the discrete spectrum of proper Levi subgroups of $\mathbf{G}^{*}$. By \cite[Section~5.8]{CG}, we know that parabolically induced representations of $\mathbf{G}^{*}$ lift to parabolically induced representations of $\tilde{\mathbf{G}}$. Therefore, since $\tilde{\tau}$ is a stable lift and therefore cuspidal, all terms occurring in  the LHS must come from the discrete spectrum $T_{disc,\chi}^{\mathbf{G}^{*}}(f^{*})$, by strong multiplicity one for $\tilde{\mathbf{G}}$. Moreover, the coefficients $m(\Pi')$ must then be the multiplicities of $\Pi'$ in the discrete spectrum. However, as in Lemma 5.1, we have an equality:
\[ T_{disc,\chi}^{\mathbf{G}^{*}}(f^{*}) = T_{cusp,\chi}^{\mathbf{G}^{*}}(f^{*}). \]
In other words, we may assume that the sum on the LHS of (3) ranges over  $\Pi' \in \mathcal{A}_{cusp,\chi}(\mathbf{G}^{*})$, and that $m(\Pi')$ denotes the multiplicity in the cuspidal automorphic spectrum. In other words, we can rewrite the LHS as
\[ \sum_{\substack{\Pi' \in \mathcal{A}_{cusp,\chi}(\mathbf{G}^{*}) \\\\
\Pi'^{S} \simeq \pi^{S}}} m(\Pi')\tr(f_{S}^{*}|\Pi'_{S}). \]

Now, for the RHS, we apply the local twisted character identities: at the finite places these are due to Chan--Gan \cite[Proposition~9.1]{CG15}, and at the archimedean places they are due to Mezo \cite{Mez13}. We note that at $v \mid \infty$ the factor $\sum_{\pi'_v} \mathrm{tr}(f^*_v \mid \pi'_v)$ is a non-zero constant, since $f^*_\infty$ is the Lefschetz function for $\xi$ and $\phi_{\tau_\infty}$ is the discrete series parameter determined by $\xi$. This tells us that we have an equality:
\[ \tr_{\sigma}(\tilde{\tau}_{S}|\tilde{f}_{S}) = c_{2}\prod_{v \in S} \sum_{\pi'_{v} \in \Pi_{\phi_{\tau_{v}}}(\mathbf{G}^{*}_{F_{v}})} \tr(f_{v}^{*}|\pi'_{v}), \]
for some non-zero constant $c_{2}$, where we have used property (3) of the representation $\tilde{\tau}$. In summary, we have concluded 
\[ \sum_{\substack{\Pi' \in \mathcal{A}_{cusp,\chi}(\mathbf{G}^{*}) \\\\
\Pi'^{S} \simeq \pi^{S}}} m(\Pi')\tr(f_{S}^{*}|\Pi'_{S}) = c \cdot \prod_{v \in S} \sum_{\pi'_{v} \in \Pi_{\phi_{\tau_{v}}}(\mathbf{G}^{*}_{F_{v}})} \tr(f_{v}^{*}|\pi'_{v})  \]
for some non-zero constant $c$. 

For the inner form $G$, we apply the inner-form character identities: at the finite places these are due to Chan--Gan \cite[Proposition~11.1]{CG15}, and at the archimedean places, they are due to Shelstad \cite{She08}.This tells us that the RHS of the previous equation is equal to
\[ c_{3}\prod_{v \in S} \sum_{\rho_{v} \in \Pi_{\phi_{\tau_{v}}}(\mathbf{G}_{F_{v}})} \tr(f_{v}|\rho_{v}) \]
for some non-zero constant $c_{3}$. Now, to rewrite the LHS, we apply the trace formula as in the proof of Theorem 5.2. By linear independence of characters, we obtain a relationship
\begin{equation*}
\sum_{\substack{\Pi' \in \mathcal{A}_{cusp,\chi}(\mathbf{G}^{*}) \\\\
\Pi'^{S} \simeq \pi^{S}}} m(\Pi')\tr(f_{S}^{*}|\Pi'_{S})  = c_{4} \cdot \sum_{\substack{\Pi \in \mathcal{A}_{cusp,\chi}(\mathbf{G}) \\\\ \Pi^{S} \simeq \pi^{S}}} m(\Pi)\tr(f_{S}|\Pi_{S})
\end{equation*}
for some non-zero constant $c_{4}$. All in all, we obtain that 
\[ \sum_{\substack{\Pi \in \mathcal{A}_{cusp,\chi}(\mathbf{G}) \\\\ \Pi^{S} \simeq \pi^{S}}} m(\Pi)\tr(f_{S}|\Pi_{S}) = c'\prod_{v \in S} \sum_{\rho_{v} \in \Pi_{\phi_{\tau_{v}}}(\mathbf{G}_{F_{v}})} \tr(f_{v}|\rho_{v}) \]
for some non-zero constant $c'$. We now explain the conclusion of these character identities.
Let $H$ be either $\mathbf{G}^*$ or $\mathbf{G}$, and let $\sigma$ be a cuspidal
automorphic representation of $H$ that is cohomological of
weight $\xi$, is an unramified twist of Steinberg at every
place in $S_{\mathrm{st}}$, and satisfies $\sigma^S \simeq \pi^S$.
The preceding identity for $H$ is valid for every choice of
the Steinberg twists in the test functions.

Choose these twists to detect $\sigma_v$ at each
$v\in S_{\mathrm{st}}$. When $H = \mathbf{G}$, we use the corresponding
characters on the two inner forms, identified via the
similitude characters. At the remaining finite places of $S$,
choose normalized compact-open idempotents, with central
character $\chi_v$, having nonzero trace on $\sigma_v$.
In the inner-form comparison, these choices are made on $H$;
the functions on $\mathbf{G}^*$ are their transfers.

The term corresponding to $\sigma$ on the spectral side is
then nonzero. As in the proof of Theorem~5.2, every nonzero
term on this side has the same sign: the archimedean
Lefschetz traces have a fixed sign, the Steinberg traces
have a fixed sign, and the remaining traces are dimensions
of spaces of invariant vectors. The one-dimensional
possibilities at the Steinberg places are excluded by
the same argument as in that proof.

If $\phi_{\sigma_v}\neq\phi_{\tau_v}$ for some
$v\in S_{\mathrm{st}}$, the corresponding factor on the
RHS of the previous trace identities vanishes. Indeed, the packet at that place
is a singleton Steinberg packet, and the chosen Lefschetz
function detects the twist belonging to $\sigma_v$.
This contradicts the nonvanishing of the LHS.
Hence $\phi_{\sigma_v}=\phi_{\tau_v}$ for all $v \in S_{\mathrm{st}}$.

Keeping these Steinberg functions and the archimedean
Lefschetz function fixed, we now vary the test functions
at the remaining finite places of $S$. Linear independence
of characters, together with the same-sign property of
the coefficients on the spectral side, implies that
$\sigma_v$ belongs to the packet over $\phi_{\tau_v}$ at
each of these places. The archimedean parameters agree
because the representations are cohomological of the
same regular weight $\xi$.

Applying this argument to $\sigma=\pi'$ gives the desired
comparison with $\tau$. If $\pi$ is a representation of
$\mathbf{G}^*$, its parameters already agree with those of $\tau$
by construction. If $\pi$ is a representation of $\mathbf{G}$, apply
the same argument to $\sigma=\pi$ as well. This gives the
asserted comparison between $\pi'$ and $\pi$.
\end{proof}
\begin{remark}
Strong multiplicity one for globally generic automorphic representations of $\GSp_{4}$ has been proven by Jiang-Soudry \cite{JS}. So, in the particular case that $\pi$ and $\pi'$ are representations of $\GSp_{4}$, we could have just assumed that $S_{st}$ is non-empty and then applied their results to a globally generic member in the global $L$-packets of $\pi$ and $\pi'$ to deduce the desired claim. 
\end{remark}
\section{Galois Representations in the Cohomology of Shimura varieties}
We now would like to combine the results of the previous section with results of Sorensen \cite{So} on the Galois representations associated to automorphic representations of $\mathbf{G}^{*} = \GSp_{4}/F$ to say something about the Galois action of the global Shimura varieties occurring in basic uniformization. Let $F/\mathbb{Q}$ be a totally real field and $\mathbb{A}_{f,F}$ the finite adeles of $F$. Throughout, we will assume that $\tau$ is a cuspidal automorphic representation of $\mathbf{G}^{*}$ satisfying the same properties as in the previous section.
\begin{enumerate}
    \item $\tau_{\infty}$ is cohomological of some regular weight $\xi$ of $\mathbf{G}^{*}(F_{\infty})$.
    \item $\tau_{v}$ is unramified at all finite places outside of $S$.
    \item $\tau_{v}$ is an unramified twist of Steinberg at some finite set of finite places $S_{st}$.
\end{enumerate}
We have the following key result of Sorensen. 
\begin{theorem}{\cite[Main Theorem~(a)]{So}}
Fix a globally generic $\tau$ as above such that $S_{st}$ is non-empty. Then there exists a unique (after fixing the isomorphism $i: \overline{\mathbb{Q}}_{\ell} \xrightarrow{\simeq} \mathbb{C}$) irreducible continuous representation $\rho_{\tau}: \Gal(\overline{F}/F) \rightarrow \GSp_{4}(\overline{\mathbb{Q}}_{\ell})$ characterized by the property that, for each finite place $v \nmid \ell$ of $F$, we have
\[ i\WD(\rho_{\tau}|_{W_{F_{v}}})^{F-s.s} \simeq \phi_{\tau_{v}} \otimes |\cdot|^{-3/2}\] 
where $(-)^{F-s.s}$ denotes the Frobenius semisimplification and $\phi_{\tau_{v}}$ is the Gan-Takeda parameter of $\tau_{v}$.
\end{theorem}
Now let us fix $\tau$ with associated $\rho_{\tau}$ as above and assume that $\tau$ is a strong transfer of some cuspidal automorphic representation $\pi$ of $\mathbf{G}$, as in Theorem 5.2. We assume that $S_{st}$ contains $q$, an odd inert prime in the number field $F$. Choose the inner form $\mathbf{G}$ to be of the following form, as in Kret-Shin \cite[Section~8]{KS} and section 4.2, 
\begin{itemize}
    \item $\mathbf{G}(\mathbb{R}) \simeq \GSp_{4}(\mathbb{R}) \times \GU_{2}(\mathbb{H})^{[F:\mathbb{Q}] - 1}$,
    \item $\mathbf{G}_{F_{v}} \simeq \GSp_{4}/F_{v}$ at all finite places $v$ if $[F:\mathbb{Q}]$ is odd,  
    \item $\mathbf{G}_{F_{v}} \simeq \GSp_{4}/F_{v}$ at all but the finite place $q$ if $[F:\mathbb{Q}]$ is even,
\end{itemize}
where $\mathbb{H}$ is the Hamilton quaternions.
Let $A(\pi)$ be the set of isomorphism classes of cuspidal automorphic representations $\Pi$ of $\mathbf{G}$ such that, for all $v \in S_{st}$, $\Pi_{v}$ is an unramified twist of Steinberg, $\Pi_{\infty}$ is $\xi$ cohomological, and, for all $v \notin S_{\infty} \cup S_{st}$, $\Pi_{v} \simeq \pi_{v}$. Our main task now is to show that $\rho_{\tau}$ is realized in the $\pi^{\infty}$ isotypic component of the Shimura variety associated to a Shimura datum $(\Res_{F/\bb{Q}}\mathbf{G},X)$, where we take $X$ to be the inverse of the cocharacter in the discussion before \cite[Lemma~7.1]{KS}. Let $\Sh(\mathbf{G},X)_{K,\overline{F}}$ be the associated Shimura variety over $\overline{F}$ which we recall is $3$-dimensional. We set $\mathcal{L}_{\xi}$ to be the $\overline{\mathbb{Q}}_{\ell}$ local system associated to an irreducible representation of $\mathbf{G}$ over $F$ of highest weight $\xi$ on it as before, and let $H^{i}_{c}(\Sh(\mathbf{G},X)_{\ol{F}},\mathcal{L}_{\xi})_{ss}$ denote the semi-simplification of the colimit of the $i$th-cohomologies over all compact opens $K \subset \mathbf{G}(\mathbb{A}_{f,F})$ as a $\mathbf{G}(\mathbb{A}_{f,F}) \times \Gamma_{F}$-module. We note this is non-zero only if $0 \leq i \leq 6$.

Fix a sufficiently small compact open subgroup $K=\prod_vK_v\subset\mathbf G(\mathbb A_{f,F})$ such
that $\pi^\infty$ has a non-zero $K$-invariant vector. Let $S_{bad}$ denote the set of prime numbers $p$ for which
either $p=2$, the group $\mathrm{Res}_{F/\mathbb Q}\mathbf G$ is ramified at $p$, or $K_p=\prod_{v\mid p}K_v$ is not
hyperspecial. Let $v_\infty$ be the infinite place with $\mathbf G(F_{v_\infty})\simeq\mathrm{GSp}_4(\mathbb R)$, and set
$K_\infty:=K_{v_\infty}\times\prod_{v\in S_\infty\setminus\{v_\infty\}}\mathbf G(F_v)$, where $K_{v_\infty}$ is the
stabilizer of $iI_2$ in $\mathrm{GSp}_4(\mathbb R)$, as in \cite[Section~7]{KS}.

We define the virtual Galois representation
\begin{equation}
 \rho^{\pi}_{shim} := (-1)^{3}\sum_{\Pi \in [A(\pi)/\sim]} \sum_{i = 0}^{6} (-1)^{i}[Hom_{\mathbf{G}(\mathbb{A}_{f,F})}(\Pi^{\infty},H^{i}_{c}(\Sh(\mathbf{G},X)_{\overline{F}},\mathcal{L}_{\xi})_{ss})] \in K_{0}(\overline{\mathbb{Q}}_{\ell}(\Gamma_{F})) 
 \end{equation}
where $K_{0}(\overline{\mathbb{Q}}_{\ell}(\Gamma_{F}))$ denotes the Grothendieck group of continuous $\Gamma_{F} := \Gal(\overline{F}/F)$-representations with coefficients in $\overline{\mathbb{Q}}_{\ell}$. Here $\sim$ denotes the equivalence relationship that $\Pi \sim \Pi'$ if there is an isomorphism of $\mathbf{G}(\bb{A}_{f,F})$-representations $\Pi^{\infty} \simeq \Pi^{'\infty}$.

We now define the rational number
\[ a(\pi) := (-1)^{3}N_{\infty}^{-1}\sum_{\Pi \in A(\pi)} m(\Pi)\cdot ep(\Pi_{\infty} \otimes \xi) \]
where,
\begin{enumerate}
\item $m(\Pi)$ is the multiplicity of $\Pi$ in the automorphic spectrum of $\mathbf{G}$,
\item $N_{\infty} = |\Pi_{\xi}^{G(F_{\infty})}|\cdot|\pi_{0}(\mathbf{G}(F_{\infty})/Z(F_{\infty}))| = 4$, where $\Pi_{\xi}^{G(F_{\infty})}$ denotes the discrete series $L$-packet of representations of $G(F_{\infty})$ cohomological of weight $\xi$,
\item $ep(\Pi_{\infty} \otimes \xi) := \sum_{i = 0}^{\infty} (-1)^{i}\dim(H^{i}(Lie(\mathbf{G}(F_{\infty})),K_{\infty};\Pi_{\infty} \otimes \xi))$.
\end{enumerate}
Then we have the following proposition of Kret-Shin.
\begin{proposition}{\cite[Proposition~8.2 and its proof]{KS}}\footnote{Note that we have taken $X$ to be the inverse of the cocharacter in \cite[Proposition~8.2]{KS}.} 
With notation as above, for almost all finite $F$-places $v$ not dividing a prime number in $S_{bad}$ and all sufficiently large integers $j$, we have:
\[ \mathrm{tr}(\rho^{\pi}_{shim}(\Frob_{v}^{j})) = a(\pi)q_{v}^{j\frac{3}{2}}\mathrm{tr}(\mathrm{std}\circ \phi_{\pi_{v}})(\Frob^{j}_{v}). \]
Moreover, the virtual representation $\rho^{\pi}_{shim}$ is a true representation, and the only non-zero term appearing in the above alternating sum occurs in middle degree ($= 3$).
\end{proposition}
\begin{remark}
The claim about it occurring in middle degree is part of the proof of the Proposition not the statement. (see the discussion after equation (8.13) in \cite{KS})
\end{remark}
For $\pi$ as above, we define the $\pi^{\infty}$-isotypic part of $R\Gamma_{c}(\Sh(\mathbf{G},X)_{\overline{F}},\mathcal{L}_{\xi})$ to be the complex of $\ol{\mathbb{Q}}_{\ell}$-vector spaces with $\Gamma_{F}$-action:
\[ \bigoplus_{i \in \mathbb{Z}} \mathrm{Hom}_{\mathbf{G}(\mathbb{A}_{f,F})} (\pi^{\infty},H^{i}_{c}(\Sh(\mathbf{G},X)_{\overline{F}},\mathcal{L}_{\xi})_{\mathrm{ss}})[-i], \]
as in the beginning of \cite[Section~8]{KS}. If $K \subset \mathbf{G}(\mathbb{A}_{f,F})$ is a sufficiently small compact open subgroup such that $\pi^{\infty}$ has a non-zero $K$-invariant vector then we note that this is identified with 
\[ \bigoplus_{i \in \mathbb{Z}} \mathrm{Hom}_{\mathcal{H}(\mathbf{G}(\mathbb{A}_{f,F}))//K} ((\pi^{\infty})^{K},H^{i}_{c}(\Sh(\mathbf{G},X)_{K,\overline{F}},\mathcal{L}_{\xi})_{\mathrm{ss}})[-i], \]
where $\mathcal{H}(\mathbf{G}(\mathbb{A}_{f,F})//K)$ is the Hecke algebra of $K$ bi-invariant compactly supported functions, and
$H^{i}_{c}(\Sh(\mathbf{G},X)_{K,\overline{F}},\mathcal{L}_{\xi})_{\mathrm{ss}}$ is the semi-simplification of the $i$th cohomology of the Shimura variety at level $K$ as a $\mathcal{H}(\mathbf{G}(\mathbb{A}_{f,F})//K) \times \Gamma_{F}$-module.

We now use the previous discussion to deduce the following Corollary.
\begin{corollary}
The $\pi^{\infty}$-isotypic component of $R\Gamma_{c}(\Sh(\mathbf{G},X)_{\overline{F}},\mathcal{L}_{\xi})$ has $\Gamma_{F}$-action given (up to multiplicity) by $\mathrm{std}\circ \rho_{\tau}$.
\end{corollary}
\begin{proof}
The first part follows immediately from the previous Proposition, and the second part follows from the identification of the traces. In particular, by the Brauer-Nesbitt Theorem, Chebotarev density theorem, and the condition characterizing $\rho_{\tau}$, we can identify (up to multiplicity) the Galois representation $\rho^{\pi}_{shim}$ with the irreducible Galois representation $\mathrm{std} \circ \rho_{\tau}$.
\end{proof}
\section{Proof of the Key Proposition}
We will now combine the results of the previous three sections to deduce some key consequences that will be used to derive Proposition 1.4. For this, using Krasner's lemma, we now fix a totally real number field $F$ with two odd distinct totally inert primes $p$ and $q$ such that $F_{p} \simeq L$ the fixed unramified extension of $\mathbb{Q}_{p}$. We fix the $\mathbb{Q}$-inner form $\mathbf{G}$ of $\mathbf{G}^{*} = \mathrm{Res}_{F/\mathbb{Q}}\GSp_{4}$ of the form considered in section 4.2 and 6, and let $\mathbf{G}'$ be an inner form of $\mathbf{G}$ of the form described in Definition 4.1. We take $(\mathbf{G},X)$ to be the Shimura datum considered in section 6. We note that this forces the associated geometric dominant cocharacter $\mu$ of $G := \mathbf{G}_{\mathbb{Q}_{p}}$ to be the Siegel cocharacter\footnote{Taking note that we are working with the inverse of the $X$ defined in discussion before \cite[Lemma~7.1]{KS}.}; in particular, we can apply the results of section 4.2. Set $\xi$ to be a regular weight of an algebraic representation of $\mathbf{G}$ over $\mathbb{Q}$. Let $K^{p} \subset \mathbf{G}(\mathbb{A}_{f}^{p})$ be an open compact subgroup. We set $S_{sc} = \{p\}$, and set $S_{st}$ to be a disjoint finite set of finite places of $F$ containing $q$. We consider the uniformization map
\begin{equation}
 \Theta: R\Gamma_{c}(G,b,\mu) \otimes^{\mathbb{L}}_ {\mathcal{H}(J_{b})} \mathcal{A}(\mathbf{G}'(\mathbb{Q})\backslash \mathbf{G}'(\mathbb{A}_{f})/K^{p},\mathcal{L}_{\xi}) \rightarrow R\Gamma_{c}(\mathcal{S}(\mathbf{G},X)_{K^{p}},\mathcal{L}_{\xi})
 \end{equation}
supplied by Proposition 4.1 and Theorem 4.2. Now fix a smooth irreducible supercuspidal representation $\rho$ of $J(\mathbb{Q}_{p}) = \mathrm{Res}_{L/\mathbb{Q}_{p}}(\GU_{2}(D))(\mathbb{Q}_{p}) = \GU_{2}(D)(L)$. We have the following lemma.
\begin{lemma}
Suppose $\rho$ is a supercuspidal representation of $J(\mathbb{Q}_{p})$, then, for sufficiently regular $\xi$ and sufficiently small $K^{p}$, we can find a globalization $\Pi'$ of an unramified twist of $\rho$ to a cuspidal automorphic representation of $\mathbf{G}'$, such that $(\Pi'^{\infty})^{K^{p}}$ occurs as a $J(\mathbb{Q}_{p})$-stable direct summand of $\mathcal{A}(\mathbf{G}'(\mathbb{Q})\backslash \mathbf{G}'(\mathbb{A}_{f})/K^{p},\mathcal{L}_{\xi})$. Moreover, for all places in $S_{st}$, we can assume that the local constituents at $v \in S_{st}$ are unramified twists of the Steinberg representation. 
\end{lemma} 
\begin{proof}
This follows from an argument using the simple trace formula. See for example \cite[Proposition~2.9]{Han} or \cite{Shin}. We note in particular that cuspidality is vacuous, since $\mathbf{G}'(\mathbb{R})$ is compact modulo center by construction. 
\end{proof}
So let $\Pi'$ be a globalization of an unramified twist of our fixed supercuspidal $\rho$ to a cuspidal automorphic representation of $\mathbf{G}'$ for some sufficiently regular $\xi$ and sufficiently small $K^{p}$. We can and do regard $\Pi'^{p\infty}$ as a representation of $\mathbf{G}(\mathbb{A}^{p}_{f}) \simeq \mathbf{G}'(\mathbb{A}^{p}_{f})$. We set $K^{p} = K^{p}_{S}K^{S}$, where $K^{S} \subset \mathbf{G}(\mathbb{A}_{f}^{S})$ is an open compact in the finite adeles away from $S$, for $S \subset \Pl_{F}$ some finite set of places of $F$ containing $S_{st} \cup \{p\} \cup \{\infty\}$, as in section 5. We assume that $S$ is sufficiently large such that outside of $S$ the automorphic representation $\Pi'$ is unramified. We assume that the subgroup $K^{S} \subset \mathbf{G}(\mathbb{A}_{f}^{S})$ is a product of hyperspecial subgroups away from $S$. We consider the abstract commutative Hecke algebra
\[ \mathbb{T}^{S} := \mathcal{H}(\mathbf{G}(\mathbb{A}_{f}^{S})//K^{S})  \]
of bi-invariant compactly supported smooth functions on $\mathbf{G}(\mathbb{A}_{f}^{S})$. We regard both sides of (5) as $\mathbb{T}^{S}$-modules and consider the maximal ideal $\mathfrak{m}$ defined by the Hecke eigenvalues of $\Pi'^{S}$. We then localize both sides of $(5)$ at $\mathfrak{m}$ to obtain a map:
\[
 \Theta_{\mathfrak{m}}: (R\Gamma_{c}(G,b,\mu) \otimes^{\mathbb{L}}_ {\mathcal{H}(J_{b})} \mathcal{A}(\mathbf{G}'(\mathbb{Q})\backslash \mathbf{G}'(\mathbb{A}_{f})/K^{p},\mathcal{L}_{\xi}))_{\mathfrak{m}} \rightarrow R\Gamma_{c}(\mathcal{S}(\mathbf{G},X)_{K^{p}},\mathcal{L}_{\xi})_{\mathfrak{m}}
 \]
We would like to apply Propositions 4.4 and 5.4 to the representations occurring on both sides of this map. However, to apply these results we need to make some more modifications. In particular, the automorphic representations of $\mathbf{G}'$ (resp. $\mathbf{G}$) occurring in the LHS (resp. RHS) of $\Theta_{\mathfrak{m}}$ are not necessarily unramified twists of Steinberg at all places in $S_{st}$. 

To remedy this, we set $K^{p} = K_{\{p\} \cup S_{st}}K^{\{p\} \cup S_{st}}$, where $K^{\{p\} \cup S_{st}} \subset \mathbf{G}(\mathbb{A}_{f}^{\{p\} \cup S_{st}}) \simeq \mathbf{G}'(\mathbb{A}_{f}^{\{p\} \cup S_{st}})$ is an open compact subgroup. Then we consider the colimits
\[ R\Gamma_{c}(\mathcal{S}(\mathbf{G},X)_{K^{\{p\} \cup S_{st}}},\mathcal{L}_{\xi}) := \colim_{K_{\{p\} \cup S_{st}} \rightarrow \{1\}} R\Gamma_{c}(\mathcal{S}(\mathbf{G},X)_{K_{\{p\} \cup S_{st}}K^{\{p\} \cup S_{st}}},\mathcal{L}_{\xi}) \]
and 
\[ \mathcal{A}(\mathbf{G}'(\mathbb{Q})\backslash \mathbf{G}'(\mathbb{A}_{f})/K^{\{p\} \cup S_{st}},\mathcal{L}_{\xi}) := \colim_{K_{\{p\} \cup S_{st}} \rightarrow \{1\}} \mathcal{A}(\mathbf{G}'(\mathbb{Q})\backslash \mathbf{G}'(\mathbb{A}_{f})/K_{\{p\} \cup S_{st}}K^{\{p\} \cup S_{st}},\mathcal{L}_{\xi}). \]
Since $S_{st} \subset S$, the map $\Theta_{\mathfrak{m}}$ gives rise to a map 
\[ (R\Gamma_{c}(G,b,\mu) \otimes^{\mathbb{L}}_ {\mathcal{H}(J_{b})} \mathcal{A}(\mathbf{G}'(\mathbb{Q})\backslash \mathbf{G}'(\mathbb{A}_{f})/K^{S_{st} \cup \{p\}},\mathcal{L}_{\xi}))_{\mathfrak{m}} \rightarrow R\Gamma_{c}(\mathcal{S}(\mathbf{G},X)_{K^{S_{st} \cup \{p\}}},\mathcal{L}_{\xi})_{\mathfrak{m}}.  \]
By Proposition 4.4, we obtain an isomorphism
\[ (R\Gamma_{c}(G,b,\mu)_{sc} \otimes^{\mathbb{L}}_ {\mathcal{H}(J_{b})} \mathcal{A}(\mathbf{G}'(\mathbb{Q})\backslash \mathbf{G}'(\mathbb{A}_{f})/K^{S_{st} \cup \{p\}},\mathcal{L}_{\xi}))_{\mathfrak{m}} \xrightarrow{\simeq} R\Gamma_{c}(\mathcal{S}(\mathbf{G},X)_{K^{S_{st} \cup \{p\}}},\mathcal{L}_{\xi})_{\mathfrak{m},sc}. \] 
Now, for all $v \in S_{st}$, we can project to the summand of the LHS where $\mathbf{G}(F_{v}) \simeq \mathbf{G}'(F_{v})$ acts via an unramified twist of Steinberg, noting that the LHS and hence the RHS is semisimple as $\mathbf{G}(F_{v})$ representation. 

This gives an isomorphism: 
\[ \Theta_{\mathfrak{m},sc}^{st}: (R\Gamma_{c}(G,b,\mu)_{sc} \otimes^{\mathbb{L}}_ {\mathcal{H}(J_{b})} \mathcal{A}(\mathbf{G}'(\mathbb{Q})\backslash \mathbf{G}'(\mathbb{A}_{f})/K^{S_{st} \cup \{p\}},\mathcal{L}_{\xi}))^{st}_{\mathfrak{m}} \xrightarrow{\simeq} R\Gamma_{c}(\mathcal{S}(\mathbf{G},X)_{K^{S_{st} \cup \{p\}}},\mathcal{L}_{\xi})^{st}_{\mathfrak{m},sc}  \]
We now apply Proposition 5.4 to obtain the following.
\begin{proposition}
Let $\rho \in \Pi(J)$ be a representation with supercuspidal Gan-Tantono parameter $\phi$. Assume that $|S_{st}| \geq 3$. Then, for $\Pi'$ a choice of globalization of an unramified twist of $\rho$ as in Lemma 7.1, unramified outside $S$ with associated maximal ideal $\mathfrak{m} \subset \mathbb{T}^{S}$ in the Hecke algebra away from $S$ defined by the Hecke eigenvalues of $\Pi'^{S}$, the LHS of the map
\[ \Theta_{\mathfrak{m},sc}^{st}: (R\Gamma_{c}(G,b,\mu)_{sc} \otimes^{\mathbb{L}}_ {\mathcal{H}(J_{b})} \mathcal{A}(\mathbf{G}'(\mathbb{Q})\backslash \mathbf{G}'(\mathbb{A}_{f})/K^{S_{st} \cup \{p\}},\mathcal{L}_{\xi}))^{st}_{\mathfrak{m}} \xrightarrow{\simeq} R\Gamma_{c}(\mathcal{S}(\mathbf{G},X)_{K^{S_{st} \cup \{p\}}},\mathcal{L}_{\xi})^{st}_{\mathfrak{m},sc}  \]
breaks up as a direct sum of $R\Gamma_{c}(G,b,\mu)_{sc} \otimes_{\mathcal{H}(J_{b})}^{\mathbb{L}} \ol{\Pi}'^{\{\infty\},K^{S_{st} \cup \{p\}}}$ of $G(\mathbb{Q}_{p}) \times W_{L}$-representations, for $\ol{\Pi}'$ a cuspidal automorphic representation of $\mathbf{G}'(\mathbb{A})$ satisfying the following:
\begin{enumerate}
    \item $\ol{\Pi}'^{S} \simeq \Pi'^{S}$,
    \item $\ol{\Pi}'$ is cohomological of regular weight $\xi$ at $\infty$,
    \item $\ol{\Pi}'$ is unramified twist of Steinberg at all $v \in S_{st}$,
    \item $\ol{\Pi}'$ has local constituent at $p$ whose associated $L$-parameter is $\phi$ twisted by the same unramified character used in globalizing $\rho$. 
\end{enumerate} 
\end{proposition}
\begin{proof}
The fact that the cohomology breaks up as a direct sum of $R\Gamma_{c}(G,b,\mu)_{sc} \otimes_{\mathcal{H}(J_{b})}^{\mathbb{L}} \ol{\Pi}'^{\{\infty\},K^{S_{st} \cup \{p\}}}$ for $\ol{\Pi}'$ an automorphic representation of $\mathbf{G}'(\mathbb{A})$ follows from the semi-simplicity of the space of algebraic automorphic forms. We localized at $\mathfrak{m}$ corresponding to $\Pi'^{S}$ and are considering algebraic automorphic representations of $\mathbf{G}'$ valued in the algebraic representation defined by $\xi$. Therefore, it is clear that any representation $\ol{\Pi}'$ giving rise to such a summand is a cuspidal automorphic representaiton of $\mathbf{G}'$ satisfying (1) and (2). Here cuspidality is automatic since $\mathbf{G}'(\mathbb{R})$ is compact modulo center. Moreover, by construction, it follows that $\mathbf{G}'(\mathbb{Q}_{v})$ acts on the LHS via representations which are an unramified twist of Steinberg for all $v \in S_{st}$. This allows us to apply Proposition 5.4, since $|S_{st}| \geq 3$ by assumption. Proposition 5.4 applied to the inner form $\mathbf{G}'$ of $\mathbf{G}^{*}$ and the cuspidal automorphic representation $\Pi'$ of $\mathbf{G}'$ tells us that $\ol{\Pi}'$ must have Langlands parameter at $\{p\} = S_{sc}$ given by $\phi$ with the appropriate unramified twist, which was the desired claim.
\end{proof}
We now combine this with Corollary 6.3 to deduce the following.
\begin{corollary}
With notation as above, the map
\[ \Theta_{\mathfrak{m},sc}^{st}: (R\Gamma_{c}(G,b,\mu)_{sc} \otimes^{\mathbb{L}}_ {\mathcal{H}(J_{b})} \mathcal{A}(\mathbf{G}'(\mathbb{Q})\backslash \mathbf{G}'(\mathbb{A}_{f})/K^{S_{st} \cup \{p\}},\mathcal{L}_{\xi}))^{st}_{\mathfrak{m}} \xrightarrow{\simeq} R\Gamma_{c}(\mathcal{S}(\mathbf{G},X)_{K^{S_{st} \cup \{p\}}},\mathcal{L}_{\xi})^{st}_{\mathfrak{m},sc}  \]
is an isomorphism of complexes of $G(\mathbb{Q}_{p}) \times W_{L}$-modules concentrated in degree $3$ with $W_{L}$-action given, up to multiplicity and semi-simplification as a $W_{L}$-module, by $\mathrm{std}\circ \phi \otimes |\cdot|^{-3/2}$ twisted by the unramified character used in globalizing $\rho$ to $\Pi'$. 
\end{corollary}
\begin{proof}
The fact that $\Theta^{\mathrm{st}}_{\mathfrak m,\mathrm{sc}}$
is an isomorphism was established above. We determine the
cohomology and Weil action on its right-hand side.

Write $\psi:=\phi_{\Pi'_p}$, so that $\psi$ is the unramified
twist of $\phi$ used in globalizing $\rho$ in Lemma~7.1.
By Theorem~5.2, choose a globally generic cuspidal
automorphic representation $\tau$ of $\mathbf G^*$ such that
\[
\tau^{S}\simeq(\Pi')^S,
\]
$\phi_{\tau_p}=\psi$, and such that $\tau$ is cohomological of weight $\xi$ and
an unramified twist of Steinberg at every place in
$S_{\mathrm{st}}$.

We use the automorphic decomposition of the cohomology on
this localized summand, as in Section~6 and the comparison
with cuspidal cohomology in the proof of \cite[Lemma~8.1]{KS}. Let $\Sigma$ be a contributing cuspidal automorphic
representation of $\mathbf G$. By the localization and
the Steinberg projection, it satisfies
\[
\Sigma^S\simeq(\Pi')^S,
\]
is cohomological of weight $\xi$, and is an unramified twist
of Steinberg at every place in $S_{\mathrm{st}}$.

Apply Theorem~5.2 to $\Sigma$, taking $S_{\mathrm{sc}}=\emptyset$
for this application. This is permitted because $\mathbf G$
is split at every finite place outside $S_{\mathrm{st}}$.
We obtain a globally generic cuspidal automorphic
representation $\tau_\Sigma$ of $\mathbf G^*$ such that
\[
\tau_\Sigma^S\simeq\Sigma^S\simeq(\Pi')^S\simeq\tau^S,
\]
with the required cohomological weight and Steinberg
components.

Let $\rho_{\tau_\Sigma}$ and $\rho_\tau$ be the Galois
representations supplied by Theorem~6.1. Their compositions
with $\operatorname{std}$ have the same characteristic
polynomials of Frobenius at almost all finite places.
Consequently, Chebotarev density and Brauer--Nesbitt give
\[
\mathrm{std} \circ\rho_{\tau_\Sigma}
\simeq
\mathrm{std} \circ\rho_\tau.
\]
Corollary~6.3, applied to $\Sigma$ and $\tau_\Sigma$, now
shows that the $\Sigma^\infty$-isotypic component of the
global Shimura cohomology is concentrated in degree $3$
and has semisimplified $\Gamma_F$-action given, up to
multiplicity, by $\operatorname{std}\circ\rho_\tau$.

Taking the sum over the contributing finite part of the automorphic representations and then taking
$K^{S_{\mathrm{st}}\cup\{p\}}$-invariants proves the same
assertion for the right-hand side of
$\Theta^{\mathrm{st}}_{\mathfrak m,\mathrm{sc}}$. Finally, Theorem~6.1 and $\phi_{\tau_p}=\psi$ give
\[
\left(
  (\operatorname{std}\circ\rho_\tau)|_{W_L}
\right)^{\mathrm{ss}}
\simeq
(\operatorname{std}\circ\psi)\otimes|\cdot|^{-3/2}.
\]
This is precisely the asserted Weil action, including
the unramified twist used in Lemma~7.1.
\end{proof}
With this in hand, we are finally ready to conclude our key Proposition. 
\begin{proposition}
Let $\phi$ be a supercuspidal parameter with associated $L$-packet $\Pi_{\phi}(J)$. Then up to replacing $\phi$ by an unramified twist the direct summand of
\[ \bigoplus_{\rho' \in \Pi_{\phi}(J)} R\Gamma_{c}(G,b,\mu)[\rho'] \]
given by the supercuspidal Bernstein components of $G(\mathbb{Q}_{p})$, denoted
\[ \bigoplus_{\rho' \in \Pi_{\phi}(J)} R\Gamma_{c}(G,b,\mu)[\rho']_{sc}, \]
is concentrated in middle degree $3$ and admits a non-zero $W_{L}$-stable sub-quotient with $W_{L}$-action given by $\mathrm{std}\circ \phi \otimes |\cdot|^{-3/2}$.
\end{proposition}
\begin{proof}
This is an immediate consequence of Proposition 7.2 and Corollary 7.3. To see nonvanishing, put $\rho_0=\Pi'_p$ for a globalization
as above. Theorem~3.13 applied to $\rho_0^\vee$, together
with Remark~3.4, gives
\[
R\Gamma_c(G,b,\mu)[\rho_0]_{\mathrm{sc}}\neq 0.
\]
Since $\Pi'$ occurs as a direct summand in Lemma~7.1,
the corresponding summand on the left-hand side of the isomorphism in Corollary~7.3 is nonzero. 
\end{proof}
In particular, using Theorem 3.6 (2) and Theorem 2.2 (4) to deal with the unramified twist and Corollary 3.21, we can deduce the following.
\begin{corollary}
If $p > 2$ and $L/\mathbb{Q}_{p}$ is an unramified extension, then, for all $\rho \in \Pi(J)$ with supercuspidal Gan-Tantono parameter $\phi_{\rho}$, the Fargues-Scholze and Gan-Tantono correspondences are compatible.
\end{corollary}
\section{Applications}
We will now apply Corollary 7.5 to deduce some applications to the strong form of the Kottwitz conjecture and conclude the proof of Theorem 1.1. We begin with the latter. 
\begin{theorem}
The following is true.
\begin{enumerate}
    \item For any $\pi \in \Pi(G)$ (resp. $\rho \in \Pi(J)$) such that the Gan-Takeda (resp. Gan-Tantono) parameter is not supercuspidal, we have that the Gan-Takeda (resp. Gan-Tantono) correspondence is compatible with the Fargues-Scholze correspondence. 
    \item If $L/\mathbb{Q}_{p}$ is unramified and $p > 2$, we have, for all $\pi \in \Pi(G)$ (resp. $\rho \in \Pi(J)$) such that the Gan-Takeda (resp. Gan-Tantono) parameter is supercuspidal, that the Gan-Takeda (resp. Gan-Tantono) correspondence is compatible with the Fargues-Scholze correspondence. 
\end{enumerate}
\end{theorem}
\begin{proof}
Part (1) follows by Corollary 3.12 and Corollary 3.16. Part (2) for the Gan-Tantono local Langlands is precisely Corollary 7.5. It remains to show that for $L/\mathbb{Q}_{p}$ unramified and $p > 2$, $\pi$ a smooth irreducible representation of $\GSp_{4}/L$ with supercuspidal Gan-Takeda  parameter $\phi_{\pi}$ that the two correspondences are compatible. To show this, we consider the complex
\[ R\Gamma_{c}^{\flat}(G,b,\mu)[\pi] \]
of $J(\mathbb{Q}_{p}) \times W_{L}$-representations. We know, by Theorem 3.13, that this admits sub-quotients as a $J(\mathbb{Q}_{p})$-module given by $\rho$, for all $\rho$ whose Gan-Tantono parameter $\phi_{\rho}$ is equal to the Gan-Takeda parameter $\phi_{\pi}$ of $\pi$. However, by Corollary 3.15, we know that these representations must have Fargues-Scholze parameter equal to $\phi_{\pi}^{\mathrm{FS}}$. Therefore, we get a chain of equalities
\[ \phi_{\pi}^{\mathrm{FS}} = \phi_{\rho}^{\mathrm{FS}} = \phi_{\rho} = \phi_{\pi}, \]
of conjugacy classes of parameters, where we have used compatibility of the Gan-Tantono and the Fargues-Scholze correspondence for the middle equality. 
\end{proof}
Now, with this out of the way, we turn our attention to proving some strong forms of the Kottwitz conjecture for these representations, verifying Theorem 1.3. 
\begin{theorem}
Let $L/\mathbb{Q}_{p}$ be an unramified extension with $p > 2$. Let $\pi$ (resp. $\rho$) be members of the $L$-packet over a supercuspidal parameter $\phi: W_{L} \rightarrow \GSp_{4}(\overline{\mathbb{Q}}_{\ell})$. Then the complexes
\[ R\Gamma_{c}(G,b,\mu)[\pi] \]
and 
\[ R\Gamma_{c}(G,b,\mu)[\rho] \]
are concentrated in middle degree $3$. 
\begin{enumerate}
\item If $\phi$ is stable supercuspidal, with singleton $L$-packets $\{\pi\} = \Pi_{\phi}(G)$ and $\{\rho\} = \Pi_{\phi}(J)$, then the cohomology of $R\Gamma_{c}(G,b,\mu)[\pi]$ in middle degree is isomorphic to 
\[ \rho \boxtimes (\mathrm{std}\circ \phi)^{\vee} \otimes |\cdot|^{-3/2}   \]
as a $J(\mathbb{Q}_{p}) \times W_{L}$-module, and the cohomology of $R\Gamma_{c}(G,b,\mu)[\rho]$ in middle degree is isomorphic to
\[ \pi \boxtimes \mathrm{std}\circ \phi \otimes |\cdot|^{-3/2} \]
as a $G(\mathbb{Q}_{p}) \times W_{L}$-module.

\item If $\phi$ is an endoscopic parameter, with $L$-packets $\Pi_{\phi}(G) = \{\pi^{+},\pi^{-}\}$ and $\Pi_{\phi}(J) = \{\rho_{1},\rho_{2}\}$, the cohomology of $R\Gamma_{c}(G,b,\mu)[\pi]$ in middle degree is isomorphic to
\[ \rho_{1} \boxtimes \phi_{1}^{\vee} \otimes |\cdot|^{-3/2} \oplus \rho_{2} \boxtimes \phi_{2}^{\vee} \otimes |\cdot|^{-3/2} \]
or
\[ \rho_{1} \boxtimes \phi_{2}^{\vee} \otimes |\cdot|^{-3/2} \oplus \rho_{2} \boxtimes \phi_{1}^{\vee} \otimes |\cdot|^{-3/2} \]
as a $J(\mathbb{Q}_{p}) \times W_{L}$-module. Similarly, the cohomology of $R\Gamma_{c}(G,b,\mu)[\rho]$ in middle degree is isomorphic to
\[  \pi^{+} \boxtimes \phi_{1} \otimes |\cdot|^{-3/2} \oplus \pi^{-} \boxtimes \phi_{2} \otimes |\cdot|^{-3/2} \]
or 
\[  \pi^{+} \boxtimes \phi_{2} \otimes |\cdot|^{-3/2} \oplus \pi^{-} \boxtimes \phi_{1} \otimes |\cdot|^{-3/2} \]
as a $G(\mathbb{Q}_{p}) \times W_{L}$-module. Here we write $\mathrm{std}\circ \phi_{\rho} \simeq \phi_{1} \oplus \phi_{2}$, with $\phi_{i}$ distinct irreducible $2$-dimensional representations of $W_{L}$ and $\det(\phi_{1}) = \det(\phi_{2})$. 
\\\\
Moreover, both possibilities for the cohomology of $R\Gamma_{c}(G,b,\mu)[\rho]$ (resp. $R\Gamma_{c}(G,b,\mu)[\pi]$) in the endoscopic case occur for some choice of representation $\rho \in \Pi_{\phi}(J)$ (resp. $\pi \in \Pi_{\phi}(G)$). In particular, knowing the precise form of either $R\Gamma_{c}(G,b,\mu)[\rho]$ or $R\Gamma_{c}(G,b,\mu)[\pi]$ for some $\rho \in \Pi_{\phi}(J)$ or $\pi \in \Pi_{\phi}(G)$ determines the precise form of the cohomology in all other cases. 
\end{enumerate}
\end{theorem}
\begin{proof}
We show the proof in the endoscopic case, with the stable case being strictly easier. We first note that, since $\phi = \phi_{\rho}^{\mathrm{FS}}$ by Theorem 8.1, it follows by assumption that the Fargues-Scholze parameter $\phi_{\rho}^{\mathrm{FS}}$ of $\rho$ is supercuspidal. Therefore, by Remark 3.12, we have an isomorphism
\[ R\Gamma_{c}(G,b,\mu)[\rho] \simeq R\Gamma_{c}^{\flat}(G,b,\mu)[\rho] \]
of $G(\mathbb{Q}_{p}) \times W_{L}$-modules. Moreover, by Theorem 3.17, we see that both are concentrated in middle degree $3$. Applying Theorem 3.13, we get the following chain of equalities in the Grothendieck group of admissible $G(\mathbb{Q}_{p})$-representations of finite length 
\[ -[H^{3}(R\Gamma_{c}(G,b,\mu)[\rho])] = [R\Gamma_{c}(G,b,\mu)[\rho]] = [R\Gamma^{\flat}_{c}(G,b,\mu)[\rho]] = -\sum_{\pi \in \Pi_{\phi}(G)} Hom_{S_{\phi}}(\delta_{\pi,\rho},\mathrm{std}\circ\phi_{\rho})\pi \]
Now we saw in the discussion following Theorem 3.13 that the RHS takes the form:
\[ -2\pi^{+} - 2\pi^{-}. \] 
We set $p_{1}$ and $p_{2}$ to be the two representations of $S_{\phi} \simeq \{(a,b) \in \overline{\mathbb{Q}}_{\ell}^{*} \times \overline{\mathbb{Q}}_{\ell}^{*}| a^{2} = b^{2} \} \subset (\GL_{2}(\overline{\mathbb{Q}}_{\ell}) \times \GL_{2}(\overline{\mathbb{Q}}_{\ell}))^{0}$ given by projecting to the first and second $\overline{\mathbb{Q}}_{\ell}^{*}$-factor, respectively. By Corollary 3.3, Corollary 3.11, and Theorem 8.1, we have an isomorphism of $G(\mathbb{Q}_{p}) \times W_{L}$-modules
\[ R\Gamma_{c}(G,b,\mu)[\rho]  \simeq   \Act_{p_{1}}(\rho)[-3] \boxtimes \phi_{1} \otimes |\cdot|^{-3/2} \oplus \Act_{p_{2}}(\rho)[-3] \boxtimes \phi_{2} \otimes |\cdot|^{-3/2} \]
where $\Act_{p_{1}}(\rho)$ and $\Act_{p_{2}}(\rho)$ are a priori direct sums of shifts of supercuspidal representations of $G(\mathbb{Q}_{p})$ with Fargues-Scholze (= Gan-Takeda) parameter equal to $\phi$. However, since we know that the LHS is a complex concentrated in middle degree $3$, this implies, by the above description in the Grothendieck group, that one of the $\Act_{p_{1}}(\rho)$ and $\Act_{p_{2}}(\rho)$ is isomorphic to $\pi^{+}$ and the other is isomorphic to $\pi^{-}$ (Note that $\Act_{p_{1}}(\rho)$ and $\Act_{p_{2}}(\rho)$ are irreducible representations since the $\Act$-functors are invertible. We will use this repeatedly below.). Without loss of generality, assume that  
\[ \Act_{p_{1}}(\rho_{1}) \simeq \pi^{+} \]
and 
\[ \Act_{p_{2}}(\rho_{1}) \simeq \pi^{-}. \]
We let $p_{+}$ and $p_{-}$ be the representations of $S_{\phi}$ determined by the trivial and non-trivial characters of the component group, respectively. Now, given two representations of $S_{\phi}$, denoted $W$ and $W'$, it follows from Remark 3.7 (3) that we have an isomorphism: 
\[ \Act_{W} \circ \Act_{W'}(\cdot) \simeq \Act_{W \otimes W'}(\cdot) \]
In turn, we get
\[ \Act_{p_{1}^{\vee}}(\pi^{+}) \simeq \Act_{p_{1}^{\vee}} \circ \Act_{p_{1}}(\rho_{1}) \simeq \Act_{p_{+}}(\rho_{1}) \simeq \rho_{1} \]
where the last isomorphism follows since $p_{+}$ is the trivial representation. Similarly, depending on the values of $\Act_{p_{1}}(\rho_{2})$ and $\Act_{p_{2}}(\rho_{2})$ we can deduce that $\Act_{p_{2}^{\vee}}(\pi^{+})$ is isomorphic to $\rho_{1}$ or $\rho_{2}$. Now by Corollary 3.3, Corollary 3.11, and Theorem 8.1, we have an isomorphism 
\[ R\Gamma_{c}(G,b,\mu)[\pi^{+}] \simeq \Act_{p_{1}^{\vee}}(\pi^{+})[-3] \boxtimes \phi_{1}^{\vee} \otimes |\cdot|^{-3/2} \oplus \Act_{p_{2}^{\vee}}(\pi^{+})[-3] \boxtimes \phi_{2}^{\vee} \otimes |\cdot|^{-3/2} \]
Since $\Act_{p_{1}^{\vee}}(\pi^{+}) \simeq \rho_{1}$ it therefore follows, by Theorem 3.13 and Remark 3.12, that $\Act_{p_{2}^{\vee}}(\pi^{+})$ must be isomorphic to $\rho_{2}$. Moreover, we know that $\Act_{p_{-}} \circ \Act_{p_{1}}(\rho_{1}) \simeq \Act_{p_{2}}(\rho_{1})$ and  $\Act_{p_{-}} \circ \Act_{p_{2}}(\rho_{1}) \simeq \Act_{p_{1}}(\rho_{1})$. Therefore, we obtain that $\Act_{p_{-}}(\pi^{+}) \simeq \pi^{-}$ and $\Act_{p_{-}}(\pi^{-}) \simeq \pi^{+}$. This allows us to determine that
\[ \Act_{p_{1}^{\vee}}(\pi^{-}) \simeq \Act_{p_{2}^{\vee}}\circ \Act_{p_{-}}(\pi^{-}) \simeq  \Act_{p_{2}^{\vee}}(\pi^{+}) \simeq \rho_{2} \]
and 
\[ \Act_{p_{2}^{\vee}}(\pi^{-}) \simeq \Act_{p_{1}^{\vee}}\circ \Act_{p_{-}}(\pi^{-}) \simeq  \Act_{p_{1}^{\vee}}(\pi^{+}) \simeq \rho_{1}, \]
which will determine the cohomology of $R\Gamma_{c}(G,b,\mu)[\pi^{-}]$. It only remains to show that the value of $R\Gamma_{c}(G,b,\mu)[\rho_{2}]$ is determined. However, this follows since
\[ \Act_{p_{2}}(\rho_{2}) \simeq \Act_{p_{2}}\circ \Act_{p_{2}^{\vee}}(\pi^{+}) \simeq \pi^{+} \]
and
\[ \Act_{p_{1}}(\rho_{2}) \simeq \Act_{p_{1}} \circ \Act_{p_{1}^{\vee}}(\pi^{-}) \simeq \pi^{-}. \]
\end{proof}
To conclude this section, we use Theorem 8.1 to deduce compatibility with the local Langlands correspondence for $\Sp_{4}$ and its unique non quasi-split inner form $\mathrm{SU}_{2}(D)$. These correspondences are described in the papers \cite{GT3} and \cite{Cho} by Gan-Takeda and Choiy, respectively. For $\Sp_{4}$, this is described as the unique correspondence which sits in the commutative diagram:
\[ \begin{tikzcd}
& \Pi(\GSp_{4}) \arrow[r,"\LLC_{\GSp_{4}}"] \arrow[d] & \Phi(\GSp_{4}) \arrow[d] \\
& \Pi(\Sp_{4}) \arrow[r,"\LLC_{\Sp_{4}}"] & \Phi(\Sp_{4}) 
\end{tikzcd} \]
Here, the left vertical arrow is not a map at all, it is a correspondence defined by the subset of $\Pi(\GSp_{4}) \times \Pi(\Sp_{4})$ consisting of pairs $(\pi,\omega)$ such that $\omega$ is a constituent of the restriction of $\pi$ to $\Sp_{4}$, and the right vertical arrow is the map on $L$-parameters induced by the natural map $\GSpin_{5}(\mathbb{C}) \rightarrow \SO_{5}(\mathbb{C})$. One has a similar characterization of the local Langlands correspondence for $\SU_{2}(D)$. With this description of the correspondence, compatibility for $\Sp_{4}$ and $\SU_{2}(D)$ follows from Theorem 8.1 and Theorem 3.6 (7).
\begin{corollary}
For $\pi$ (resp. $\rho$) a smooth irreducible representation of $\Sp_{4}/L$ (resp. $\SU_{2}(D)$), with associated Gan-Takeda (resp. Choiy) parameter $\phi_{\pi}: W_{L} \times \SL_{2}(\mathbb{C}) \rightarrow \SO_{5}(\mathbb{C})$ (resp. $\phi_{\rho}$) we have that:
\begin{enumerate}
    \item The Fargues-Scholze and Gan-Takeda (resp. Choiy) local Langlands correspondences are compatible for any representation  $\pi$ (resp. $\rho$) such that $\phi_{\pi}$ (resp. $\phi_{\rho}$) is not supercuspidal.
    \item If $L/\mathbb{Q}_{p}$ is unramified and $p > 2$ then the Fargues-Scholze and Gan-Takeda (resp. Choiy) local Langlands correspondences are compatible for any representation $\pi$ (resp. $\rho$) such that $\phi_{\pi}$ (resp. $\phi_{\rho}$) is supercuspidal. 
\end{enumerate}
\end{corollary}
\begin{remark}
We note that Corollary 3.11, Remark 3.12, Theorem 3.18, and \cite[Theorem~1.0.2]{KW} apply to a triple $(G,b,\mu)$, where $\mu$ is any cocharacter and $b \in B(G,\mu)$ is the unique basic element. Therefore, by applying the same kind of analysis as in the proof of Theorem 8.2, we can prove the analogue of Theorem 8.2 in the Grothendieck group of finite length admissible representations with a smooth action of $W_{E}$ (cf. \cite[Conjecture~1.0.1]{KW}) for the cohomology of the local shtuka spaces defined by the triple $(G,b,\mu)$. By Corollary 8.3, this works even in the case when $G = \Sp_{4}$ and there are no Shimura varieties that these spaces uniformize. Moreover, in the case that $G = \GSp_{4}$, one can also deduce from Theorem 8.2 that it is concentrated in middle degree $0$ (after taking coefficients in the relevant Satake sheaf of $\mu$), using the monoidal property of the $\Act$-functors and Corollary 3.11, for $\mu$ any cocharacter. This in particular will imply some form of Fargues' Conjecture for these groups (see e.g. \cite[Pages~37-40]{BMNH}, for this worked out in the more complicated case of $G = \U_{n}$). 
\end{remark}

\printbibliography

@article {GT1,
    AUTHOR = {Gan, Wee Teck and Takeda, Shuichiro},
     TITLE = {The local {L}anglands conjecture for {${\rm GSp}(4)$}},
   JOURNAL = {Ann. of Math. (2)},
  FJOURNAL = {Annals of Mathematics. Second Series},
    VOLUME = {173},
      YEAR = {2011},
    NUMBER = {3},
     PAGES = {1841--1882},
      ISSN = {0003-486X},
   MRCLASS = {22E50 (20G25 22E35)},
  MRNUMBER = {2800725},
MRREVIEWER = {B. Sury},
       DOI = {10.4007/annals.2011.173.3.12},
       URL = {https://doi.org/10.4007/annals.2011.173.3.12},
}

@article {GanTakedaTheta,
    AUTHOR = {Gan, Wee Teck and Takeda, Shuichiro},
     TITLE = {Theta correspondences for {${\rm GSp}(4)$}},
   JOURNAL = {Represent. Theory},
  FJOURNAL = {Representation Theory. An Electronic Journal of the American
              Mathematical Society},
    VOLUME = {15},
      YEAR = {2011},
     PAGES = {670--718},
      ISSN = {1088-4165},
   MRCLASS = {22E50 (11F27 11S37)},
  MRNUMBER = {2846304},
MRREVIEWER = {Ivan\ Mati\'c},
       DOI = {10.1090/S1088-4165-2011-00405-2},
       URL = {https://doi.org/10.1090/S1088-4165-2011-00405-2},
}

@article{GH,
    author =       {Hamann, Linus and Imai, Naoki},
    title =        "Dualizing Complexes on the Moduli of Parabolic Bundles",
    journal = "Preprint",
    year =         "2024",
    Note = "arXiv:2401.06342"
}

@book {Soudry,
     TITLE = {Automorphic Forms on $GSp_{4}$ Festschrift in honor of {I}. {I}. {P}iatetski-{S}hapiro on the
              occasion of his sixtieth birthday. {P}art {II}},
    Author = {D. Soudry},
    SERIES = {Israel Mathematical Conference Proceedings},
    VOLUME = {3},
    EDITOR = {Gelbart, S. and Howe, R. and Sarnak, P.},
      NOTE = {Papers in analysis, number theory and automorphic
              $L$-functions,
              Papers from the Workshop on $L$-Functions, Number Theory, and
              Harmonic Analysis held at Tel-Aviv University, Ramat Aviv, May
              14--19, 1989},
 PUBLISHER = {Weizmann Science Press of Israel, Jerusalem},
      YEAR = {1990},
     PAGES = {291–303},
   MRCLASS = {00B30 (11-06)},
  MRNUMBER = {1159106},
}

@article{Ecod,
    author = {P. Scholze},
    journal = "Preprint",
    title = "\'Etale Cohomology of Diamonds",
    journal = "Preprint",
    year = "2018",
    note  = "arXiv:1709.07343"
}

@article {Huber,
    AUTHOR = {Huber, R.},
     TITLE = {A comparison theorem for {$l$}-adic cohomology},
   JOURNAL = {Compositio Math.},
  FJOURNAL = {Compositio Mathematica},
    VOLUME = {112},
      YEAR = {1998},
    NUMBER = {2},
     PAGES = {217--235},
      ISSN = {0010-437X,1570-5846},
   MRCLASS = {14F20},
  MRNUMBER = {1626021},
MRREVIEWER = {Abdellah\ Mokrane},
       DOI = {10.1023/A:1000345530725},
       URL = {https://doi.org/10.1023/A:1000345530725},
}

@article {GHN,
    AUTHOR = {G\"{o}rtz, Ulrich and He, Xuhua and Nie, Sian},
     TITLE = {Fully {H}odge-{N}ewton decomposable {S}himura varieties},
   JOURNAL = {Peking Math. J.},
  FJOURNAL = {Peking Mathematical Journal},
    VOLUME = {2},
      YEAR = {2019},
    NUMBER = {2},
     PAGES = {99--154},
      ISSN = {2096-6075},
   MRCLASS = {11G18 (14G35 20G25)},
  MRNUMBER = {4060001},
MRREVIEWER = {Jinbo Ren},
       DOI = {10.1007/s42543-019-00013-2},
       URL = {https://doi.org/10.1007/s42543-019-00013-2},
}

@article {Ar2,
    AUTHOR = {Arthur, James},
     TITLE = {The trace formula in invariant form},
   JOURNAL = {Ann. of Math. (2)},
  FJOURNAL = {Annals of Mathematics. Second Series},
    VOLUME = {114},
      YEAR = {1981},
    NUMBER = {1},
     PAGES = {1--74},
      ISSN = {0003-486X},
   MRCLASS = {10D40 (22E55)},
  MRNUMBER = {625344},
MRREVIEWER = {Freydoon Shahidi},
       DOI = {10.2307/1971376},
       URL = {https://doi.org/10.2307/1971376},
}

@article{Ar1,
    AUTHOR = {Arthur, James},
     TITLE = {A stable trace formula. {I}. {G}eneral expansions},
   JOURNAL = {J. Inst. Math. Jussieu},
  FJOURNAL = {Journal of the Institute of Mathematics of Jussieu. JIMJ.
              Journal de l'Institut de Math\'{e}matiques de Jussieu},
    VOLUME = {1},
      YEAR = {2002},
    NUMBER = {2},
     PAGES = {175--277},
      ISSN = {1474-7480},
   MRCLASS = {11F72 (11R39 22E55)},
  MRNUMBER = {1954821},
MRREVIEWER = {Volker J. Heiermann},
       DOI = {10.1017/S1474-748002000051},
       URL = {https://doi.org/10.1017/S1474-748002000051},
}

@article{MW,
    AUTHOR = {M\oe glin, Colette and Waldspurger, J.-L.},
     TITLE = {La formule des traces locale tordue},
   JOURNAL = {Mem. Amer. Math. Soc.},
  FJOURNAL = {Memoirs of the American Mathematical Society},
    VOLUME = {251},
      YEAR = {2018},
    NUMBER = {1198},
     PAGES = {v+183},
      ISSN = {0065-9266},
      ISBN = {978-1-4704-2771-9; 978-1-4704-4280-4},
   MRCLASS = {11F72 (11F70)},
  MRNUMBER = {3743601},
MRREVIEWER = {Micha\l  Zydor},
       DOI = {10.1090/memo/1198},
       URL = {https://doi.org/10.1090/memo/1198},
}

@book {LW,
    AUTHOR = {Labesse, Jean-Pierre and Waldspurger, Jean-Loup},
     TITLE = {La formule des traces tordue d'apr\`es le {F}riday {M}orning
              {S}eminar},
    SERIES = {CRM Monograph Series},
    VOLUME = {31},
      NOTE = {With a foreword by Robert Langlands [dual English/French
              text]},
 PUBLISHER = {American Mathematical Society, Providence, RI},
      YEAR = {2013},
     PAGES = {xxvi+234},
      ISBN = {978-0-8218-9441-5},
   MRCLASS = {11F72 (11R56 20G35)},
  MRNUMBER = {3026269},
MRREVIEWER = {Erez M. Lapid},
       DOI = {10.1090/crmm/031},
       URL = {https://doi.org/10.1090/crmm/031},
}

@article {We,
    AUTHOR = {Weselmann, Uwe},
     TITLE = {A twisted topological trace formula for {H}ecke operators and
              liftings from symplectic to general linear groups},
   JOURNAL = {Compos. Math.},
  FJOURNAL = {Compositio Mathematica},
    VOLUME = {148},
      YEAR = {2012},
    NUMBER = {1},
     PAGES = {65--120},
      ISSN = {0010-437X},
   MRCLASS = {11F75 (11E72 11F23 11F46)},
  MRNUMBER = {2881309},
MRREVIEWER = {Stefan K\"{u}hnlein},
       DOI = {10.1112/S0010437X11005641},
       URL = {https://doi.org/10.1112/S0010437X11005641},
}

@article {KoShe,
    AUTHOR = {Kottwitz, Robert E. and Shelstad, Diana},
     TITLE = {Foundations of twisted endoscopy},
   JOURNAL = {Ast\'{e}risque},
  FJOURNAL = {Ast\'{e}risque},
    NUMBER = {255},
      YEAR = {1999},
     PAGES = {vi+190},
      ISSN = {0303-1179},
   MRCLASS = {22E55 (11F70 11R34 22-02 22E50)},
  MRNUMBER = {1687096},
MRREVIEWER = {Volker J. Heiermann},
}

@article{KoShe2,
    author = "Kottwitz, Robert E. and Shelstad, Diana",
    journal = "Preprint",
    title = "On Splitting Invariants and Sign Conventions in Endoscopic Transfer",
    year = "2012",
    note  = "arXiv:1201.5658",
}

@article {KS,
    AUTHOR = {Kret, Arno and Shin, Sug Woo},
     TITLE = {Galois representations for general symplectic groups},
   JOURNAL = {J. Eur. Math. Soc. (JEMS)},
  FJOURNAL = {Journal of the European Mathematical Society (JEMS)},
    VOLUME = {25},
      YEAR = {2023},
    NUMBER = {1},
     PAGES = {75--152},
      ISSN = {1435-9855,1435-9863},
   MRCLASS = {11R39 (11F70 11F80 11G18)},
  MRNUMBER = {4556781},
MRREVIEWER = {Jack\ Shotton},
       DOI = {10.4171/jems/1179},
       URL = {https://doi.org/10.4171/jems/1179},
}

@article{RW,
    author = "Rosner, Mirko and Weissauer, Rainer",
    journal = "Preprint",
    title = "Global liftings between inner forms of $GSp_{4}$",
    year = "2021",
    note  = "arXiv:2103.14715",
}

@article {JS,
    AUTHOR = {Jiang, Dihua and Soudry, David},
     TITLE = {The multiplicity-one theorem for generic automorphic forms of
              {${\rm GSp}(4)$}},
   JOURNAL = {Pacific J. Math.},
  FJOURNAL = {Pacific Journal of Mathematics},
    VOLUME = {229},
      YEAR = {2007},
    NUMBER = {2},
     PAGES = {381--388},
      ISSN = {0030-8730},
   MRCLASS = {11F70 (22E50 22E55)},
  MRNUMBER = {2276516},
MRREVIEWER = {Andre Reznikov},
       DOI = {10.2140/pjm.2007.229.381},
       URL = {https://doi.org/10.2140/pjm.2007.229.381},
}

@article{Vi,
    author = "Viehmann, Eva",
    journal = "Preprint",
    title = "On Newton strata in the $B_{dR}^{+}$-Grassmannian",
    year = "2021",
    note  = "arXiv:2101.07510",
}

@article {Kaz,
    AUTHOR = {Kazhdan, David},
     TITLE = {Cuspidal geometry of {$p$}-adic groups},
   JOURNAL = {J. Analyse Math.},
  FJOURNAL = {Journal d'Analyse Math\'{e}matique},
    VOLUME = {47},
      YEAR = {1986},
     PAGES = {1--36},
      ISSN = {0021-7670},
   MRCLASS = {22E50},
  MRNUMBER = {874042},
MRREVIEWER = {Ernst-Wilhelm Zink},
       DOI = {10.1007/BF02792530},
       URL = {https://doi.org/10.1007/BF02792530},
}

@article {Kott,
    AUTHOR = {Kottwitz, Robert E.},
     TITLE = {Sign changes in harmonic analysis on reductive groups},
   JOURNAL = {Trans. Amer. Math. Soc.},
  FJOURNAL = {Transactions of the American Mathematical Society},
    VOLUME = {278},
      YEAR = {1983},
    NUMBER = {1},
     PAGES = {289--297},
      ISSN = {0002-9947},
   MRCLASS = {22E35 (22E30)},
  MRNUMBER = {697075},
MRREVIEWER = {Martin L. Karel},
       DOI = {10.2307/1999316},
       URL = {https://doi.org/10.2307/1999316},
}

@article {KW,
    AUTHOR = {Hansen, David and Kaletha, Tasho and Weinstein, Jared},
     TITLE = {On the {K}ottwitz conjecture for local shtuka spaces},
   JOURNAL = {Forum Math. Pi},
  FJOURNAL = {Forum of Mathematics. Pi},
    VOLUME = {10},
      YEAR = {2022},
     PAGES = {Paper No. e13, 79},
      ISSN = {2050-5086},
   MRCLASS = {14G45 (11S37)},
  MRNUMBER = {4430954},
MRREVIEWER = {Kazuhiro\ Ito},
       DOI = {10.1017/fmp.2022.7},
       URL = {https://doi.org/10.1017/fmp.2022.7},
}

@article{BMNH,
    author = "Bertoloni Meli, Alexander and Hamann, Linus and Nguyen, Kieu-Hieu",
    journal = "Preprint",
    title = "Compatibility of Fargues-Scholze correspondence for unitary groups",
    year = "2022",
    note  = "arXiv:2207.13193",
}

@article {Cho,
    AUTHOR = {Choiy, Kwangho},
     TITLE = {The local {L}anglands conjecture for the {$p$}-adic inner form
              of {$\rm Sp_4$}},
   JOURNAL = {Int. Math. Res. Not. IMRN},
  FJOURNAL = {International Mathematics Research Notices. IMRN},
      YEAR = {2017},
    NUMBER = {6},
     PAGES = {1830--1889},
      ISSN = {1073-7928},
   MRCLASS = {22E50 (11S37)},
  MRNUMBER = {3658185},
MRREVIEWER = {B. Sury},
       DOI = {10.1093/imrn/rnw043},
       URL = {https://doi.org/10.1093/imrn/rnw043},
}

@article {So,
    AUTHOR = {Sorensen, Claus M.},
     TITLE = {Galois representations attached to {H}ilbert-{S}iegel modular forms},
   JOURNAL = {Doc. Math.},
  FJOURNAL = {Documenta Mathematica},
    VOLUME = {15},
      YEAR = {2010},
     PAGES = {623--670},
      ISSN = {1431-0635},
   MRCLASS = {11F80 (11F33 11F46 11F70 11G18 22E55)},
  MRNUMBER = {2735984},
MRREVIEWER = {Anton Deitmar},
       DOI = {10.1007/s00031-010-9092-7},
       URL = {https://doi.org/10.1007/s00031-010-9092-7},
}

@article {GT2,
    AUTHOR = {Gan, Wee Teck and Tantono, Welly},
     TITLE = {The local {L}anglands conjecture for {$\rm GSp(4)$}, {II}:
              {T}he case of inner forms},
   JOURNAL = {Amer. J. Math.},
  FJOURNAL = {American Journal of Mathematics},
    VOLUME = {136},
      YEAR = {2014},
    NUMBER = {3},
     PAGES = {761--805},
      ISSN = {0002-9327},
   MRCLASS = {22E50 (11F27 11F70 11S37)},
  MRNUMBER = {3214276},
MRREVIEWER = {Volker J. Heiermann},
       DOI = {10.1353/ajm.2014.0016},
       URL = {https://doi.org/10.1353/ajm.2014.0016},
}

@incollection {Ar,
    AUTHOR = {Arthur, James},
     TITLE = {Automorphic representations of {${\rm GSp(4)}$}},
 BOOKTITLE = {Contributions to automorphic forms, geometry, and number
              theory},
     PAGES = {65--81},
 PUBLISHER = {Johns Hopkins Univ. Press, Baltimore, MD},
      YEAR = {2004},
   MRCLASS = {11F70 (22E50 22E55)},
  MRNUMBER = {2058604},
MRREVIEWER = {Jeff Hakim},
}

@article {KST,
    AUTHOR = {Kim, Ju-Lee and Shin, Sug Woo and Templier, Nicolas},
     TITLE = {Asymptotic behavior of supercuspidal representations and
              {S}ato-{T}ate equidistribution for families},
   JOURNAL = {Adv. Math.},
  FJOURNAL = {Advances in Mathematics},
    VOLUME = {362},
      YEAR = {2020},
     PAGES = {106955, 57},
      ISSN = {0001-8708},
   MRCLASS = {20G25 (11F55 22E35)},
  MRNUMBER = {4046074},
MRREVIEWER = {Feng Wei},
       DOI = {10.1016/j.aim.2019.106955},
       URL = {https://doi.org/10.1016/j.aim.2019.106955},
}

@article {VZ,
    AUTHOR = {Vogan, Jr., David A. and Zuckerman, Gregg J.},
     TITLE = {Unitary representations with nonzero cohomology},
   JOURNAL = {Compositio Math.},
  FJOURNAL = {Compositio Mathematica},
    VOLUME = {53},
      YEAR = {1984},
    NUMBER = {1},
     PAGES = {51--90},
      ISSN = {0010-437X},
   MRCLASS = {22E47 (11F67 11F70 32N10)},
  MRNUMBER = {762307},
MRREVIEWER = {Dan Barbasch},
       URL = {http://www.numdam.org/item?id=CM_1984__53_1_51_0},
}

@article {GeTa,
    AUTHOR = {Gee, Toby and Ta\"{i}bi, Olivier},
     TITLE = {Arthur's multiplicity formula for {${\bf GSp}_4$} and
              restriction to {${\bf Sp}_4$}},
   JOURNAL = {J. \'{E}c. polytech. Math.},
  FJOURNAL = {Journal de l'\'{E}cole polytechnique. Math\'{e}matiques},
    VOLUME = {6},
      YEAR = {2019},
     PAGES = {469--535},
      ISSN = {2429-7100},
   MRCLASS = {11F72 (11F46 11F55)},
  MRNUMBER = {3991897},
MRREVIEWER = {Han Wu},
       DOI = {10.5802/jep.99},
       URL = {https://doi.org/10.5802/jep.99},
}

@article {She,
    AUTHOR = {Shen, Xu},
     TITLE = {On some generalized {R}apoport-{Z}ink spaces},
   JOURNAL = {Canad. J. Math.},
  FJOURNAL = {Canadian Journal of Mathematics. Journal Canadien de
              Math\'{e}matiques},
    VOLUME = {72},
      YEAR = {2020},
    NUMBER = {5},
     PAGES = {1111--1187},
      ISSN = {0008-414X,1496-4279},
   MRCLASS = {11G18 (14G35)},
  MRNUMBER = {4152538},
MRREVIEWER = {Alan\ Koch},
       DOI = {10.4153/s0008414x19000269},
       URL = {https://doi.org/10.4153/s0008414x19000269},
}

@article {He,
    AUTHOR = {Henniart, Guy},
     TITLE = {Une preuve simple des conjectures de {L}anglands pour {${\rm
              GL}(n)$} sur un corps {$p$}-adique},
   JOURNAL = {Invent. Math.},
  FJOURNAL = {Inventiones Mathematicae},
    VOLUME = {139},
      YEAR = {2000},
    NUMBER = {2},
     PAGES = {439--455},
      ISSN = {0020-9910},
   MRCLASS = {11F70 (11R39 11S37 22E50 22E55)},
  MRNUMBER = {1738446},
MRREVIEWER = {Dihua Jiang},
       DOI = {10.1007/s002220050012},
       URL = {https://doi.org/10.1007/s002220050012},
}

@inproceedings {HPS,
    AUTHOR = {Howe, R. and Piatetski-Shapiro, I. I.},
     TITLE = {A counterexample to the ``generalized {R}amanujan conjecture''
              for (quasi-) split groups},
 BOOKTITLE = {Automorphic forms, representations and {$L$}-functions
              ({P}roc. {S}ympos. {P}ure {M}ath., {O}regon {S}tate {U}niv.,
              {C}orvallis, {O}re., 1977), {P}art 1},
    SERIES = {Proc. Sympos. Pure Math., XXXIII},
     PAGES = {315--322},
 PUBLISHER = {Amer. Math. Soc., Providence, R.I.},
      YEAR = {1979},
   MRCLASS = {22E55 (10D40 20G05)},
  MRNUMBER = {546605},
MRREVIEWER = {A. B. Venkov},
}

@article {Ku,
    AUTHOR = {Kurokawa, Nobushige},
     TITLE = {Examples of eigenvalues of {H}ecke operators on {S}iegel cusp
              forms of degree two},
   JOURNAL = {Invent. Math.},
  FJOURNAL = {Inventiones Mathematicae},
    VOLUME = {49},
      YEAR = {1978},
    NUMBER = {2},
     PAGES = {149--165},
      ISSN = {0020-9910},
   MRCLASS = {10D20},
  MRNUMBER = {511188},
MRREVIEWER = {Rita Hall},
       DOI = {10.1007/BF01403084},
       URL = {https://doi.org/10.1007/BF01403084},
}

@article {VL,
    AUTHOR = {Lafforgue, Vincent},
     TITLE = {Chtoucas pour les groupes r\'{e}ductifs et param\'{e}trisation de
              {L}anglands globale},
   JOURNAL = {J. Amer. Math. Soc.},
  FJOURNAL = {Journal of the American Mathematical Society},
    VOLUME = {31},
      YEAR = {2018},
    NUMBER = {3},
     PAGES = {719--891},
      ISSN = {0894-0347},
   MRCLASS = {14G35 (11F70 14D24 14H60)},
  MRNUMBER = {3787407},
MRREVIEWER = {James H. Stankewicz},
       DOI = {10.1090/jams/897},
       URL = {https://doi.org/10.1090/jams/897},
}

@book {HT,
    AUTHOR = {Harris, Michael and Taylor, Richard},
     TITLE = {The geometry and cohomology of some simple {S}himura
              varieties},
    SERIES = {Annals of Mathematics Studies},
    VOLUME = {151},
      NOTE = {With an appendix by Vladimir G. Berkovich},
 PUBLISHER = {Princeton University Press, Princeton, NJ},
      YEAR = {2001},
     PAGES = {viii+276},
      ISBN = {0-691-09090-4},
   MRCLASS = {11G18 (11F70 11S37 14G35 22E45)},
  MRNUMBER = {1876802},
MRREVIEWER = {James Milne},
}

@article {CG,
    AUTHOR = {Chan, Ping-Shun and Gan, Wee Teck},
     TITLE = {The local {L}anglands conjecture for {$\rm GSp(4)$} {III}:
              {S}tability and twisted endoscopy},
   JOURNAL = {J. Number Theory},
  FJOURNAL = {Journal of Number Theory},
    VOLUME = {146},
      YEAR = {2015},
     PAGES = {69--133},
      ISSN = {0022-314X},
   MRCLASS = {22E50 (11F70)},
  MRNUMBER = {3267112},
MRREVIEWER = {Dubravka Ban},
       DOI = {10.1016/j.jnt.2013.07.009},
       URL = {https://doi.org/10.1016/j.jnt.2013.07.009},
}

@article {RV,
    AUTHOR = {Rapoport, Michael and Viehmann, Eva},
     TITLE = {Towards a theory of local {S}himura varieties},
   JOURNAL = {M\"{u}nster J. Math.},
  FJOURNAL = {M\"{u}nster Journal of Mathematics},
    VOLUME = {7},
      YEAR = {2014},
    NUMBER = {1},
     PAGES = {273--326},
      ISSN = {1867-5778},
   MRCLASS = {11G18 (14G35)},
  MRNUMBER = {3271247},
}

@article {SW1,
    AUTHOR = {Scholze, Peter and Weinstein, Jared},
     TITLE = {Moduli of {$p$}-divisible groups},
   JOURNAL = {Camb. J. Math.},
  FJOURNAL = {Cambridge Journal of Mathematics},
    VOLUME = {1},
      YEAR = {2013},
    NUMBER = {2},
     PAGES = {145--237},
      ISSN = {2168-0930},
   MRCLASS = {14L05 (11G25 14C30 14D20)},
  MRNUMBER = {3272049},
MRREVIEWER = {Rui Miguel Saramago},
       DOI = {10.4310/CJM.2013.v1.n2.a1},
       URL = {https://doi.org/10.4310/CJM.2013.v1.n2.a1},
}

@article{Kosh,
    author = "T. Koshikawa",
    journal = "Preprint",
    title = "On Eichler-Shimura relations for local Shimura varieties",
    year = "2021",
    note  = "arXiv:2106.10603",
}

@book {SW2,
    AUTHOR = {Scholze, P. and Weinstein, J.},
     TITLE = {Berkeley lectures on $p$-adic Geometry},
    SERIES = {Annals of Mathematics Studies},
    VOLUME = {207},
 PUBLISHER = {Princeton University Press},
      YEAR = {2020},
}

@article{Han,
    author =       {Hansen, D.},
    title =        "On the supercuspidal cohomology of basic local Shimura varieties",
    year =         "2020",
    Note = "J. Reine Angew. Math, to appear"
}

@article{Han1,
    author =       {Hansen, D.},
    title =        "Period Morphisms and variations of $p$-adic Hodge structures",
    journal =       "Draft", 
    year =         "2016",
    Note = "Available at home page of author"
}

@article {FGV,
    AUTHOR = {Frenkel, E. and Gaitsgory, D. and Vilonen, K.},
     TITLE = {On the geometric {L}anglands conjecture},
   JOURNAL = {J. Amer. Math. Soc.},
  FJOURNAL = {Journal of the American Mathematical Society},
    VOLUME = {15},
      YEAR = {2002},
    NUMBER = {2},
     PAGES = {367--417},
      ISSN = {0894-0347},
   MRCLASS = {11R39 (11F70 14D20 14H60 22E55)},
  MRNUMBER = {1887638},
MRREVIEWER = {Igor Yu. Potemine},
       DOI = {10.1090/S0894-0347-01-00388-5},
       URL = {https://doi.org/10.1090/S0894-0347-01-00388-5},
}

@article {Gait,
    AUTHOR = {Gaitsgory, D.},
     TITLE = {On a vanishing conjecture appearing in the geometric
              {L}anglands correspondence},
   JOURNAL = {Ann. of Math. (2)},
  FJOURNAL = {Annals of Mathematics. Second Series},
    VOLUME = {160},
      YEAR = {2004},
    NUMBER = {2},
     PAGES = {617--682},
      ISSN = {0003-486X},
   MRCLASS = {11R39 (11F70 14D20 22E55)},
  MRNUMBER = {2123934},
MRREVIEWER = {Peter Fiebig},
       DOI = {10.4007/annals.2004.160.617},
       URL = {https://doi.org/10.4007/annals.2004.160.617},
}

@article{AL,
    author =       {Ansch\"{u}tz, Johannes and Le Bras, Arthur-C\'esar},
    title =        "Averaging Functors in Fargues' Program for $GL_{n}$",
    journal = "Preprint",
    year =         "2021",
    Note = "arXiv:2104.04701"
}

@article {GT3,
    AUTHOR = {Gan, Wee Teck and Takeda, Shuichiro},
     TITLE = {The local {L}anglands conjecture for {S}p(4)},
   JOURNAL = {Int. Math. Res. Not. IMRN},
  FJOURNAL = {International Mathematics Research Notices. IMRN},
      YEAR = {2010},
    NUMBER = {15},
     PAGES = {2987--3038},
      ISSN = {1073-7928},
   MRCLASS = {22E50 (20G25 22E35)},
  MRNUMBER = {2673717},
MRREVIEWER = {B. Sury},
       DOI = {10.1093/imrn/rnp203},
       URL = {https://doi.org/10.1093/imrn/rnp203},
}

@article{IM,
    author =       "Ito, Tetsushi and Mieda, Yoichi",
    title =        "Local Saito-Kurokawa $A$-packets and $\ell$-adic cohomology of Rapoport-Zink tower for $GSp_{4}$",
    journal =       "Preparation", 
    year =         "2021",
}

@article{Ng,
    author =       "Nguyen, Kieu Hieu",
    title =        "Un cas PEL de la conjecture de Kottwitz",
    journal =       "Preprint", 
    year =         "2019",
    Note = "arXiv:1903.11505"
}

@book {RenReprp,
    AUTHOR = {Renard, David},
     TITLE = {Repr\'esentations des groupes r\'eductifs {$p$}-adiques},
    SERIES = {Cours Sp\'ecialis\'es},
    VOLUME = {17},
 PUBLISHER = {Soci\'et\'e Math\'ematique de France, Paris},
      YEAR = {2010},
     PAGES = {vi+332},
      ISBN = {978-2-85629-278-5},
   MRCLASS = {22E50 (20G25)},
  MRNUMBER = {2567785},
MRREVIEWER = {Anne-Marie H. Aubert},
}

@inproceedings {Bor,
    AUTHOR = {Borel, A.},
     TITLE = {Automorphic {$L$}-functions},
 BOOKTITLE = {Automorphic forms, representations and {$L$}-functions
              ({P}roc. {S}ympos. {P}ure {M}ath., {O}regon {S}tate {U}niv.,
              {C}orvallis, {O}re., 1977), {P}art 2},
    SERIES = {Proc. Sympos. Pure Math., XXXIII},
     PAGES = {27--61},
 PUBLISHER = {Amer. Math. Soc., Providence, R.I.},
      YEAR = {1979},
   MRCLASS = {10D40 (12A67 22E50)},
  MRNUMBER = {546608},
MRREVIEWER = {Yasuhiro Asoo},
}

@article {SSt,
    AUTHOR = {Schneider, Peter and Stuhler, Ulrich},
     TITLE = {Representation theory and sheaves on the {B}ruhat-{T}its
              building},
   JOURNAL = {Inst. Hautes \'{E}tudes Sci. Publ. Math.},
  FJOURNAL = {Institut des Hautes \'{E}tudes Scientifiques. Publications
              Math\'{e}matiques},
    NUMBER = {85},
      YEAR = {1997},
     PAGES = {97--191},
      ISSN = {0073-8301},
   MRCLASS = {22E50 (11F70 20G25)},
  MRNUMBER = {1471867},
MRREVIEWER = {Ernst-Wilhelm Zink},
       URL = {http://www.numdam.org/item?id=PMIHES_1997__85__97_0},
}

@article{DH,
    author =       {Dat, Jean-Fran\c{c}ois and Helm, David and Kurinczuk, Robert and Moss, Gilbert},
    title =        {Moduli of Langlands Parameters},
    journal =       "Preprint", 
    year =         {2020},
    Note = {arXiv:2009.06708}
}

@article{Zhu1,
    author =       "X. Zhu",
    title =        "Coherent Sheaves on the stack of Langlands Parameters",
    journal =       "Preprint", 
    year =         "2020",
    Note = "arXiv:2008.02998 "
}

@article{I,
    author =       "Ito, Tetsushi",
    title =        "Supercuspidal representations in the cohomology of the Rapoport-Zink space for the unitary group in three variables",
    year =         "2013",
    Note = "Available \hyperlink{http://www.kurims.kyoto-u.ac.jp/~kyodo/kokyuroku/contents/pdf/1871-14.pdf}{here}"
}

@book {NT,
    AUTHOR = {Townsend, Nelson J.},
     TITLE = {Properties of {G}amma factors for $GSp(4) \times GL(r)$ with r =
              1, 2},
      NOTE = {Thesis (Ph.D.)--University of California, San Diego},
 PUBLISHER = {ProQuest LLC, Ann Arbor, MI},
      YEAR = {2013},
     PAGES = {69},
      ISBN = {978-1303-62050-8},
   MRCLASS = {Thesis},
  MRNUMBER = {3211516},
       URL =
              {http://gateway.proquest.com/openurl?url_ver=Z39.88-2004&rft_val_fmt=info:ofi/fmt:kev:mtx:dissertation&res_dat=xri:pqm&rft_dat=xri:pqdiss:3605550},
}

@article {R,
    AUTHOR = {Rapoport, Michael},
     TITLE = {Appendix to On the {$p$}-adic cohomology of the {L}ubin-{T}ate tower},
   JOURNAL = {Ann. Sci. \'{E}c. Norm. Sup\'{e}r. (4)},
  FJOURNAL = {Annales Scientifiques de l'\'{E}cole Normale Sup\'{e}rieure. Quatri\`eme
              S\'{e}rie},
    VOLUME = {51},
      YEAR = {2018},
    NUMBER = {4},
     PAGES = {811--863},
      ISSN = {0012-9593},
   MRCLASS = {14F30 (11F80 11S37)},
  MRNUMBER = {3861564},
       DOI = {10.24033/asens.2367},
       URL = {https://doi.org/10.24033/asens.2367},
}

@article {Gross,
    AUTHOR = {Gross, Benedict H.},
     TITLE = {Algebraic modular forms},
   JOURNAL = {Israel J. Math.},
  FJOURNAL = {Israel Journal of Mathematics},
    VOLUME = {113},
      YEAR = {1999},
     PAGES = {61--93},
      ISSN = {0021-2172},
   MRCLASS = {11F55 (11F80 20G30 22E55)},
  MRNUMBER = {1729443},
MRREVIEWER = {Stefan K\"{u}hnlein},
       DOI = {10.1007/BF02780173},
       URL = {https://doi.org/10.1007/BF02780173},
}

@article {Ko,
    AUTHOR = {Kottwitz, Robert E.},
     TITLE = {Isocrystals with additional structure},
   JOURNAL = {Compositio Math.},
  FJOURNAL = {Compositio Mathematica},
    VOLUME = {56},
      YEAR = {1985},
    NUMBER = {2},
     PAGES = {201--220},
      ISSN = {0010-437X},
   MRCLASS = {14L25 (14F30 14L15 20G25)},
  MRNUMBER = {809866},
MRREVIEWER = {K. F. Lai},
       URL = {http://www.numdam.org/item?id=CM_1985__56_2_201_0},
}

@article {MV,
    AUTHOR = {Mirkovi\'{c}, I. and Vilonen, K.},
     TITLE = {Geometric {L}anglands duality and representations of algebraic
              groups over commutative rings},
   JOURNAL = {Ann. of Math. (2)},
  FJOURNAL = {Annals of Mathematics. Second Series},
    VOLUME = {166},
      YEAR = {2007},
    NUMBER = {1},
     PAGES = {95--143},
      ISSN = {0003-486X},
   MRCLASS = {22E55 (11R39 20G05)},
  MRNUMBER = {2342692},
MRREVIEWER = {Peter Fiebig},
       DOI = {10.4007/annals.2007.166.95},
       URL = {https://doi.org/10.4007/annals.2007.166.95},
}

@article {K,
    AUTHOR = {Kedlaya, Kiran S.},
     TITLE = {Noetherian properties of {F}argues-{F}ontaine curves},
   JOURNAL = {Int. Math. Res. Not. IMRN},
  FJOURNAL = {International Mathematics Research Notices. IMRN},
      YEAR = {2016},
    NUMBER = {8},
     PAGES = {2544--2567},
      ISSN = {1073-7928},
   MRCLASS = {14F30 (13F35 14G22)},
  MRNUMBER = {3519123},
MRREVIEWER = {Liang Xiao},
       DOI = {10.1093/imrn/rnv227},
       URL = {https://doi.org/10.1093/imrn/rnv227},
}

@article {RR,
    AUTHOR = {Rapoport, M. and Richartz, M.},
     TITLE = {On the classification and specialization of {$F$}-isocrystals
              with additional structure},
   JOURNAL = {Compositio Math.},
  FJOURNAL = {Compositio Mathematica},
    VOLUME = {103},
      YEAR = {1996},
    NUMBER = {2},
     PAGES = {153--181},
      ISSN = {0010-437X},
   MRCLASS = {14F30 (22E50)},
  MRNUMBER = {1411570},
MRREVIEWER = {Abdellah Mokrane},
       URL = {http://www.numdam.org/item?id=CM_1996__103_2_153_0},
}

@article {Kott1,
    AUTHOR = {Kottwitz, Robert E.},
     TITLE = {Isocrystals with additional structure. {II}},
   JOURNAL = {Compositio Math.},
  FJOURNAL = {Compositio Mathematica},
    VOLUME = {109},
      YEAR = {1997},
    NUMBER = {3},
     PAGES = {255--339},
      ISSN = {0010-437X},
   MRCLASS = {20G25 (11S25 14F30 14L05)},
  MRNUMBER = {1485921},
MRREVIEWER = {Guy Rousseau},
       DOI = {10.1023/A:1000102604688},
       URL = {https://doi.org/10.1023/A:1000102604688},
}

@article{HamLeeTorsVan,
    author = {Hamann, Linus and Lee, Si Ying},
    journal = "arXiv",
    title =  "Torsion Vanishing for some Shimura Varieties",
    year = "2023",
    note = "arXiv:2309.08705"
}

@article{MingjiaPolI,
    author = {Daniels, Patrick and van Hoften Pol and Kim, Dongryul and Zhang, Mingjia},
    journal = "arXiv",
    title =  "Igusa Stacks and the Cohomology of Shimura varieties",
    year = "2024",
    note = "arXiv:2408.01348"
}

@article {BailyBorel,
    AUTHOR = {Baily, Jr., W. L. and Borel, A.},
     TITLE = {Compactification of arithmetic quotients of bounded symmetric
              domains},
   JOURNAL = {Ann. of Math. (2)},
  FJOURNAL = {Annals of Mathematics. Second Series},
    VOLUME = {84},
      YEAR = {1966},
     PAGES = {442--528},
      ISSN = {0003-486X},
   MRCLASS = {32.65},
  MRNUMBER = {216035},
MRREVIEWER = {A.\ Kor\'anyi},
       DOI = {10.2307/1970457},
       URL = {https://doi.org/10.2307/1970457},
}

@article{Mann2022NuclearSheaves,
	title={The 6-Functor formalism for $\mathbb{Z}_{\ell}$ and $\mathbb{Q}_{\ell}$-sheaves on Diamonds},
	author={Mann, Lucas},
	journal={arXiv preprint arXiv:2209.08135},
	year={2022}
}

@misc{MingjiaPolII,
      title={Igusa Stacks and the Cohomology of Shimura Varieties II}, 
      author={Patrick Daniels and Pol van Hoften and Dongryul Kim and Mingjia Zhang},
      year={2026},
      eprint={2603.24921},
      archivePrefix={arXiv},
      primaryClass={math.NT},
      url={https://arxiv.org/abs/2603.24921}, 
}

@article {Shin,
    AUTHOR = {Shin, Sug Woo},
     TITLE = {Automorphic {P}lancherel density theorem},
   JOURNAL = {Israel J. Math.},
  FJOURNAL = {Israel Journal of Mathematics},
    VOLUME = {192},
      YEAR = {2012},
    NUMBER = {1},
     PAGES = {83--120},
      ISSN = {0021-2172},
   MRCLASS = {22E35 (43A85)},
  MRNUMBER = {3004076},
MRREVIEWER = {Valeri\u{\i} Vladimirovich Volchkov},
       DOI = {10.1007/s11856-012-0018-z},
       URL = {https://doi.org/10.1007/s11856-012-0018-z},
}

@article{IG,
    author = {Gaisin, I. and Imai, N.},
    title = "Non-semi-stable loci in Hecke stacks and Fargues' Conjecture",
    journal = "Preprint",
    year = "2016",
    note = "arXiv:1608.07446"
}

@article {Boy1,
    AUTHOR = {Boyer, P.},
     TITLE = {Mauvaise r\'{e}duction des vari\'{e}t\'{e}s de {D}rinfeld et
              correspondance de {L}anglands locale},
   JOURNAL = {Invent. Math.},
  FJOURNAL = {Inventiones Mathematicae},
    VOLUME = {138},
      YEAR = {1999},
    NUMBER = {3},
     PAGES = {573--629},
      ISSN = {0020-9910},
   MRCLASS = {11S37 (11G09 11R39)},
  MRNUMBER = {1719811},
MRREVIEWER = {Laurent Lafforgue},
       DOI = {10.1007/s002220050354},
       URL = {https://doi.org/10.1007/s002220050354},
}

@article {Boy,
    AUTHOR = {Boyer, Pascal},
     TITLE = {Monodromie du faisceau pervers des cycles \'{e}vanescents de
              quelques vari\'{e}t\'{e}s de {S}himura simples},
   JOURNAL = {Invent. Math.},
  FJOURNAL = {Inventiones Mathematicae},
    VOLUME = {177},
      YEAR = {2009},
    NUMBER = {2},
     PAGES = {239--280},
      ISSN = {0020-9910},
   MRCLASS = {14G35 (11G18 11G35 11G45 14G22 14L05)},
  MRNUMBER = {2511742},
MRREVIEWER = {Ulrich G\"ortz},
       DOI = {10.1007/s00222-009-0183-9},
       URL = {https://doi.org/10.1007/s00222-009-0183-9},
}

@article{FS,
    author = {Fargues, L. and Scholze, P.},
    title = "Geometrization of the local Langlands Correspondence",
    journal = "Preprint",
    year = "2021",
    note = "arXiv:2102.13459 "
}

@article{Zou,
    author = {Zou, K.},
    title = "The Categorical Form of Fargues' conjecture for Tori",
    journal = "Preprint",
    year = "2022",
    note = "arXiv:2202.13238 "
}

@book {BW,
    AUTHOR = {Borel, A. and Wallach, N.},
     TITLE = {Continuous cohomology, discrete subgroups, and representations
              of reductive groups},
    SERIES = {Mathematical Surveys and Monographs},
    VOLUME = {67},
   EDITION = {Second},
 PUBLISHER = {American Mathematical Society, Providence, RI},
      YEAR = {2000},
     PAGES = {xviii+260},
      ISBN = {0-8218-0851-6},
   MRCLASS = {22E41 (22-02 22E40 22E45 57T15)},
  MRNUMBER = {1721403},
MRREVIEWER = {F. E. A. Johnson},
       DOI = {10.1090/surv/067},
       URL = {https://doi.org/10.1090/surv/067},
}

@article{Mez13,
  author  = {Mezo, Paul},
  title   = {Character identities in the twisted endoscopy of real reductive groups},
  journal = {Mem. Amer. Math. Soc.},
  volume  = {222},
  number  = {1042},
  year    = {2013}
}

@article{She08,
  author  = {Shelstad, D.},
  title   = {Tempered endoscopy for real groups. {III}. {I}nversion of transfer and {$L$}-packet structure},
  journal = {Represent. Theory},
  volume  = {12},
  year    = {2008},
  pages   = {369--402}
}

@book{MW16,
  author    = {M{\oe}glin, Colette and Waldspurger, Jean-Loup},
  title     = {Stabilisation de la formule des traces tordue},
  series    = {Progress in Mathematics},
  volume    = {316--317},
  publisher = {Birkh{\"a}user/Springer, Cham},
  year      = {2016}
}

@misc{LiHue,
  url = {https://arxiv.org/abs/2208.09962},
  author = {Li-Huerta, Daniel Siyuan},
  YEAR = {2022},
  title = {The plectic conjecture for local fields},
  publisher = {arXiv},
  }

@incollection {LR,
    AUTHOR = {Lapid, Erez M. and Rallis, Stephen},
     TITLE = {On the local factors of representations of classical groups},
 BOOKTITLE = {Automorphic representations, {$L$}-functions and applications:
              progress and prospects},
    SERIES = {Ohio State Univ. Math. Res. Inst. Publ.},
    VOLUME = {11},
     PAGES = {309--359},
 PUBLISHER = {de Gruyter, Berlin},
      YEAR = {2005},
   MRCLASS = {11F70 (22E45)},
  MRNUMBER = {2192828},
MRREVIEWER = {Mahdi Asgari},
       DOI = {10.1515/9783110892703.309},
       URL = {https://doi.org/10.1515/9783110892703.309},
}

@incollection {Ka,
    AUTHOR = {Kaletha, Tasho},
     TITLE = {The local {L}anglands conjectures for non-quasi-split groups},
 BOOKTITLE = {Families of automorphic forms and the trace formula},
    SERIES = {Simons Symp.},
     PAGES = {217--257},
 PUBLISHER = {Springer, [Cham]},
      YEAR = {2016},
   MRCLASS = {11S37},
  MRNUMBER = {3675168},
MRREVIEWER = {Victor Alexandru},
}

@article {Sh,
    AUTHOR = {Shahidi, Freydoon},
     TITLE = {A proof of {L}anglands' conjecture on {P}lancherel measures;
              complementary series for {$p$}-adic groups},
   JOURNAL = {Ann. of Math. (2)},
  FJOURNAL = {Annals of Mathematics. Second Series},
    VOLUME = {132},
      YEAR = {1990},
    NUMBER = {2},
     PAGES = {273--330},
      ISSN = {0003-486X},
   MRCLASS = {11R39 (11F70 11S37 22E35 22E55)},
  MRNUMBER = {1070599},
MRREVIEWER = {Stephen Gelbart},
       DOI = {10.2307/1971524},
       URL = {https://doi.org/10.2307/1971524},
}
\end{document}